\documentclass{article}

\usepackage{graphicx}
\graphicspath{ {./images/} }

\usepackage{amsmath, amsthm, amssymb}
\usepackage{wasysym}

\usepackage{enumitem}

\usepackage[pdftex,plainpages=false,hypertexnames=false,pdfpagelabels]{hyperref}
\newcommand{\arxiv}[1]{\href{http://arxiv.org/abs/#1}{\tt arXiv:\nolinkurl{#1}}}
\newcommand{\arXiv}[1]{\href{http://arxiv.org/abs/#1}{\tt arXiv:\nolinkurl{#1}}}
\newcommand{\doi}[1]{\href{http://dx.doi.org/#1}{{\tt DOI:#1}}}

\newcommand{\googlebooks}[1]{(preview at \href{http://books.google.com/books?id=#1}{google books})}

\usepackage{xcolor}
\definecolor{medium-blue}{rgb}{0,0,.8}
\definecolor{DarkGreen}{RGB}{0,150,0}
\hypersetup{
   colorlinks, linkcolor={purple},
   citecolor={medium-blue}, urlcolor={medium-blue}
}

\usepackage{fullpage}
\theoremstyle{plain}
\newtheorem{thm}{Theorem}[section]
\newtheorem*{thm*}{Theorem}
\newtheorem{thmalpha}{Theorem}

\newtheorem{cor}[thm]{Corollary}
\newtheorem{coralpha}[thmalpha]{Corollary}
\newtheorem*{cor*}{Corollary}

\newtheorem*{conj*}{Conjecture}
\newtheorem{hypothesis}[thm]{Hypothesis}
\newtheorem*{hypothesis*}{Hypothesis}
\newtheorem*{philosophy*}{Philosophy}
\newtheorem{lem}[thm]{Lemma}

\newtheorem{prop}[thm]{Proposition}

\newtheorem*{quest*}{Question}
\newtheorem*{claim*}{Claim}

\theoremstyle{definition}
\newtheorem{defn}[thm]{Definition}
\newtheorem{construction}[thm]{Construction}

\newtheorem{nota}[thm]{Notation}

\newtheorem{ex}[thm]{Example}
\newtheorem{sub-ex}[thm]{Sub-Example}
\newtheorem{rem}[thm]{Remark}
\newtheorem*{rem*}{Remark}
\newtheorem{remark}[thm]{Remark}

\newtheorem{warn}[thm]{Warning}

\def\semicolon{;}
\def\applytolist#1{
    \expandafter\def\csname multi#1\endcsname##1{
        \def\multiack{##1}\ifx\multiack\semicolon
            \def\next{\relax}
        \else
            \csname #1\endcsname{##1}
            \def\next{\csname multi#1\endcsname}
        \fi
        \next}
    \csname multi#1\endcsname}

\def\calc#1{\expandafter\def\csname c#1\endcsname{{\mathcal #1}}}
\applytolist{calc}QWERTYUIOPLKJHGFDSAZXCVBNM;
\def\bbc#1{\expandafter\def\csname bb#1\endcsname{{\mathbb #1}}}
\applytolist{bbc}QWERTYUIOPLKJHGFDSAZXCVBNM;
\def\bfc#1{\expandafter\def\csname bf#1\endcsname{{\mathbf #1}}}
\applytolist{bfc}QWERTYUIOPLKJHGFDSAZXCVBNM;
\def\sfc#1{\expandafter\def\csname s#1\endcsname{{\sf #1}}}
\applytolist{sfc}QWERTYUIOPLKJHGFDSAZXCVBNM;
\def\fc#1{\expandafter\def\csname f#1\endcsname{{\mathfrak #1}}}
\applytolist{fc}QWERTYUIOPLKJHGFDSAZXCVBNM;
\def\rm#1{\expandafter\def\csname rm#1\endcsname{{\mathrm #1}}}
\applytolist{rm}QWERTYUIOPLKJHGFDSAZXCVBNM;

\newcommand{\bfa}{\mathbf{a}}
\newcommand{\bfb}{\mathbf{b}}

\newcommand{\MR}[1]{}

\usepackage{tikz}
\usetikzlibrary{cd}
\usetikzlibrary{decorations,decorations.pathreplacing,decorations.markings}
\tikzstyle{mid>}=[decoration={markings, mark=at position 0.5 with {\arrow{>}}}, postaction={decorate}]
\tikzstyle{mid<}=[decoration={markings, mark=at position 0.5 with {\arrow{<}}}, postaction={decorate}]
\tikzstyle{upper>}=[decoration={markings, mark=at position 0.8 with {\arrow{>}}}, postaction={decorate}]
\tikzstyle{upper<}=[decoration={markings, mark=at position 0.8 with {\arrow{<}}}, postaction={decorate}]
\tikzstyle{lower>}=[decoration={markings, mark=at position 0.2 with {\arrow{>}}}, postaction={decorate}]
\tikzstyle{lower<}=[decoration={markings, mark=at position 0.2 with {\arrow{<}}}, postaction={decorate}]

\tikzset{
    tripleline/.style args={[#1] in [#2] in [#3]}{
        #1,preaction={preaction={draw,#3},draw,#2}
    }
}
\tikzstyle{triple}=[tripleline={[line width=.15mm,black] in
      [line width=.7mm,white] in
      [line width=1mm,black]}] 
\tikzset{
    quadrupleline/.style args={[#1] in [#2] in [#3] in [#4]}{
        #1,preaction={preaction={preaction={draw,#4},draw,#3}, draw,#2}
    }
}
\tikzstyle{quadruple}=[quadrupleline={[line width=.3mm,white] in
      [line width=.6mm,black] in
      [line width=1.2mm,white] in
      [line width=1.5mm,black]}]

\DeclareMathOperator{\Inv}{Inv}
\DeclareMathOperator{\Aut}{Aut}

\DeclareMathOperator{\coev}{coev}
\DeclareMathOperator{\ev}{ev}
\DeclareMathOperator{\End}{End}

\DeclareMathOperator{\Hom}{Hom}
\DeclareMathOperator{\id}{id}

\DeclareMathOperator{\op}{op}
\DeclareMathOperator{\Ob}{Ob}

\DeclareMathOperator{\pt}{pt}
\DeclareMathOperator{\st}{st}

\newcommand{\CrsBrd}{\mathsf{CrsBrd}}
\newcommand{\Gray}{\mathsf{Gray}}
\newcommand{\Cat}{\mathsf{Cat}}

\newcommand{\Rep}{\mathsf{Rep}}

\newcommand{\ModTens}{\mathsf{ModTens}}

\newcommand{\BiCat}{\mathsf{2Cat}}

\newcommand{\TriCat}{3\Cat}

\newcommand{\To}{\Rightarrow}
\newcommand{\xz}{\otimes}
\newcommand{\xo}{\circ}
\newcommand{\xt}{*}
\newcommand{\Iz}{1_\cC}
\newcommand{\Io}{\id}
\newcommand{\It}{\id}
\newcommand{\RRightarrow}{\,\,\tikzmath{
\clip (0,-.145) -- (.285,-.145) arc (160:120:.3cm) -- (.4,.15) -- (0,.15);
\clip (0,.145) -- (.285,.145) arc (-160:-120:.3cm) -- (.4,-.15) -- (0,-.15);
\draw[quadrupleline={[line width=.5mm,white] in
      [line width=.8mm,black] in
      [line width=1.8mm,white] in
      [line width=2.1mm,black]}] 
      (0,0) -- (.4,0);
\draw[line width=.3mm] (.285,.15) arc (-160:-120:.3cm);
\draw[line width=.3mm] (.285,-.15) arc (160:120:.3cm);
}\,}

\tikzset{Rightarrow/.style={double equal sign distance,>={Implies},->},
triplecd/.style={-,preaction={draw,Rightarrow}},
quadruplecd/.style={preaction={draw,Rightarrow,
shorten >=0pt
},
shorten >=1pt,
-,double,double
distance=0.2pt}}

\usetikzlibrary{calc}
\newcommand{\tikzmath}[2][]
{\vcenter{\hbox{\begin{tikzpicture}[#1]#2
                     \end{tikzpicture}}}
}

\newcommand{\roundNbox}[6]{
	\draw[rounded corners=5pt, thick, #1] ($#2+(-#3,-#3)+(-#4,0)$) rectangle ($#2+(#3,#3)+(#5,0)$);
	\coordinate (ZZa) at ($#2+(-#4,0)$);
	\coordinate (ZZb) at ($#2+(#5,0)$);
	\node at ($1/2*(ZZa)+1/2*(ZZb)$) {#6};
}

\newcommand{\AlphaAction}[3]{
\draw[thick, red, mid>] ($ #1 + (#2,0) $)  .. controls ++(90:#2) and ++(90:#2) .. ($ #1 - (#2,0) + (0,#3) $) -- ($ #1 - (#2,0) - (0,#3) $) .. controls ++(270:#2) and ++(270:#2) .. ($ #1 + (#2,0) $);
}

\newcommand{\LongAlphaAction}[4]{
\draw[thick, red, mid>] 
($ #1 + (#2,0) $) 
-- 
($ #1 + (#2,0) + (0,#4) $)  
.. controls ++(90:#2) and ++(90:#2) .. 
($ #1 - (#2,0) + (0,#3) $) 
-- 
($ #1 - (#2,0) - (0,#3) $) 
.. controls ++(270:#2) and ++(270:#2) .. 
($ #1 + (#2,0) - (0,#4) $) 
-- 
($ #1 + (#2,0) $);
}

\newcommand{\DoubleAlphaAction}[4]{
\draw[thick, red, mid>] ($ #1 - .5*(#2,0) + (0,#4) $) .. controls ++(90:{.5*#2}) and ++(270:{#2}) .. ($ #1 + (#2,0) + 2*(0,#4) $) .. controls ++(90:#2) and ++(90:#2) .. ($ #1 - (#2,0) + 2*(0,#4) + (0,#3) $) -- ($ #1 - (#2,0) - (0,#3) $) .. controls ++(270:#2) and ++(270:#2) .. ($ #1 + (#2,0) $) .. controls ++(90:{#2}) and ++(270:{.5*#2}) .. ($ #1 - .5*(#2,0) + (0,#4) $);
}

\newcommand{\TripleAlphaAction}[3]{
\draw[thick, red, mid>] 
($ #1 + (#2,0) $) 
.. controls ++(90:{#2}) and ++(270:{.5*#2}) .. 
($ #1 - .5*(#2,0) + (0,#3) + (0,.2) $) 
.. controls ++(90:{.5*#2}) and ++(270:{#2}) .. 
($ #1 + (#2,.7) $) 
.. controls ++(90:#2) and ++(90:#2) .. 
($ #1 - (#2,0) + (0,#3) + (0,.7) $) --  ($ #1 - (#2,0) - (0,#3) - (0,.7) $)
.. controls ++(270:#2) and ++(270:#2) .. 
($ #1 + (#2,0) - (0,.7) $) 
.. controls ++(90:{#2}) and ++(270:{.5*#2}) .. 
($ #1 - .5*(#2,0) - (0,#3) - (0,.2) $) 
.. controls ++(90:{.5*#2}) and ++(270:{#2}) .. 
($ #1 + (#2,0) $) ;
}

\newcommand{\CircleMorphism}[2]{
\filldraw[fill=#2, thick] #1 circle (.1cm);
}
\newcommand{\RectangleMorphism}[2]{
\filldraw[fill=#2, thick] ($ #1 + (-.1,-.1) $) rectangle ($ #1 + (.1,.1) $);
}
\newcommand{\RedRectangleMorphism}[2]{
\filldraw[fill=#2, thick, draw=red] ($ #1 + (-.1,-.1) $) rectangle ($ #1 + (.1,.1) $);
}
\newcommand{\GreenRectangleMorphism}[2]{
\filldraw[fill=#2, thick, draw=DarkGreen] ($ #1 + (-.1,-.1) $) rectangle ($ #1 + (.1,.1) $);
}
\newcommand{\BlueRectangleMorphism}[2]{
\filldraw[fill=#2, thick, draw=blue] ($ #1 + (-.1,-.1) $) rectangle ($ #1 + (.1,.1) $);
}

\newcommand{\DeferProof}[1]{
\hfill$\tikzmath{
\node[draw,inner sep=1pt, rectangle] (qed) at (0,0) {\normalfont \scriptsize \S\ref{#1}};
}$
}

\newcommand{\gColor}{white}
\newcommand{\hColor}{green!60}
\newcommand{\kColor}{blue!60}
\newcommand{\BlueGray}{blue!30}

\tikzstyle{BackAlphaSheet}=[fill=red!15!white]
\tikzstyle{FrontAlphaSheet}=[fill=red!40!white]
\tikzstyle{MixedAlphaSheet}=[fill=red!40!black!10!white]
\tikzstyle{gSheet}=[fill=black!10!white]
\pgfdeclarelayer{bg}    % declare background layer
\pgfsetlayers{bg,main}

\title{
A 3-categorical perspective on 
\texorpdfstring{$G$}{G}-crossed braided categories}
\date{\today}
\begin{document}
\author{Corey Jones, David Penneys, and David Reutter}
\maketitle
\begin{abstract}
A braided monoidal category may be considered a $3$-category with one object and one $1$-morphism. In this paper, we show that, more generally, $3$-categories with one object and $1$-morphisms given by elements of a group $G$ correspond to $G$-crossed braided categories, certain mathematical structures which have emerged as important invariants of low-dimensional quantum field theories. 
More precisely, we show that the 4-category of $3$-categories $\mathcal{C}$ equipped with a 3-functor $\mathrm{B}G \to \mathcal{C}$ which is essentially surjective on objects and $1$-morphisms is equivalent to the $2$-category of $G$-crossed braided categories. 
This provides a uniform approach to various constructions of $G$-crossed braided categories.

%This is the submitted version of \arxiv{???}.
\end{abstract}

%%%%%%%%%%%%%%%%%%%%%%%%%%%%%%%%%%%%%%%%%%%%%%
%auto-ignore
%this ensures the arxiv doesn't try to start TeXing here.
%!TEX root =../Equivalence.tex

%%%%%%%%%%%%%%%%%%%%%%%%%%%%%%%%%%%%%%%%%%%%%%%%%%
%%%%%%%%%%%%%%%%%%%%%%%%%%%%%%%%%%%%%%%%%%%%%%%%%%
%%%%%%%%%%%%%%%%%%%%%%%%%%%%%%%%%%%%%%%%%%%%%%%%%%
\section{Introduction}

$G$-\emph{crossed braided categories}~\cite[\S8.24]{MR3242743} (see also \S\ref{defn:Definitions}) have emerged as important mathematical structures describing symmetry enriched invariants of quantum field theories in low dimensions. 
In particular, $G$-crossed braided categories arise from
%\begin{enumerate}[label=(\arabic*)]
%\item
global symmetries in (1+1)D chiral conformal field theory (\cite{MR1923177,math/0110221,MR2183964})
%\item
and (2+1)D topological phases of matter \cite{1410.4540}, and as invariants of three-dimensional homotopy quantum field theories \cite{MR2674592,1802.08512}. 
They are a central object of study in the theory of $G$-extensions of fusion categories \cite{MR2677836,MR2587410,MR3555361}.
%All these constructions typically use the specific structure of the underlying quantum field theory to construct the associated $G$-crossed braided category.
%In this article we give a uniform approach to $G$-crossed braided categories using higher categories.
In this article we describe a higher categorical approach to $G$-crossed braided categories, which unifies these perspectives.

When $G$ is trivial, a $G$-crossed braided category is exactly a braided monoidal category.
It is well-known that braided monoidal categories are `the same as' 3-categories\footnote{
In this article, by a \emph{$3$-category} we mean an algebraic tricategory in the sense of~\cite[Def 4.1]{MR3076451}, and by \emph{functor}, \emph{transformation}, \emph{modification}, and \emph{perturbation}, we mean the corresponding notions of trihomomorphism, tritransformation, trimodification, and perturbation of~\cite[Def 4.10, 4.16, 4.18, 4.21]{MR3076451}.
} 
with exactly one object and one 1-morphism
\cite[Table~21]{MR1355899} and \cite{MR2839900}.
This is an instance of the \emph{Delooping Hypothesis} \cite[\S5.6 and Hypothesis 22]{MR2664619}
relating $k$-fold degenerate $(n+k)$-categories with $k$-fold monoidal $n$-categories. %; see \S\ref{sec:DeloopingHypothesis} below for further discussion.
However, twice degenerate $3$-categories, 3-functors, transformations,  modifications, and perturbations
form a $4$-category, whereas braided monoidal categories, braided monoidal functors, and monoidal natural transformations
only form a $2$-category.
This discrepancy can be resolved by viewing `2-fold degeneracy' as a \emph{structure} on a $3$-category rather than a \emph{property}, namely the structure of a 1-surjective pointing\footnote{\label{footnote:kSurjective}
A functor between $n$-categories $\cG\to \cC$ is $k$-\emph{surjective} if it is
essentially surjective on objects and on $j$-morphisms for all $1\leq j\leq k$. 
A $k$-\emph{surjective pointing} on an $n$-category $\cC$ is a $k$-surjective functor $* \to \cC$.}
~\cite[Sec 5.6]{MR2664619}. 
Explicitly, the Delooping Hypothesis may then be understood as asserting that 
the 4-category of
3-categories equipped with 1-surjective pointings and pointing-preserving higher morphisms between them
is in fact a 2-category (all hom 2-categories between 2-morphisms are contractible) and is equivalent to the 2-category of braided monoidal categories.

Rather than pointing by something contractible (i.e., a point), we can also study `pointings' by other categories. 
In this article, we show that 1-surjective $G$-\emph{pointed} 3-categories, i.e $3$-categories equipped with a 1-surjective $3$-functor from a group $G$ viewed as a 1-category $\rmB G$ with one object, 
are `the same as' $G$-crossed braided categories.

\begin{thmalpha}
\label{thm:Main}
The 4-category\footnote{\label{footnote:4CategoryIntro}
We never actually work with a 4-category, as all our results can be stated and proven at the level of 2-categories.
See Remark \ref{rem:No4Categories} for more details.} 
$\TriCat_G$ of
1-surjective
$G$-pointed $3$-categories 
and pointing-preserving higher morphisms
(see Definition~\ref{defn:TricatG})
is equivalent to the 2-category $G\CrsBrd$ of $G$-crossed braided categories.
In particular, 
%the $4$-category $\TriCat_G$ is $2$-\emph{truncated}, i.e., 
every hom $2$-category between parallel $2$-morphisms in $\TriCat_G$ is contractible. 
\end{thmalpha}

We prove Theorem \ref{thm:Main} as follows.
First, we show in 
Theorem \ref{thm:tricatpt} and 
Corollary \ref{cor:Contract4CategoryTo2Category}
that $\TriCat_G$ is 2-truncated by showing it is equivalent to the \emph{strict} sub-2-category $\TriCat_G^{\st}$ of \emph{strict} $G$-pointed $3$-categories, 
whose objects are $\Gray$-categories with precisely one object, whose sets of $1$-morphisms is exactly $G$, and composition of $1$-morphisms is the group multiplication.
Then in Theorem \ref{thm:ThmA-2}, we construct a strict 2-equivalence between $\TriCat_G^{\st}$ and the strict 2-category $G\CrsBrd^{\st}$ of strict $G$-crossed braided categories.
Finally, by \cite{MR3671186},
every $G$-crossed braided category is equivalent to a strict one (see Definition \ref{defn:StrictGCrossedBraided} for more details), so that $G\CrsBrd$ is equivalent to its full 2-subcategory $G\CrsBrd^{\st}$. 
In summary, we construct the following zig-zag of strict equivalences, where the hooked arrows denote inclusions of full subcategories.
\begin{equation}
\label{eq:ZigZagOfStrictEquivalences}
\begin{tikzcd}[column sep=4em]
\TriCat_G
\arrow[r,hookleftarrow, "\sim","\text{\scriptsize Thm.~\ref{thm:tricatpt}}"']
&
\TriCat_G^{\st}
\arrow[r,rightarrow, "\sim", "\text{\scriptsize Thm.~\ref{thm:ThmA-2}}"']
&
G\CrsBrd^{\st}
\arrow[r,hookrightarrow, "\sim", "\text{\scriptsize \cite{MR3671186}}"']
&
G\CrsBrd
\end{tikzcd}
\end{equation}

For the trivial group $G=\{e\}$, Theorem~\ref{thm:Main} specializes to the Delooping Hypothesis for twice degenerate $3$-categories (also see~\cite{MR2839900}, which uses so called `iconic natural transformations' rather than pointings). 
%\footnote{Icons violate the principle of equivalence, so they are not compatible with strictification.}

\begin{coralpha}
The 4-category$^{ \textup{\ref{footnote:4CategoryIntro}}}$ $\TriCat_{\{e\}}$
of 1-surjective pointed 3-categories 
is equivalent to the 2-category of braided monoidal categories.
\end{coralpha}

Our main theorem was inspired by, and is closely related to, the following two results:
Passing from a $G$-pointed $3$-category to the associated $G$-crossed braided category generalizes a result of~\cite{MR3936135} which constructs $G$-crossed braided categories from group actions on 2-categories; see Example \ref{ex:Galindo} for more details. A version of the construction of a $G$-pointed $3$-category from a $G$-crossed braided category is discussed in~\cite{MR4033513}, and we use this construction in Section \ref{sec:2FunctorEquivalence} to prove essential surjectivity of the $2$-functor $\TriCat_G^{\st}\to G\CrsBrd^{\st}$. 

%%%%%%%%%%%%%%%%%%%%%%%%%%%%%%%%%%%%%%%%%%%%%%%%%%%%%%
\subsection{\texorpdfstring{$G$}{G}-crossed braided categories from $G$-pointed $3$-categories}
\label{sec:WeakGCrossedConstruction}
In the proof of Theorem~\ref{thm:Main}, we construct the equivalence $\TriCat_G  \sim G\CrsBrd$ by passing through appropriate strictifications, resulting in the zig-zag~\eqref{eq:ZigZagOfStrictEquivalences} of strict equivalences. 
For the reader's convenience, we now sketch 
%in the following section 
a direct construction of a $G$-crossed braided category (defined in \S\ref{sec:GCrossed} below) from a $G$-pointed $3$-category, without passing through strictifications. 

For a group $G$, we denote by $\rmB G$ the \emph{delooping} of $G$, i.e., $G$ considered as a $1$-category with one object.
Let $\cC$ be a $3$-category equipped with a 
%$1$-surjective \nn{do we need 1-surjectivity here?} 
$3$-functor\footnote{Since a $k$-category may be viewed as an $n$-category for $n\geq k$ with only identity $r$-morphism for $n\geq r> k$, it  makes sense to talk about an $n$-functor from a $k$-category to an $n$-category.} 
$\pi: \mathrm{B}G \to \cC$. 

To construct the $G$-crossed braided category, we will make use of the graphical calculus of $\Gray$-categories (outlined in Section~\ref{sec:GraphicalCalculus} below) and hence assume that $\cC$ has been strictified to a $\Gray$-category.\footnote{In fact,~\cite{1903.05777} justifies working with this graphical calculus even in the context of weak $3$-categories.
} 
Unpacking the (weak) $3$-functor into the data $(\pi, \mu^\pi, \iota^\pi, \omega^\pi, \lambda^\pi, \rho^\pi)$ as described in Appendix~\ref{sec:Weak3CategoryCoherences}, the $G$-crossed braided category $\fC$ may be constructed as follows.
Strictifying the situation slightly, we may assume that $\cC$ has only one object, i.e. is a $\Gray$-monoid, a monoid object in $\Gray$ viewed as a monoidal 2-category, 
and that the underlying $2$-functor of $\pi$ is strict (the unitor and compositor data $\pi^1, \pi^2$ of $\pi$ is trivial).

We write $g_\cC := \pi(g) \in \cC$, %(\pi(*)\to\pi(*))$, 
and we define $\fC_g := \cC(1_\cC \to g_\cC)$ for each $g\in G$.
We denote the tensorator $\mu^\pi_{g,h}\in \cC(g_\cC \xz h_\cC \to gh_\cC)$ and unitor $\iota^\pi_*\in \cC(1_\cC \to e_\cC)$ of $\pi$ by trivialent and univalent vertices respectively
$$
\mu^\pi_{g,h}
=
\tikzmath{
\draw (0,0) -- (0,.3) node[above] {$\scriptstyle gh_\cC$};
\draw (-.3,-.3) node[below] {$\scriptstyle g_\cC$} arc (180:0:.3cm) node[below] {$\scriptstyle h_\cC$};
\filldraw (0,0) circle (.05cm);
}
\qquad\qquad
\iota^\pi_*
=
\tikzmath{
\draw (0,0) -- (0,.3) node[above] {$\scriptstyle e_\cC$};
\filldraw (0,0) circle (.05cm);
}\,.
$$
We denote 1-morphisms $a_g \in \cC(1_{\cC} \to g_{\cC})$ by shaded disks as follows:
$$
a_g=
\tikzmath{
    \draw (0,0)  -- (0,.3) node [above] {\scriptsize{$g_\cC$}};
    \filldraw[fill=\gColor, thick] (0,0) circle (.1cm);
}
\qquad
\qquad
b_h=
\tikzmath{
    \draw (0,0)  -- (0,.3) node [above] {\scriptsize{$h_\cC$}};
    \filldraw[fill=\hColor, thick] (0,0) circle (.1cm);
}
\qquad
\qquad
c_k=
\tikzmath{
    \draw (0,0)  -- (0,.3) node [above] {\scriptsize{$k_\cC$}};
    \filldraw[fill=\kColor, thick] (0,0) circle (.1cm);
}\,.
$$
For $g,h\in G$, we define a tensor product
$(a_g, b_h)\mapsto a_g\otimes b_h$
by
\begin{equation}
\label{eq:WeakGCrossedTensor}    
\tikzmath{
\draw (-.4,0) -- (-.4,.4) node [above] {\scriptsize{$g_\cC$}};
\draw (.4,0) -- (.4,.4) node [above] {\scriptsize{$h_\cC$}};
\filldraw[fill=\gColor, thick] (-.4,0) circle (.1cm);
\filldraw[fill=\hColor, thick] (.4,0) circle (.1cm);
\node at (0,.2) {$\times$};
}
\longmapsto
\tikzmath{
    \draw (-.3,-.1) -- (-.3,0) arc (180:0:.3cm) -- (.3,-.3);
    \draw (0,.3) -- (0,.6) node [above] {$\scriptstyle gh_\cC$};
    \filldraw (0,.3) circle (.05cm);
    \filldraw[fill=\gColor, thick] (-.3,-.1) circle (.1cm);
    \filldraw[fill=\hColor, thick] (.3,-.3) circle (.1cm);
}\,,
\end{equation}
and we define the associator
$\otimes_{gh,k} \circ (\otimes_{g,h}\times \fC_k) \Rightarrow \otimes_{g,hk} \circ (\fC_g \times \otimes_{h,k})$ 
by
\begin{equation}
\label{eq:WeakGCrossedAssociator}
\tikzmath{
\draw (-.3,-.5) node [left, yshift=.15cm] {\scriptsize{$g_\cC$}} -- (-.3,-.3) arc (180:0:.3cm)  to node [right,yshift=.05cm]  {\scriptsize{$h_\cC$}}(.3,-.7);
\draw (0,0) node [left,xshift=.1cm, yshift=.2cm] {\scriptsize{${gh}_\cC$}} arc (180:0:.45) -- ++(0,-.8) node [right, yshift=.15cm] {\scriptsize{${k}_\cC$}};
\draw (.45,.45) -- (.45,.7) node [above] {\scriptsize{${ghk}_\cC$}};
\filldraw (0,0) circle (.05cm);
\filldraw (.45,.45) circle (.05cm);
\filldraw[fill=\gColor, thick] (-.3,-.5) circle (.1cm);
\filldraw[fill=\hColor, thick] (0.3,-.7) circle (.1cm);
\filldraw[fill=\kColor, thick] (0.9,-.9) circle (.1cm);
\draw[rounded corners= 5pt, dashed, thick] (-.6, -.2) rectangle (1.2, .6);
}
\overset{\omega^\pi_{g,h,k}}{\Longrightarrow}
\tikzmath{
\draw (.3,-.7) -- (.3,-.3) arc (180:0:.3cm)  to (.9,-.9);
\draw (.6,0) arc (0:180:.45) -- ++(0,-.3);
\draw (.15,.45) -- (.15,.7) node [above] {\scriptsize{${ghk}_\cC$}};
\filldraw (.6,0) circle (.05cm);
\filldraw (.15,.45) circle (.05cm);
\filldraw[fill=\gColor, thick] (-.3,-.3) circle (.1cm);
\filldraw[fill=\hColor, thick] (0.3,-.7) circle (.1cm);
\filldraw[fill=\kColor, thick] (0.9,-.9) circle (.1cm);
\draw[rounded corners= 5pt, dashed, thick] (-.7, -.5) rectangle (1.2, .2);
}
\overset{\phi^{-1}}{\Longrightarrow}
\tikzmath{
\draw (.3,-.7) node [left, yshift=.15cm] {\scriptsize{$h_\cC$}} arc (180:0:.3cm)  to node [right,yshift=.05cm]  {\scriptsize{$k_\cC$}}(.9,-.9);
\draw (.6,-.4) -- (.6,0) node [right,xshift=-.1cm, yshift=.2cm] {\scriptsize{${hk}_\cC$}} arc (0:180:.45) -- ++ (0,0) node [left, yshift=.25cm] {\scriptsize{${g}_\cC$}};
\draw (.15,.45) -- (.15,.7) node [above] {\scriptsize{${ghk}_\cC$}};
\filldraw (.6,-.4) circle (.05cm);
\filldraw (.15,.45) circle (.05cm);
\filldraw[fill=\gColor, thick] (-.3,0) circle (.1cm);
\filldraw[fill=\hColor, thick] (0.3,-.7) circle (.1cm);
\filldraw[fill=\kColor, thick] (0.9,-.9) circle (.1cm);
}\,,
\end{equation}
where $\phi$ denotes the interchanger in $\cC$ (see \S\ref{sec:3Categories} below).
We define the unit object $1_\fC := \iota_*^\pi\in \fC_e$. Unitors 
%$\lambda_g : $
$\otimes_{e,g} \circ (i \times -) \Rightarrow \id_{\fC_g}$
and
$\otimes_{g,e} \circ (- \times i) \Rightarrow \id_{\fC_g}$
are given respectively by
$$
\tikzmath{
    \draw (-.3,-.1) node [left, yshift = .2cm] {\scriptsize{$e_\cC$}} -- (-.3,0) arc (180:0:.3cm) node [right] {\scriptsize{$g_\cC$}}  -- (.3,-.5);
    \draw (0,.3) -- (0,.7) node [above] {\scriptsize{${g}_\cC$}};
    \draw[rounded corners=5pt, thick, dashed] (-.9,-.3) rectangle (.9,.5);
    \filldraw (0,.3) circle (.05cm);
    \filldraw (-.3,-.1) circle (.05cm);
    \filldraw[fill=\gColor, thick] (.3,-.5) circle (.1cm);
}
\,\,\,
\overset{\lambda^\pi_g}{\Longrightarrow}
\tikzmath{
\draw (-.4,0) -- (-.4,1) node [above] {\scriptsize{$g_\cC$}};
\filldraw[fill=\gColor, thick] (-.4,0) circle (.1cm);
}
\qquad\qquad\text{and}\qquad\qquad
\tikzmath[xscale=-1]{
    \draw (-.3,-.5) node [right, yshift = .2cm] {\scriptsize{$e_\cC$}} -- (-.3,0) arc (180:0:.3cm) node [left, yshift = .2cm] {\scriptsize{$g_\cC$}}  -- (.3,-.3);
    \draw (0,.3) -- (0,.7) node [above] {\scriptsize{${g}_\cC$}};
    \draw[rounded corners=5pt, thick, dashed] (-.9,-.8) rectangle (.9,0); 
    \filldraw (0,.3) circle (.05cm);
    \filldraw (-.3,-.5) circle (.05cm);
    \filldraw[fill=\gColor, thick] (.3,-.3) circle (.1cm);
}
\overset{\phi}{\Longrightarrow}\,\,\,
\tikzmath[xscale=-1]{
    \draw (-.3,-.1) node [right, yshift = .2cm] {\scriptsize{$e_\cC$}} -- (-.3,0) arc (180:0:.3cm) node [left] {\scriptsize{$g_\cC$}}  -- (.3,-.5);
    \draw (0,.3) -- (0,.7) node [above] {\scriptsize{${g}_\cC$}};
    \draw[rounded corners=5pt, thick, dashed] (-.9,-.3) rectangle (.9,.5);
    \filldraw (0,.3) circle (.05cm);
    \filldraw (-.3,-.1) circle (.05cm);
    \filldraw[fill=\gColor, thick] (.3,-.5) circle (.1cm);
}
\,\,\,
\overset{\rho^\pi_g}{\Longrightarrow}
\tikzmath{
\draw (0,0) -- (0,1) node [above] {\scriptsize{$g_\cC$}};
\filldraw[fill=\gColor, thick] (0,0) circle (.1cm);
}\,.
$$
We define a $G$-action $F_g:\fC_h \to \fC_{ghg^{-1}}$ by
\begin{equation}
\label{eq:WeakGCrossedAction}    
F_g\left(
\tikzmath{
\draw (0,0) -- (0,.4) node [above] {\scriptsize{$h_\cC$}};
\filldraw[fill=\hColor, thick] (0,0) circle (.1cm);
}
\right)
:=
\tikzmath{
\draw (0,0) -- (0,.3) node [right] {\scriptsize{$h_\cC$}} arc (0:180:.3cm) node [left] {\scriptsize{$g_\cC$}} -- (-.6,.15) arc (-180:0:.6cm) node [right] {\scriptsize{${g^{-1}}_\cC$}} -- (.6,.6) arc (0:180:.45) node [left, yshift = .2cm] {\scriptsize{${gh}_\cC$}};
\draw (.15,1.05) -- (.15,1.4) node [above] {\scriptsize{$ghg^{-1}_\cC$}};
\draw (0,-.75) -- node [right] {\scriptsize{${e}_\cC$}} (0,-.45);
\filldraw[fill=\hColor, thick] (0,0) circle (.1cm);
\filldraw (-.3,.6) circle (.05cm);
\filldraw (0,-.45) circle (.05cm);
\filldraw (0,-.75) circle (.05cm);
\filldraw (.15,1.05) circle (.05cm);
}\,.
\end{equation}
The functors $F_g$ come equipped with natural isomorphisms
$\psi^g: 
\otimes_{ghg^{-1},gkg^{-1}}\circ (F_g \times F_g)
\Rightarrow 
F_g\circ \otimes_{h,k}
$ 
built from the coherence isomorphisms $\omega^\pi, \lambda^\pi, \rho^\pi$ and interchangers between two black nodes and between a black node and a shaded disk.
For example, $\psi^g_{b_h, c_k}$ is given by
\begin{equation}
\label{eq:WeakGCrossedTensorator}
\tikzmath[scale=.5]{
\draw (0,0) -- (0,.3) arc (0:180:.3cm) -- (-.6,.15) arc (-180:0:.6cm) -- (.6,.6) arc (0:180:.45) ;
\draw (0,-1) -- (0,-.45);
\filldraw[fill=\hColor, thick] (0,0) circle (.2cm);
\filldraw (-.3,.6) circle (.1cm);
\filldraw (0,-.45) circle (.1cm);
\filldraw (0,-1) circle (.1cm);
\filldraw (.15,1.05) circle (.1cm);
\begin{scope}[yshift=-2.5cm, xshift=1.5cm]
\draw (0,0) -- (0,.3) arc (0:180:.3cm) -- (-.6,.15) arc (-180:0:.6cm) -- (.6,.6) arc (0:180:.45) ;
\draw (0,-1) -- (0,-.45);
\filldraw[fill=\kColor, thick] (0,0) circle (.2cm);
\filldraw (-.3,.6) circle (.1cm);
\filldraw (0,-.45) circle (.1cm);
\filldraw (0,-1) circle (.1cm);
\filldraw (.15,1.05) circle (.1cm);
\end{scope}
\draw (.15,1.05) arc (180:0:.75cm) -- (1.65,-1.5);
\filldraw (.9,1.8) circle (.1cm);
\draw (.9,1.8) -- (.9,2.2) ;
}
\quad
\overset{\phi}{\Rightarrow}
\quad
\tikzmath[scale=.5]{
\coordinate (a) at (0,0);
\coordinate (b) at (1.8,-.4);
\draw ($ (a) + (0,.2) $) arc (0:180:.3cm) -- (-.6,-1.6);
\draw (-.3,.5) arc (180:0:.45cm) -- (.6,-1.6) arc (0:-180:.6cm);
\draw (b) -- ($ (b)+ (0,1.6) $) arc (0:180:.3cm) -- (1.2,-.6);
\draw ($ (b)+ (-.3,1.9) $) arc (180:0:.45cm) -- (2.4,-.6) arc (0:-180:.6cm);
\draw (0,-2.6) -- (0,-2.2);
\draw (1.8,-1.2) -- (1.8,-1.6);
\draw (.15,.95) -- (.15,1.95) arc (180:0:.9cm);
\draw (1.05, 2.85) -- (1.05, 3.15);
\filldraw (1.05,2.85) circle (.1cm);
\filldraw (-.3,.5) circle (.1cm);
\filldraw (.15,.95) circle (.1cm);
\filldraw (1.5,1.5) circle (.1cm);
\filldraw (1.95,1.95) circle (.1cm);
\filldraw (0,-2.6) circle (.1cm);
\filldraw (0,-2.2) circle (.1cm);
\filldraw (1.8,-1.6) circle (.1cm);
\filldraw (1.8,-1.2) circle (.1cm);
\filldraw[fill=\hColor, thick] (a) circle (.2cm);
\filldraw[fill=\kColor, thick] (b) circle (.2cm);
}
\quad
\overset{\omega,\lambda,\rho}{\Rightarrow}
\quad
\tikzmath[scale=.5]{
\coordinate (a) at (0,0);
\coordinate (b) at (1.8,-.4);
\draw (a) -- ($ (a) + (0,1.2) $) arc (0:180:.3cm) -- (-.6,-1.6) arc (-180:0:.6cm);
\draw (b) -- ($ (b)+ (0,2) $) arc (180:0:.3cm) -- (2.4,-.6) arc (0:-180:.6cm);
\draw (.6,-1.6) -- ($ (a) + (.6,.3) $) arc (180:0:.3cm) -- (1.2,-.6);
\draw (.9,.6) -- (.9,1);
\draw (-.3,1.5) -- (-.3,1.9) arc (180:0:1.2cm);
\draw (0,-2.6) -- (0,-2.2);
\draw (1.8,-1.2) -- (1.8,-1.6);
\draw (.9,3.1) -- (.9, 3.5);
\filldraw (.9,.6) circle (.1cm);
\filldraw (.9,1) circle (.1cm);
\filldraw (-.3,1.5) circle (.1cm);
\filldraw (2.1,1.9) circle (.1cm);
\filldraw (.9,3.1) circle (.1cm);
\filldraw (0,-2.6) circle (.1cm);
\filldraw (0,-2.2) circle (.1cm);
\filldraw (1.8,-1.6) circle (.1cm);
\filldraw (1.8,-1.2) circle (.1cm);
\filldraw[fill=\hColor, thick] (a) circle (.2cm);
\filldraw[fill=\kColor, thick] (b) circle (.2cm);
}
\quad
\overset{\phi}{\Rightarrow}
\quad
\tikzmath[scale=.5]{
\coordinate (a) at (0,1);
\coordinate (b) at (1.8,.6);
\draw (a) -- ($ (a) + (0,.2) $) arc (0:180:.3cm) -- (-.6,-1.6) arc (-180:0:.6cm);
\draw (b) -- ($ (b)+ (0,1) $) arc (180:0:.3cm) -- (2.4,-.6) arc (0:-180:.6cm);
\draw (.6,-1.6) -- (.6,-.6) arc (180:0:.3cm) -- (1.2,-.6);
\draw (.9,-.3) -- (.9,.1);
\draw (-.3,1.5) -- (-.3,1.9) arc (180:0:1.2cm);
\draw (0,-2.6) -- (0,-2.2);
\draw (1.8,-1.2) -- (1.8,-1.6);
\draw (.9,3.1) -- (.9, 3.5);
\filldraw (.9,.1) circle (.1cm);
\filldraw (.9,-.3) circle (.1cm);
\filldraw (-.3,1.5) circle (.1cm);
\filldraw (2.1,1.9) circle (.1cm);
\filldraw (.9,3.1) circle (.1cm);
\filldraw (0,-2.6) circle (.1cm);
\filldraw (0,-2.2) circle (.1cm);
\filldraw (1.8,-1.6) circle (.1cm);
\filldraw (1.8,-1.2) circle (.1cm);
\filldraw[fill=\hColor, thick] (a) circle (.2cm);
\filldraw[fill=\kColor, thick] (b) circle (.2cm);
}
\quad
\overset{\omega,\lambda,\rho}{\Rightarrow}
\quad
\tikzmath[scale=.5]{
\coordinate (a) at (0,0);
\coordinate (b) at (.8,-.4);
\draw (a) -- ($ (a) + (0,.2) $) arc (0:180:.3cm) -- (-.6,-.4) arc (-180:0:1cm);
\draw (b) -- ($ (b)+ (0,1) $) arc (180:0:.3cm) -- (1.4,-.4);
\draw (-.3,.5) -- (-.3,.9) arc (180:0:.7cm);
\draw (.4,-1.4) -- (.4,-1.8);
\draw (.4,1.6) -- (.4,2);
\filldraw (-.3,.5) circle (.1cm);
\filldraw (1.1,.9) circle (.1cm);
\filldraw (.4,1.6) circle (.1cm);
\filldraw (.4,-1.4) circle (.1cm);
\filldraw (.4,-1.8) circle (.1cm);
\filldraw[fill=\hColor, thick] (a) circle (.2cm);
\filldraw[fill=\kColor, thick] (b) circle (.2cm);
}
\quad
\overset{\omega,\lambda,\rho}{\Rightarrow}
\quad
\tikzmath{
\draw (0,.3) arc (0:180:.3cm) -- (-.6,.0) arc (-180:0:.6cm) -- (.6,.6) arc (0:180:.45) ;
\draw (-.2,0) -- (-.2,.1) arc (180:0:.2cm) -- (.2,-.1);
\draw (.15,1.05) -- (.15,1.4) ;
\draw (0,-.9) -- (0,-.6);
\filldraw[fill=\hColor, thick] (-.2,0) circle (.1cm);
\filldraw[fill=\kColor, thick] (.2,-.2) circle (.1cm);
\filldraw (-.3,.6) circle (.05cm);
\filldraw (0,-.6) circle (.05cm);
\filldraw (0,-.9) circle (.05cm);
\filldraw (0,.3) circle (.05cm);
\filldraw (.15,1.05) circle (.05cm);
}\,.
\end{equation}
The tensorator $\mu_{g,h} : F_g \circ F_h \Rightarrow F_{gh}$ 
and the unit map
$\iota_h: \id_{\fC_h}\to F_e|_{\fC_h}$
are defined similarly.
The $G$-crossed braiding natural isomorphisms $\beta^{g,h}:a_g \otimes b_h \to F_g(b_h)\otimes a_g$ are also defined similarly using the interchanger isomorphism $\phi$ of $\cC$:
\begin{equation}
\label{eq:WeakGCrossedBraiding}
\tikzmath{
    \draw (-.3,-.1) node [left, yshift = .2cm] {\scriptsize{$g_\cC$}} -- (-.3,0) arc (180:0:.3cm) node [right] {\scriptsize{$h_\cC$}}  -- (.3,-.3);
    \draw (0,.3) -- (0,.7);
    \filldraw (0,.3) circle (.05cm);
    \filldraw[fill=\gColor, thick] (-.3,-.1) circle (.1cm);
    \filldraw[fill=\hColor, thick] (.3,-.3) circle (.1cm);
}
\overset{\phi}{\Rightarrow}
\tikzmath[xscale=-1]{
    \draw (-.3,-.1) node [right, yshift = .2cm] {\scriptsize{$h_\cC$}} -- (-.3,0) arc (180:0:.3cm) node [left] {\scriptsize{$g_\cC$}}  -- (.3,-.3);
    \draw (0,.3) -- (0,.7) node [above] {\scriptsize{${gh}_\cC$}};
    \filldraw (0,.3) circle (.05cm);
    \filldraw[fill=\hColor, thick] (-.3,-.1) circle (.1cm);
    \filldraw[fill=\gColor, thick] (.3,-.3) circle (.1cm);
}
\overset{\cong}{\Rightarrow}
\tikzmath{
\draw (0,0) -- (0,.3) node [right] {\scriptsize{$h_\cC$}} arc (0:180:.3cm) node [left] {\scriptsize{$g_\cC$}} -- (-.6,.15) arc (-180:0:.6cm) node [right] {\scriptsize{${g^{-1}_\cC}$}} -- (.6,.6) arc (0:180:.45) node [left, yshift = .2cm] {\scriptsize{$gh_\cC$}};
\draw (.15,1.05) node [left, yshift = .5cm] {\scriptsize{$ghg^{-1}_\cC$}} arc (180:0:.75cm) 
node [right] {\scriptsize{${g}_\cC$}}
-- (1.65,-1.1);
\draw (.9,1.8) -- (.9,2) node [above] {\scriptsize{${gh}_\cC$}};
\draw (0,-.75) -- node [right] {\scriptsize{${e}_\cC$}} (0,-.45);
\filldraw[fill=\hColor, thick] (0,0) circle (.1cm);
\filldraw (-.3,.6) circle (.05cm);
\filldraw (0,-.45) circle (.05cm);
\filldraw (0,-.75) circle (.05cm);
\filldraw (.15,1.05) circle (.05cm);
\filldraw (.9,1.8) circle (.05cm);
\filldraw[fill=\gColor, thick] (1.65,-1.1) circle (.1cm);
}\,.
\end{equation}

%%%%%%%%%%%%%%%%%%%%%%%%%%%%%%%%%%%%%%%%%%%%%%%%%%%%%%
\subsection{The delooping hypothesis}
\label{sec:DeloopingHypothesis}
Recall $(n+k)$-categories form an $(n+k+1)$-category, whereas $k$-fold monoidal $n$-categories only form an $(n+1)$-category.
Thus one should not think of `$k$-fold degeneracy' as a \emph{property} of an $(n+k)$-category $\cC$ but rather as additional \emph{structure}, namely the structure of a \emph{$(k-1)$-surjective pointing},
and require all morphisms and higher morphisms between these categories to preserve pointings~\cite[Sec 5.6]{MR2664619}. 
Explicitly, the Delooping Hypothesis may then be understood as asserting that $(k-1)$-connected pointed $(n+k)$-categories and pointing-preserving higher morphisms form an $(n+1)$-category which is equivalent to the $(n+1)$-category of $k$-fold monoidal $n$-categories. 
This is an instance of a more general higher categorical principle. 

\begin{defn}
\label{defn:kSurjective}
We call a functor $F:\cC\to\cD$ of $n$-categories  $k$-\emph{surjective}\footnote{This notion of $k$-surjectivity does not coincide with the one used in~\cite{MR2664619}, where a functor is said to be $k$-surjective if it is essentially surjective on $k$-morphisms.} if it is essentially surjective on objects and parallel $r$-morphisms for $r\leq k$. By convention, any functor is $(-1)$-surjective.
\end{defn}

\begin{hypothesis}
\label{conj:OverCatTruncates}
Let $\cG$ be a $n$-category.
The full $(n+1)$-subcategory of the under-$(n+1)$-category $n\Cat_{\cG/}$ on the $k$-surjective functors out of $\cG$ is an $(n-k)$-category, i.e., all hom $(k+1)$-categories between parallel $(n-k)$-morphisms are contractible.
\end{hypothesis}

\begin{rem}
We expect hypothesis~\ref{conj:OverCatTruncates} is a direct consequence of more common assumptions on the $(n+1)$-category of $n$-categories: 
Namely, following~\cite[\S5.5]{MR2664619}, we say that a functor $F:\cC \to \cD$ between $n$-categories is \emph{$j$-monic}\footnote{
Many of the definitions and statements in this remark are extensively developed in the setting of $(\infty,1)$-categories \cite[Sec 5]{MR2522659}, and in particular in the $(n+1,1)$-category of $n$-categories. 
However, we are not able to use these $(\infty,1)$-notions and statements for our purposes, as we are working in 
the $(n+1,n+1)$-category of $n$-categories. 
For example, our $j$-monomorphisms do not coincide with the $(\infty,1)$-categorical $j$-monomorphisms (in this context also known as \emph{$(j-1)$-truncated morphisms}) as the latter only fulfill essential surjectivity conditions with respect to invertible cells.
}
if it is essentially surjective on $k$-morphisms for all $k>j$ (including $k=n+1$, where we interpret surjectivity to mean faithfulness on $n$-morphisms).
By~\cite[Hypothesis 17]{MR2664619}, the (weak) fibers of such a $j$-monic functor are expected to be (possibly poset-enriched) $(j-1)$-categories.\footnote{A functor between $n$-groupoids is $j$-monic if and only if its fibers are $(j-1)$-categories. 
For functors between general $n$-categories, $j$-monomorphisms have truncated fibers but the converse is not necessarily true.} 
Dually, a functor $G:\cC \to \cD$ between $n$-categories is \emph{$j$-epic} if for every $n$-category $\cE$, the pre-composition functor $ n\Cat(p, \cE): n\Cat(\cD\to \cE) \to n\Cat(\cC \to \cE)$ is $(n-1-j)$-monic. 
In particular, any $j$-surjective functor in the sense of Definition~\ref{defn:kSurjective} is $j$-epic.\footnote{\label{footnote:StrongEpi}
More generally, $j$-surjective functors are expected to correspond to `strong $j$-epimorphisms'~\cite[Hypothesis 21]{MR2664619}, that is, functors that have the left lifting property with respect to $j$-monomorphisms. 
Since the $(n+1)$-category of $n$-categories has finite limits, any such `strong $j$-epimorphism' is in particular a $j$-epimorphism; see~\cite[Sec 5.5]{MR2664619}.}
Combining these observations, given a $k$-surjective functor $p:\cG \to \cC$ and an $n$-category $\cE$, the pre-composition functor $n\Cat(p,\cE): n\Cat(\cC \to \cE) \to n\Cat(\cG \to \cE)$ is $(n-1-k)$-monic. 
Hence, its fiber at a $g\in n\Cat(\cG\to \cE)$ is a (possibly poset-enriched) $(n-1-k)$-category.
But the fiber of the pre-composition functor $n\Cat(p, \cE)$ at $g:\cG \to \cE$ is the hom-category $n\Cat_{\cG/}(g,p)$ of the under-category of $n$-categories under $\cG$. 
Therefore, the full subcategory of $n\Cat_{\cG/}$ on the $k$-surjective functors is a (possibly poset-enriched) $(n-k)$-category.
Moreover, essential ($0$-)surjectivity of $p:\cG\to \cC$ (also cf. Footnote~\ref{footnote:StrongEpi}) should imply that the pre-composition functor $n\Cat(p,\cE)
:n\Cat(\cC\to \cE)\to n\Cat(\cG\to \cE)$ is $n$-conservative\footnote{An $n$-functor $F:\cC\to \cD$ is \emph{$n$-conservative} if it reflects $n$-isomorphisms, i.e., for every $n$-morphism $\alpha: f\To g$ in $\cD$ for which $F(\alpha)$ is an isomorphism, it follows that $\alpha$ is an isomorphism.}, and hence that the enriching posets of the (weak) fibers of $n\Cat(p,\cE)$ are honest sets.
%Second version:
%Since our $k$-surjective functor $p:\cG\to \cC$ is in particular essentially surjective on objects and since we expect the pre-composition 
%Since we expect our $k$-surjective functor $p: \cG \to \cC$ to be a strong $k$-epimorphism %(see Footnote \ref{footnote:StrongEpi}), we expect the 
%enriching posets of the (weak) fibers of the precomposition functor $n\Cat(p,\cE)
%:n\Cat(\cC\to \cE)\to n\Cat(\cG\to \cE)$ to be honest sets.  
\end{rem}

\begin{ex}[$k$-fold monoidal $n$-categories]
In the case where $\cG=*$ is the terminal category, Hypothesis~\ref{conj:OverCatTruncates} asserts that $(k-1)$-surjective ($k$-fold degenerate) pointed $(n+k)$-categories form an $(n+1)$-category. 
The \emph{Delooping Hypothesis} \cite[\S5.6 and Hypothesis 22]{MR2664619} identifies this $(n+1)$-category with the $(n+1)$-category of $k$-fold monoidal $n$-categories.
\end{ex}

An important consequence of Hypothesis~\ref{conj:OverCatTruncates} is that it allows us to study certain higher-categorical objects, namely $k$-surjective functors and their higher transformations, using lower-categorical machinery.
In many instances, there exist concrete descriptions of the resulting low-dimensional categories which have been developed and appear in mathematics and physics independently. 

As a concrete example, 
it is easier to describe and work with
the $1$-category of monoids and monoid homomorphisms 
than its unpointed variant, the $2$-category of categories, functors, and natural transformations.
Similarly, it is easier to describe and work with the 2-category of monoidal categories, monoidal functors, and monoidal natural transformations than its upointed variant, the 3-category of 2-categories, 2-functors, 2-transformations, and 2-modifications. Similar examples are shown in Figure~\ref{fig:PointingTable}.

\begin{figure}[!ht]
\begin{tabular}{|c|c|c|c|c|c|}
\hline
&
$\cG$
&
$n+k=0$
&
$n+k=1$
&
$n+k=2$
&
$n+k=3$
\\\hline
$k=-1$
&
$\emptyset$
&
0-category
&
1-category
&
2-category
&
3-category
\\\hline
$k=0$
&
$*$
&
point
&
monoid
&
monoidal category
&
monoidal 2-category
\\\hline
$k=1$
&
$\rmB G$
&
&
normal subgroup of $G$
&
$G$-crossed monoid
&
$G$-crossed braided category
\\\hline
\end{tabular}

\caption{
\label{fig:PointingTable}
$(n+k)$-categories equipped with $k$-surjective functors from $\cG$ form an $n$-category}
\end{figure}

In this article, we focus on 1-surjective functors from the delooping $\rmB G$ of $G$, i.e., the 1-category with one object and endomorphisms $G$.

\begin{hypothesis}[$G$-crossed delooping] 
\label{hyp:Gcrossed}
For $n\geq -1$,
the $(n+3)$-category 
of 1-surjective functors from $\rmB G$ into $(n+2)$-categories
is equivalent to 
the $(n+1)$-category of $G$-crossed braided $n$-categories.
\end{hypothesis}
While we do not present a general definition of 
$G$-crossed braided $n$-category here,
this hypothesis is a desideratum for any such definition (such as for example via M\"uller and Woike's `little bundles' operad~\cite{1901.04850}). Observe that the $k=1$ version of the delooping hypothesis follows as a consequence for the trivial group $G=\{e\}$. 

In the following, as a warm-up to our main theorem, we discuss the low-dimensional versions ($n=0$ and $n=-1$) of Hypothesis~\ref{hyp:Gcrossed} appearing in the last row of Figure~\ref{fig:PointingTable}.

\begin{ex}[$G$-crossed monoids as $G$-pointed $2$-categories]
\label{ex:GCrossedBraided0Categories}
The 3-category $\BiCat_G$ of $2$-categories $\cC$ equipped with $1$-surjective $2$-functors $\rmB G \to \cC$ is equivalent to the $1$-category of \emph{$G$-crossed monoids}, or `$G$-crossed braided $0$-categories', defined below. Explicitly, the $2$-category $\BiCat_G$ has
\begin{itemize}
    \item 
    objects $(\cC, \pi^\cC)$ where $\cC$ is a 2-category and $\pi^\cC:\rmB G \to \cC$ is a 1-surjective 2-functor,
    \item
    1-morphisms $(A,\alpha): (\cC, \pi^\cC) \to (\cD, \pi^\cD)$ where  
    $A: \cC \to \cD$ is a 2-functor
    and $\alpha : \pi^\cD \Rightarrow A\circ \pi^\cC$ is an invertible 2-transformation,
    \item
    2-morphisms $(\eta,m) : (A, \alpha) \Rightarrow (B, \beta)$ where 
    $\eta: A \Rightarrow B$ is a 2-transformation 
\begin{equation*}
\begin{tikzcd}
\rmB G\arrow[rr, "\pi^\cC"]
\arrow[ddrr, swap, "\pi^\cD"]
&&
\cC\arrow[dd,"B"]
\arrow[dl,Leftarrow,shorten <= 1em, shorten >= 1em, "\beta"]
%    \arrow[d,bend left = 60, "B"]
\\
&\mbox{}&
\\
&&
\cD.
\end{tikzcd}
\qquad
\overset{m}{\Rrightarrow}
\qquad
\begin{tikzcd}
\rmB G\arrow[rr, "\pi^\cC"]
\arrow[ddrr, swap, "\pi^\cD"]
&&
\cC\arrow[dd,swap,"A"]
\arrow[dl,Leftarrow,shorten <= 1em, shorten >= 1em, "\alpha"]
\arrow[dd,bend left = 90, "B"]
\\
&\mbox{}&
\arrow[r,Rightarrow,shorten >= 1em, "\!\!\!\!\eta"]
&\mbox{}
\\
&&
\cD
\end{tikzcd}
\end{equation*}
\item
3-morphisms $p: (\eta,m) \Rrightarrow (\zeta, n)$ where $p: \eta \Rrightarrow \zeta$ is a 2-modification such that
\begin{equation*}
\begin{tikzcd}
\pi^\cD\arrow[rr,Rightarrow, "\alpha"]
\arrow[ddrr, Rightarrow,swap, "\beta"]
&&
A\circ \pi^\cC\arrow[dd,Rightarrow, "\zeta\circ \pi^\cC"]
%    \arrow[d,bend left = 60, "B"]
\\
&
\mbox{}
\arrow[ur,triplecd,shorten <= 1em, shorten >= 1em, "n"]
&
\\
&&
B\circ \pi^\cC.
\end{tikzcd}
\qquad
=
\qquad
\begin{tikzcd}
\pi^\cD\arrow[rr,Rightarrow, "\alpha"]
\arrow[ddrr, Rightarrow,swap, "\beta"]
&&
A\circ \pi^\cC\arrow[dd,swap,Rightarrow, "\eta\circ \pi^\cC"]
\arrow[dd,Rightarrow, bend left = 90, "\zeta \circ \pi^\cC"]
\\
&
\arrow[ur,triplecd,shorten <= 1em, shorten >= 1em, "m"]
\mbox{}
&
\arrow[r,triplecd,shorten >=1em, "\!\!\!\!p \circ \pi^\cC"]
&\mbox{}
\\
&&
B\circ \pi^\cC
\end{tikzcd}
\end{equation*}
\end{itemize}

On the other hand, a natural decategorification of a $G$-crossed braided monoidal category is a $G$-graded monoid
%A $G$-crossed braided 0-category is a $G$-graded monoid 
$M=\amalg_{g\in G} M_g$  
%with a `grading' monoid homomorphism $\gr: M \to G$ sending $m_g \in M_g$ to $g\in G$ 
together with
a group homomorphism $\pi^M: G \to \Aut(M)$ such that the following axioms are satisfied:
\begin{itemize}
    \item 
    $\pi^M_g(m_h) \in M_{ghg^{-1}}$
    for all $g\in G$ and $m_h\in M_h$, and
    %$\gr(F_g(m)) = g(\gr(m))g^{-1}$
    %for all $g\in G$ and $m\in M$, and
    \item
    $m_g \cdot n_h= \pi^M_{g}(n_h) \cdot m_g$
    for all $m\in M_g$ and $n_h \in M_h$.
\end{itemize}
We call such a pair $(M, \pi^M)$ a \emph{$G$-crossed monoid}, or a `$G$-crossed braided $0$-category'.
Morphisms $(M,\pi^M) \to (N, \pi^N)$ are $G$-graded monoid homomorphisms that intertwine the $G$-actions.

To see that $\BiCat_G$ is equivalent to the category of $G$-crossed monoids,
we mirror our proof of Theorem~\ref{thm:Main}.
One first shows that $\BiCat_G$ is
equivalent to the 1-category $\BiCat_G^{\st}$ with
\begin{itemize}
    \item 
    objects strict monoidal categories $\cC$ whose set of objects is $\{g_{\cC}\}_{g\in G}$ with $1_\cC= e_{\cC}$ and tensor product given by the group multiplication, and
    \item
    morphisms $A: \cC \to \cD$ are strict monoidal functors such that $A(g_\cC)=g_\cD$ for all $g\in G$.
\end{itemize}

The equivalence from $\BiCat_G^{\st}$ to $G$-crossed monoids is given by taking hom from $1_\cC$.
We set
$M_g:=\cC(1_\cC \to g_\cC)$,
and the multiplication on $M:= \amalg_{g\in G} M_g$ is $\xz$ in $\cC$.
The $G$-action $\pi^M: G \to \Aut(M)$ is given by conjugation: 
$$
\pi^M_g(m_h) := 
\id_{g_\cC} \xz m_h \xz \id_{g_\cC^{-1}} 
\in
M_{ghg^{-1}}
=
\cC(1_\cC \to ghg^{-1}_\cC).
$$
One then verifies the $G$-crossed braiding axiom by a $G$-graded version of Eckmann-Hilton.
A 1-morphism $A\in \BiCat_G^{\st}( \cC \to \cD)$ yields a $G$-graded monoid homomorphism by restricting to $M_g= \cC(1_\cC \to g_\cC)$.
This monoid homomorphism is compatible with the $G$-actions by strictness of $A$.
Finally, one verifies this construction is an equivalence of categories.
\end{ex}

\begin{ex}[Normal subgroups as $G$-pointed $1$-categories]
\label{ex:GCrossedBraided(-1)Categories}
The 2-category $\Cat_G$ of $1$-categories $\cC$ equipped with $1$-surjective functors $\rmB G \to \cC$ is equivalent to the set of normal subgroups of $G$ (which we may think of as the `$0$-category of $G$-crossed braided $(-1)$-categories', see below). Explicitly, $\Cat_G$ has
\begin{itemize}
    \item 
    objects $(\cC, \pi^\cC)$ where $\cC$ is a category and $\pi^\cC:\rmB G \to \cC$ is a 1-surjective functor,
    \item
    1-morphisms $(A,\alpha): (\cC, \pi^\cC) \to (\cD, \pi^\cD)$ where  
    $A: \cC \to \cD$ is a functor
    and $\alpha : \pi^\cD \Rightarrow A\circ \pi^\cC$ is a natural isomorphism, and
    \item
    2-morphisms $\eta : (A, \alpha) \Rightarrow (B, \beta)$ are natural transformations $\eta: A \Rightarrow B$ such that
\begin{equation*}
\label{eq:2MorphismCriterionForGCrossed(-1)Cat}
\begin{tikzcd}
\rmB G\arrow[rr, "\pi^\cC"]
\arrow[ddrr, swap, "\pi^\cD"]
&&
\cC\arrow[dd,"B"]
\arrow[dl,Leftarrow,shorten <= 1em, shorten >= 1em, "\beta"]
%    \arrow[d,bend left = 60, "B"]
\\
&\mbox{}&
\\
&&
\cD.
\end{tikzcd}
\qquad
=
\qquad
\begin{tikzcd}
\rmB G\arrow[rr, "\pi^\cC"]
\arrow[ddrr, swap, "\pi^\cD"]
&&
\cC\arrow[dd,swap,"A"]
\arrow[dl,Leftarrow,shorten <= 1em, shorten >= 1em, "\alpha"]
\arrow[dd,bend left = 90, "B"]
\\
&\mbox{}&
\arrow[r,Rightarrow,shorten >= 1em, "\!\!\!\!\eta"]
&\mbox{}
\\
&&
\cD
\end{tikzcd}
\end{equation*}
\end{itemize}
It is straightforward to verify that this 2-category is equivalent to a set. Moreover, up to equivalence, the data of a $1$-surjective functor $\pi^\cC: \rmB G \to \cC$ is equivalent to the data of a normal subgroup of $G$, obtained as the kernel of the surjective group homomorphism $G\to \Aut_{\cC}(\pi^\cC(*))$. Hence, the $2$-category $\Cat_G$ is equivalent to the set of normal subgroups of $G$.

Employing `categorical negative thinking' as in \cite[\S2]{MR2664619}, we may in fact think of a normal subgroup of $G$ as a `$G$-crossed braided $(-1)$-category', and hence of the set of normal subgroups as `the $0$-category of $G$-crossed braided $(-1)$-category' as it appears in Hypothesis~\ref{hyp:Gcrossed}:
Since a $(-1)$-category may be thought of as a truth value~\cite[\S2]{MR2664619}, one may define a $G$-graded $(-1)$-category to be a monoid homomorphism $G\to \mathsf{Bool}=(\{T,F\}, \wedge)$, where $\mathsf{Bool}$ denotes the Booleans which one may think of as the commutative monoid (symmetric monoidal 0-category) of $(-1)$-categories. 
Indeed, by taking the kernel, such `$G$-graded $(-1)$-categories' correspond to normal subgroups of $G$. 
This correspondence may be seen a further decategorified analogue of our construction.
Indeed, given $(\cC,\pi^\cC)$, the corresponding monoid homomorphism $G\to \mathsf{Bool}$ is exactly given by $g\mapsto \cC(\id_{\pi^\cC(*)} \to \pi^\cC(g))$,
where the latter is the Boolean which is true if $\id_{\pi^\cC(*)}=\pi(g)$ and false otherwise.
\end{ex}

% \begin{rem}
% We note that the mathematical gadget underlying both Examples \ref{ex:GCrossedBraided(-1)Categories} and \ref{ex:GCrossedBraided0Categories} above and $G$-crossed braided categories in \S\ref{defn:Definitions} 
% starts with the data of a lax monoidal functor $L:G \to n\Cat$ 
% together with a map $G\to \Aut(L)$.
% \nn{phrase this better...}
% \end{rem}

% It is well-known that monoidal categories are 2-categories (bicategories) with exactly one object \cite[Table~21]{MR1355899} and \cite{MR2342826}.
% \nn{maybe also mention \cite{2573304}}
% However, monoidal categories with monoidal functors and monoidal natural transformations form a 2-category, whereas 2-categories form a 3-category (algebraic tricategory).
% As discussed in \cite[\S5.6]{MR2664619}, one can consider the 3-category of \emph{pointed} 2-categories, whose objects are pairs $(\cC, I_\cC)$ where $\cC$ is a 2-category and $I_\cC: *\to \cC$ is a 0-connected 2-functor (pseudofunctor) from the trivial 2-category to $\cC$, and whose higher morphisms are required to be compatible with these pointing.
% This 3-category is in fact 1-truncated and equivalent to the 2-category of monoidal categories, monoidal functors, and monoidal natural transformations.

\begin{ex}[Shaded monoidal algebras]
\label{ex:ShadedMonoidalAlgebras}
In \cite[Defn.~3.18 and ~3.26]{1810.06076}, the authors define the notion of a shaded monoidal algebra, which is an operadic approach to 2-categories with a chosen set of objects and a set of generating 1-morphisms.
The statements of \cite[Thm.~3.21 and Cor.~3.23]{1810.06076} can be understood as examples of 
Hypothesis \ref{conj:OverCatTruncates}.
Indeed, equipping a 2-category with a set of objects and a generating set of 1-morphisms is equivalent to pointing by the free category on a graph $\Gamma$.
Hence the 3-category of 1-surjective $\Gamma$-pointed 2-categories is equivalent to the 1-category of $\Gamma$-shaded monoidal algebras.
%\nn{special subcategory when they are dualizble
%DP: This is a real pain to explain, and I'm not sure it's worth it at all. We chose to omit the rigid story completely from \cite[\S3]{1810.06076} for this reason. However in pivotal land, it's extra structure again, not just property, and so it's reasonable again.}
\end{ex}

\begin{rem}[Planar algebras]
Expanding on Example~\ref{ex:ShadedMonoidalAlgebras}, Jones' planar algebras \cite{math.QA/9909027} reflect the philosophy of Hypothesis \ref{conj:OverCatTruncates}. 
A 2-shaded planar algebra may be understood as a pivotal 2-category $\cC$ with precisely two objects `unshaded' and `shaded' together with a generating dualizable $1$-morphism between them with loop modulus $\delta$.
%\nn{loop modulus left right or just a single number?}. 
This choice of generating $1$-morphism may be understood as equipping $\cC$ with a $1$-surjective pivotal functor $\pi^\cC:\cT\cL\cJ(\delta) \to \cC$,
where $\cT\cL\cJ(\delta)$ is the free spherical 2-category on a dualizble 1-morphism with quantum dimension $\delta$.
%\nn{need say what $\cT\cL\cJ$ stands for somewhere around here}. 
By (a pivotal version of) Hypothesis \ref{conj:OverCatTruncates}, such pivotal 2-categories and functors preserving this `TLJ-pointing' actually form a $1$-category, which is equivalent to the 1-category of 2-shaded planar algebras and planar algebra homomorphisms. %\cite{math.QA/9909027}. 

Another instance of this philosophy appears in~\cite{1607.06041} which shows the 2-category $\ModTens_*$ of \emph{pointed module tensor categories} over a braided pivotal category $\cV$ (defined in \cite[\S3.1]{1607.06041}) is $1$-truncated~\cite[Lem.~3.6]{1607.06041}.
By \cite[Thm.~A]{1607.06041}, $\ModTens_*$ is equivalent to the 1-category of anchored planar algebras in $\cV$.
\end{rem}

%%%%%%%%%%%%%%%%%%%%%%%%%%%%%%%%%%%%%%%%%%%%%%
\subsection{Examples}
Our main theorem asserts an equivalence between $1$-surjective functors $\rmB G \to \cC$ and $G$-crossed braided categories. 
Starting with an arbitrary $3$-functor $\pi:\rmB G\to \cC$ we may factor it through a $1$-surjective functor $\pi':\rmB G\to \cC'$ (where $\cC'$ is the subcategory of $\cC$ with objects and $1$-morphisms in the essential image of $\pi$, and all $2$- and $3$-morphisms between them) and apply our construction from \S\ref{sec:WeakGCrossedConstruction} to obtain a $G$-crossed braided category. Most examples discussed below arise in this way.

\begin{ex}[Delooped braided monoidal categories]
Let $\cB$ be a braided monoidal category, and denote the corresponding 3-category with one object and one $1$-morphism by $\rmB^2\cB$. 
Observe that every weak $3$-functor $\rmB G \to \rmB^2\cB$ is automatically $1$-surjective. 
Such $3$-functors $\rmB G\to \rmB^2 \cB$ factor through 
the maximal sub-3-groupoid $\rmB^2 \cB^\times$ of $\rmB^2 \cB$, delooping the braided monoidal groupoid $\cB^\times$ of invertible objects and morphisms in $\cB$. 
Assuming the homotopy hypothesis for algebraic trigroupoids,\footnote{
The article \cite{MR2842929} constructs a model category structure on the category of $\Gray$-categories and $\Gray$-functors which restrict to a model structure on $\Gray$-groupoids. 
Even though it is shown that the corresponding homotopy category of $\Gray$-groupoids localized at the $\Gray$-equivalences is equivalent to the category of homotopy 3-types and homotopy classes of continuous maps, to the best of our knowledge, it has not yet been shown that this category is also equivalent to the 1-category whose objects are $\Gray$-groupoids (or algebraic trigroupoids) and whose morphisms are natural equivalence classes of weak 3-functors.
}
such functors correspond to homotopy classes of maps from the classifying space $\rmB G$ to the $1$-connected homotopy 3-type $\rmB^2 \cB^\times$.

Such $1$-connected $3$-types are completely determined by the abelian group $ \pi_2(\rmB^2\cB^\times) = \mathrm{Inv}(\cB)$ of isomorphism classes of invertible objects of $\cB$, 
the abelian group $\pi_3(\rmB^2\cB^\times) = \mathrm{Aut}(1_{\cB})$ of automorphisms of the tensor unit $1_{\cB}$ of $\cB$, and the $k$--invariant $q\in H^4(K(\mathrm{Inv}(\cB),2), \Aut(1_\cB))\cong \mathrm{Quad}(\mathrm{Inv}(\cB), \Aut(1_{\cB}))$, the group of quadratic functions on $\mathrm{Inv}(\cB)$ valued in $\Aut(1_\cB)$ \cite{MR65163}, which is explicitly given by the quadratic function $$
q:\mathrm{Inv}(\cB) \to \mathrm{Aut}(1_{\cB})
\qquad
\text{given by}
\qquad
q(b) := \ev_b \circ \beta^\cB_{b,b^{-1}}\circ \coev_b.
$$ 
%This is actually inverse to the convention in EGNO.
Here, $\ev_b: b^{-1} \otimes b \to I$ and $\coev_b:I \to   b\otimes b^{-1}$ denote a choice of pairing between $b$ and $b^{-1}$ and $\beta_{b,b^{-1}}: b\otimes b^{-1} \to b^{-1}\otimes b$ denotes the braiding.

By~\cite{MR0045115,MR65163}, 
the group 
$\mathrm{Quad}(\mathrm{Inv}(\cB),\Aut(1_\cB))$ 
is further isomorphic to the group $H^3_{ab}(\mathrm{Inv}(\cB),\Aut(1_\cB))$ of \emph{abelian $3$-cocycles}  $(\alpha, \beta)$, consisting of pairs of a group $3$-cocycle $\alpha: \mathrm{Inv}(\cB)^3\to \mathrm{Aut}(1_{\cB})$ and a certain `$\alpha$-twisted-bilinear' form $\beta: \mathrm{Inv}(\cB)^2 \to \mathrm{Aut}(1_{\cB})$. 
We refer the reader to \cite[(1.2) and \S11]{2005.05243} for more details.

By the obstruction theory for homotopy classes of maps into such Postnikov towers (cf.~\cite[Theorem 1.3]{MR2677836}), it follows that, up to natural isomorphism, $3$-functors $\rmB G \to \rmB^2 \cB$ correspond to the following data: 
\begin{itemize}
\item 
a $2$-cocycle $\mu \in Z^2(G, \Inv(\cB))$, up to coboundary;
\item 
a $3$-cochain $\omega \in C^3(G,\Aut(1_{\cB}))$ such that $d\omega = (\alpha, \beta)_*\mu$, where $(\alpha,\beta)_*\mu \in Z^4(G, \Aut(1_{\cB}))$ is the $4$-cocycle in the image of the Pontryagin-Whitehead morphism\footnote{
Under the isomorphism 
$H^4(K(\mathrm{Inv}(\cB),2), \mathrm{Aut}(1_{\cB}))\cong H^3_{ab}(\mathrm{Inv}(\cB),\Aut(1_\cB))$, 
the abelian $3$-cocycle $(\alpha, \beta)$ corresponds to a map $(\alpha, \beta)_*:K(\mathrm{Inv}(\cB),2) \to K(\Aut(1_{\cB}), 4)$.
From this perspective, the Pontryagin-Whitehead morphism $(\alpha, \beta)_*:H^2(G,\mathrm{Inv}(\cB)) \to H^4(G, \Aut_\cB(1_\cB))$ is simply given by postcomposing a class $\omega: \rmB G \to K(\mathrm{Inv}(\cB),2)$ with $(\alpha, \beta)_*$.
} $(\alpha, \beta)_*:H^2(G,\mathrm{Inv}(\cB)) \to H^4(G, \Aut_\cB(1_\cB)) $ for the $k$-invariant $(\alpha, \beta) \in H^4(K(\mathrm{Inv}(\cB),2), \mathrm{Aut}(1_{\cB}))$.
An explicit expression for the 4-cocycle $(\alpha, \beta)_*\mu\in Z^4(G, \Aut(1_{\cB}))$ is given by
    \begin{equation}
    \label{eq:o_4}
    \begin{split}
    (\alpha,\beta)_*\mu(g,h,k,\ell)
    =
    \beta_{\mu_{k,\ell}, \mu_{g,h}}
    &
    \alpha^{-1}_{\mu_{ghk,\ell}, \mu_{gh,k},\mu_{g,h}}
    \alpha_{\mu_{ghk,\ell}, \mu_{g,hk},\mu_{h,k}}
    \alpha^{-1}_{\mu_{g,hk\ell}, \mu_{hk,\ell},\mu_{h,k}}
    \\&
    \alpha_{\mu_{g,hk\ell}, \mu_{h,k\ell},\mu_{k,\ell}}
    \alpha^{-1}_{\mu_{gh,k\ell}, \mu_{g,h},\mu_{k,\ell}}
    \alpha_{\mu_{gh,k\ell}, \mu_{k,\ell},\mu_{g,h}}
    \end{split}
    \end{equation}
    This explicit expression can also be obtained, up to conventions, by taking the trivial $G$-action in
    \cite[Eq.~(5.6)]{MR3555361}.
\end{itemize}

In fact, after strictifying $\cB$ to a strict braided monoidal category, so that $\rmB^2 \cB$ is a $\Gray$-category, 
this cohomological data may be directly read off from the components of the weak 3-functor $\pi: \rmB G \to \rmB^2 \cB$, using notation from Appendix \ref{sec:Weak3CategoryCoherences}, as follows:
We may assume the underlying 2-functor of $\pi$ is strictly unital, i.e.,
$\pi^1_g = \id_{1_\cB}$ for all $g\in G$.
By \ref{Functor:A.unital}, this implies $\pi^2_{g,h}=\id_{1_\cB}$ for all $g,h\in G$.
We write $\mu_{g,h}:= \mu^\pi_{g,h}\in \operatorname{Inv}(\cB)$.
By \ref{Functor:mu.unital}, $\mu^\pi_{\id_g, \id_h}= \id\in \End(\mu_{g,h})$.
Using the isomorphism
$ \omega^\pi_{g,h,k} :\mu_{gh,k}\xz \mu_{g,h} \to \mu_{g,hk}\xz \mu_{h,k},$ $\mu$ descends to a 2-cocycle in $Z^2(G, \operatorname{Inv}(\cB))$. To translate $\omega^\pi$ into a $3$-cochain in $C^3(G, \Aut_\cB(1_\cB))$, we let $\cC$ be a skeletalization of $\cB^\times$. 
In $\cC$, we may identify all automorphism spaces of $\cC$ with $\Aut(1_\cB)$, and hence recover the associator $\alpha$ in $\cC$ as an element of $Z^3(\Inv(\cC), \Aut(1_\cB))$, and descend the isomorphisms $\omega^\pi_{g,h,k} :\mu_{gh,k}\xz \mu_{g,h} \to \mu_{g,hk}\xz \mu_{h,k}$ to a 3-cochain $\omega$ in $C^3(G, \Aut_\cB(1_\cB))$. 
Unpacking\footnote{Unpacking 
\ref{Functor:PentagonCoherence}
in $\cC$
introduces six additional associator terms, one for every vertex of the hexagon commutative diagram.
As these terms correspond to the two different ways to associate each of the vertex 1-cells in \ref{Functor:PentagonCoherence}, the associators alternate $\alpha$ and $\alpha^{-1}$ around the diagram.
The resulting 12 sided commutative diagram exactly reproduces, up to conventions, a simplification of
\cite[Fig.~1]{MR3555361} where the $G$-action is trivial.
Five of these 12 terms give $d\omega$, while the other 7 terms give \eqref{eq:o_4}.
Since the diagram commutes, we have $d\omega = (\alpha, \beta)_*\mu$ as desired.
}
\ref{Functor:PentagonCoherence} leads to $d\omega = (\alpha, \beta)_*\mu$.

We can now explicitly describe the $G$-crossed braided category resulting from our construction from this cohomological data by interpreting the diagrams 
\eqref{eq:WeakGCrossedTensor}, 
\eqref{eq:WeakGCrossedAssociator}, 
\eqref{eq:WeakGCrossedAction},
\eqref{eq:WeakGCrossedTensorator},
\eqref{eq:WeakGCrossedBraiding}.
\begin{itemize}
\item 
All $g$-graded components are $\cB$,
\item 
the monoidal structure is given by interpreting \eqref{eq:WeakGCrossedTensor}:
$a_g \otimes b_h := \mu_{g,h} \otimes a_g \otimes b_h$,
with associator given by interpreting \eqref{eq:WeakGCrossedAssociator}:
$$
\mu_{gh,k} \otimes \mu_{g,h} \otimes a_g \otimes b_h \otimes c_k
\xrightarrow{\omega_{g,h,k}^\pi\otimes \id}
\mu_{g,hk} \otimes \mu_{h,k} \otimes a_g \otimes b_h \otimes c_k
\xrightarrow{\id\otimes \beta^{-1}_{a_g,\mu_{h,k}}\otimes \id}
\mu_{g,hk} \otimes a_g \otimes \mu_{h,k} \otimes b_h \otimes c_k.
$$
\item
the $G$-action is given by interpreting \eqref{eq:WeakGCrossedAction}:
$F_g(b_h):= \mu_{gh, g^{-1}} \otimes \mu_{g,h} \otimes b_h \otimes \mu_{g,g^{-1}}^{-1}$,
with tensorator $\psi^g$ given by interpreting \eqref{eq:WeakGCrossedTensorator}.
\item
the $G$-crossed braiding is given by interpreting \eqref{eq:WeakGCrossedBraiding}.
\end{itemize}

One can view the resulting $G$-crossed extension as a twisting of the trivial extension by a $2$-cocycle \cite[Pf.~of~Thm.~1.3]{MR2677836}.
When $\cB$ is fusion, this is a $G$-crossed  \emph{zesting} of the trivial $G$-crossed extension $\cB\boxtimes \mathrm{Vec}(G)$ of $\cB$ \cite{2005.05544}.
\end{ex}

\begin{ex}[Generalized relative center construction]
\label{ex:Galindo}
The article~\cite{MR3936135} shows that every (weak) $G$-action on a $2$-category may be strictified to a strict $G$-action on a strict $2$-category, encoded by a group homomorphism $\pi: G \to \Aut^{\st}(\cB)$, where $\Aut^{\st}(\cB)$ is the group of strict 2-equivalences of $\cB$ which admit strict inverses.
From such a strict $G$-action, the authors then construct a $G$-crossed braided monoidal category
$Z_G(\cB)$ whose $g$-graded component is the category of pseudonatural transformations and modifications $\mathsf{PseudoNat}(\id_\cB \Rightarrow \pi(g)).
$ 
Since $\pi(e) = \id_\cB$, the trivial graded component is the Drinfeld center $Z(\cB)$.
This construction generalizes the construction of the relative center $Z_\cC(\cD)$ of a $G$-extension $\cD$ of a fusion category $\cC$;
by \cite{MR2587410}, $Z_\cC(\cD)$ is a $G$-crossed braided fusion category whose trivial graded component is $Z(\cC)$. 

Our construction of a $G$-crossed braided monoidal category from a $G$-pointed $3$-category may be understood as a generalization of~\cite{MR3936135} from $G$-actions on $2$-categories, encoded by $3$-functors $\rmB G\to 2\Cat$ from $\rmB G$ into the $3$-category of $2$-categories, to arbitrary $3$-functors $\rmB G \to \cC$. 
In particular, we show in Section~\ref{sec:Strictifying1Morphisms} that we may strictify a $1$-surjective weak $3$-functor $\rmB G \to \cC$ to a $\Gray$-functor $\rmB G \to \cC'$ into a $\Gray$-category $\cC'$ equivalent to $\cC$, and construct a $G$-crossed braided category from this data.
\end{ex}

\begin{ex}[$G$-crossed extension theory for braided fusion categories]
\label{ex:ENO}
Let $\cC$ be a braided fusion category, and consider the monoidal 2-category $\mathsf{Mod}(\cC)$ of finite semisimple module categories~\cite{MR2678824,1812.11933}. 
Given a monoidal $2$-functor $\pi:G\rightarrow  \mathsf{Mod}(\cC)$, our construction produces the $G$-crossed braided fusion category
$$
\bigoplus_{g\in G}
\Hom(\cC_\cC \to \pi(g)_\cC)
\cong
\bigoplus_{g\in G} \pi(g)
$$
which is a $G$-crossed braided extension of the $e$-graded piece
$\operatorname{End}_{\mathrm{Mod}(\cC)}(\cC_{\cC})\cong\cC$. 
This $G$-crossed braided category is equivalent to the $G$-crossed extension constructed in~\cite{MR2677836} (which moreover gives an alternate proof that faithful $G$-crossed extensions of braided fusion categories are in fact classified by monoidal 2-functors $G \to \mathsf{Mod}(\cC)$).
\end{ex}

\begin{ex}[Permutation crossed extensions]
\label{ex:PermutationCrossedExtension}
Let $\cC$ be a symmetric monoidal 3-category, and let $A$ be an object of $\cC$. 
Then there exists a monoidal $2$-functor $\pi: S_{n}\rightarrow \End(A^{\boxtimes n})$, where $\boxtimes$ denotes the symmetric monoidal product in $\cC$. 
Our construction produces a $S_n$-crossed braided category whose trivially graded piece is 
$\End(\id_{A^{\boxtimes n}})$.
For example, if $\cA$ is an object in the 3-category of fusion categories \cite{1312.7188,MR3650080,MR3590516}, there is an equivalence 
$$
\End(\id_{\cA^{\boxtimes n}}) 
\cong 
Z(\cA^{\boxtimes n}) 
\cong 
Z(\cA)^{\boxtimes n} 
,
$$ 
where $Z(\cA)$ is the Drinfeld center of $\cA$, and the resulting $S_n$-crossed braided category is what is known as a  \emph{permutation crossed extension} of $Z(\cA)^{\boxtimes n}$.
More generally, the article~\cite{MR3959559} shows that such permutation crossed extensions of $\cC^{\boxtimes n}$ exist for any modular tensor category.
\end{ex}

\begin{ex}[Conformal nets]
\label{ex:ConformalNets}
Consider the symmetric monoidal 3-category\footnote{The notion of tricategory used in~\cite{1206.4284,MR3773743}, namely an internal bicategory in $\Cat$, is expected, but not proven to be equivalent to the notion of algebraic tricategory~\cite{MR3076451} used in the present article.
} of coordinate free conformal nets $\mathsf{CN}$ defined in~\cite{1206.4284,MR3439097,MR3656522,MR3927541,MR3773743}.
A $3$-functor $\rmB G \to \mathsf{CN}$ amounts to a conformal net $\cA \in \mathsf{CN}$ together with a generalized action of $G$ on the net $\cA$ by invertible topological defects. 
Applied to such a $3$-functor, our construction produces a $G$-crossed braided category whose trivial graded component is the braided category $\End_{\mathsf{CN}}(1_{\cA}) = \Rep(\cA)$ of (super-selection) sectors~\cite[Sec 1.B]{MR3439097} of $\cA$. 
We expect this generalizes a construction of M\"uger \cite{MR2183964}, which produces a $G$-crossed braided category from the action of global symmetries on a \emph{coordinatized} conformal net.
However, it is difficult to compare these two $G$-crossed braided categories, since it is not obvious how to construct a symmetric monoidal 3-category of coordinatized conformal nets.
\end{ex}

\begin{ex}[Topological phases]
\label{ex:TopologicalPhases}
The collection of (2+1)D gapped topological phases is expected to form a 3-category \cite{MR3978827,1905.09566}. 
Given a global, onsite symmetry, there is an associated $G$-crossed braided category of twist defects \cite{1410.4540}. 
Our construction can be understood as a direct generalization of this heuristic.
Indeed, our pictures and arguments can be viewed as a more mathematically precise version of the arguments and structure given in the physical context (e.g., see \cite[Fig.~7]{1410.4540}).
\end{ex}

\begin{ex}[Homotopy quantum field theory]
\label{ex:HomotopyQFT}
\emph{Homotopy quantum field theories} are topological field theories on bordisms equipped with a map to a fixed target space. 
If this target space is the classifying space $\rmB G$ of a finite group $G$, such field theories are also known as \emph{$G$-equivariant field theories}. 
Following the cobordism hypothesis \cite{MR1355899,MR2555928}, such a fully extended (framed) $3$-dimensional $G$-equivariant topological field theory valued in a fully dualizable symmetric monoidal $3$-category corresponds to a $3$-functor $\rmB G \to \cC$ (i.e. a fully dualizable object $A$ in $\cC$ equipped with an `internal $G$-action', given by a monoidal 2-functor $X:G\to \End_{\cC}(A)$). 
It therefore follows from Theorem~\ref{thm:Main} that to any such field theory, there is an associated $G$-crossed braided category.

In particular, if $\mathsf{Fus}$ is the $3$-category of fusion categories introduced in~\cite{1312.7188}, we expect the $G$-crossed braided category constructed via Theorem~\ref{thm:Main} from a fully extended $G$-equivariant three-dimensional field theory valued in $\mathsf{Fus}$ to coincide with the $G$-crossed braided category constructed in~\cite{1802.08512} by evaluating the field theory on ($G$-structured) circles. 
In particular, if $G$ is trivial, this recovers the construction of the Drinfeld center of a fusion category $\cA$ as ${\mathsf{FusCat}}({}_{\cA}\cA_{\cA} \Rightarrow {}_{\cA} \cA_{\cA})$. 
\end{ex}

\subsection{Outline}
Section~\ref{sec:3Categories} contains basic definitions and a brief introduction to the graphical calculus of $\Gray$-monoids used throughout.

Section \ref{sec:Truncation} proves various strictification results for $1$-surjective pointed $3$-categories (\S\ref{sec:StrictifyingObjects}) and higher morphisms between them (\S\ref{sec:Strictifying1Morphisms}, \S\ref{sec:Strictifying2Morphisms}, \S\ref{sec:Strictifying3Morphisms}, \S\ref{sec:Strictifying4Morphisms}) and shows that $\TriCat_G$ (Definition~\ref{defn:TricatG}) is equivalent to its strict sub-2-category $\TriCat_G^{\st}$ (Corollary~\ref{cor:Contract4CategoryTo2Category}).

Section~\ref{sec:GCrossed} defines the $2$-category $G\CrsBrd$ of $G$-crossed braided categories (\S\ref{defn:Definitions}) and its equivalent full sub-2-category $G\CrsBrd^{\st}$, constructs the strict $2$-functor $\TriCat_G^{\st} \to G\CrsBrd^{\st}$ (\S\ref{sec:FromGBoring3CatsToGCrossedBriadedCats}) and proves that it is an equivalence (\S\ref{sec:2FunctorEquivalence}). 

Section~\ref{sec:inducedprops} discusses how various properties and structures on a $1$-surjective $G$-pointed $3$-category, such as linearity and rigidity, may be translated across the equivalence of Theorem~\ref{thm:Main} to the resulting $G$-crossed braided category.

Appendix~\ref{sec:Weak3CategoryCoherences} unpacks the definitions of (weak) $3$-functors, transformations, modifications and perturbations between $\Gray$-monoids in terms of the graphical calculus. 

Appendices~\ref{sec:CoherenceProofs} and \ref{sec:CoherenceProofsGCrossed} contain most of the coherence proofs from Sections~\ref{sec:Truncation} and~\ref{sec:GCrossed}, respectively.

%%%%%%%%%%%%%%%%%%%%%%%%%%%%%%%%%%%%%%%%%%%%%%%%%%
\paragraph{Acknowledgements}
The authors would like to thank
Shawn Cui,
Nick Gurski,
Niles Johnson,
and
Andr\'e Henriques
for helpful conversations.

This material is based upon work supported by the National Science Foundation under Grant No. DMS-1440140, while the authors DP and DR were in residence at the Mathematical Sciences Research Institute in Berkeley, California, during the Spring 2020 semester.
CJ was supported by NSF DMS grant 1901082.
DP was supported by NSF DMS grant 1654159.
DR is grateful for the financial support and hospitality of the Max Planck Institute for Mathematics where part of this work was carried out.
%auto-ignore
%this ensures the arxiv doesn't try to start TeXing here.
%!TEX root =../Equivalence.tex

%%%%%%%%%%%%%%%%%%%%%%%%%%%%%%%%%%%%%%%%%%%%%%%%%%
%%%%%%%%%%%%%%%%%%%%%%%%%%%%%%%%%%%%%%%%%%%%%%%%%%
%%%%%%%%%%%%%%%%%%%%%%%%%%%%%%%%%%%%%%%%%%%%%%%%%%
\section{Background on 3-categories and monoidal 2-categories}
\label{sec:3Categories}

In this article, by a \emph{$3$-category} we mean an algebraic tricategory in the sense of~\cite[Def 4.1]{MR3076451}, and by \emph{functor}, \emph{transformation}, \emph{modification}, and \emph{perturbation}, we mean the corresponding notions of trihomomorphism, tritransformation, trimodification, and perturbation of~\cite[Def 4.10, 4.16, 4.18, 4.21]{MR3076451}.
We include Appendix \ref{sec:Weak3CategoryCoherences} below which unpacks the full definitions of these notions for $\Gray$-monoids using the graphical calculus discussed in \S\ref{sec:GraphicalCalculus} below.
When we consider stricter notions of categories or functors we will always use appropriate adjectives such as `Gray' or `strict'.

\begin{rem}
\label{rem:InvertibleExtendsToAdjointEq}
In this article, we use the term \emph{invertible} as a property, i.e., the existence of a coherent inverse.
Indeed, by \cite{MR2972968}, every invertible 1-morphism (biequivalence) in a 3-category is part of a biadjoint biequivalence, and every invertible 2-morphism is part of an adjoint equivalence. 
Moreover, there is a contractible space of choices for these coherent inverses.
Whenever we need to make such choices, we will refer back to this remark.
\end{rem}

%%%%%%%%%%%%%%%%%%%%%%%%%%%%%%%%%%%%
\subsection{\texorpdfstring{$\Gray$}{Gray}-categories and \texorpdfstring{$\Gray$}{Gray}-monoids}
\label{sec:GrayCategories}

In this section, we give a terse definition of $\Gray$-category and $\Gray$-monoid, and a brief discussion on the diagrammatic calculus for $\Gray$-monoids.
We refer the reader to \cite{MR2717302} for a more detailed treatment of $\Gray$-categories and to \cite[\S2.6]{1211.0529} or \cite{1409.2148} for a more detailed treatment of the graphical calculus.

\begin{defn}
\label{defn:GrayCategory}
The symmetric monoidal category $\Gray$ is the 1-category 
of strict 2-categories and strict 2-functors
equipped with the Gray monoidal structure \cite[\S5]{MR2717302}. 
A $\Gray$-category is a category enriched in $\Gray$ in the sense of \cite{MR2177301}.
A $\Gray$-monoid is a monoid object in $\Gray$.
Given a $\Gray$-monoid $\cC$, its \emph{delooping} $\rmB \cC$ is the $\Gray$-category with one object and endomorphisms $\cC$.
\end{defn}

We now unpack the notion of $\Gray$-monoid from
Definition \ref{defn:GrayCategory}.

\begin{nota}
\label{nota:CellsInGrayMonoids}
Given a $\Gray$-monoid $\cC$, we refer to its objects, 1-morphisms, and 2-morphisms as 0-cells, 1-cells, and 2-cells respectively in order to distinguish these basic components of $\cC$ from morphisms in an ambient category in which $\cC$ lives.
\end{nota}

The remarks and warning below are adapted directly from \cite{1812.11933}.

\begin{rem}\label{def:monoidal2cat}
Unpacking Definition \ref{defn:GrayCategory}, a $\Gray$-\emph{monoid} consists of the following data:
\begin{enumerate}[label=(D\arabic*)]
\item 
\label{Gray:2cat}
a strict $2$-category $\cC$, where composition of 1-morphisms is denoted by $\xo$ and composition of 2-morphisms is denoted by $\xt$;
\item 
\label{Gray:Id}
an \emph{identity} 0-cell $\Iz \in \cC$;
\item 
\label{Gray:tensor}
strict left and right \emph{tensor product} $2$-functors $L_a= a\xz -$ and $R_a= -\xz a$ for each object $a\in \cC$:
\begin{align*}
L_a&= a\xz -: \cC \to \cC\\
R_a&= -\xz a: \cC \to \cC, 
\end{align*}
\item
\label{Gray:Interchanger}
an \emph{interchanger} $2$-isomorphism
$\phi_{x,y}$ for each pair of $1$-cells $x:a\to b$ and $g:c\to d$:
\[
\phi_{x,y}: \left( x\xz \Io_{d}\right)\xo \left(\Io_a\xz y\right) \To \left(\Io_b \xz y\right)\xo \left(x \xz \Io_c \right)
\]
\end{enumerate}
subject to the following conditions:
\begin{enumerate}[label=(C\arabic*)]
\item 
left and right tensor product agree: 
for all objects $a,b \in \cC$, 
$L_a b = R_b a = a\xz b$;
\item 
tensor product is strictly unital and associative:
\begin{align*} 
&L_{\Iz} = \mathrm{id}_{\cC} = R_{\Iz} \\
&L_a L_b = L_{a\xz b} \\
&R_bR_a = R_{a\xz b} \\
&L_a R_b = R_b L_a;
\end{align*}
\item the interchanger $\phi$ respects identities, i.e., 
for a 0-cell $A\in \cC$ and a 1-cell $f: C \to D$,
\begin{align*}
\phi_{f, \Io_A} &= \It_{f\xz A} \\
\phi_{\Io_A, f} &= \It_{A\xz f}
\end{align*}
\item
\label{Interchanger:Composition}
the interchanger $\phi$ respects composition, i.e., 
for $x:a\to a'$, $x':a'\to a''$, $y:b\to b'$ and $y':b'\to b''$,
\begin{align*}
\phi_{x'\xo x, y}&= \left(\phi_{x',y} \xo (x\xz \Io_b ) \right)\xt \left( (x' \xz \Io_{b'})\xo \phi_{x,y}\right)\\
\phi_{x,y'\xo y} &= \left((\Io_{a'} \xz y')  \xo \phi_{x,y}\right) \xt \left(\phi_{x,y'}\xo ( \Io_a \xz y )  \right)
\end{align*}
\item 
\label{Interchanger:Natural}
the interchanger $\phi$ is natural, i.e., 
for $1$-cells $x,x':a\to a', y,y':b\to b'$ and $2$-cells $\alpha: x\To x'$, $\beta:y\To y'$,
\begin{align*}
\phi_{x',y} \xt \left((\alpha \xz \Io_{b'}) \xo (\Io_{a}\xz y ) \right) &= \left((\Io_{a'}\xz y) \xo (\alpha \xz \Io_b) \right)\xt \phi_{x,y}\\
\phi_{x,y'}\xt \left( \left(  x\xz \Io_{b'}\right) \xo \left(\Io_a \xz  \beta \right) \right) &=  \left( \left( \Io_{a'} \xz \beta \right) \xo \left( x \xz \Io_b \right) \right)\xt\phi_{x,y}
\end{align*}
\item 
the interchanger $\phi$ respects tensor product, i.e., 
for $x:a \to a'$, $y:b\to b'$ and $z:c\to c'$,
\begin{align*}
\phi_{\Io_a\xz y, z} &= \Io_a \xz \phi_{y,z} \\
\phi_{x\xz \Io_b, z} &= \phi_{x, \Io_b\xz z} \\
\phi_{x,y\xz \Io_c} &= \phi_{x,y} \xz \Io_c
\end{align*}
\end{enumerate}
A $\Gray$-monoid is called \emph{linear} 
if the underlying 2-category is linear 
and 
for all objects $a$ the functors $a \xz -$ and $- \xz a$ are linear.
\end{rem}

\begin{warn}[Horizontal composition of 1-morphisms] \label{notation:nudging}
We warn the reader that
the tensor product in a $\Gray$-monoid does \emph{not} provide a unique definition of the tensor product of two 1-cells.
Given 
$x:a\to b$ and $y:c\to d$,
we define 
\begin{equation}
\label{eq:Nudging}
x\xz y
:=
\left(x\xz \Io_d\right)\xo \left(\Io_a\xz y\right);
\end{equation}
this convention is known as \emph{nudging} \cite[\S4.5]{MR1261589}. 
We use a similar nudging convention for the tensor product of 2-cells. 
With this convention, the data of a $\Gray$-monoid $\cC$ as described in Definition~\ref{def:monoidal2cat} gives rise to an (\emph{opcubical} cf~\cite[\S8]{MR3076451}) algebraic tricategory $\rmB \cC$~\cite[Thm.~8.12]{MR3076451}.
\end{warn}

\begin{rem}[Strictification for monoidal 2-categories]
By the strictification for tricategories from \cite{MR1261589} or \cite[Cor.~9.16]{MR3076451},
every (linear) weakly monoidal weak 2-category admits a monoidal 2-equivalence to a (linear) $\Gray$-monoid of the form in Definition~\ref{def:monoidal2cat}.
\end{rem}

%%%%%%%%%%%%%%%%%%%%%%%%%%%%%%%%%%%%%%%%%%%%%%%%%%%%%%%%%%%%%
\subsection{Graphical calculus for \texorpdfstring{$\Gray$}{Gray}-monoids}
\label{sec:GraphicalCalculus}

$\Gray$-categories admit a graphical calculus of surfaces, lines, and vertices in three-dimensional space.
We refer the reader to \cite[\S2.6]{1211.0529} for a rigorous discussion. 
Here, we will only ever work in a two-dimensional projection of this graphical calculus for $\Gray$-monoids. 
Our exposition below follows \cite{1409.2148}.

The 0-cells of our strict 2-category $\cC$ \ref{Gray:2cat} are denoted by strands in the plane
$$
\tikzmath{
\draw (0,-.4) -- node[left] {$\scriptstyle{a}$} (0,.4);
}
$$
and the identity 0-cell $1_\cC$ \ref{Gray:Id} is denoted by the empty strand.
The 1-cells are denoted by coupons between labelled strands
$$
x: a\to b
\qquad\qquad
\tikzmath{
\draw (0,-.6) node[right] {$\scriptstyle{a}$} -- (0,.6) node[right] {$\scriptstyle{b}$};
\filldraw[very thick, fill=white] (0,0) node {$\scriptstyle{x}$} circle (.25cm);
}
$$
The composition of 1-cells is denoted by vertical stacking of such diagrams.

The strict tensor product $\xz$ is denoted by horizontal juxtaposition.
For example, the tensor product functors $L_a$ and $R_a$ \ref{Gray:tensor} are denoted by placing a strand labelled by $a$ to the left or right respectively.
$$
L_a(x: b \to c) :=
\id_a \xz x=
\tikzmath{
\draw (0,-.6) node[right] {$\scriptstyle{b}$} -- (0,.6) node[right] {$\scriptstyle{c}$};
\draw (-.4,-.6) -- node[left] {$\scriptstyle{a}$} (-.4,.6);
\filldraw[very thick, fill=white] (0,0) node {$\scriptstyle{x}$} circle (.25cm);
}
\qquad\qquad
R_a(x: b \to c) :=
x\xz \id_a=
\tikzmath{
\draw (0,-.6) node[left] {$\scriptstyle{b}$} -- (0,.6) node[left] {$\scriptstyle{c}$};
\draw (.4,-.6) -- node[right] {$\scriptstyle{a}$} (.4,.6);
\filldraw[very thick, fill=white] (0,0) node {$\scriptstyle{x}$} circle (.25cm);
}
$$
Given $x: a\to b$ and $y: c\to d$, we define their tensor product using the nudging convention from Warning \ref{notation:nudging}.
$$
x\xz y
:=
\left(x\xz \Io_d\right)\xo \left(\Io_a\xz y\right)
=
\tikzmath{
\draw (0,-1) node[left] {$\scriptstyle{a}$} -- (0,1) node[left] {$\scriptstyle{b}$};
\filldraw[very thick, fill=white] (0,.4) node {$\scriptstyle{x}$} circle (.25cm);
\draw (.6,-1) node[right] {$\scriptstyle{c}$} -- (.6,1) node[right] {$\scriptstyle{d}$};
\filldraw[very thick, fill=white] (.6,-.4) node {$\scriptstyle{y}$} circle (.25cm);
}
$$
Observe that no two coupons ever share the same vertical height.

The 2-cells are inherently 3-dimensional, 
and can be thought of as `movies' between our 2-dimensional string diagrams.
Rather than drawing 2-cells, we denote them by arrows $\Rightarrow$ between diagrams corresponding to their source and target 1-cells.
For example, the interchanger 
$
\phi_{x,y}
$
from
\ref{Gray:Interchanger} 
is simply denoted by 
$$
\tikzmath{
\draw (0,-1) node[left] {$\scriptstyle{a}$} -- (0,1) node[left] {$\scriptstyle{b}$};
\filldraw[very thick, fill=white] (0,.4) node {$\scriptstyle{x}$} circle (.25cm);
\draw (.6,-1) node[right] {$\scriptstyle{c}$} -- (.6,1) node[right] {$\scriptstyle{d}$};
\filldraw[very thick, fill=white] (.6,-.4) node {$\scriptstyle{y}$} circle (.25cm);
}
\overset{\phi_{x,y}}{\Longrightarrow}
\tikzmath{
\draw (0,-1) node[left] {$\scriptstyle{a}$} -- (0,1) node[left] {$\scriptstyle{b}$};
\filldraw[very thick, fill=white] (0,-.4) node {$\scriptstyle{x}$} circle (.25cm);
\draw (.6,-1) node[right] {$\scriptstyle{c}$} -- (.6,1) node[right] {$\scriptstyle{d}$};
\filldraw[very thick, fill=white] (.6,.4) node {$\scriptstyle{y}$} circle (.25cm);
}\,.
$$

\begin{nota}
\label{nota:DashedBoxForWhiskering}
When working with $\Gray$-monoids, one often needs to whisker 2-cells between 1-cells, and the notation can quickly become cumbersome.
Instead, we use the convention of a dashed box when we apply a 2-cell locally to a 1-cell, and we simply label the whiskered 2-cell by the name of the locally applied 2-cell.
Later on, we will draw commutative diagrams whose vertices are 1-cells.
When we want to apply two 2-cells locally in different places to the same 1-cell, we will use two dashed boxes with different colors, usually \textcolor{red}{red} and \textcolor{blue}{blue}.
When one of these two 2-cells is applied to the entire diagram, we do not use a dashed box, and we only use one dashed box of another color, usually \textcolor{red}{red}.
As an explicit example, the second equation in \ref{Interchanger:Composition} in string diagrams is given by:
$$
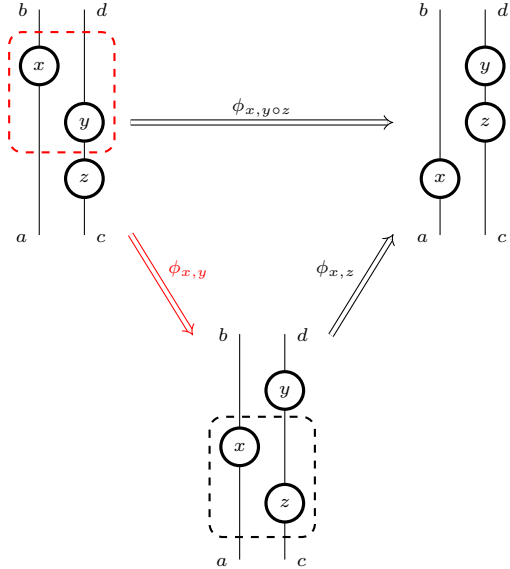
\begin{tikzpicture}[baseline= (a).base]\node[scale=1] (a) at (0,0){\begin{tikzcd}
\tikzmath{
\draw (-.6,-1.5) node[left] {$\scriptstyle{a}$} -- (-.6,1.5) node[left] {$\scriptstyle{b}$};
\filldraw[very thick, fill=white] (-.6,.75) node[yshift=-.05cm] {$\scriptstyle{x}$} circle (.25cm);
\draw (0,-1.5) node[right] {$\scriptstyle{c}$} -- (0,1.5) node[right] {$\scriptstyle{d}$};
\filldraw[very thick, fill=white] (0,0) node[yshift=-.05cm] {$\scriptstyle{y}$} circle (.25cm);
\filldraw[very thick, fill=white] (0,-.75) node[yshift=-.05cm] {$\scriptstyle{z}$} circle (.25cm);
\draw[thick, dashed, red, rounded corners=5pt] (-1,-.4) rectangle (.4,1.2);
}
\arrow[dr, red, Rightarrow, "\phi_{x,y}"]
\arrow[rr, Rightarrow, "\phi_{x,y\xo z}"]
&&
\tikzmath{
\draw (-.6,-1.5) node[left] {$\scriptstyle{a}$} -- (-.6,1.5) node[left] {$\scriptstyle{b}$};
\filldraw[very thick, fill=white] (-.6,-.75) node[yshift=-.05cm] {$\scriptstyle{x}$} circle (.25cm);
\draw (0,-1.5) node[right] {$\scriptstyle{c}$} -- (0,1.5) node[right] {$\scriptstyle{d}$};
\filldraw[very thick, fill=white] (0,.75) node[yshift=-.05cm] {$\scriptstyle{y}$} circle (.25cm);
\filldraw[very thick, fill=white] (0,0) node[yshift=-.05cm] {$\scriptstyle{z}$} circle (.25cm);
}
\\
&
\tikzmath{
\draw (-.6,-1.5) node[left] {$\scriptstyle{a}$} -- (-.6,1.5) node[left] {$\scriptstyle{b}$};
\filldraw[very thick, fill=white] (-.6,0) node[yshift=-.05cm] {$\scriptstyle{x}$} circle (.25cm);
\draw (0,-1.5) node[right] {$\scriptstyle{c}$} -- (0,1.5) node[right] {$\scriptstyle{d}$};
\filldraw[very thick, fill=white] (0,.75) node[yshift=-.05cm] {$\scriptstyle{y}$} circle (.25cm);
\filldraw[very thick, fill=white] (0,-.75) node[yshift=-.05cm] {$\scriptstyle{z}$} circle (.25cm);
\draw[thick, dashed, rounded corners=5pt] (-1,.4) rectangle (.4,-1.2);
}
\arrow[ur, Rightarrow, "\phi_{x,z}"]
\end{tikzcd}};\end{tikzpicture}
$$
\end{nota}

For the convenience of the reader, we have included Appendix \ref{sec:Weak3CategoryCoherences} which unpacks the notions of 3-functor, transformation, modification, and perturbation for $\Gray$-monoids using this graphical calculus.

%auto-ignore
%this ensures the arxiv doesn't try to start TeXing here.
%!TEX root =../Equivalence.tex

%%%%%%%%%%%%%%%%%%%%%%%%%%%%%%%%%%%%%%%%%%%%%%%%%%
%%%%%%%%%%%%%%%%%%%%%%%%%%%%%%%%%%%%%%%%%%%%%%%%%%
%%%%%%%%%%%%%%%%%%%%%%%%%%%%%%%%%%%%%%%%%%%%%%%%%%
\section{Strictifying \texorpdfstring{$G$}{G}-pointed 3-categories}
\label{sec:Truncation}

Let $G$ be a group.
We recall from \S\ref{sec:WeakGCrossedConstruction} that $\rmB G$ denoted the \emph{delooping} of $G$, i.e., $G$ considered as a 1-category with one object.
As discussed at the beginning of \S\ref{sec:3Categories}, the terms $n$-\emph{category} and $n$-\emph{functor} for $n\leq 3$ will always mean weak $n$-categories and weak $n$-functors.
Observe that since a $k$-category may be viewed as an $n$-category for $n\geq k$ with only identity higher morphisms, we may talk about an $n$-functor from a $k$-category to an $n$-category.
Recall from Remark \ref{rem:InvertibleExtendsToAdjointEq} that we use the adjective \emph{invertible} for (bi)adjoint (bi)equivalences.

\begin{defn}
A $3$-functor $A: \cC\to \cD$ is 1-\emph{surjective} if it is essentially surjective on objects and if for every pair of objects $c_1,c_2$ of $\cC$, the $2$-functors $A_{c_1,c_2}: \cC(c_1\to c_2) \to \cD(A(c_1)\to A(c_2))$ are essentially surjective on objects.
\end{defn}

\begin{defn}
\label{defn:TricatG}
Let $G$ be a group. 
We define the 4-category\footnote{\label{footnote:4Category}
All results in this section can be stated and proved at the level of various $2$-categories of $(k-1)$-morphisms, $k$-morphisms and equivalence classes of $(k+1)$-morphisms of $\TriCat_G$; 
we therefore will not show that $\TriCat_G$ forms a $4$-category --- and in fact will not even choose any definition of $4$-category.
We only use the conceptual idea of a $4$-category as an underlying organizational principle for our results.}
 $\TriCat_G$ of \emph{$G$-pointed 3-categories} to be the full sub-4-category of the under-category $\TriCat_{\rmB G/}$ on the $1$-surjective $3$-functors $\rmB G \to \cC$. 
Explicitly, this $4$-category can be described as follows:
\begin{itemize}
\item 
objects are 3-categories $\cC$ equipped with a 1-surjective 3-functor $\pi^\cC: \rmB G \to \cC$.
\item
1-morphisms $(A,\alpha):(\cC, \pi^\cC) \to (\cD, \pi^\cD)$ are pairs where $A: \cC \to \cD$ is a 3-functor and $\alpha: \pi^\cD\Rightarrow A \circ \pi^\cC $ is an invertible natural transformation;
    
\item
2-morphisms $(\eta, m): (A,\alpha) \Rightarrow (B,\beta)$ are pairs where $\eta: A \Rightarrow B$ is a natural transformation and $m$ is an invertible modification
\begin{equation}
\label{eq:ModificationConditionFor2Morphisms}
\begin{tikzcd}
\rmB G\arrow[rr, "\pi^\cC"]
\arrow[ddrr, swap, "\pi^\cD"]
&&
\cC\arrow[dd,"B"]
\arrow[dl,Leftarrow,shorten <= 1em, shorten >= 1em, "\beta"]
\\
&\mbox{}&
\\
&&
\cD.
\end{tikzcd}
\qquad
\overset{m}{\Rrightarrow}
\qquad
\begin{tikzcd}
\rmB G\arrow[rr, "\pi^\cC"]
\arrow[ddrr, swap, "\pi^\cD"]
&&
\cC\arrow[dd,swap,"A"]
\arrow[dl,Leftarrow,shorten <= 1em, shorten >= 1em, "\alpha"]
\arrow[dd,bend left = 90, "B"]
\\
&\mbox{}&
\arrow[r,Rightarrow,shorten >= 1em, "\!\!\!\!\eta"]
&\mbox{}
\\
&&
\cD
\end{tikzcd}
\,.
\end{equation}

\item
3-morphisms $(p,\rho): (\eta, m) \Rrightarrow (\zeta, n)$ are modifications $p: \eta \Rrightarrow \zeta$  together with an invertible perturbation $\rho$:
\begin{equation}
\begin{tikzcd}
\pi^\cD\arrow[rr,Rightarrow, "\alpha"]
\arrow[ddrr, Rightarrow,swap, "\beta"]
&&
A\circ \pi^\cC\arrow[dd,Rightarrow, "\zeta\circ \pi^\cC"]
\\
&
\mbox{}
\arrow[ur,triplecd,shorten <= 1em, shorten >= 1em, "n"]
&
\\
&&
B\circ \pi^\cC.
\end{tikzcd}
\qquad
\overset{\rho}{\RRightarrow}
\qquad
\begin{tikzcd}
\pi^\cD\arrow[rr,Rightarrow, "\alpha"]
\arrow[ddrr, Rightarrow,swap, "\beta"]
&&
A\circ \pi^\cC\arrow[dd,swap,Rightarrow, "\eta\circ \pi^\cC"]
\arrow[dd,Rightarrow, bend left = 90, "\zeta \circ \pi^\cC"]
\\
&
\arrow[ur,triplecd,shorten <= 1em, shorten >= 1em, "m"]
\mbox{}
&
\arrow[r,triplecd,shorten >=1em, "\!\!\!\!p \circ \pi^\cC"]
&\mbox{}
\\
&&
B\circ \pi^\cC
\end{tikzcd}
\end{equation}

\item
4-morphisms $\xi:(p,\rho) \RRightarrow (q, \delta)$ are perturbations $\xi: p \RRightarrow q$ satisfying
\begin{equation}
\label{eq:4MorphismCriterion}
\begin{tikzcd}
\beta\arrow[rr,triplecd, "m"]
\arrow[ddrr, triplecd,swap, "n"]
&&
(\eta\circ \pi^\cC)* \alpha\arrow[dd,triplecd,"(q\circ \pi^\cC) *\alpha"]
\\
&\mbox{}
\arrow[ur,quadruplecd,shorten <= 1em, "\delta"]
&
\\
&&
(\zeta\circ \pi^\cC)*\alpha.
\end{tikzcd}
\qquad
=
\qquad
\begin{tikzcd}[column sep=1em]
\beta\arrow[rrrr,triplecd, "m"]
\arrow[ddrrrr, swap,triplecd, "n"]
&&&&
(\eta\circ \pi^\cC)*\alpha\arrow[dd,swap,triplecd,"(p\circ \pi^\cC)*\alpha"]
%\arrow[dd, triplecd,bend left = 90, "(q \circ \pi^\cC)* \alpha"]
\\
&&\mbox{}
\arrow[urr,quadruplecd,shorten <= 1em, "\rho"]
&&
\arrow[r,quadruplecd, "\!\!\!\!(\xi\circ \pi^\cC)*\alpha"]
&\mbox{}
\\
&&&&
(\zeta\circ \pi^\cC)*\alpha
\end{tikzcd}
\hspace{-1cm}\raisebox{-.16cm}{
$\tikzmath{
\clip (.125,.125) arc (-20:-60:.25cm) -- (-.1,-.15) -- (2.3,-.15) -- (2.3,2.2) -- (-.15,2.2) -- (.125,.125);
\clip (.125,-.125) arc (20:60:.25cm) -- (-.1,.15) -- (-.1,2.2) -- (2.3,2.2) -- (2.3,-.15) -- (.125,-.125);
\draw[triple] (0,2.15)  .. controls ++(0:1cm) and ++(0:1cm) .. node[right] {$\scriptstyle (q\xo \pi^\cC)\xt \alpha$} (0,0);
\draw[line width=.3mm] (.125,.125) arc (-20:-60:.25cm);
\draw[line width=.3mm] (.125,-.125) arc (20:60:.25cm);
}$
}
\end{equation}

\end{itemize}
\end{defn}
\begin{remark}
\label{rem:No4Categories}
As stated, Definition~\ref{defn:TricatG} and Theorem~\ref{thm:tricatpt} below assume the existence of a (weak) $4$-category $\TriCat$ of algebraic tricategories, trifunctors, tritransformations, modifications, and perturbations which has the appropriate homotopy bicategories between parallel $k$-morphisms. Assuming the existence of such a $4$-category $\TriCat$, we may define $\TriCat_G$ as a certain full sub-4-category of the under-category as in Definition~\ref{defn:TricatG}. 
In Theorem~\ref{thm:tricatpt} we show, working a bicategory at a time, that this 4-category $\TriCat_G$ is equivalent to a sub-$4$-category $\TriCat_G^{\pt}$ with only identity $3$- and $4$-morphisms, and is hence equivalent to a bicategory. 
After having established Theorem~\ref{thm:tricatpt}, we will from then on only work with this bicategory $\TriCat_G^{\pt}$.

Unfortunately, to the best of our knowledge, such a $4$-category $\TriCat$ has not yet been constructed in any of the established models of weak $4$-category.
However, none of the results in this article truly depend on the specifics of $4$-categories, and $4$-categories only appear as a convenient conceptual organizing tool. 

The reader uncomfortable with this sort of model-independent argument may unpack the statement of our main Theorem~\ref{thm:Main} to assert the following:
\begin{enumerate}[label=(\arabic*)]
    \item For a pair of parallel `$1$-morphisms' as in Definition~\ref{defn:TricatG}, the bicategory of $2$-morphisms, $3$-morphisms and $4$-morphisms between them is equivalent to a set.
    \item The bicategory of objects, $1$-morphisms and $2$-morphisms up to invertible $3$-morphisms of Definition~\ref{defn:TricatG} is equivalent to the $2$-category of $G$-crossed braided categories.
\end{enumerate}
\end{remark}

\begin{thm}
\label{thm:tricatpt}
The 4-category $\TriCat_G$
is equivalent to the 4-subcategory $\TriCat_G^{\pt}$ where
\begin{itemize}
\item 
objects $(\rmB \cC, \pi^\cC)$ are those objects of $\TriCat_G$ for which the $3$-category is a $\Gray$-category with one object, and hence given by $\rmB \cC$ for some $\Gray$-monoid $\cC$, and for which $\pi^\cC: \rmB G \to \rmB\cC$ is a strictly $1$-bijective $\Gray$-functor.
Equivalently, an object is a $\Gray$-monoid $\cC$ %with precisely one object $*_{\cC}$, 
whose set of $0$-cells is $\{g_{\cC}:=\pi^\cC(g)\}_{g\in G}$ %with $\id_{*_{\cC}}=e_\cC$ 
and composition of 0-cells given by group multiplication.
\item
$1$-morphisms $(A,\alpha): (\rmB\cC,\pi^\cC) \to (\rmB\cD, \pi^\cD)$ 
satisfy:
\begin{itemize}
\item
%$A(*_{\cC})=*_{\cD}$ and 
$A(g_\cC) = g_\cD$ for all $g\in G$, 
\item
the adjoint equivalence
$\mu: \xz_\cD \circ (A \times A) \Rightarrow A \circ \xz_\cC$
satisfies
$\mu_{g_\cC,h_\cC} = \id_{gh_\cD}: g_\cD\xz h_\cD \Rightarrow gh_\cD$, 
\item
the adjoint equivalence $\iota^A=(\iota^A_*, \iota^A_1): I_\cD \Rightarrow A \circ I_\cC$ 
satisfies $\iota^A_* =\id_{e_\cD}$, and $\iota^A_{1}=A^1_e$, 
\item
the associators and unitors $\omega^A, \ell^A,r^A$ are identities,
\item
$\alpha_* = e_{\cD}$ and $\alpha_g = \id_{g_{\cD}}$, and $\alpha_{\id_g} = A^1_g$,
\item
$\alpha^1 = \id_{\id_{e_\cD}}$ and $\alpha^2_{g,h} = \id_{\id_{gh_\cD}}$ for all $g,h\in G$.
\end{itemize}
\item
$2$-morphisms $(\eta,m): (A,\alpha) \To (B,\beta)$ satisfy 
$\eta_{*_{\cC}} = e_\cD$, $\eta_{g_\cC} = \id_{g_{\cD}}$, $\eta^1 = \id_{\id_{e_\cD}}$ and $\eta^2_{g_{\cC},h_{\cC}} = \id_{\id_{gh_\cD}}$
and
$m_* = e_\cD$, $m_g = \id_{\id_{g_\cD}}$.
That is, $m$ is the identity modification.
\item
3-morphisms $(p,\rho): (\eta,m) \Rrightarrow (\zeta,n)$
satisfy
$p_{*_{\cC}}=\id_{e_\cD}$, $p_{g_{\cC}}=\id_{\id_{g_\cD}}$, 
and
$\rho_* = \id_{\id_{e_\cD}}$.
That is, there are only identity 3-morphisms.
\item
4-morphisms $\xi:(p,\rho)\RRightarrow (q,\delta)$
satisfy
$\xi_{*_{\cC}} = \id_{\id_{e_\cD}}$.
That is, the only 4-endomorphism of an identity 3-morphism is the identity.
\end{itemize}
\end{thm}
\begin{proof}
In 
\S
\ref{sec:StrictifyingObjects}, 
\ref{sec:Strictifying1Morphisms},
\ref{sec:Strictifying2Morphisms},
\ref{sec:Strictifying3Morphisms},
\ref{sec:Strictifying4Morphisms}
below
we show that every object, 1-morphism, 2-morphism, 3-morphism, and 4-morphism respectively in $\TriCat_G$ is equivalent to one of the desired form in $\TriCat_G^{\pt}$.
All proofs in these further subsections amount to checking the appropriate coherences for 3-functors, 3-natural transformations, 3-modifications, and 3-perturbations outlined in Appendix \ref{sec:Weak3CategoryCoherences} and are deferred to Appendix \ref{sec:CoherenceProofs}.
We signify where the reader may find the deferred proof of a statement by including a small box with a link to the appropriate appendix after the statement.
\end{proof}

Since the only $3$- and $4$-morphisms of $\TriCat_G^{\pt}$ are identities, it is evident that $\TriCat_G^{\pt}$ --- and hence by Theorem~\ref{thm:tricatpt} also $\TriCat_G$ --- is $2$-truncated and actually defines a $2$-category. 
In the following corollary, we give a streamlined description of this $2$-category without the redundant data.
\begin{cor}\label{cor:Contract4CategoryTo2Category}
The $4$-category $\TriCat_G^{\pt}$ is isomorphic to the strict $2$-category $\TriCat_G^{\st}$, defined as follows:
\begin{itemize}
\item 
An object is a $\Gray$-monoid $\cC$ %with precisely one object $*_{\cC}$, 
whose set of $0$-cells is $G$ (below, we will denote the elements of $G$ seen as $0$-cells in $\cC$ by $g_{\cC}$) 
%with $\id_{*_{\cC}}=e_\cC$ 
and composition of 0-cells is given by group multiplication.
\item 
A $1$-morphism $A: \cC \to \cD$ is a $3$-functor $A:\rmB\cC \to \rmB\cD$ such that 
\begin{itemize}
\item
%$A(*_{\cC})=*_{\cD}$ and 
$A(g_\cC) = g_\cD$ for all $g\in G$, 
\item
the adjoint equivalence
$\mu: \xz_\cD \circ (A \times A) \Rightarrow A \circ \xz_\cC$
satisfies
$\mu_{g,h} = \id_{gh}: g_\cD\xz h_\cD \Rightarrow gh_\cD$,
\item
the adjoint equivalence $\iota^A=(\iota^A_*, \iota^A_1): I_\cD \Rightarrow A \circ I_\cC$ 
satisfies $\iota^A_* = \id_{e_\cD}$, and $\iota^A_{1}=A^1_e$, 
\item
the associators and unitors $\omega^A, \ell^A,r^A$ are identities.
\end{itemize}    
\item 
A $2$-morphism $\eta: A \To B$ is a natural transformation such that $\eta_* = e_\cD$, $\eta_g = \id_{g_{\cD}}$, $\eta^1 = \id_{\id_{e_\cD}}$ and $\eta^2_{g,h} = \id_{\id_{gh_\cD}}$ for all $g,h\in G$. 
\end{itemize}
Composition of 1- and 2-morphisms is the usual composition of $3$-functors and natural transformations~\cite{MR3076451}.
\end{cor}
\begin{proof}
The natural transformation $\alpha$, the modifications $m$ and $p$ and the perturbations $\rho$ and $\xi$ in the statement of Theorem~\ref{thm:tricatpt} are completely determined by the imposed conditions on their coefficients. 
Moreover, the so defined coefficients always assemble into natural transformations, modifications, and perturbations, respectively, between the respective morphisms described in Corollary~\ref{cor:Contract4CategoryTo2Category}.

We now show that $\TriCat_G^{\st}$ is indeed a strict $2$-category.
Suppose we have two composable 1-morphisms
$(A, A^1, A^2, \mu^A, \iota^A) \in \TriCat_G^{\st}(\cD \to \cE)$
and
$(B, B^1, B^2, \mu^B, \iota^B) \in \TriCat_G^{\st}(\cC \to \cD)$.
Then the formulas for the components for the composite $(A\circ B, (A\circ B)^2, (A\circ B)^2, \mu^{A\circ B}, \iota^{A\circ B})$ are given by 
\begin{align*}
(A \circ B)^1_g
&= 
A(B^1_g) \xt A^1_{g}
&&
\forall\, g\in G
\\
(A \circ B)^2_{x,y} 
&= 
A(B^2_{x,y}) \xt A^2_{B(x), B(y)}
&&
\forall\, x\in \cC(h_\cC \to k_\cC),
\forall\, y \in \cC(g_\cC \to h_\cC)
\\
\mu_{x,y}^{A\circ B}
&=
A(\mu^B_{x,y}) 
\xt 
\mu^A_{B(x),B(y)}  
&&
\forall\, x\in \cC(g_\cC \to k_\cC),
\forall\, y \in \cC(h_\cC \to \ell_\cC)
\\
\iota^{A\circ B}_1 
&=
A(\iota^B_1)\xt \iota^A_1,
\end{align*}
which are easily seen to be strictly associative and strictly unital.
It is also straightforward to see that composition of 2-morphisms is strictly associative and strictly unital as well.
\end{proof}

%%%%%%%%%%%%%%%%%%%%%%%%%%%%%%%%%%%%%%%%
\subsection{Strictifying objects}
\label{sec:StrictifyingObjects}

In the following section, we prove the `object part' of Theorem~\ref{thm:tricatpt} and show that every object $\pi=\pi^\cC
:\rmB G\to \cC$ of the 4-category $\TriCat_{G}$ is equivalent to a strictly $1$-bijective $\Gray$-functor $\pi':\rmB G\to \rmB\cC'$, where $\cC'$ is a $\Gray$-monoid whose set of 0-cells is $G$ with composition the group multiplication. 
The following lemma is a direct consequence of Gurski's strictification of $3$-categories~\cite[Cor.~9.15]{MR3076451}.

\begin{lem}
Any $1$-surjective $3$-functor $\pi: \rmB G \to \cC$
is equivalent, in $\TriCat_G$, to a $1$-surjective $3$-functor $\pi':\rmB G \to \rmB\cC'$ where $\cC'$ is a $\Gray$-monoid.
\end{lem}
\begin{proof}
By~\cite[Cor.~9.15]{MR3076451}, there is a $\Gray$-category $\cC_0'$ and a $3$-equivalence $\cC \to \cC_0'$. 
By $1$-surjectivity of $\pi$, it follows that the composite $\rmB G \to \cC \to \cC_0'$ factors through the full endomorphism $\Gray$-monoid $\cC'$ of $\cC'_0$ on the single object in the image of the composite, resulting in a $3$-functor $\pi': \rmB G \to \rmB\cC'$ which is equivalent to $\pi: \rmB G \to \cC$ in $\TriCat_G$.
\end{proof}

To further strictify $\pi: \rmB G  \to \rmB\cC$, we use the following direct consequence of a theorem of Buhn\'e~\cite{1408.3481}.
Recall that a $3$-functor $F:\cA\to \cB$ between $\Gray$-categories $\cA$ and $\cB$ is \emph{locally strict} if the $2$-functors $F_{a,b}: \cA(a\to b) \to \cB(F(a)\to F(b))$ are strict.
\begin{prop}
\label{prop:Monoidal2FunctorStrictification}
Given $\Gray$-monoids $\cG,\cC$
and a locally strict 3-functor $\pi: \rmB\cG\to \rmB\cC$,
there exists a $\Gray$-monoid $\cC'$, 
an equivalence $A:\rmB\cC \to \rmB\cC'$, a $\Gray$-functor $\pi':\rmB\cG \to \rmB\cC'$ and a natural isomorphism $\pi' \To A \circ \pi$.
\end{prop}
\begin{proof}
By~\cite[Thm. 8]{1408.3481}, every locally strict $3$-functor from a (small) $\Gray$-category into a cocomplete $\Gray$-category is equivalent to a $\Gray$-functor. 
Here, \emph{cocomplete} is used in the sense of enriched category theory~\cite[\S3.2]{MR2177301}.

Given two $\Gray$-categories $\cA,\cB$, we denote by $[\cA,\cB]$ the $\Gray$-category of $\Gray$-functors $\cA\to\cB$.
Consider the 
$\Gray$-enriched Yoneda embedding
$y:\rmB \cC \to  [\rmB \cC^{\op}, \Gray]$, where the target is cocomplete as $\Gray$ is cocomplete \cite[\S3.3]{MR2177301}.
The composite 
$$
\rmB \cG \xrightarrow{\pi} \rmB \cC \xrightarrow{y} [(\rmB \cC)^{\op}, \Gray]
$$ 
is a composite of a locally strict 3-functor with a $\Gray$-functor and hence itself locally strict. Therefore, there is a $\Gray$-functor $\pi': \rmB \cG \to [(\rmB \cC)^{\op}, \Gray]$ which is equivalent to the composite. 

Now we define $\rmB\cC'$ to be the full sub-$\Gray$-category of $[(\rmB \cC)^{\op}, \Gray]$ on the object $\pi'(*)$ and define $\pi': \rmB \cG \to \rmB \cC'$ as the codomain-restriction of $\pi'$ to $\rmB \cC'$.
Finally, observe that both the $\Gray$-Yoneda embedding $y: \rmB \cC \to [(\rmB \cC)^{\op},\Gray]$ and the inclusion $\rmB \cC' \to [(\rmB\cC)^{\op}, \Gray]$ are fully faithful\footnote{
Here, by a \emph{fully faithful $\Gray$-functor} we mean a $\Gray$-functor $F:\cA \to \cB$ whose induced 2-functors $F_{a,b}:\cA(a\to b) \to \cB(F(a)\to F(b))$ are isomorphisms in $\Gray$.} 
$\Gray$-functors which map the single objects of $\rmB\cC$ and $\rmB\cC'$ to equivalent objects. Hence, there is an equivalence $A: \rmB\cC \to \rmB\cC'$ and a natural isomorphism $\alpha: \pi' \To A \circ \pi$.
\end{proof}

\begin{rem}
In general, we cannot get rid of the local strictness assumption on $\pi$ by the example given in \cite[Ex 2.2]{BuhneThesis}.
\end{rem}

\begin{thm}[Strictifying objects]
\label{thm:StrictifyingObjects}
Every object $(\cC, \pi)\in \TriCat_G$ is equivalent to an object $(\rmB \cC', \pi')$ of the subcategory $\TriCat_G^{\pt}$ where $\pi': \rmB G \to \rmB\cC'$ is a strictly $1$-bijective $\Gray$-functor into a $\Gray$-monoid $\cC'$ whose set of 0-cells is $G$ with composition the group multiplication. 
\end{thm}
\begin{proof}
Since $\rmB G$ is a $1$-category, it follows from \cite[Cor 2.6]{BuhneThesis} that every 3-functor $\rmB G \to \cC$ is equivalent to a locally strict 3-functor.
Applying Proposition~\ref{prop:Monoidal2FunctorStrictification}, we obtain a $\Gray$-monoid $\cD$ and a $\Gray$-functor $\pi^\cD: \rmB G \to \rmB\cD$ such that $(\rmB \cD, \pi^{\cD})$ is equivalent to $(\cC,\pi)$ in $\TriCat_G$.

Let $\cD'$ be the full 2-subcategory of $\cD$ whose objects are exactly those in the image of $\pi^\cD$.
Since $\pi^\cD$ is a $\Gray$-functor, $\cD'$ is a $\Gray$-submonoid of $\cD$, which comes equipped with the corestricted $\Gray$-functor $\pi^{\cD'}: \rmB G \to \rmB \cD'$ which is strictly 1-surjective, i.e., onto $\Ob(\cD')$.
Since $\pi^\cD$ is 1-surjective, $(\rmB \cD, \pi^\cD)$ is equivalent in $\TriCat_G$ to $(\rmB \cD', \pi^{\cD'})$.

Since $\pi^{\cD'}: \rmB G \to \rmB \cD'$ is a strictly $1$-surjective $\Gray$-functor, there is in particular a surjective homomorphism $\phi: G \to \Ob (\cD')$.
We define a $\Gray$-monoid $\cC'$ as follows.
The $0$-cells of $\cC'$ are the elements of $G$, and hom categories are given by $\Hom_{\cC'}(g \to h):= \Hom_{\cD'}(\phi(g) \to \phi(h))$.
Since $\phi$ is a homomorphism, $\cC'$ inherits a $\Gray$-monoid structure from $\cD'$ together with an obvious strictly 1-bijective $\Gray$-homomorphism $\pi': \rmB G \to \rmB \cC'$.
Since $\phi$ is surjective, $(\rmB \cC',\pi')$ is equivalent to $(\rmB \cD',\pi^{\cD'})$ in $\TriCat_G$.
\end{proof}

%%%%%%%%%%%%%%%%%%%%%%%%%%%%%%%%%%%%%%%%
\subsection{Strictifying 1-morphisms}
\label{sec:Strictifying1Morphisms}
Given objects $(\rmB \cC, \pi^\cC)$ and $(\rmB \cD, \pi^{\cD})$ in $\TriCat_G^{\pt}$ comprised of $\Gray$-monoids $\cC$ and $\cD$ and a strictly $1$-bijective $\Gray$-functor $(A,\alpha) \in \TriCat_G\left((\rmB\cC, \pi^\cC) \to (\rmB\cD, \pi^\cD)\right)$, we construct a 1-morphism $(B,\beta)\in \TriCat_G^{\pt}$
and an equivalence $(B,\beta) \Rightarrow (A,\alpha)$.
As $\cC,\cD$ are $\Gray$-monoids, we make heavy use of the graphical calculus discussed in \S\ref{sec:GraphicalCalculus}.

Recall that the 3-functor $A$ consists of the data from Definition \ref{defn:3functor}.
The invertible natural transformation $\alpha:\pi^\cD\Rightarrow A\circ \pi^\cC$ is comprised of 
the data from Definition \ref{defn:Transformation}.
We depict $\alpha_*$ by an oriented \textcolor{red}{red} strand:
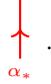
$$
\tikzmath{
\draw[thick, red, mid>] (0,-.4) node [below] {$\scriptstyle\alpha_*$} -- (0,.4) ;
}\,.
$$
By the third unitality bullet point in \ref{Transformation:eta_c}, we have that $\alpha_{\id_g}=A^1_g$ since $\pi^\cC$ is strict.
By Remark \ref{rem:InvertibleExtendsToAdjointEq},
there is a contractible choice of ways to extend
the invertible 0-cell $\alpha_*$ to a biadjoint biequivalence \eqref{eq:BiadjointBiequivalence};
we do so arbitrarily.

We now define $B: \rmB\cC \to \rmB\cD$ as follows.
First, %$B(*) := *$ and 
$B(g_\cC):=g_\cD$ for all $g\in G$.
Given $x \in \cC(g_\cC \to h_\cC)$, we define
$$
B\left(
\tikzmath{
\draw (0,-.4) node[below] {$\scriptstyle g_\cC$} -- (0,.4) node[above] {$\scriptstyle h_\cC$};
\roundNbox{fill=white}{(0,0)}{.2}{0}{0}{$\scriptstyle x$}
}\right)
:=
\tikzmath{
\draw[thick, red, mid>] (.5,0) .. controls ++(90:.5cm) and ++(90:.5cm) .. (-.5,.3) -- (-.5,-.3) .. controls ++(270:.5cm) and ++(270:.5cm) .. (.5,0);
\draw (0,-.7) node[below] {$\scriptstyle g_\cD$} -- (0,.7) node[above] {$\scriptstyle h_\cD$};
\roundNbox{fill=white}{(0,0)}{.25}{.05}{.05}{$\scriptstyle A(x)$}
}
=
\tikzmath{
\draw[thick, red, mid<] (-.1,-.8) arc (0:-180:.2cm)  -- (-.5,0) node [left] {$\scriptstyle \alpha_*^{-1}$} -- (-.5,.8) arc (180:0:.2cm);
\draw[thick, red, mid<] (.1,.6) .. controls ++(-45:.1cm) and ++(90:.4cm) .. (.5,0) node [right] {$\scriptstyle \alpha_*$} .. controls ++(270:.4cm) and ++(45:.1cm) .. (.1,-.6);
\draw (0,-1.2) node[below] {$\scriptstyle g_\cD$} -- (0,1.2) node[above] {$\scriptstyle h_\cD$};
\roundNbox{fill=white}{(0,0)}{.25}{.05}{.05}{$\scriptstyle A(x)$}
\roundNbox{fill=white}{(0,-.7)}{.2}{0}{0}{$\scriptstyle \alpha_g$}
\roundNbox{fill=white}{(0,.7)}{.2}{.1}{.1}{$\scriptstyle \alpha_h^{-1}$}
}\,.
$$
Given $x,y\in \cC(g_\cC \to h_\cC)$ and $f\in \cC(x \Rightarrow y)$, we define $B(f)$ to be the following 2-cell in $\cD$:
$$
\tikzmath{
\draw[thick, red, mid>] (.7,0) .. controls ++(90:.5cm) and ++(90:.5cm) .. (-.7,.5) -- (-.7,-.5) .. controls ++(270:.5cm) and ++(270:.5cm) .. (.7,0);
\draw (0,-.9) node[below] {$\scriptstyle g_\cD$} -- (0,.9) node[above] {$\scriptstyle h_\cD$};
\roundNbox{fill=white}{(0,0)}{.25}{.05}{.05}{$\scriptstyle A(x)$}
\roundNbox{dashed}{(0,0)}{.35}{.05}{.05}{}
}
\overset{A(f)}{\Rightarrow}
\tikzmath{
\draw[thick, red, mid>] (.7,0) .. controls ++(90:.5cm) and ++(90:.5cm) .. (-.7,.5) -- (-.7,-.5) .. controls ++(270:.5cm) and ++(270:.5cm) .. (.7,0);
\draw (0,-.9) node[below] {$\scriptstyle g_\cD$} -- (0,.9) node[above] {$\scriptstyle h_\cD$};
\roundNbox{fill=white}{(0,0)}{.25}{.05}{.05}{$\scriptstyle A(y)$}
}
\,.
$$
For $g\in G$, we define $B^1_g \in \cD(\id_{g_\cD} \Rightarrow B(\id_{g_\cC}))$ to be the composite
\begin{equation}
    \label{eq:B1}
\tikzmath{
\draw (0,-.7) node[below] {$\scriptstyle g_\cD$} -- (0,.7);
}
\Rightarrow
\tikzmath{
\draw (.6,-.7) node[below] {$\scriptstyle g_\cD$} -- (.6,.7);
\draw[thick, red, mid>] (.4,0) .. controls ++(90:.4cm) and ++(90:.4cm) .. (-.4,.2) -- (-.4,-.2) .. controls ++(270:.4cm) and ++(270:.4cm) .. (.4,0);
}
\Rightarrow
\tikzmath{
\draw[thick, red, mid>] (.6,0) .. controls ++(90:.5cm) and ++(90:.5cm) .. (-.6,.3) -- (-.6,-.3) .. controls ++(270:.5cm) and ++(270:.5cm) .. (.6,0);
\draw (0,-.7) node[below] {$\scriptstyle g_\cD$} -- (0,.7) node[above] {$\scriptstyle g_\cD$};
\roundNbox{dashed}{(0,0)}{.25}{.2}{.2}{}
}
\overset{A^1_g}{\Rightarrow}
\tikzmath{
\draw[thick, red, mid>] (.6,0) .. controls ++(90:.5cm) and ++(90:.5cm) .. (-.6,.3) -- (-.6,-.3) .. controls ++(270:.5cm) and ++(270:.5cm) .. (.6,0);
\draw (0,-.7) node[below] {$\scriptstyle g_\cD$} -- (0,.7) node[above] {$\scriptstyle g_\cD$};
\roundNbox{fill=white}{(0,0)}{.25}{.2}{.2}{$\scriptstyle A(\id_g)$}
}
\,.
\end{equation}
For $x\in \cC(g_\cC \to h_\cC)$ and $y\in \cC(h_\cC\to k_\cC)$, we define $B^2_{x,y} \in \cD(B(y)\circ B(x) \Rightarrow B(y\circ x))$ to be the composite
\begin{equation}
    \label{eq:B2}
\tikzmath{
\draw[thick, red, mid>] (.5,.7) .. controls ++(90:.5cm) and ++(90:.5cm) .. (-.5,1) -- (-.5,.4) .. controls ++(270:.5cm) and ++(270:.5cm) .. (.5,.7);
\draw[thick, red, mid>] (.5,-.7) .. controls ++(90:.5cm) and ++(90:.5cm) .. (-.5,-.4) -- (-.5,-1) .. controls ++(270:.5cm) and ++(270:.5cm) .. (.5,-.7);
\draw (0,-1.5) node[below] {$\scriptstyle g_\cD$} -- node[right] {$\scriptstyle h_\cD$} (0,1.5) node[above] {$\scriptstyle k_\cD$};
\roundNbox{fill=white}{(0,.7)}{.25}{.1}{.1}{$\scriptstyle A(y)$}
\roundNbox{fill=white}{(0,-.7)}{.25}{.1}{.1}{$\scriptstyle A(x)$}
}
\Rightarrow
\tikzmath{
\draw[thick, red, mid>] (-.2,0) .. controls ++(90:.3cm) and ++(270:.5cm) .. (.5,.7) .. controls ++(90:.5cm) and ++(90:.5cm) .. (-.5,1) -- (-.5,-1) .. controls ++(270:.5cm) and ++(270:.5cm) .. (.5,-.7) .. controls ++(90:.5cm) and ++(270:.3cm) .. (-.2,0);
\draw (0,-1.5) node[below] {$\scriptstyle g_\cD$} -- node[right] {$\scriptstyle h_\cD$} (0,1.5) node[above] {$\scriptstyle k_\cD$};
\roundNbox{fill=white}{(0,.7)}{.25}{.1}{.1}{$\scriptstyle A(y)$}
\roundNbox{fill=white}{(0,-.7)}{.25}{.1}{.1}{$\scriptstyle A(x)$}
}
\Rightarrow
\tikzmath{
\draw[thick, red, mid>] (.5,0) -- (.5,.7) .. controls ++(90:.5cm) and ++(90:.5cm) .. (-.5,1) -- (-.5,-1) .. controls ++(270:.5cm) and ++(270:.5cm) .. (.5,-.7) -- (.5,0);
\draw (0,-1.5) node[below] {$\scriptstyle g_\cD$} -- (0,1.5) node[above] {$\scriptstyle k_\cD$};
\roundNbox{fill=white}{(0,.6)}{.25}{.1}{.1}{$\scriptstyle A(y)$}
\roundNbox{fill=white}{(0,-.6)}{.25}{.1}{.1}{$\scriptstyle A(x)$}
}
\overset{A^2_{x,y}}{\Rightarrow}
\tikzmath{
\draw[thick, red, mid>] (.7,0) .. controls ++(90:.5cm) and ++(90:.5cm) .. (-.7,.5) -- (-.7,-.5) .. controls ++(270:.5cm) and ++(270:.5cm) .. (.7,0);
\draw (0,-1) node[below] {$\scriptstyle g_\cD$} -- (0,1) node[above] {$\scriptstyle k_\cD$};
\roundNbox{fill=white}{(0,0)}{.25}{.25}{.25}{$\scriptstyle A(y\circ x)$}
}
\,.
\end{equation}

\begin{lem}
\label{lem:Truncation2Functor}
The data $(B, B^1, B^2) : \cC \to \cD$ defines a 2-functor.
\DeferProof{sec:CoherenceProofs1Morphisms}
\end{lem}

We now endow $B$ with the structure of a weak 3-functor $\rmB\cC \to \rmB\cD$.
\begin{construction}
\label{const:B3Functor}
We define an adjoint equivalence $\mu^B: \xz_\cD \circ (B \times B) \Rightarrow B \circ \xz_\cC$ as follows.
First we define $\mu^B_{g,h}\in \cD(g_\cD \xz h_\cD \to gh_\cD)$ to be the identity.
Next, for $x \in \cC(g_\cC\to h_\cC)$ and $y\in \cC(k_\cC \to \ell_\cC)$, we define the natural isomorphism
$\mu^B_{x,y} \in \cD(\mu^B_{g,\ell} \xo(B(x) \xz B(y)) \Rightarrow B( x\xz y)\circ \mu^B_{g,k})$ to be the composite
\begin{align*}
\tikzmath{
\draw[thick, red, mid>] (1.5,0) .. controls ++(90:.5cm) and ++(90:.5cm) .. (.5,.3) -- (.5,-.3) .. controls ++(270:.5cm) and ++(270:.5cm) .. (1.5,0);
\draw[thick, red, mid>] (.5,1.3) .. controls ++(90:.5cm) and ++(90:.5cm) .. (-.5,1.6) -- (-.5,1) .. controls ++(270:.5cm) and ++(270:.5cm) .. (.5,1.3);
\draw (0,-.7) node[below] {$\scriptstyle g_\cD$} -- (0,1.3); \draw (1,-.7) node[below] {$\scriptstyle k_\cD$} -- (1,0); 
\draw (0,1.55) .. controls ++(90:.5cm) and ++(270:.5cm) .. (.485,2.6) -- (.485,3);
\draw (1,.25) -- (1,1.55) .. controls ++(90:.5cm) and ++(270:.5cm) .. (.515,2.6) node [right] {$\scriptstyle \mu^B_{h,\ell}$} -- (.515,3) node[above] {$\scriptstyle h\ell_\cD$};
\roundNbox{fill=white}{(1,0)}{.25}{.05}{.05}{$\scriptstyle A(y)$}
\roundNbox{fill=white}{(0,1.3)}{.25}{.05}{.05}{$\scriptstyle A(x)$}
}
&\Rightarrow
\tikzmath{
\draw[thick, red, mid>] (1.5,0) .. controls ++(90:.5cm) and ++(90:.5cm) .. (.5,.7) -- (.5,-.3) .. controls ++(270:.5cm) and ++(270:.5cm) .. (1.5,0);
\draw[thick, red, mid>] (.5,1.3) .. controls ++(90:.5cm) and ++(90:.5cm) .. (-.5,1.6) -- (-.5,.6) .. controls ++(270:.5cm) and ++(270:.5cm) .. (.5,1.3);
\draw (0,-.7) node[below] {$\scriptstyle g_\cD$} -- (0,1.3); 
\draw (1,-.7) node[below] {$\scriptstyle k_\cD$} -- (1,0);
\draw (0,1.55) .. controls ++(90:.5cm) and ++(270:.5cm) .. (.485,2.6) -- (.485,3);
\draw (1,.25) -- (1,1.55) .. controls ++(90:.5cm) and ++(270:.5cm) .. (.515,2.6) node [right] {$\scriptstyle \mu^B_{h,\ell}$} -- (.515,3) node[above] {$\scriptstyle h\ell_\cD$};
\roundNbox{fill=white}{(1,0)}{.25}{.05}{.05}{$\scriptstyle A(y)$}
\roundNbox{fill=white}{(0,1.3)}{.25}{.05}{.05}{$\scriptstyle A(x)$}
}
\Rightarrow
\tikzmath{
\draw[thick, red, mid>] (1.5,0) .. controls ++(90:.3cm) and ++(270:.3cm) .. (.9,.6) -- (.9,.9) arc (0:180:.1cm) arc (0:-180:.1cm) -- (.5,1.3) .. controls ++(90:.5cm) and ++(90:.5cm) .. (-.5,1.6) -- (-.5,.6) .. controls ++(270:.7cm) and ++(270:.3cm) .. (.3,.6) arc (180:0:.1cm) -- (.5,-.3) .. controls ++(270:.5cm) and ++(270:.5cm) .. (1.5,0);
\draw (0,-.7) node[below] {$\scriptstyle g_\cD$} -- (0,1.3); 
\draw (1,-.7) node[below] {$\scriptstyle k_\cD$} -- (1,0);
\draw (0,1.55) .. controls ++(90:.5cm) and ++(270:.5cm) .. (.485,2.6) -- (.485,3);
\draw (1,.25) -- (1,1.55) .. controls ++(90:.5cm) and ++(270:.5cm) .. (.515,2.6) node [right] {$\scriptstyle \mu^B_{h,\ell}$} -- (.515,3) node[above] {$\scriptstyle h\ell_\cD$};
\roundNbox{fill=white}{(1,-.05)}{.25}{.05}{.05}{$\scriptstyle A(y)$}
\roundNbox{fill=white}{(0,1.3)}{.25}{.05}{.05}{$\scriptstyle A(x)$}
}
\Rightarrow
\tikzmath{
\draw[thick, red, mid>] (1.5,0) .. controls ++(90:.3cm) and ++(270:.4cm) .. (.8,.8) .. controls ++(90:.9cm) and ++(90:.5cm) .. (-.5,1.6) -- (-.5,.6) .. controls ++(270:.7cm) and ++(270:.3cm) .. (.3,.7) arc (180:0:.1cm) -- (.5,-.3) .. controls ++(270:.5cm) and ++(270:.5cm) .. (1.5,0);
\draw (0,-.7) node[below] {$\scriptstyle g_\cD$} -- (0,1.3); 
\draw (1,-.7) node[below] {$\scriptstyle k_\cD$} -- (1,0);
\draw (0,1.55) .. controls ++(90:.5cm) and ++(270:.5cm) .. (.485,2.6) -- (.485,3);
\draw (1,.25) -- (1,1.55) .. controls ++(90:.5cm) and ++(270:.5cm) .. (.515,2.6) node [right] {$\scriptstyle \mu^B_{h,\ell}$} -- (.515,3) node[above] {$\scriptstyle h\ell_\cD$};
\roundNbox{fill=white}{(1,-.05)}{.25}{.05}{.05}{$\scriptstyle A(y)$}
\roundNbox{fill=white}{(0,1.3)}{.25}{.05}{.05}{$\scriptstyle A(x)$}
}
\Rightarrow
\tikzmath{
\draw[thick, red, mid>] (1.3,0) .. controls ++(90:.8cm) and ++(270:.5cm) .. (.1,1.5) -- (.1,2) 
arc (0:180:.3cm)
-- (-.5,-.8) .. controls ++(270:.5cm) and ++(270:.3cm) .. (.3,-.8) arc (180:0:.1cm) arc (-180:0:.1cm) -- (.7,-.7)
.. controls ++(90:.3cm) and ++(270:.5cm) .. (1.3,0);
\draw (0,-1.3) node[below] {$\scriptstyle g_\cD$} -- (0,.6); 
\draw (.8,-1.3) node[below] {$\scriptstyle k_\cD$} -- (.8,0);
\draw (0,.85) .. controls ++(90:.5cm) and ++(270:.5cm) .. (.385,1.9) -- (.385,2.5);
\draw (.8,.25) -- (.8,.85) .. controls ++(90:.5cm) and ++(270:.5cm) .. (.415,1.9) node [right] {$\scriptstyle \mu^B_{h,\ell}$} -- (.415,2.5) node[above] {$\scriptstyle h\ell_\cD$};
\roundNbox{fill=white}{(.8,0)}{.25}{.05}{.05}{$\scriptstyle A(y)$}
\roundNbox{fill=white}{(0,.6)}{.25}{.05}{.05}{$\scriptstyle A(x)$}
}
\Rightarrow
\tikzmath{
\draw (0,-1.3) node[below] {$\scriptstyle g_\cD$} -- (0,.4); 
\draw (.8,-1.3) node[below] {$\scriptstyle k_\cD$} -- (.8,-.2);
\draw (0,.65) .. controls ++(90:.5cm) and ++(270:.5cm) .. (.385,1.7) -- (.385,2.5);
\draw (.8,.05) -- (.8,.65) .. controls ++(90:.5cm) and ++(270:.5cm) .. (.415,1.7) node [right] {$\scriptstyle \mu^B_{h,\ell}$} -- (.415,2.5) node[above] {$\scriptstyle h\ell_\cD$};
\roundNbox{fill=white}{(.8,-.2)}{.25}{.05}{.05}{$\scriptstyle A(y)$}
\roundNbox{fill=white}{(0,.4)}{.25}{.05}{.05}{$\scriptstyle A(x)$}
\draw[thick, red, mid>] (1.3,0) .. controls ++(90:.8cm) and ++(270:.5cm) .. (.1,1.5)  -- (.1,2)
arc (0:180:.3cm)
-- (-.5,.8) -- (-.5,-.8) .. controls ++(270:.7cm) and ++(270:.8cm) .. (1.3,0);
}
\Rightarrow
\\&
\Rightarrow
\tikzmath{
\draw (0,-1.3) node[below] {$\scriptstyle g_\cD$} -- (0,.4); 
\draw (.8,-1.3) node[below] {$\scriptstyle k_\cD$} -- (.8,-.2);
\draw (0,.65) .. controls ++(90:.5cm) and ++(270:.5cm) .. (.385,1.7) -- (.385,3);
\draw (.8,.05) -- (.8,.65) .. controls ++(90:.5cm) and ++(270:.5cm) .. (.415,1.7) node [right] {$\scriptstyle \mu^B_{h,\ell}$} -- (.415,3) node[above] {$\scriptstyle h\ell_\cD$};
\roundNbox{fill=white}{(.8,-.2)}{.25}{.05}{.05}{$\scriptstyle A(y)$}
\roundNbox{fill=white}{(0,.4)}{.25}{.05}{.05}{$\scriptstyle A(x)$}
\draw[thick, red, mid>] (1.3,0) .. controls ++(90:.8cm) and ++(270:.5cm) .. (.1,1.5)  -- (.1,1.7)
.. controls ++(90:.5cm) and ++(270:.3cm) ..  (.6,2.5)
.. controls ++(90:.5cm) and ++(90:.5cm) ..  (-.5,2.5)
-- (-.5,.8) -- (-.5,-.8) .. controls ++(270:.7cm) and ++(270:.8cm) .. (1.3,0);
}
\overset{(\alpha^2_{h,\ell})^{-1}}{\Rightarrow}
\tikzmath{
\draw (0,-1.3) node[below] {$\scriptstyle g_\cD$} -- (0,.4); 
\draw (.8,-1.3) node[below] {$\scriptstyle k_\cD$} -- (.8,-.2);
\draw (0,.65) -- (0,1.05) .. controls ++(90:.5cm) and ++(270:.5cm) .. (.3,2.1);
\draw (.8,.05) -- (.8,1.05) .. controls ++(90:.5cm) and ++(270:.5cm) .. (.5,2.1);
\draw[double] (.4,2.35) -- (.4,3) node[above] {$\scriptstyle h\ell_\cD$};
\roundNbox{fill=white}{(.8,-.2)}{.25}{.05}{.05}{$\scriptstyle A(y)$}
\roundNbox{fill=white}{(0,.4)}{.25}{.05}{.05}{$\scriptstyle A(x)$}
\roundNbox{fill=white}{(.4,2.35)}{.25}{.05}{.05}{$\scriptstyle \mu^A_{h,\ell}$}
\draw[thick, red, mid>] (1.3,0) .. controls ++(90:.8cm) and ++(270:.5cm) .. (-.2,1.2) 
.. controls ++(90:.5cm) and ++(270:.5cm) ..  (.9,2.25)
.. controls ++(90:.7cm) and ++(90:.7cm) ..  (-.5,2.5)
-- (-.5,.8) -- (-.5,-.8) .. controls ++(270:.7cm) and ++(270:.8cm) .. (1.3,0);
}
\Rightarrow
\tikzmath{
\draw[thick, red, mid>] (1.3,-.2) -- (1.3,1) .. controls ++(90:.8cm) and ++(90:.7cm) .. (-.5,1.8) -- (-.5,-.8) .. controls ++(270:.7cm) and ++(270:.8cm) .. (1.3,-.2);
\draw (0,-1.3) node[below] {$\scriptstyle g_\cD$} -- (0,.4); 
\draw (.8,-1.3) node[below] {$\scriptstyle k_\cD$} -- (.8,-.2);
\draw (0,.65) .. controls ++(90:.2cm) and ++(270:.2cm) .. (.3,1.15);
\draw (.8,.05) -- (.8,.65) .. controls ++(90:.2cm) and ++(270:.2cm) .. (.5,1.15);
\draw[double] (.4,1.4) -- (.4,2.3) node[above] {$\scriptstyle h\ell_\cD$};
\roundNbox{fill=white}{(.8,-.2)}{.25}{.05}{.05}{$\scriptstyle A(y)$}
\roundNbox{fill=white}{(0,.4)}{.25}{.05}{.05}{$\scriptstyle A(x)$}
\roundNbox{fill=white}{(.4,1.4)}{.25}{.05}{.05}{$\scriptstyle \mu^A_{h,\ell}$}
}
\overset{\mu^A_{x,y}}{\Rightarrow}
\tikzmath{
\draw[thick, red, mid>] (1.3,0) -- (1.3,.8) .. controls ++(90:.8cm) and ++(90:.7cm) .. (-.5,1.4) -- (-.5,-.8) .. controls ++(270:.7cm) and ++(270:.8cm) .. (1.3,0);
\draw (0,-1.3) node[below] {$\scriptstyle g_\cD$} -- (0,-.8) .. controls ++(90:.2cm) and ++(270:.2cm) .. (.3,-.25);
\draw (.8,-1.3) node[below] {$\scriptstyle k_\cD$} -- (.8,-.8) .. controls ++(90:.2cm) and ++(270:.2cm) .. (.5,-.25);
\draw[double] (.4,0) -- (.4,2.3) node[above] {$\scriptstyle h\ell_\cD$};
\roundNbox{fill=white}{(0,.8)}{.25}{.05}{.85}{$\scriptstyle A(x\xz y)$}
\roundNbox{fill=white}{(.4,0)}{.25}{.05}{.05}{$\scriptstyle \mu^A_{g,k}$}
}
\overset{(\alpha^2_{g,k})^{-1}}{\Rightarrow}
\tikzmath{
\draw[thick, red, mid>] (1.3,0) .. controls ++(90:.8cm) and ++(90:.7cm) .. (-.5,.8) -- (-.5,-.8) .. controls ++(270:.7cm) and ++(270:.8cm) .. (1.3,0);
\draw (0,-2.3) node[below] {$\scriptstyle g_\cD$} .. controls ++(90:.5cm) and ++(270:.5cm) .. (.385,-1.25) -- (.385,0);
\draw (.8,-2.3) node[below] {$\scriptstyle k_\cD$} .. controls ++(90:.5cm) and ++(270:.5cm) .. (.415,-1.25) node [right] {$\scriptstyle \mu^B_{g,k}$} -- (.415,0);
\draw[double] (.4,0) -- (.4,1.3) node[above] {$\scriptstyle h\ell_\cD$};
\roundNbox{fill=white}{(.8,0)}{.25}{.85}{.05}{$\scriptstyle A(x\xz y)$}
}
\,.
\end{align*}
%By invertibility, $\mu^B$ extends arbitrarily to an adjoint equivalence.
We define an adjoint equivalence $\iota^B = (\iota^B_*, \iota^B_{1}): I_\cD \Rightarrow B\circ I_\cC$
by $\iota^B_* = \id_{e_\cD}$, and $\iota^B_{1} := B^1_e\in \cD(\id_{e_\cD}\Rightarrow B(\id_{e_\cC}))$ 
from \eqref{eq:B1}.
Finally, we define the associator $\omega^B$ and unitors $\ell^B, r^B$ to be identities.
\end{construction}

\begin{lem}
\label{lem:Truncation3Functor}
The data $(\mu^B, \iota^B, \omega^B, \ell^B, r^B)$ endows $B: \rmB\cC \to \rmB\cD$ with the structure of a weak 3-functor.

\DeferProof{sec:CoherenceProofs1Morphisms}
\end{lem}

\begin{lem}
\label{lem:Truncation3FunctorGPointed}
The data
$\beta=(\beta_*:=e_\cD, \beta_g := \id_{g_{\cD}}, \beta_{\id_g}:=B^1_g, \beta^1:=\id_{\id_{e_\cD}}, \beta^2_{g,h}:=\id_{\id_{gh_\cD}},): \pi^\cD \Rightarrow B\circ \pi^\cC$ defines a natural isomorphism. 
\DeferProof{sec:CoherenceProofs1Morphisms}
\end{lem}

We now define for $x\in \cC(g_\cC \to h_\cC)$ the 2-cell
$\gamma_x$ given by
\begin{equation}
\label{eq:DefinitionOfGamma_x}
\tikzmath{
\draw[thick, red, mid>] (-.8,-.7) -- (-.8,.3) .. controls ++(90:.5cm) and ++(270:.5cm) .. (.5,1.7);
\draw[thick, red, mid>] (.5,0) .. controls ++(90:.5cm) and ++(90:.5cm) .. (-.5,.3) -- (-.5,-.3) .. controls ++(270:.5cm) and ++(270:.5cm) .. (.5,0);
\draw (0,-.7) node[below] {$\scriptstyle g_\cD$} -- (0,1.7) node[above] {$\scriptstyle A(h_\cC)$};
\roundNbox{fill=white}{(0,0)}{.25}{.05}{.05}{$\scriptstyle A(x)$}
}
\Rightarrow
\tikzmath{
\draw[thick, red, mid>] (-.7,-.7) -- (-.7,.6) arc (180:0:.1cm) -- (-.5,-.3) .. controls ++(270:.5cm) and ++(270:.5cm) .. (.5,0) .. controls ++(90:.5cm) and ++(270:.5cm) .. (-.3,.9) arc (0:180:.1cm) arc (0:-180:.1cm) .. controls ++(90:.5cm) and ++(270:.5cm) .. (.5,1.7);
\draw (0,-.7) node[below] {$\scriptstyle g_\cD$} -- (0,1.7) node[above] {$\scriptstyle A(h_\cC)$};
\roundNbox{fill=white}{(0,0)}{.25}{.05}{.05}{$\scriptstyle A(x)$}
}
\Rightarrow
\tikzmath{
\draw[thick, red, mid>] (-.7,-.7) -- (-.7,.6) arc (180:0:.1cm) -- (-.5,-.3) .. controls ++(270:.5cm) and ++(270:.5cm) .. (.5,0) .. controls ++(90:.5cm) and ++(270:.5cm) .. (-.3,.85) .. controls ++(90:.5cm) and ++(270:.5cm) .. (.5,1.7);
\draw (0,-.7) node[below] {$\scriptstyle g_\cD$} -- (0,1.7) node[above] {$\scriptstyle A(h_\cC)$};
\roundNbox{fill=white}{(0,0)}{.25}{.05}{.05}{$\scriptstyle A(x)$}
}
\Rightarrow
\tikzmath{
\draw[thick, red, mid>] (-.7,-.7) -- (-.7,-.4) arc (180:0:.1cm) -- (-.5,-.4) .. controls ++(270:.5cm) and ++(270:.5cm) .. (.5,0) .. controls ++(90:.5cm) and ++(270:.5cm) .. (-.3,.85) .. controls ++(90:.5cm) and ++(270:.5cm) .. (.5,1.7);
\draw (0,-.7) node[below] {$\scriptstyle g_\cD$} -- (0,1.7) node[above] {$\scriptstyle A(h_\cC)$};
\roundNbox{fill=white}{(0,0)}{.25}{.05}{.05}{$\scriptstyle A(x)$}
}
\Rightarrow
\tikzmath{
\draw[thick, red, mid>] (-.7,-.7) -- (-.7,-.4) arc (180:0:.1cm) -- (-.5,-.4) .. controls ++(270:.5cm) and ++(270:.5cm) .. (.5,0) -- (.5,1.7);
\draw (0,-.7) node[below] {$\scriptstyle g_\cD$} -- (0,1.7) node[above] {$\scriptstyle A(h_\cC)$};
\roundNbox{fill=white}{(0,0)}{.25}{.05}{.05}{$\scriptstyle A(x)$}
}
\Rightarrow  
\tikzmath{
\draw[thick, red, mid>] (-.5,-.7) .. controls ++(90:.5cm) and ++(270:.5cm) .. (.5,.7) -- (.5,1.7);
\draw (0,-.7) node[below] {$\scriptstyle g_\cD$} -- (0,1.7) node[above] {$\scriptstyle A(h_\cC)$};
\roundNbox{fill=white}{(0,1)}{.25}{.05}{.05}{$\scriptstyle A(x)$}
}
\,.
\end{equation}

\begin{thm}
\label{thm:TruncationStrictifying1Morphisms}
The 1-morphisms $(A,\alpha), (B, \beta)
\in \TriCat_G((\rmB\cC, \pi^\cC) \to (\rmB\cD, \pi^\cD))$ are equivalent via the 2-morphism $(\gamma, \id): (B, \beta)\Rightarrow (A, \alpha)$ where $\gamma=(\gamma_*:=\alpha_*, \gamma_g:=\alpha_g, \gamma_x, \gamma^1:=\alpha^1, \gamma^2_{g,h}:=\alpha^2_{g,h}): B \Rightarrow A$
is the natural isomorphism where
$\gamma_x$ is given in \eqref{eq:DefinitionOfGamma_x} above.
\DeferProof{sec:CoherenceProofs1Morphisms}
\end{thm}

\begin{rem}
Working a bit harder, we can actually make $(B,\beta)$ \emph{strictly unital}, i.e., 
$B(\id_{g_\cC}) = \id_{g_\cD}$ and $B^1_g = \id_{g_\cD}$ for all $g\in G$.
This has the following advantages:
$\iota$ becomes trivial,
$\mu^B_{\id_{e_\cC}, x} = \id_{B(x)}$ for all $x\in \cC(g_\cC \to h_\cC)$ by \ref{Functor:lr},
$\pi^\cD = B\circ \pi^\cC$ on the nose,
and $\beta : \pi^\cD \Rightarrow B\circ \pi^\cC$ is the identity transformation.
Unfortunately, this would complicate our definition of the coherence data for $B$ considerably, and it would further obfuscate the reasons why certain commuting diagrams commute in the sequel.
Moreover, it has not yet been shown in the literature that every $G$-crossed braided functor is equivalent to a strictly unital one, although this would follow as a corollary of our main theorem.
We are thus content to work with our $(B, \beta)$ with $\beta$ completely determined by $B$.
\end{rem}

%%%%%%%%%%%%%%%%%%%%%%%%%%%%%%%%%%%%%%%%
\subsection{Strictifying 2-morphisms}
\label{sec:Strictifying2Morphisms}

Suppose $(\rmB\cC,\pi^\cC),(\rmB\cD, \pi^\cD)\in \TriCat_G^{\pt}$ and 
$(A,\alpha), (B,\beta):(\rmB\cC,\pi^\cC)\to (\rmB\cD, \pi^\cD)$ are two 1-morphisms in $\TriCat_G^{\pt}$.
Since $(A,\alpha), (B,\beta)$ are 1-morphisms in $\TriCat_G^{\pt}$, 
$A(g_\cC) = g_\cD = B(g_\cC)$ for all $g\in G$, and
$\alpha_*= e_\cD = \beta_*$ and $\alpha_g = \id_{g_\cD} = \beta_g$.
Suppose $(\eta, m): \in \TriCat_G((A,\alpha)\Rightarrow (B,\beta))$.
We prove that $(\eta, m)$ is equivalent to a 2-morphism $(\zeta,\id) \in \TriCat_G^{\pt}((A,\alpha)\Rightarrow (B,\beta))$.
%such that $\zeta_* = e$ and $\zeta_g = \id_g$ for all $g\in G$.

As in Defintion \ref{defn:Transformation}, we denote the 0-cell $\eta_*$ by an oriented \textcolor{DarkGreen}{green} strand.
The modification $m=(m_*, m_g)$ as in Definition \ref{defn:Modification}
consists of an invertible 1-cell $m_*: \beta_* \Rightarrow \eta_* \xz \alpha_*$
together with coherent invertible 2-cells
\begin{equation}
\label{eq:mgIso}
\tikzmath{
\node[blue] (beta) at (0,0) {$\scriptstyle \beta_*$};
\node (g1) at (1,0) {$\scriptstyle g_\cD$};
\node[DarkGreen] (eta) at (.5,3) {$\scriptstyle \eta_*$};
\node[red] (alpha) at (1.5,3) {$\scriptstyle \alpha_*$};
\coordinate (g2) at (1,1);
\node (g3) at (-.5,3) {$\scriptstyle g_\cD$};
\node[draw,rectangle, thick, rounded corners=5pt] (m) at (0,1) {$\scriptstyle m_*$};
\draw (g1) to[in=-90,out=90] (g2) to[in=-90,out=90] (g3);
\draw[blue,dashed,thick,mid>] (beta) to[in=-90,out=90] (m);
\draw[red, dashed, thick, mid>] (m.60) to[in=-90,out=90] (alpha);
\draw[DarkGreen, thick, mid>] (m.120) to[in=-90,out=90] (eta);
}
\overset{m_g}{\Rightarrow}
\tikzmath{
\node[blue] (beta) at (0,0) {$\scriptstyle \beta_*$};
\node (g1) at (1,0) {$\scriptstyle g_\cD$};
\node[DarkGreen] (eta) at (.5,3) {$\scriptstyle \eta_*$};
\node[red] (alpha) at (1.5,3) {$\scriptstyle \alpha_*$};
\coordinate (g2) at (-.5,2);
\node (g3) at (-.5,3) {$\scriptstyle g_\cD$};
\node[draw,rectangle, thick, rounded corners=5pt] (m) at (1,2) {$\scriptstyle m_*$};
\draw (g1) to[in=-90,out=90] (g2) to[in=-90,out=90] (g3);
\draw[blue,dashed,thick,mid>] (beta) to[in=-90,out=90] (m);
\draw[red, dashed, thick, mid>] (m.60) to[in=-90,out=90] (alpha);
\draw[DarkGreen, thick, mid>] (m.120) to[in=-90,out=90] (eta);
}
\,.
\end{equation}
Observe that since $\beta_*=e_\cD=\alpha_*$ and $\beta_g = \id_{g} = \alpha_g$, we may completely omit the dashed lines in \eqref{eq:mgIso}.
As in Remark \ref{rem:InvertibleExtendsToAdjointEq}, we extend the invertible 1-cell $m_* \in \cD$ to an adjoint equivalence arbitrarily.

For $x\in \cC(g_\cC \to h_\cC)$, we define an invertible 2-cell $\zeta_x$ as the following composite:
\begin{equation}
\label{eq:DefinitionOfZeta_x}
\begin{split}
\tikzmath{
\node (g) at (0,0) {$\scriptstyle g_\cD$};
\node (h) at (0,5) {$\scriptstyle h_\cD$};
\node[draw,rectangle, thick, rounded corners=5pt] (Bx) at (0,1) {$\scriptstyle A(x)$};
\draw (g) to[in=-90,out=90] (Bx);
\draw (Bx) to[in=-90,out=90] (h);
\draw[thick, dashed, rounded corners=5pt] (.5,2.5) rectangle (1.5,4.5);
}
&\overset{}{\Rightarrow}
\tikzmath{
\node (g) at (0,-1) {$\scriptstyle g_\cD$};
\node (h) at (0,4) {$\scriptstyle h_\cD$};
\draw[thick, dashed, rounded corners=5pt] (-.5,.5) rectangle (1.5,2.5);
\node[draw,rectangle, thick, rounded corners=5pt] (Ax) at (0,0) {$\scriptstyle A(x)$};
\node[draw,rectangle, thick, rounded corners=5pt] (m1) at (1,2) {$\scriptstyle m_*$};
\node[draw,rectangle, thick, rounded corners=5pt] (m2) at (1,3) {$\scriptstyle m_*^{-1}$};
\draw (g) to[in=-90,out=90] (Ax);
\draw (Ax) to[in=-90,out=90] (h);
\draw[DarkGreen, thick, mid>] (m1) to[in=-90,out=90] (m2);
}
\overset{m_h^{-1}}{\Rightarrow}
\tikzmath{
\node (g) at (1,-1) {$\scriptstyle g_\cD$};
\coordinate (h1) at (1,1);
\coordinate (h2) at (0,3);
\node (h3) at (0,4) {$\scriptstyle h_\cD$};
\node[draw,rectangle, thick, rounded corners=5pt] (Ax) at (1,0) {$\scriptstyle A(x)$};
\node[draw,rectangle, thick, rounded corners=5pt] (m1) at (0,1) {$\scriptstyle m_*$};
\node[draw,rectangle, thick, rounded corners=5pt] (m2) at (1,3) {$\scriptstyle m_*^{-1}$};
\draw (g) to[in=-90,out=90] (Ax);
\draw (Ax) to[in=-90,out=90] (h1) to[in=-90,out=90] (h2) to[in=-90,out=90] (h3);
\draw[DarkGreen, thick, mid>] (m1) to[in=-90,out=90] (m2);
\draw[thick, dashed, rounded corners=5pt] (-.5,-.5) rectangle (1.5,1.5);
}
\overset{\phi}{\Rightarrow}
\tikzmath{
\node (g) at (1,-1) {$\scriptstyle g_\cD$};
\coordinate (h1) at (0,3);
\node (h2) at (0,4) {$\scriptstyle h_\cD$};
\coordinate (mm) at (0,1);
\draw[thick, dashed, rounded corners=5pt] (-.5,.5) rectangle (1.5,2.5);
\node[draw,rectangle, thick, rounded corners=5pt] (Ax) at (1,1) {$\scriptstyle A(x)$};
\node[draw,rectangle, thick, rounded corners=5pt] (m1) at (0,0) {$\scriptstyle m_*$};
\node[draw,rectangle, thick, rounded corners=5pt] (m2) at (1,3) {$\scriptstyle m_*^{-1}$};
\draw (g) to[in=-90,out=90] (Ax);
\draw (Ax) to[in=-90,out=90] (h1) to[in=-90,out=90] (h2);
\draw[DarkGreen, thick, mid>] (m1) to[in=-90,out=90] (mm) to[in=-90,out=90] (m2);
}
\\&\overset{\eta_x}{\Rightarrow}
\tikzmath{
\node (g1) at (1,0) {$\scriptstyle g_\cD$};
\coordinate (g2) at (1,1);
\node (h) at (0,5) {$\scriptstyle h_\cD$};
\coordinate (mm) at (1,3);
\draw[thick, dashed, rounded corners=5pt] (-.5,.5) rectangle (1.5,2.5);
\node[draw,rectangle, thick, rounded corners=5pt] (Bx) at (0,3) {$\scriptstyle B(x)$};
\node[draw,rectangle, thick, rounded corners=5pt] (m1) at (0,1) {$\scriptstyle m_*$};
\node[draw,rectangle, thick, rounded corners=5pt] (m2) at (1,4) {$\scriptstyle m_*^{-1}$};
\draw (g1) to[in=-90,out=90] (g2) to[in=-90,out=90] (Bx);
\draw (Bx) to[in=-90,out=90] (h);
\draw[DarkGreen, thick, mid>] (m1) to[in=-90,out=90] (mm) to[in=-90,out=90] (m2);
}
\overset{m_g}{\Rightarrow}
\tikzmath{
\node (g) at (0,0) {$\scriptstyle g_\cD$};
\node (h) at (0,5) {$\scriptstyle h_\cD$};
\node[draw,rectangle, thick, rounded corners=5pt] (Bx) at (0,3) {$\scriptstyle B(x)$};
\node[draw,rectangle, thick, rounded corners=5pt] (m1) at (1,2) {$\scriptstyle m_*$};
\node[draw,rectangle, thick, rounded corners=5pt] (m2) at (1,4) {$\scriptstyle m_*^{-1}$};
\draw (g) to[in=-90,out=90] (Bx);
\draw (Bx) to[in=-90,out=90] (h);
\draw[DarkGreen, thick, mid>] (m1) to[in=-90,out=90] (m2);
\draw[thick, dashed, rounded corners=5pt] (-.5,1.5) rectangle (1.5,3.5);
}
\overset{\phi}{\Rightarrow}
\tikzmath{
\node (g) at (0,0) {$\scriptstyle g_\cD$};
\node (h) at (0,5) {$\scriptstyle h_\cD$};
\node[draw,rectangle, thick, rounded corners=5pt] (Bx) at (0,2) {$\scriptstyle B(x)$};
\node[draw,rectangle, thick, rounded corners=5pt] (m1) at (1,3) {$\scriptstyle m_*$};
\node[draw,rectangle, thick, rounded corners=5pt] (m2) at (1,4) {$\scriptstyle m_*^{-1}$};
\draw (g) to[in=-90,out=90] (Bx);
\draw (Bx) to[in=-90,out=90] (h);
\draw[DarkGreen, thick, mid>] (m1) to[in=-90,out=90] (m2);
\draw[thick, dashed, rounded corners=5pt] (.5,2.5) rectangle (1.5,4.5);
}
\overset{}{\Rightarrow}
\tikzmath{
\node (g) at (0,0) {$\scriptstyle g_\cD$};
\node (h) at (0,5) {$\scriptstyle h_\cD$};
\node[draw,rectangle, thick, rounded corners=5pt] (Bx) at (0,2) {$\scriptstyle B(x)$};
\draw (g) to[in=-90,out=90] (Bx);
\draw (Bx) to[in=-90,out=90] (h);
}.
\end{split}
\end{equation}
We define the unit map as in \ref{Transformation:eta^1} by $\zeta^1:=\id_{\id_{e_\cD}}$ and the monoidal map as in \ref{Transformation:eta^2}
 by $\zeta^2_{g,h}:=\id_{\id_{gh_\cD}}$.

\begin{lem}
\label{lem:TruncationDefine2Morphism}
The data $\zeta:=(\zeta_*=e, \zeta_g = \id_g, \zeta_x, \zeta^1:=\id_{\id_{e_\cD}}, \zeta^2_{g,h}:=\id_{\id_{gh_\cD}})$ together with the identity modification defines a 2-morphism $(\zeta, \id) \in \TriCat_G^{\pt} ((A,\alpha) \Rightarrow (B, \beta))$.
\DeferProof{sec:CoherenceProofs2Morphisms}
\end{lem}

Observe now that by the strictness properties of $\alpha$ and $\beta$,
$m_* : e_\cD \Rightarrow \eta_*$.
Erasing the dotted lines from \eqref{eq:mgIso} for $m_g$, we see that
the same data as $m=(m_*, m_g)$ actually defines an invertible modification $\zeta\Rrightarrow \eta$! 

\begin{thm}
\label{thm:TruncationStrictify2Morphisms}
The 2-morphisms
$(\eta, m),(\zeta, \id) \in\TriCat_G( (A,\alpha) \Rightarrow (B, \beta))$
are equivalent via the 3-morphism
$(m,\id)\in \TriCat_G((\zeta,\id)\Rrightarrow (\eta,\id))$.
\DeferProof{sec:CoherenceProofs2Morphisms}
\end{thm}

%%%%%%%%%%%%%%%%%%%%%%%%%%%%%%%%%%%%%%%%
\subsection{Strictifying 3-morphisms}
\label{sec:Strictifying3Morphisms}

Suppose now that $(\eta,m=\id),(\zeta, n=\id): (A,\alpha) \Rightarrow (B,\beta)$ are two 2-morphisms in $\TriCat_G^{\pt}$ and $(p,\rho):(\eta, \id) \Rrightarrow (\zeta, \id)$ is a 3-morphism in $\TriCat_G$.

First, since $(\eta,\id),(\zeta,\id)$ are 2-morphisms in $\TriCat_G^{\pt}$, we have that
$\eta_*=e_\cD = \zeta_*$ and $\eta_g = \id_{g_\cD} = \zeta_g$ for all $g\in G$, and the modifications are identities.
This means the perturbation $\rho$ is a 2-cell
$$
\tikzmath{
\draw[dashed, thick, rounded corners=5pt] (-.2,.3) rectangle (-.8,.7);
}
\quad
=
\tikzmath{
\node[blue] (beta) at (0,0) {$\scriptstyle \beta_*=e_\cD$};
\node[red] (alpha) at (.5,2) {$\scriptstyle \alpha_*=e_\cD$};
\node[orange] (zeta) at (-.5,2) {$\scriptstyle \zeta_*=e_\cD$};
\node[draw,rectangle, thick, rounded corners=5pt] (m) at (0,1) {$\scriptstyle m_*=\id_{e_\cD}$};
\draw[blue, dashed, thick, mid>] (beta) to[in=-90,out=90] (m);
\draw[red, dashed, thick, mid>] (m.60) to[in=-90,out=90] (alpha);
\draw[orange, dashed, thick, mid>] (m.120) to[in=-90,out=90] (zeta);
}
\overset{\rho}{\Rightarrow}
\tikzmath{
\node[blue] (beta) at (0,0) {$\scriptstyle \beta_*=e_\cD$};
\coordinate (alpha1) at (.5,2);
\node[red] (alpha2) at (.5,3) {$\scriptstyle \alpha_*=e_\cD$};
\node[orange] (zeta) at (-.5,3) {$\scriptstyle \zeta_*=e_\cD$};
\node[draw,rectangle, thick, rounded corners=5pt] (m) at (0,1) {$\scriptstyle n_*=\id_{e_\cD}$};
\node[draw,rectangle, thick, rounded corners=5pt] (p) at (-.5,2) {$\scriptstyle p_*$};
\draw[blue, dashed, thick, mid>] (beta) to[in=-90,out=90] (m);
\draw[red, dashed, thick, mid>] (m.60) to[in=-90,out=90] (alpha1) to[in=-90,out=90] (alpha2);
\draw[DarkGreen, dashed, thick, mid>] (m.120) to[in=-90,out=90] node[left] {$\scriptstyle \eta_*=e_\cD$} (p);
\draw[orange, dashed, thick, mid>] (p) to[in=-90,out=90] (zeta);
}
=
\quad
\tikzmath{
\node[draw,rectangle, thick, rounded corners=5pt] (p) at (0,0) {$\scriptstyle p_*$};
}
$$
satisfying \ref{Petrubtation:Condition} in Definition \ref{defn:Perturbation}.
We may thus view $\rho$ as an invertible 2-cell $\id_{\id_{e_\cD}} \Rightarrow p_*$, under which \ref{Petrubtation:Condition} becomes
\begin{equation}
\label{eq:PerturbationEquationForRho}
\left(
\tikzmath{
\node (g1) at (0,0) {$\scriptstyle g_\cD$};
\node (g2) at (0,1) {};
\draw[dashed, thick, rounded corners=5pt] (-.2,.3) rectangle (-.8,.7);
\draw (g1) to[in=-90,out=90] (g2);
}
\overset{\rho_*}{\Rightarrow}
\tikzmath{
\node (g1) at (0,0) {$\scriptstyle g_\cD$};
\node (g2) at (0,1) {};
\node[draw,rectangle, thick, rounded corners=5pt] (p) at (-.5,.5) {$\scriptstyle p_*$};
\draw (g1) to[in=-90,out=90] (g2);
}
\overset{p_g}{\Rightarrow}
\tikzmath{
\node (g1) at (0,0) {$\scriptstyle g_\cD$};
\node (g2) at (0,1) {};
\node[draw,rectangle, thick, rounded corners=5pt] (p) at (.5,.5) {$\scriptstyle p_*$};
\draw (g1) to[in=-90,out=90] (g2);
}
\right)
=
\left(
\tikzmath{
\node (g1) at (0,0) {$\scriptstyle g_\cD$};
\node (g2) at (0,1) {};
\draw[dashed, thick, rounded corners=5pt] (.2,.3) rectangle (.8,.7);
\draw (g1) to[in=-90,out=90] (g2);
}
\overset{\rho_*}{\Rightarrow}
\tikzmath{
\node (g1) at (0,0) {$\scriptstyle g_\cD$};
\node (g2) at (0,1) {};
\node[draw,rectangle, thick, rounded corners=5pt] (p) at (.5,.5) {$\scriptstyle p_*$};
\draw (g1) to[in=-90,out=90] (g2);
}
\right)
\qquad\qquad \forall\,g\in G.
\end{equation}

\begin{lem}
\label{lem:Only3Endos}
%If $(\eta,\id),(\zeta,\id)$ are two 2-morphisms in $\TriCat_G^{\pt}$ and there is a 3-morphism $(p,\rho):(\eta,\id)\Rrightarrow (\zeta,\id)$ in $\TriCat_G$, then $\eta=\zeta$. 
Any $3$-morphism in $\TriCat_G$ between $2$-morphisms in the subcategory $\TriCat_G^{\pt}$ is an endomorphism.
\DeferProof{sec:CoherenceProofs3Morphisms}
\end{lem}

\begin{thm}
Any 3-morphism in $\TriCat_G$ between $2$-morphisms in the subcategory $\TriCat_G^{\pt}$ is isomorphic to the identity $3$-morphism.
\end{thm}
\begin{proof}
First, by Lemma \ref{lem:Only3Endos}, every 3-morphism is a 3-endomorphism.
Suppose $(\eta, \id)$ is a 2-morphism in $\TriCat_G^{\pt}$
and $(p,\rho)$ is a 3-endomorphism of $(\eta, \id)$.
As above, we may view $\rho_*$ as an invertible 2-morphism $\id_{e_\cD} \Rightarrow p_*$ that satisfies \eqref{eq:PerturbationEquationForRho}.
This is exactly saying that $\rho_*$ is a perturbation $\id_{(\eta, \id)} \RRightarrow (p,\rho)$.
\end{proof}

%%%%%%%%%%%%%%%%%%%%%%%%%%%%%%%%%%%%%%%%
\subsection{Strictifying 4-morphisms}
\label{sec:Strictifying4Morphisms}

\begin{thm}
The only 4-endomorphism in $\TriCat_G$ of an identity 3-morphism in the subcategory $\TriCat_G^{\pt}$ is the identity.
\end{thm}
\begin{proof}
Suppose $\xi$ is a 4-endomorphism of an identity 3-morphism $(p=\id,\rho=\id)$ in $\TriCat_G^{\pt}$.
Then $\xi$ satisfies the criterion \eqref{eq:4MorphismCriterion}, which in diagrams is
$$
\left(
\tikzmath{
\node[blue] (beta) at (0,0) {$\scriptstyle \beta_*=e_\cD$};
\node[red] (alpha) at (.5,2) {$\scriptstyle \alpha_*=e_\cD$};
\node[DarkGreen] (zeta) at (-.5,2) {$\scriptstyle \eta_*=e_\cD$};
\node[draw,rectangle, thick, rounded corners=5pt] (m) at (0,1) {$\scriptstyle m_*=\id_{e_\cD}$};
\draw[blue, dashed, thick, mid>] (beta) to[in=-90,out=90] (m);
\draw[red, dashed, thick, mid>] (m.60) to[in=-90,out=90] (alpha);
\draw[DarkGreen, dashed, thick, mid>] (m.120) to[in=-90,out=90] (zeta);
}
\overset{\rho=\id}{\Rightarrow}
\tikzmath{
\coordinate (beta) at (0,0);
\coordinate (alpha1) at (.5,2);
\coordinate (alpha2) at (.5,3);
\coordinate (zeta) at (-.5,3);
\node[draw,rectangle, thick, rounded corners=5pt] (m) at (0,1) {$\scriptstyle n_*=\id_{e_\cD}$};
\node[draw,rectangle, thick, rounded corners=5pt] (p) at (-.5,2) {$\scriptstyle p_*=\id_{e_\cD}$};
\draw[dashed, thick, rounded corners=5pt] (-1.3,1.6) rectangle (.3,2.4);
\draw[blue, dashed, thick, mid>] (beta) to[in=-90,out=90] (m);
\draw[red, dashed, thick, mid>] (m.60) to[in=-90,out=90] (alpha1) to[in=-90,out=90] (alpha2);
\draw[DarkGreen, dashed, thick, mid>] (m.120) to[in=-90,out=90] (p);
\draw[DarkGreen, dashed, thick, mid>] (p) to[in=-90,out=90] (zeta);
}
\overset{\xi}{\Rightarrow}
\tikzmath{
\coordinate (beta) at (0,0);
\coordinate (alpha1) at (.5,2);
\coordinate (alpha2) at (.5,3);
\coordinate (zeta) at (-.5,3);
\node[draw,rectangle, thick, rounded corners=5pt] (m) at (0,1) {$\scriptstyle n_*=\id_{e_\cD}$};
\node[draw,rectangle, thick, rounded corners=5pt] (p) at (-.5,2) {$\scriptstyle p_*=\id_{e_\cD}$};
\draw[blue, dashed, thick, mid>] (beta) to[in=-90,out=90] (m);
\draw[red, dashed, thick, mid>] (m.60) to[in=-90,out=90] (alpha1) to[in=-90,out=90] (alpha2);
\draw[DarkGreen, dashed, thick, mid>] (m.120) to[in=-90,out=90] (p);
\draw[DarkGreen, dashed, thick, mid>] (p) to[in=-90,out=90] (zeta);
}
\right)
=
\left(
\tikzmath{
\coordinate (beta) at (0,0);
\coordinate (alpha) at (.5,2);
\coordinate (zeta) at (-.5,2);
\node[draw,rectangle, thick, rounded corners=5pt] (m) at (0,1) {$\scriptstyle m_*=\id_{e_\cD}$};
\draw[blue, dashed, thick, mid>] (beta) to[in=-90,out=90] (m);
\draw[red, dashed, thick, mid>] (m.60) to[in=-90,out=90] (alpha);
\draw[DarkGreen, dashed, thick, mid>] (m.120) to[in=-90,out=90] (zeta);
}
\overset{\rho=\id}{\Rightarrow}
\tikzmath{
\coordinate (beta) at (0,0);
\coordinate (alpha1) at (.5,2);
\coordinate (alpha2) at (.5,3);
\coordinate (zeta) at (-.5,3);
\node[draw,rectangle, thick, rounded corners=5pt] (m) at (0,1) {$\scriptstyle n_*=\id_{e_\cD}$};
\node[draw,rectangle, thick, rounded corners=5pt] (p) at (-.5,2) {$\scriptstyle p_*=\id_{e_\cD}$};
\draw[blue, dashed, thick, mid>] (beta) to[in=-90,out=90] (m);
\draw[red, dashed, thick, mid>] (m.60) to[in=-90,out=90] (alpha1) to[in=-90,out=90] (alpha2);
\draw[DarkGreen, dashed, thick, mid>] (m.120) to[in=-90,out=90] (p);
\draw[DarkGreen, dashed, thick, mid>] (p) to[in=-90,out=90] (zeta);
}
\right).
$$
We conclude that $\xi=\id$.
\end{proof}

%auto-ignore
%this ensures the arxiv doesn't try to start TeXing here.
%!TEX root =../Equivalence.tex

%%%%%%%%%%%%%%%%%%%%%%%%%%%%%%%%%%%%%%%%%%%%%%%%%%
%%%%%%%%%%%%%%%%%%%%%%%%%%%%%%%%%%%%%%%%%%%%%%%%%%
%%%%%%%%%%%%%%%%%%%%%%%%%%%%%%%%%%%%%%%%%%%%%%%%%%
\section{\texorpdfstring{$G$}{G}-crossed braided categories}
\label{sec:GCrossed}

In \S\ref{defn:Definitions} below, we define the strict 2-category $G\CrsBrd$ of $G$-crossed braided categories.
By \cite{MR3671186}, 
$G\CrsBrd$ is equivalent to the full 2-subcategory $G\CrsBrd^{\st}$ of strict $G$-crossed braided categories.
In this section, we prove our second main theorem.

\begin{thm}
\label{thm:ThmA-2}
The 2-category $\TriCat_G^{\st}$ is equivalent to $G\CrsBrd^{\st}$.
\end{thm}
\begin{proof}
In \S\ref{sec:FromGBoring3CatsToGCrossedBriadedCats} below, we construct a strict 2-functor $\TriCat_G^{\st} \to G\CrsBrd^{\st}$.
In \S\ref{sec:2FunctorEquivalence} below, we show this 2-functor is an equivalence.
Indeed, 
the 2-functor is essentially surjective on objects by \cite{MR4033513} explained at the beginning of \S\ref{sec:2FunctorEquivalence}, essentially surjective on 1-morphisms by Theorem \ref{thm:Preimage3Functor}, and fully faithful on 2-morphisms by Theorem \ref{thm:FullyFaithfulOn2Morphisms}.
We defer all further proofs in this section to Appendix \ref{sec:CoherenceProofsGCrossed}.
\end{proof}

We thus have the following zig-zag of strict equivalences denoted $\sim$ and an isomorphism $\cong$, where the hooked arrows denote inclusions of full subcategories.
$$
\begin{tikzcd}[column sep=4em]
\TriCat_G
\arrow[r,hookleftarrow, "\sim","\text{\scriptsize Thm.~\ref{thm:tricatpt}}"']
&
\TriCat_G^{\pt}
\arrow[r,leftrightarrow, "\cong","\text{\scriptsize Cor.~\ref{cor:Contract4CategoryTo2Category}}"']
&
\TriCat_G^{\st}
\arrow[r,rightarrow, "\sim", "\text{\scriptsize Thm.~\ref{thm:ThmA-2}}"']
&
G\CrsBrd^{\st}
\arrow[r,hookrightarrow, "\sim", "\text{\scriptsize \cite{MR3671186}}"']
&
G\CrsBrd
\end{tikzcd}
$$

%%%%%%%%%%%%%%%%%%%%%%%%%%%%%%%%%%%%

\subsection{Definitions}
\label{defn:Definitions}

Let $G$ be a group.
We now give a definition of a (possibly non-additive) $G$-crossed braided category.
Below, we give a definition in terms of the component categories $\fC_g$.
When each component $\fC_g$ is linear and the tensor product functors and $G$-action functors are linear, $\fC:=\bigoplus_{g\in G} \fC_g$ is an ordinary $G$-crossed braided monoidal category in the sense of \cite[\S8.24]{MR3242743} (except possibly not rigid nor fusion).

\begin{defn}
A $G$-\emph{crossed braided category} $\fC$ consists of the following data:
\begin{itemize}
    \item 
    a collection of categories $(\fC_g)_{g\in G}$;
    \item
    a family of bifunctors $\xz_{g,h}: \fC_{g}\times \fC_{h} \to \fC_{gh}$;
    \item
    an associator natural isomorphism
    $\alpha_{g,h,k}: \xz_{gh,k}\circ (\xz_{g,h}\times \id_{\fC_k})\Rightarrow \xz_{g,hk}\circ (\id_{\fC_g} \times \xz_{h,k})$;
    \item
    a unit object $1_\fC \in \fC_e$;
    \item
    unitor natural isomorphisms
    $\lambda: \xz_{e,g}\circ (1_\fC \times -)\Rightarrow \id_{\fC_g}$
    and
    $\rho:\xz_{g,e}\circ (-\times 1_\fC) \Rightarrow \id_{\fC_g}$.
    \end{itemize}
Using the convention
$$
a_g\xz b_h 
:= 
 \xz_{g,h}(a_g \times b_h)
\qquad\qquad
\forall\, a_g \in \fC_g
\text{ and }
b_h \in \fC_h,
$$
this data should satisfy the obvious pentagon and triangle axioms of a monoidal category.

Moreover, $\fC$ is equipped with a $G$-action $F_g : \fC_h \to \fC_{ghg^{-1}}$ together with
an isomorphism
$i_g:1_\fC \to F_g(1_\fC)$
and
natural isomorphisms $\psi^g$, $\mu^{g,h}$, and $\iota_g$
$$
\begin{tikzcd}
\fC_h \times \fC_k 
\arrow[bend left=50, rr, "F_g\circ \xz_{h,k}"]
\arrow[bend right=50, swap, rr, "\xz_{ghg^{-1}, gkg^{-1}}\circ (F_g \times F_g)"]
&
\arrow[phantom, "\Uparrow \psi^g_{h,k}"]
& 
\fC_{ghkg^{-1}}
\end{tikzcd}
\qquad\qquad
\begin{tikzcd}
\fC_k 
\arrow[bend left=50, rr, "F_{gh}"]
\arrow[bend right=50, swap, rr, "F_g \circ F_h"]
&
\arrow[phantom, "\Uparrow \mu^{g,h}_k"]
& 
\fC_{ghkh^{-1}g^{-1}}
\end{tikzcd}
\qquad\qquad
\begin{tikzcd}
\fC_g
\arrow[bend left=50, rr, "F_e"]
\arrow[bend right=50, swap, rr, "\id_{\fC_g}"]
&
\arrow[phantom, "\Uparrow \iota_g"]
& 
\fC_{g}
\end{tikzcd}
$$
which satisfy the following associativity and unitality conditions where we suppress whiskering:
\begin{enumerate}[label=($\psi$\arabic*)]
\item\label{rel:psi1} (associativity)
The following diagram commutes:
$$
\begin{tikzpicture}[baseline= (a).base]
\node[scale=.75] (a) at (0,0){
\begin{tikzcd}
\xz_{ghkg^{-1},g\ell g^{-1}}\circ (\xz_{ghg^{-1}, gkg^{-1}}\circ (F_g \times F_g) \times F_g)
\arrow[Rightarrow, rr, "\alpha"]
\arrow[Rightarrow, d, "\xz_{ghkg^{-1}}\circ (\psi^g_{h,k}\times \id_{\fC_{g\ell g^{-1}}})"]
&&
\xz_{ghg^{-1},gk\ell g^{-1}}\circ (F_g \times (\xz_{gkg^{-1}, g\ell g^{-1}}\circ (F_g \times F_g)))
\arrow[Rightarrow, swap, d, "\xz_{ghg^{-1}, gk\ell g^{-1}} \circ (\id_{\fC_{ghg^{-1}}}\times \psi^g_{k,\ell})"]
\\
\xz_{ghkg^{-1}, g\ell g^{-1}}\circ ((F_g \circ \xz_{h,k})\times F_g)
\arrow[Rightarrow, swap, dr, "\psi^g_{hk,\ell}"]
&&
\xz_{ghg^{-1},gk\ell g^{-1}}\circ (F_g \times (F_g \circ \xz_{k,\ell}))
\arrow[Rightarrow, dl, "\psi^g_{h,k\ell}"]
\\
&
F_g\circ \xz_{h,k\ell}\circ (\id_{\fC_h}\times \xz_{k,\ell})
\end{tikzcd}};
\end{tikzpicture}
$$
\item\label{rel:psi2} (unitality)
For every $a_h\in \fC_h$, the following diagram commutes:
$$
\begin{tikzcd}
1_{\fC}\xz F_g(a_h)
\arrow[rrrr, "\xz_{e,ghg^{-1}}(i_g \times \id_{\fC_{ghg^{-1}}})"]
\arrow[d, "\lambda_{F_g(a_h)}"]
&&&&
F_g(1_\fC) \xz F_g(a_h)
\arrow[d, "\psi^g_{e,h}"]
\\
F_g(a_h)
&&&&
F_g(1_\fC \xz a_h)
\arrow[llll, "F_g(\lambda_{a_h})"]
\end{tikzcd}
$$
as does a similar diagram where $1_\fC$ appears on the right with $\rho$.
\end{enumerate}
\begin{enumerate}[label=($\mu$\arabic*)]
\item\label{rel:mu1} (monoidality)
The following diagram commutes:
$$
\begin{tikzcd}
\xz_{ghkg^{-1}h^{-1}, gh\ell g^{-1}h^{-1}}\circ ((F_g\circ F_h)\times (F_g\circ F_h))
\arrow[Rightarrow, rr, "\psi^g_{hkh^{-1},h\ell h^{-1}}"]
\arrow[Rightarrow, d,"\xz_{ghkg^{-1}h^{-1}, gh\ell g^{-1}h^{-1}}(\mu^{g,h}_{k}\times \mu^{g,h}_{\ell})"]
&&
F_g\circ \xz_{hkh^{-1}, h\ell h^{-1}}\circ (F_h \times F_h)
\arrow[Rightarrow, d,swap,  "F_g(\psi^h_{k,\ell})"]
\\
\xz_{ghkg^{-1}h^{-1}, gh\ell g^{-1}h^{-1}} \circ(F_{gh}\times F_{gh})
\arrow[Rightarrow, dr,swap, "\psi^{gh}_{k,\ell}"]
&&
F_g\circ F_h \circ \xz_{k,\ell}
\arrow[Rightarrow, dl,"\mu^{g,h}_{k\ell}"]
\\
&
F_{gh}\circ \xz_{k,\ell}
\end{tikzcd}
$$
\item\label{rel:mu2} (associativity)
The following diagram commutes:
$$
\begin{tikzcd}
F_g \circ F_h \circ F_k
\arrow[Rightarrow, d,"\mu^{g,h}_{k\ell k^{-1}}"]
\arrow[Rightarrow, r, "F_g(\mu^{h,k}_\ell)"]
&
F_g\circ F_{hk}
\arrow[Rightarrow, d,"\mu^{g,hk}_\ell"]
\\
F_{gh}\circ F_k
\arrow[Rightarrow, r,"\mu^{gh,k}_{\ell}"]
&
F_{ghk}
\end{tikzcd}
$$
\end{enumerate}
\begin{enumerate}[label=($\iota$\arabic*)]
\item\label{rel:iota1} (monoidality)
The following diagram commutes:
$$
\begin{tikzcd}
\id_{\fC_{hk}}\circ \xz_{h,k}
\arrow[equal ,r,""]
\arrow[Rightarrow, d, "\iota_{hk}"]
&
\xz_{h,k}\circ (\id_{\fC_h}\times \id_{\fC_k})
\arrow[Rightarrow, d, "\xz_{h,k}(\iota_h \times \iota_k)"]
\\
F_e\circ \xz_{h,k}
\arrow[Rightarrow, r, "\psi^e_{h,k}"]
&
\xz_{h,k} \circ (F_e \times F_e)
\end{tikzcd}
$$
\item\label{rel:iota2} (unitality)
The following diagrams commute:
$$
\begin{tikzcd}
\id_{\fC_{ghg^{-1}}}\circ F_g
\arrow[equal ,rr,""]
\arrow[Rightarrow, swap, dr, "\iota_{ghg^{-1}}"]
&&
F_g
\\
&
F_e\circ F_g
\arrow[Rightarrow, swap, ur, "\mu^{e,g}_h"]
\end{tikzcd}
\qquad\text{and}
\qquad
\begin{tikzcd}
F_g
\arrow[equal ,rr,""]
&&
F_g\circ \id_{\fC_h}
\arrow[Rightarrow, dl, "F_g(\iota_h)"]
\\
&
F_g\circ F_e
\arrow[Rightarrow, ul, "\mu^{g,e}_h"]
\end{tikzcd}.
$$
\end{enumerate}
Finally, we have the $G$-crossed braiding natural isomorphism
$$
\begin{tikzcd}
\fC_h \times \fC_g
\arrow[bend left=30, drr, "\xz_{ghg^{-1},g}\circ (F_g\times \id_{\fC_g})"]
\\
&
\arrow[phantom, "\Uparrow \beta^{g,h}"]
& 
\fC_{gh}
\\
\fC_g \times \fC_h 
\arrow[uu, "\text{\scriptsize swap}"]
\arrow[bend right=30, swap, urr, " \xz_{g,h}"]
\end{tikzcd}
\qquad\qquad
a_g \otimes b_h \xrightarrow{\beta^{g,h}_{a_g, b_h}} F_g(b_h) \otimes a_g
\qquad
\forall a_g \in \fC_G, \, b_h \in \fC_h.
$$
The $G$-action and $G$-crossed braiding are subject to the following coherence axioms taken from \cite{MR3242743}.
For all $a_g \in \fC_g$, $b_h \in \fC_h$, and $c_k \in \fC_k$, the following diagrams commute, where suppress all labels.
\begin{equation}
\label{eq:Hexagon}
\begin{tikzcd}
F_g(b_h) \xz F_g(c_k)
\arrow[rr]
&& 
F_{ghg^{-1}}F_g(c_k) \xz F_g(b_h)
\arrow[d]
\\
F_{g}(b_h\xz c_k)
\arrow[d]
\arrow[u]
&&
F_{gh}(c_k)\xz F_g(b_h)
\\
F_{g}(F_h(c_k)\xz b_h)
\arrow[rr]
&&
F_gF_h(c_k) \xz F_g(b_h)
\arrow[u]
\end{tikzcd}
\tag{$\beta$1}
\end{equation}

\begin{equation}
\label{eq:Heptagon1}
\begin{tikzcd}
&
(a_g \xz b_h)\xz c_k
\arrow[dr]
\arrow[dl]
\\
a_g\xz (b_h \xz c_k)
\arrow[d]
&& 
(F_g(b_h)\xz a_g) \xz c_k
\arrow[d]
\\
F_{g}(b_h\xz c_k)\xz a_g
\arrow[d]
&&
F_g(b_h)\xz (a_g \xz c_k)
\arrow[d]
\\
(F_g(b_h) \xz F_g(c_k)) \xz a_g
\arrow[rr]
&&
F_g(b_h) \xz (F_g(c_k) \xz a_g)
\end{tikzcd}
\tag{$\beta$2}
\end{equation}

\begin{equation}
\label{eq:Heptagon2}
\begin{tikzcd}
&
a_g \xz (b_h\xz c_k)
\arrow[dr]
\arrow[dl]
\\
(a_g\xz b_h) \xz c_k 
\arrow[d]
&& 
a_g\xz (F_h(c_k) \xz b_h)
\arrow[d]
\\
F_{gh}(c_k)\xz (a_g \xz b_h)
\arrow[d]
&&
(a_g\xz F_h(c_k)) \xz b_h
\arrow[d]
\\
F_gF_h(c_k) \xz (a_g \xz b_h)
\arrow[rr]
&&
(F_gF_h(c_k)\xz a_g) \xz b_h
\end{tikzcd}
\tag{$\beta$3}
\end{equation}
\end{defn}

\begin{defn}
Given two $G$-crossed braided categories $\fC$ and $\fD$, a $G$-\emph{crossed braided functor} $(\bfA,\bfa): \fC \to \fD$ consists of 
a family of functors $(\bfA_g : \fC_g \to \fD_g)_{g\in G}$
together with a unitor isomorphism $\bfA^1: 1_\fD \to \bfA(1_\fC)$
and
a tensorator natural isomorphism $\bfA^2_{a_g, b_h} : \bfA(a_g)\xz \bfA(b_h) \to \bfA(a_g \xz b_h)$ for all $a_g \in \fC_g$ and $b_h \in \fC_h$
satisfying the obvious coherences.
The monoidal functor $\bfA=(\bfA_g, \bfA^1, \bfA^2)$ comes equipped with a family $\bfa=\{\bfa_g : F^\fD_g \circ \bfA \Rightarrow \bfA \circ F^\fC_g\}_{g\in G}$
of monoidal natural isomorphisms such that
for all $g,h\in G$, the following diagrams commute, where we suppress whiskering from the notation.
\begin{equation}
\label{eq:GCrossedFunctor-Gamma1}
\begin{tikzcd}
F^\fD_g \circ F^\fD_h \circ \bfA
\arrow[rr,"\mu^\fD_{g,h}"]
\arrow[d,"\bfa_h"]
&&
F^\fD_{gh} \circ \bfA
\arrow[d,"\bfa_{gh}"]
\\
F^\fD_g \circ \bfA \circ F^\fC_h
\arrow[dr,"\bfa_g"]
&&
\bfA\circ F^\fC_{gh}
\\
&
\bfA \circ F^\fD_g \circ F^\fC_h
\arrow[ur,"\mu^\fC_{g,h}"]
\end{tikzcd}
\tag{$\gamma1$}
\end{equation}
\begin{equation}
\label{eq:GCrossedFunctor-Gamma2}
\begin{tikzcd}
\bfA(a)\xz \bfA(b)
\arrow[rr,"\bfA^2_{a,b}"]
\arrow[d,"\beta^\fD"]
&&
\bfA(a\xz b)
\arrow[d,"\bfA(\beta^\fC)"]
\\
F^\fD_h(\bfA(b)) \xz \bfA(a)
\arrow[dr,"\bfa_h \xz \id"]
&&
\bfA(F^\fC(b)\xz a)
\\
&\bfA(F^\fC(b)) \xz \bfA(a)
\arrow[ur,"\bfA^2_{F^\fC(b), a}"]
\end{tikzcd}
\tag{$\gamma2$}
\end{equation}
\end{defn}

\begin{defn}
If $(\bfA, \bfa), (\bfB,\bfb): \fC \to \fD$ are $G$-crossed braided functors,
a $G$-crossed braided natural transformation
$h: (\bfA, \bfa) \Rightarrow (\bfB, \bfb)$
is a monoidal natural transformation $h: \bfA \Rightarrow \bfB$
such that
for all $g\in G$, the following diagram commutes.%, where again we suppress whiskering:
\begin{equation}
\label{eq:ActionCohrenceForTransformation}
\begin{tikzcd}[column sep=4em]
    F^\fD_g \circ \bfA
    \arrow[r,"F_g^\fD(h_{(\cdot)})"]
    \arrow[d, "\bfa_g"]
    &
    F_g^\fD \circ \bfB
    \arrow[d, "\bfb_g"]
    \\
    \bfA\circ F_g^\fC
    \arrow[r,"h_{F^\fC_g(\cdot)}"]
    &
    \bfB\circ F_g^\fC
\end{tikzcd}
\end{equation}
\end{defn} 

It is straightforward to verify that $G$-crossed braided categories, functors, and natural transformations assemble into a strict 2-category called $G\CrsBrd$ with familiar composition formulas similar to those from the strict 2-category of monoidal categories.
(See the proof of Proposition \ref{prop:2FunctorStrictlyUnitalAndAssociative} in Appendix \ref{sec:CoherenceProofsGCrossed} for full details.)

\begin{defn}[{Adapted from \cite[p.6]{MR3671186}}]
\label{defn:StrictGCrossedBraided}
A $G$-crossed braided category is called \emph{strict} if $\alpha, \lambda, \rho$ are all identities, and all $i_g$, $\psi^g$, $\mu^{g,h}$, and $\iota^g$ are identities.
Observe this implies that $F_e$ is the identity as well.
\end{defn}

By the main theorem of \cite{MR3671186}, every $G$-crossed braided category is equivalent (via a $G$-crossed braided functor which is an equivalence of categories) to a strict $G$-crossed braided category. In particular, the $2$-category $G\CrsBrd$ is equivalent to the full subcategory $G\CrsBrd^{\st}$ of strict $G$-crossed braided categories.

%%%%%%%%%%%%%%%%%%%%%%%%%%%%%%%%%%%%%%%%%%%%%%%%%%
\subsection{A strict 2-functor 
\texorpdfstring{$\TriCat_G^{\st}$}{TriCatGst}
to
\texorpdfstring{$G\CrsBrd^{\st}$}{GCrsBrdst}}
\label{sec:FromGBoring3CatsToGCrossedBriadedCats}

In this section, we construct a strict 2-functor $\TriCat_G^{\st} \to G\CrsBrd^{\st}$.
We begin by explaining how to obtain a strict $G$-crossed braided category $\fC$ from an object $\cC\in \TriCat_G^{\st}$, i.e., $\cC$ is a $\Gray$-monoid %with one object $*_\cC$ and 1-morphisms 
with 0-cells $\{g_\cC\}_{g\in G}$ with %$e_\cC = \id_{*_\cC}$ and 
0-composition the group multiplication.

\begin{construction}
\label{const:UnderlyingCategories}
For each $g\in \cG$, we define the category
$\fC_g := {\cC}(1_\cC \to g_\cC)$.
We denote $1$-cells in $\fC_g$ by small disks. 
For better readability, we distinguish different 1-cells in a given diagram by different shadings of the corresponding disks.
We will use the shorthand notation that white, green, and blue shaded disks correspond to 1-cells into $g_\cC, h_\cC,$ and $k_\cC$, respectively:
$$
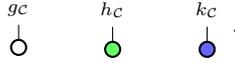
\tikzmath{
    \draw (0,0)  -- (0,.3) node [above] {\scriptsize{$g_\cC$}};
    \filldraw[fill=\gColor, thick] (0,0) circle (.1cm);
}
\qquad
\tikzmath{
    \draw (0,0)  -- (0,.3) node [above] {\scriptsize{$h_\cC$}};
    \filldraw[fill=\hColor, thick] (0,0) circle (.1cm);
}
\qquad
\tikzmath{
    \draw (0,0)  -- (0,.3) node [above] {\scriptsize{$k_\cC$}};
    \filldraw[fill=\kColor, thick] (0,0) circle (.1cm);
}\,.
$$
We define the bifunctor
$\xz_{g,h}: \fC_g \times \fC_h \to \fC_{gh}$
by $-\xz-$:
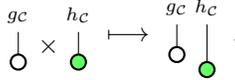
$$
\tikzmath{
	\draw (-.4,0) -- (-.4,.4) node [above] {\scriptsize{$g_\cC$}};
	\draw (.4,0) -- (.4,.4) node [above] {\scriptsize{$h_\cC$}};
	\filldraw[fill=\gColor, thick] (-.4,0) circle (.1cm);
	\filldraw[fill=\hColor, thick] (.4,0) circle (.1cm);
	\node at (0,.2) {$\times$};
}
\longmapsto
\tikzmath{
	\coordinate (a) at (-.2,0);
	\coordinate (b) at (.2,-.2);
	\draw (a) -- ($ (a) + (0,.4) $) node [above] {$\scriptstyle g_\cC$};
	\draw (b) -- ($ (b) + (0,.6) $) node [above] {$\scriptstyle h_\cC$};
	\filldraw[fill=\gColor, thick] (a) circle (.1cm);
	\filldraw[fill=\hColor, thick] (b) circle (.1cm);
}\,.
$$
The associator $\xz_{gh,k} \circ (\xz_{g,h}\times \fC_k) \Rightarrow \xz_{g,hk} \circ (\fC_g \times \xz_{h,k})$ is the identity.
The unit object $1_\fC := \id_e \in \fC_e$, which we denote by a univalent vertex attached to a dashed string. 
The unitors 
$\xz_{e,g} \circ (i \times -) \Rightarrow \id_{\fC_g}$
and
$\xz_{g,e} \circ (- \times i) \Rightarrow \id_{\fC_g}$
are also identities
$$
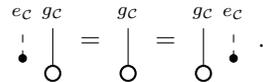
\tikzmath{
	\coordinate (a) at (-.2,0);
	\coordinate (b) at (.2,-.2);
	\draw[dashed] (a) -- ($ (a) + (0,.4) $) node [above] {$\scriptstyle e_\cC$};
	\draw (b) -- ($ (b) + (0,.6) $) node [above] {$\scriptstyle g_\cC$};
	\filldraw[fill=\gColor, thick] (b) circle (.1cm);
	\filldraw (a) circle (.05cm);
}
=
\tikzmath{
	\coordinate (a) at (-.2,-.2);
	\draw (a) -- ($ (a) + (0,.6) $) node [above] {$\scriptstyle g_\cC$};
	\filldraw[fill=\gColor, thick] (a) circle (.1cm);
}
=
\tikzmath{
	\coordinate (a) at (-.2,-.2);
	\coordinate (b) at (.2,0);
	\draw[dashed] (b) -- ($ (b) + (0,.4) $) node [above] {$\scriptstyle e_\cC$};
	\draw (a) -- ($ (a) + (0,.6) $) node [above] {$\scriptstyle g_\cC$};
	\filldraw[fill=\gColor, thick] (a) circle (.1cm);
	\filldraw (b) circle (.05cm);
}\,.
$$
Clearly the associators and unitors satisfy the obvious pentagon and triangle axioms of a $G$-crossed braided category.
\end{construction}

\begin{construction}[$G$-action]
\label{const:GAction}
We define a $G$-action $F_g:\fC_h \to \fC_{ghg^{-1}}$ by
$$
F_g\left(
\tikzmath{
\draw (0,0) -- (0,.4) node [above] {\scriptsize{$h_\cC$}};
\filldraw[fill=\hColor, thick] (0,0) circle (.1cm);
}
\right)
:=
\tikzmath{
\draw (-.4,.4) node [above] {\scriptsize{$g_\cC$}} -- (-.4,0) arc (-180:0:.4cm) -- (.4,.4) node [above, xshift=.1cm] {\scriptsize{$g^{-1}_\cC$}};
\draw (0,0) -- (0,.4) node [above] {\scriptsize{$h_\cC$}};
\filldraw[fill=\hColor, thick] (0,0) circle (.1cm);
}
=:
\tikzmath{
    \coordinate (a) at (0,0);
    \draw (a) -- ($ (a) + (0,.4) $) node [above,xshift=.15cm] {$\scriptstyle ghg^{-1}_\cC$};
    \draw[red] ($ (a) + (-.2,.4) $) -- ($ (a) + (-.2,0) $) arc (-180:0:.2cm) -- ($ (a) + (.2,.4) $); 
    \filldraw[fill=\hColor, thick] (a) circle (.1cm);
}\,.
$$
On the right hand side, we abbreviate this `cup' action by a single $g$-labelled red cup drawn under the respective node.
The functors $F_g$ are strict tensor functors, i.e., the tensorators $\psi^g: 
\xz_{ghg^{-1},gkg^{-1}}\circ (F_g \times F_g)
\Rightarrow 
F_g\circ \xz_{h,k}
$
are identity natural isomorphisms.
The tensorator $\mu_{g,h} : F_g \circ F_h \Rightarrow F_{gh}$ 
and the unit map
$\iota_h: \id_{\fC_h}\to F_e$
are also both identities.
It is straightforward to see that these identity natural isomorphisms $\psi^g$, $\mu^{g,h}$, and $\iota_h$ satisfy \ref{rel:psi1}, \ref{rel:psi2}, \ref{rel:mu1}, \ref{rel:mu2}, \ref{rel:iota1}, \ref{rel:iota2}.
\end{construction}

\begin{construction}[$G$-crossed braiding]
\label{const:GCrossedBraiding}
The $G$-crossed braiding natural isomorphisms $\beta^{g,h}$ are given by interchangers in $\cC$:
$$
\tikzmath{
	\coordinate (a) at (-.2,0);
	\coordinate (b) at (.2,-.2);
	\draw (a) -- ($ (a) + (0,.4) $) node [above] {$\scriptstyle g_\cC$};
	\draw (b) -- ($ (b) + (0,.6) $) node [above] {$\scriptstyle h_\cC$};
	\filldraw[fill=\gColor, thick] (a) circle (.1cm);
	\filldraw[fill=\hColor, thick] (b) circle (.1cm);
}
\overset{\phi}{\Rightarrow}
\tikzmath{
	\coordinate (a) at (-.2,-.2);
	\coordinate (b) at (.2,0);
	\draw (a) -- ($ (a) + (0,.6) $) node [above] {$\scriptstyle g_\cC$};
	\draw (b) -- ($ (b) + (0,.4) $) node [above] {$\scriptstyle h_\cC$};
	\filldraw[fill=\gColor, thick] (a) circle (.1cm);
	\filldraw[fill=\hColor, thick] (b) circle (.1cm);
}
=
\tikzmath{
\draw (-.4,.4) node [above] {\scriptsize{$g_\cC$}} -- (-.4,0) arc (-180:0:.4cm) -- (.4,.4) node [above, xshift=.1cm] {\scriptsize{$g^{-1}_\cC$}};
\draw (0,0) -- (0,.4) node [above] {\scriptsize{$h_\cC$}};
\filldraw[fill=\hColor, thick] (0,0) circle (.1cm);
\draw (.8,-.6) -- (.8,.4) node [above] {\scriptsize{$g_\cC$}};
\filldraw[fill=\gColor, thick] (.8,-.6) circle (.1cm);
}
=
\tikzmath{
	\coordinate (a) at (-.35,0);
	\coordinate (b) at (.35,-.2);
	\draw (a) -- ($ (a) + (0,.4) $) node [above,xshift=.15cm] {$\scriptstyle ghg^{-1}_\cC$};
	\draw[red] ($ (a) + (-.2,.4) $) -- ($ (a) + (-.2,0) $) arc (-180:0:.2cm) -- ($ (a) + (.2,.4) $); 
	\draw (b) -- ($ (b) + (0,.6) $) node [above] {$\scriptstyle h_\cC$};
	\filldraw[fill=\hColor, thick] (a) circle (.1cm);
	\filldraw[fill=\gColor, thick] (b) circle (.1cm);
}\,.
$$
\end{construction}

\begin{thm}
\label{thm:ConstructionOfGCrossedBraidedCategory}
The data 
$(\fC, \xz_{g,h}, F_g, \beta^{g,h})$
from Constructions \ref{const:UnderlyingCategories}, \ref{const:GAction}, and \ref{const:GCrossedBraiding}
forms a strict $G$-crossed braided category.
\DeferProof{sec:CoherenceProofsForGCrossedObjects}
\end{thm}

%%%%%%%%%%%%%%%%%%%%%%%%%%%%%%%%%%%%%%%%%%%%%%%%%%

Now suppose that $\cC, \cD\in \TriCat_G^{\st}$
and $A\in \TriCat_G^{\st}(\cC \to \cD)$ 
This means %$A(*_\cC) = *_\cD$ and 
$A(g_\cC) = g_\cD$ on the nose for all $g\in G$, 
the adjoint equivalence
$\mu^A: \xz_\cD \circ (A \times A) \Rightarrow A \circ \xz_\cC$
satisfies
$\mu^A_{g,h} = \id_{gh}\in \cD(g_\cD\xz h_\cD \to gh_\cD)$,
the adjoint equivalence
$\iota^A: (\iota^A_*,\iota^A_1): I_\cD \Rightarrow A\circ I_\cC$
satisfies $\iota^A_* = \id_{e_\cD}$, and $\iota^A_1 := A^1_e\in \cD(\id_{e_\cD}\Rightarrow B(\id_{e_\cC}))$,
and the associators and unitors $\omega, \ell,r$ are identities.

Let $\fC$ and $\fD$ be the strict $G$-crossed braided categories obtained from $\cC$ and $\cD$
respectively from Theorem \ref{thm:ConstructionOfGCrossedBraidedCategory}.
We now define a $G$-crossed braided functor $(\bfA,\bfa):\fC \to \fD$.

\begin{construction}
\label{const:From3FunctorToGCrossedBraidedFunctor}
First, for $a \in \fC_{g}:=\cC(e_\cC \to g_\cC)$, we define $\bfA(a) := A(a) \in \cD(e_\cD \to g_\cD)=\fD_g$.
For $x \in \fC_g(a \to b)$, we define $\bfA(x) := A(x) \in \fD_g(\bfA(a) \to \bfA(b))$.
It is straightforward to verify $\bfA$ is a functor.
We now endow $\bfA$ with a tensorator.
For $a\in \fC_g$ and $b\in \fC_h$, we define $\bfA^2_{a,b} \in \fD(\bfA(a) \xz \bfA(b) \to \bfA(a\xz b))$ to be 
$$
\mu_{a,b}^A:
(A(a)\xz \id_h)\circ (\id_e \xz A(b))
\Rightarrow
A((a\xz \id_h) \circ (\id_e \xz b)).
$$
We define the unitor by $\bfA^1 := A^1_e \in \fD(1_\fD \to \bfA(1_\fC)) = \cD(\id_{e_\cD} \to A(\id_{e_\cC}))$.
\end{construction}

\begin{lem}
\label{lem:FromUnderlying2FunctorToGMonoidalFunctor}
The data $(\bfA, \bfA^1, \bfA^2): \fC \to \fD$ is a $G$-graded monoidal functor.
\DeferProof{sec:CoherenceProofsForGCrossedObjects}
\end{lem}

We now construct the compatibility $\bfa$ between the $G$-actions on $\fC$ and $\fD$. 
For $a\in \fC_h=\cC(1_\cC \to h_\fC)$, we define
$\bfa^a_g : F^\fD_g (\bfA(a)) \Rightarrow \bfA (F^\fC_g(a))$ using the tensorator $\mu^A$:
\begin{equation}
\label{eq:ActionatorForGCrossedFunctor}
\begin{split}
F^\fD_g(\bfA(a))
=
\tikzmath{
	\draw (0,0) -- (0,.5) node [above] {$\scriptstyle h_\cD$};
	\draw[red] (-.5,.5) -- (-.5,0) node [left,xshift=.1cm] {$\scriptstyle g$} arc (-180:0:.5cm) -- (.5,.5); 
	\roundNbox{fill=white}{(0,0)}{.25}{.1}{.1}{$\scriptstyle \bfA(a)$}
}
&\overset{A^1_{g}}{\Rightarrow}
\tikzmath{
	\draw (0,-.5) -- (0,.5) node [above] {$\scriptstyle h_\cD$};
	\draw[red] (-1,.5) node [above] {$\scriptstyle g_\cD$} -- (-1,0) arc (-180:0:1cm) -- (1,.5) node [above] {$\scriptstyle g_\cD^{-1}$} ;
	\roundNbox{fill=white}{(0,-.5)}{.25}{.1}{.1}{$\scriptstyle A(a)$}
	\roundNbox{fill=white}{(-1,0)}{.25}{.3}{.3}{$\scriptstyle A(\id_{g_\cC})$}
}
\overset{\mu^{A}_{\id_{g_\cC}, a}}{\Rightarrow}
\tikzmath{
	\draw (0,0) -- (0,.5) node [above] {$\scriptstyle h_\cD$};
	\draw[red] (-1,.5) node [above] {$\scriptstyle g_\cD$} -- (-1,0) arc (-180:0:1cm) -- (1,.5) node [above] {$\scriptstyle g_\cD^{-1}$} ;
	\roundNbox{fill=white}{(-1,0)}{.25}{0}{1}{$\scriptstyle A(\id_{g_\cC}\otimes a)$}
}
\\&\overset{A^1_{g^{-1}_\cC}}{\Rightarrow}
\tikzmath{
	\draw (0,0) -- (0,.5) node [above] {$\scriptstyle h_\cD$};
	\draw[red] (-1,.5) node [above] {$\scriptstyle g_\cD$} -- (-1,0) -- (-1,-.5) arc (-180:0:1cm) -- (1,.5) node [above] {$\scriptstyle g_\cD^{-1}$} ;
	\roundNbox{fill=white}{(-1,0)}{.25}{0}{1}{$\scriptstyle A(\id_{g_\cC}\otimes a)$}
	\roundNbox{fill=white}{(1,-.5)}{.25}{.4}{.4}{$\scriptstyle A(\id_{g^{-1}_\cC})$}
}
\overset{\mu^{A}_{\id_{g_\cC}\xz a, \id_{g^{-1}_\cC}}}{\Rightarrow}
\tikzmath{
	\draw (0,0) -- (0,.5) node [above] {$\scriptstyle h_\cD$};
	\draw[red] (-.5,.5) -- (-.5,0)  node [below, yshift=-.25cm] {$\scriptstyle g$} arc (-180:0:.5cm) -- (.5,.5); 
	\roundNbox{fill=white}{(0,0)}{.25}{1}{1}{$\scriptstyle A(\id_{g_\cC} \xz a \xz \id_{g_\cC^{-1}})$}
}
\\&=
\bfA\left(
\tikzmath{
	\draw (0,0) -- (0,.5) node [above] {$\scriptstyle h_\cC$};
	\draw[red] (-.5,.5) -- (-.5,0)  node [left, xshift=.1cm] {$\scriptstyle g$} arc (-180:0:.5cm) -- (.5,.5); 
	\roundNbox{fill=white}{(0,0)}{.25}{0}{0}{$\scriptstyle a$}
}\right)
=
\bfA(F^\fC_g(a)).
\end{split}
\end{equation}

\begin{thm}
\label{thm:From3FunctorToGCrossedBraidedFunctor}
The data $(\bfA, \bfA^1,\bfA^2, \bfa)$ is a $G$-crossed braided monoidal functor.
\DeferProof{sec:CoherenceProofsForGCrossedObjects}
\end{thm}

\begin{prop}
\label{prop:2FunctorStrictlyUnitalAndAssociative}
The map $(A,\mu^A, \iota^A)\mapsto (\bfA, \bfa)$ strictly preserves identity 1-morphisms and composition of 1-morphisms.
\DeferProof{sec:CoherenceProofsForGCrossedObjects}
\end{prop}

%%%%%%%%%%%%%%%%%%%%%%%%%%%%%%%%%%%%%%%%%%%%%%%%%%

Suppose $\cC, \cD \in \TriCat_G^{\st}$, $A, B\in \TriCat_G^{\st}(\cC \to \cD)$, and $\eta\in \TriCat_G^{\st}(A \Rightarrow B)$.
This means that $\eta_*=e_\cD$ and $\eta_g = \id_{g_\cD}$ for all $g\in G$.
Let $\fC, \fD$ be the $G$-crossed braided categories obtained from $\cC,\cD$ respectively
from Theorem \ref{thm:ConstructionOfGCrossedBraidedCategory}.
Let $(\bfA,\bfa),(\bfB,\bfb):\fC \to \fD$ be the $G$-crossed braided functors obtained from $A,B$ respectively from Theorem \ref{thm:From3FunctorToGCrossedBraidedFunctor}, 

\begin{construction}
\label{const:GMonoidalTransformationFrom3Transformation}
We define $h:(\bfA,\bfa) \Rightarrow (\bfB,\bfb)$
by $h_a:=\eta_a \in \cD(\bfA(a) \Rightarrow \bfB(a))$ for $a\in \fC_g = \cC(1_\cC \to g_\cC)$.
\end{construction}

\begin{thm}
\label{thm:GMonoidalTransformation}
The data $h$ defines a $G$-crossed braided natural transformation $(\bfA,\bfa) \Rightarrow (\bfB, \bfb)$.
\DeferProof{sec:CoherenceProofsForGCrossedObjects}
\end{thm}

\begin{thm}
\label{thm:2Functor}
The map $\cC \mapsto (\fC, \otimes_{g,h}, F_g, \beta^{g,h})$, $(A,\mu^A, \iota^A) \mapsto (\bfA, \bfa)$, $\eta\mapsto h$ is a strict 2-functor $\TriCat_G^{\st} \to G\CrsBrd^{\st}$.
\end{thm}
\begin{proof}
By Proposition \ref{prop:2FunctorStrictlyUnitalAndAssociative}, we saw that this candidate 2-functor strictly preserves identity 1-morphisms and composition of 1-morphisms.
It remains to prove that the map
$\eta \mapsto h$ 
preserves identities and 2-composition.
This is immediate from Construction \ref{const:GMonoidalTransformationFrom3Transformation} as $(\eta\xt \zeta)_a = \eta_a \xt \zeta_a$ as $\TriCat_G^{\st}$ is strict.
\end{proof}

%%%%%%%%%%%%%%%%%%%%%%%%%%%%%%%%%%%%%%%%%%%%%%%%%%

%%%%%%%%%%%%%%%%%%%%%%%%%%%%%%%%%%%%%%%%%%%%%%%%%%
\subsection{The 2-functor is an equivalence}
\label{sec:2FunctorEquivalence}

We now show our 2-functor $\TriCat_G^{\st} \to G\CrsBrd^{\st}$ constructed in \S\ref{sec:FromGBoring3CatsToGCrossedBriadedCats} is an equivalence.

\paragraph{Essential surjectivity on objects.}
We begin by showing essential surjective, applying the techniques from \cite{MR4033513}.
Suppose $\fC$ is a strict $G$-crossed braided category.
We define a $\Gray$-monoid $\cC \in \TriCat_G^{\st}$
as follows.
The 0-cells of $\cC$ are simply the elements of $G$, and 0-composition $\xz$ is the group multiplication.
%, and the unit object is $1_\cC :=e$.
For $g,h\in G$, we define the hom category $\cC(g\to h):= \fC_{hg^{-1}}$.
These hom categories comprise a strict 2-category by defining vertical composition by the tensor product in $\fC$, i.e., if $a\in \cC(g\to h)=\fC_{hg^{-1}}$ and $b \in \cC(h\to k) = \fC_{kh^{-1}}$, we define
$$
b \xo_\cC a :=  b\otimes_\fC a = \otimes_{kh^{-1}, hg^{-1}}(b \times a).
$$
It is straightforward to verify that $\cC$ is a strict 2-category by strictness of the associator and unitor of $\fC$.

We now endow $\cC$ with a monoidal product and interchanger.
We define the monoidal product with identity 1-morphisms as follows.
Given $a\in \cC(g\to h)=\fC_{hg^{-1}}$ and $k\in G$, we set
$$
a\xz \id_k := a \in \fC_{hg^{-1}}=\fC_{hkk^{-1}g^{-1}} = \cC(gk \to hk),
$$
i.e., tensoring on the right with $\id_k$ does nothing.
Tensoring on the left, however, implements the $G$-action:
$$
\id_k \xz a := F_k(a) \in \fC(kgh^{-1}k^{-1}) = \cC(hg \to kh).
$$
The interchanger $\phi$ is given by the $G$-crossed braiding.
In more detail, given $a\in \cC(g\to h)=\fC_{hg^{-1}}$ and $b\in \cC(k\to \ell)_{\ell k^{-1}}$, we define
$$
\phi_{a,b}
:=
\beta^{hg^{-1},g\ell k^{-1}\ell^{-1}}_{a,b}
\in
\cC((a\xz \id_\ell)\xo(\id_g \xz b)
\Rightarrow
(\id_h\xz b)\xo(a \xz \id_k))
$$
Indeed, since $\fC$ is strict, $F_{hg^{-1}}=F_hF_{g^{-1}}$ on the nose, and
$\beta^{hg^{-1},g\ell k^{-1}g^{-1}}$ is a natural isomorphism
$$
\begin{tikzcd}
    \fC_{g\ell k^{-1}g^{-1}} \times \fC_{hg^{-1}}
    \arrow[bend left=30, drrr, "\xz_{hg^{-1}g\ell k^{-1}g^{-1}gh^{-1},hg^{-1}}\circ (F_{hg^{-1}}\times \id_{\fC_{hg^{-1}}})"]
    \\
    &
    \arrow[phantom, "\Uparrow \beta^{hg^{-1},g\ell k^{-1}g^{-1}}"]
    & &
    \fC_{h\ell k^{-1}g^{-1}}
    \\
    \fC_{hg^{-1}} \times \fC_{g\ell k^{-1}g^{-1}} 
    \arrow[uu, "\sigma"]
    \arrow[bend right=30, swap, urrr, " \xz_{hg^{-1},g\ell k^{-1}g^{-1}}"]
\end{tikzcd}.
$$

\begin{nota}
%One can view the construction of $\cC$ from $\fC$ as a categorification of the category of endomorphisms of a von Neumann algebra or a Cuntz $\rmC^*$-algebra \cite[\S2]{MR3635673}.
In the graphical calculus, one can think of a 1-cell in $\cC(g \to h)$ as a 1-cell in $\cC(1_\cC \to hg^{-1})$ with a $g$-strand on the right hand side, which does nothing.
$$
\tikzmath{
\draw (0,-.4) node[below] {$\scriptstyle g$} -- (0,.4) node[above] {$\scriptstyle h$};
    \filldraw[fill=\gColor, thick] (0,0) circle (.1cm);
}
:=
\tikzmath{
\draw[double] (0,0) -- (0,.4) node[above, xshift=.15cm] {$\scriptstyle hg^{-1}$};
\draw (.3,-.4) node[below] {$\scriptstyle g$} -- (.3,.4);
    \filldraw[fill=\gColor, thick] (0,0) circle (.1cm);
}
$$
Vertical composition is then given by
$$
\tikzmath{
\draw (0,-.7) node[below] {$\scriptstyle g$} -- node[right] {$\scriptstyle h$} (0,.7) node[above] {$\scriptstyle k$};
    \filldraw[fill=\gColor, thick] (0,-.3) circle (.1cm);
    \filldraw[fill=\hColor, thick] (0,.3) circle (.1cm);
}
:=
\tikzmath{
\draw[double] (0,.3) -- (0,.7) node[above, xshift=.15cm] {$\scriptstyle kh^{-1}$};
\draw[double] (.6,-.3) -- (.6,.7) node[above, xshift=.15cm] {$\scriptstyle hg^{-1}$};
\draw (1,.7) -- (1,-.7) node[below] {$\scriptstyle g$};
\filldraw[fill=\hColor, thick] (0,.3) circle (.1cm);
\filldraw[fill=\gColor, thick] (.6,-.3) circle (.1cm);
}
\in
\xz_{kh^{-1},hg^{-1}}(\fC_{kh^{-1}} \times \fC_{hg^{-1}}) = \fC_{kg^{-1}}.
$$
Tensoring by an identity strand on the right adds a strand on the right which does nothing, whereas tensoring by an identity on the left implements the $G$-action.\footnote{
This graphical calculus is analogous to 
diagrams for endomorphisms of a von Neumann algebra or a Cuntz $\rmC^*$-algebra \cite[\S2]{MR3635673} 
where
adding a strand labelled by an endomorphisms of a von Neumann algebra on the right does nothing, and adding a strand on the left implements the action of that endomorphism.
}
$$
\tikzmath{
\draw (0,-.4) node[below] {$\scriptstyle g$} -- (0,.4) node[above] {$\scriptstyle h$};
    \filldraw[fill=\gColor, thick] (0,0) circle (.1cm);
}
\xz
\tikzmath{
\draw (0,-.4) node[below] {$\scriptstyle k$} -- (0,.4) node[above] {$\scriptstyle k$};
}
:=
\tikzmath{
\draw (0,-.4) node[below] {$\scriptstyle g$} -- (0,.4) node[above] {$\scriptstyle h$};
    \filldraw[fill=\gColor, thick] (0,0) circle (.1cm);
\draw (.3,-.4) node[below] {$\scriptstyle k$} -- (.3,.4) node[above] {$\scriptstyle k$};
}
:=
\tikzmath{
\draw[double] (0,0) -- (0,.4) node[above, xshift=.15cm] {$\scriptstyle hg^{-1}$};
\draw (.3,-.4) node[below] {$\scriptstyle gk$} -- (.3,.4);
    \filldraw[fill=\gColor, thick] (0,0) circle (.1cm);
}
\in \fC_{hg^{-1}}
\qquad\qquad
\tikzmath{
\draw (0,-.4) node[below] {$\scriptstyle g$} -- (0,.4) node[above] {$\scriptstyle g$};
}
\xz
\tikzmath{
\draw (0,-.4) node[below] {$\scriptstyle h$} -- (0,.4) node[above] {$\scriptstyle k$};
    \filldraw[fill=\gColor, thick] (0,0) circle (.1cm);
}
:=
\tikzmath{
\draw[double] (0,0) -- (0,.4) node[above, xshift=.15cm] {$\scriptstyle kh^{-1}$};
\draw (-.3,-.4) node[below] {$\scriptstyle g$} -- (-.3,.4);
    \filldraw[fill=\gColor, thick] (0,0) circle (.1cm);
}
:=
\tikzmath{
\draw[double] (0,0) -- (0,.4) node[above, xshift=.15cm] {$\scriptstyle kh^{-1}$};
\draw (.4,-.4) node[below] {$\scriptstyle gh$} -- (.4,.4);
    \filldraw[fill=\gColor, thick] (0,0) circle (.1cm);
    \draw[red] (-.2,.4) -- (-.2,0) node [left] {\scriptsize{$g$}} arc (-180:0:.2cm) -- (.2,.4); 
}
\in \fC_{gkh^{-1}g^{-1}}.
$$ 
That the interchanger is given by the $G$-crossed braiding can now be represented graphically by
\begin{align*}
\tikzmath{
\draw (.5,-.5) node[below] {$\scriptstyle k$} -- (.5,.5) node[above] {$\scriptstyle \ell$};
\draw (0,-.5) node[below] {$\scriptstyle g$} -- (0,.5) node[above] {$\scriptstyle h$};
    \filldraw[fill=\gColor, thick] (0,.15) circle (.1cm);
    \filldraw[fill=\hColor, thick] (.5,-.15) circle (.1cm);
}
&=
\tikzmath{
\draw[double] (0,.15) -- (0,.5) node[above, xshift=.15cm] {$\scriptstyle hg^{-1}$};
\draw[double] (.8,-.15) -- (.8,.5) node[above, xshift=.15cm] {$\scriptstyle \ell k^{-1}$};
\draw (1.2,-.5) node[below] {$\scriptstyle k$} -- (1.2,.5);
\draw (.4,-.5) node[below] {$\scriptstyle g$} -- (.4,.5);
    \filldraw[fill=\gColor, thick] (0,.15) circle (.1cm);
    \filldraw[fill=\hColor, thick] (.8,-.15) circle (.1cm);
}
=
\tikzmath{
\draw[double] (0,.15) -- (0,.5) node[above,xshift=.15cm] {$\scriptstyle hg^{-1}$};
\draw[double] (.6,-.15) -- (.6,.5) node[above, xshift=.15cm] {$\scriptstyle \ell k^{-1}$};
\draw (1,-.5) node[below] {$\scriptstyle gk$} -- (1,.5);
    \filldraw[fill=\gColor, thick] (0,.15) circle (.1cm);
    \filldraw[fill=\hColor, thick] (.6,-.15) circle (.1cm);
    \draw[red] (.4,.5) -- (.4,-.15) node [left] {\scriptsize{$g$}} arc (-180:0:.2cm) -- (.8,.5); 
}
\xrightarrow{\beta^{hg^{-1},g\ell k^{-1}g^{-1}}}
\tikzmath{
\draw[double] (0,.15) -- (0,.5) node[above, xshift=.15cm] {$\scriptstyle \ell k^{-1}$};
\draw[double] (.6,-.15) -- (.6,.5) node[above, xshift=.15cm] {$\scriptstyle hg^{-1}$};
\draw (1,-.5) node[below] {$\scriptstyle gk$} -- (1,.5);
    \filldraw[fill=\hColor, thick] (0,.15) circle (.1cm);
    \filldraw[fill=\gColor, thick] (.6,-.15) circle (.1cm);
    \draw[red] (-.2,.5) -- (-.2,.15) node [left] {\scriptsize{$h$}} arc (-180:0:.2cm) -- (.2,.5); 
}
=
\tikzmath{
\draw (.5,-.5) node[below] {$\scriptstyle k$} -- (.5,.5) node[above] {$\scriptstyle \ell$};
\draw (0,-.5) node[below] {$\scriptstyle g$} -- (0,.5) node[above] {$\scriptstyle h$};
    \filldraw[fill=\gColor, thick] (0,-.15) circle (.1cm);
    \filldraw[fill=\hColor, thick] (.5,.15) circle (.1cm);
}\,.
\end{align*}
\end{nota}

One then checks that $\cC$ defined as above is a $\Gray$-monoid and thus defines an object of $\TriCat_G^{\st}$.
Indeed, the verification is entirely similar to \cite{MR4033513} (see also \cite[Construction~2.1.23]{1812.11933}).
Moreover, applying our construction from Theorem~\ref{thm:ConstructionOfGCrossedBraidedCategory} to the so defined $\cC$, recovers the $G$-crossed braided category $\fC$ on the nose. Hence, the strict $2$-functor $\TriCat_G^{\st} \to G\CrsBrd^{\st}$ is in fact strictly surjective on objects. 

%%%%%%%%%%%%%%%%%%%%%%%%%%%%%%%%%%%%%%%%%%%%%%%%%%
\paragraph{Essential surjectivity on $1$-morphisms.}
Let $\fC, \fD$ be the $G$-crossed braided categories obtained from $\cC,\cD\in \TriCat_G^{\st}$ respectively from Theorem \ref{thm:ConstructionOfGCrossedBraidedCategory}, and suppose $(\bfA,\bfa): \fC \to \fD$ is a $G$-crossed braided functor.
We now construct an $A\in \TriCat_G^{\st}(\cC \to \cD)$
which maps to $(\bfA,\bfa)$ under Construction \ref{const:From3FunctorToGCrossedBraidedFunctor} and \eqref{eq:ActionatorForGCrossedFunctor}.

\begin{construction}
\label{defn:PreimageUnderlying2Functor}
First, we must have $A(g_\cC) = g_\cD$ for all $g\in G$.
Recall that we have an isomorphism of categories $\cC(g_\cC \to h_\cC) \cong \cC(1_\cC \to hg_\cC^{-1})$ given by the strict 2-functor $R_{g^{-1}_\cC}=-\xz g^{-1}_\cC$.
For a 1-cell $x\in \cC(g_\cC \to h_\cC)$, we define
$A(x):=\bfA(x \xz g^{-1}_\cC)\xz g_\cD$,
and similarly for 2-cells $f\in \cC(x \Rightarrow y)$.
We define the unitor 
$$
A^1_g := \bfA^1_e \xz g_\cD\in \cD(\id_{g_\cD} \Rightarrow A(\id_{g_\cC}) = \bfA(\id_{e_\cC})\xz g_\cD),
$$
and 
for $x\in \cC(g_\cC \to h_\cC)$ and $y\in \cC(h_\cC \to k_\cC)$, the compositor $A^2_{y,x}$
as the composite
\begin{align*}
A(y)\xo A(x) 
&=
(\bfA(y \xz h^{-1}_\cC)\xz h_\cD)
\xo
(\bfA(x \xz g^{-1}_\cC)\xz g_\cD)
\\&=
(\bfA(y \xz h^{-1}_\cC)\xz hg^{-1}_\cD \xz g_\cD)
\xo
(\bfA(x \xz g^{-1}_\cC)\xz g_\cD)
\\&=
((\bfA(y \xz h^{-1}_\cC)\xz hg^{-1}_\cD)
\xo
(\bfA(x \xz g^{-1}_\cC))\xz g_\cD
&&
\text{$-\otimes g_\cD$ strict}
\\&=
((\bfA(y \xz h^{-1}_\cC)
\xz 
\bfA(x \xz g^{-1}_\cC))\xz g_\cD
&&
\text{Nudging \eqref{eq:Nudging}}
\\&\xrightarrow{\bfA^2}
(\bfA((y \xz h^{-1}_\cC) \xz (x \xz g^{-1}_\cC))\xz g_\cD
\\&=
(\bfA((y \xz h^{-1}_\cC \xz hg^{-1}_\cC) \xo (x \xz g^{-1}_\cC))\xz g_\cD
&&
\text{Nudging \eqref{eq:Nudging}}
\\&=
(\bfA((y \xz g^{-1}_\cC) \xo (x \xz g^{-1}_\cC))\xz g_\cD
\\&=
(\bfA((y \xo x) \xz g^{-1}_\cC)\xz g_\cD
&&
\text{$-\otimes g^{-1}_\cC$ strict}
\\&=
A(y\circ x).
\end{align*}
\end{construction}

\begin{lem}
\label{lem:PreimageUnderlying2Functor}
The data $(A,A^1,A^2)$ defines a 2-functor $\cC\to \cD$ such that 
$A(g_\cC) = g_\cD$ for all $g\in G$.
\DeferProof{sec:CoherenceProofsForEquivalence}
\end{lem}

\begin{construction}
\label{defn:PreimageMonoidal2Functor}
The adjoint equivalence $\mu^A: \xz_\cD \circ (A\times A) \Rightarrow A \circ \xz_\cC$ is defined as follows.
First, $\mu^A_{g,h}:= \id_{gh_\cD}\in \cD(g_\cD \xz h_\cD \Rightarrow gh_\cD)$.
For $x\in \cC(g_\cC \to h_\cC)$ and $y\in \cC(k_\cC \to \ell_\cC)$, we define the natural isomorphism $\mu^A_{x,y}\in \cD(A(x)\xz A(y) \Rightarrow A(x \xz y))$ by the composite
\begin{align*}
A(x)\xz A(y)
&=
\bfA(x \xz g^{-1}_\cC))\xz g_\cD \xz \bfA(y \xz k^{-1}_\cC))\xz k_\cD
\\&=
\bfA(x \xz g^{-1}_\cC))\xz F^\fD_g(\bfA(y \xz k^{-1}_\cC)))\xz g_\cD \xz k_\cD
\\&=
\bfA(x \xz g^{-1}_\cC))\xz F^\fD_g(\bfA(y \xz k^{-1}_\cC)))\xz gk_\cD
\\&\xrightarrow{\bfa}
\bfA(x \xz g^{-1}_\cC))\xz \bfA(F^\fC_g(y \xz k^{-1}_\cC)))\xz gk_\cD
\\&\xrightarrow{\bfA^2}
\bfA(x \xz g^{-1}_\cC\xz F^\fC_g(y \xz k^{-1}_\cC))\xz gk_\cD
\\&=
\bfA(x \xz F^\fC_{g^{-1}}F^\fC_g(y \xz k^{-1}_\cC)\xz g^{-1}_\cC)\xz gk_\cD
\\&=
\bfA(x \xz F^\fC_e(y \xz k^{-1}_\cC)\xz g^{-1}_\cC)\xz gk_\cD
\\&=
\bfA(x \xz y \xz k^{-1}_\cC\xz g^{-1}_\cC)\xz gk_\cD
\\&=
\bfA(x \xz y \xz (gk)^{-1}_\cC)\xz gk_\cD
\\&=
A(x\xz y).
\end{align*}
The adjoint equivalence $\iota^A = (\iota^!_*, \iota^A_1): I_\cD \Rightarrow A\circ I_\cC$ is defined by $\iota^A_* := \id_{e_\cD}$, $\iota^A_1 := A^1_e$.
The associator $\omega^A$ and the unitors $\ell^A, r^A$ are all defined to be identities.
\end{construction}

\begin{thm}
\label{thm:Preimage3Functor}
The data $(A,\mu^A, \iota^A)$ defines a 1-morphism in $\TriCat_G^{\st}( \cC \to \cD)$.
\DeferProof{sec:CoherenceProofsForEquivalence}
\end{thm}

Finally, we observe that the $G$-crossed braided functor constructed from $A$ 
in Theorem \ref{thm:From3FunctorToGCrossedBraidedFunctor}
is exactly $(\bfA,\bfa)$ by construction.
Indeed, the strict 2-functors $-\xz e_\cC$ and $-\xz e_\cD$ are the identity on the nose.
Hence, $\TriCat_G^{\st} \to G\CrsBrd^{\st}$ is in fact surjective on $1$-morphisms on the nose.

%%%%%%%%%%%%%%%%%%%%%%%%%%%%%%%%%%%%%%%%%%%%%%%%%%
%\subsection{Fully faithful on 2-morphisms}
\paragraph{Fully faithfulness on $2$-morphisms.}
For $\cC, \cD \in \TriCat_G^{\st}$, $A, B\in \TriCat_G^{\st}(\cC \to \cD)$, and $\eta\in \TriCat_G^{\st}(A \Rightarrow B)$,
let $\fC, \fD$ be the $G$-crossed braided categories obtained from $\cC,\cD$ respectively
from Theorem \ref{thm:ConstructionOfGCrossedBraidedCategory}, and let $(\bfA,\bfa),(\bfB,\bfb):\fC \to \fD$ be the $G$-crossed braided functors obtained from $A,B$ respectively from Theorem \ref{thm:From3FunctorToGCrossedBraidedFunctor}.
In Construction \ref{const:GMonoidalTransformationFrom3Transformation}
we defined $h:(\bfA,\bfa) \Rightarrow (\bfB,\bfb)$ by $h_a := \eta_a \in \cD(\bfA(a) \Rightarrow \bfB(a))$ for $a\in \fC_g = \cC(1_\cC \to g_\cC)$.

\begin{thm}
\label{thm:FullyFaithfulOn2Morphisms}
The map $\eta \mapsto h$ is a bijection $\TriCat_G^{\st}(A \Rightarrow B)\to G\CrsBrd^{\st}(\bfA \Rightarrow \bfB)$.
\DeferProof{sec:CoherenceProofsForEquivalence}
\end{thm}

%auto-ignore
%this ensures the arxiv doesn't try to start TeXing here.
%!TEX root =../Equivalence.tex

%%%%%%%%%%%%%%%%%%%%%%%%%%%%%%%%%%%%%%%%%%%%%%%%
%%%%%%%%%%%%%%%%%%%%%%%%%%%%%%%%%%%%%%%%%%%%%%%%
%%%%%%%%%%%%%%%%%%%%%%%%%%%%%%%%%%%%%%%%%%%%%%%%
\section{Induced properties and structures}
\label{sec:inducedprops}
Theorem~\ref{thm:Main} constructs an equivalence between $G$-crossed braided categories and $1$-surjective $G$-pointed $3$-categories. In this section, we investigate how various additional structures and properties of $3$-categories, such as linearity and dualizability translate into the corresponding properties of $G$-crossed braided categories.
Let $\pi^\cC:\rmB G\to \cC$ be a $1$-surjective $G$-pointed $3$-category and let $\{\fC_g\}_{g\in G}$ be the corresponding $G$-crossed braided category constructed via Theorem~\ref{thm:Main}.

The first result below is immediate.

\begin{prop}[Linearity]
\label{prop:Linearity}
If $\cC$ is a linear $3$-category, then $\fC:=\bigoplus_{g\in G} \fC_g $ is a $G$-crossed braided category in these sense of~\cite[\S8.24]{MR3242743}. 
%If $\cC$ is a linear $\Gray$-monoid, then
%$\fC:=\bigoplus_{g\in G} \fC_g$ is a $G$-crossed braided linear monoidal category in the sense of~\cite{}\todo{citation}.
\qed
\end{prop}

%Before the next proposition, we state our conventions for left/right adjoints in a monoidal 2-category $\cD$ and left/right duals in a rigid monoidal category $\cB$.
Following the conventions in~\cite[Defs.~2.1.1,~2.1.2,~2.1.4]{1312.7188}, given a $1$-morphism $f:c\to d$ in a $2$-category, we write $(f^L: d\to c,  \ev_f:f^L\circ f \To \id_c,\coev_f: \id_d \To f\circ f^L )$ for the left adjoint of $f$ and $(f^R:d\to c, {}_f\ev: f \circ f^R \Rightarrow \id_d, {}_f\coev: \id_c \Rightarrow f^R\circ f)$ for its right adjoint. 
Given an object $x$ in a monoidal category $\fM$, we write $(x^\vee,\ev_x : x^\vee\otimes x\to 1_{\fM},\coev_x: 1_{\fM} \to x\otimes x^\vee)$ for the right dual of $x$, and $({}^\vee x,{}_x\ev : x\otimes{}^\vee x\to 1_{\fM},{}_x\coev: 1_{\fM} \to {}^\vee x\otimes x)$ for the left dual of $x$.

Recall that a braided monoidal category has right duals if and only if it has left duals.

\begin{lem}
\label{lem:GCrossedOneDualImpliesOther}
A $G$-crossed braided monoidal category $\fC=\bigoplus_{g\in G} \fC_g$ has right duals if and only if it has left duals.
\end{lem}
\begin{proof}
We prove that having right duals implies having left duals; the other direction is analogous.
Suppose $x\in \fC_g$ has a right dual
$(x^\vee\in \fC_{g^{-1}},\ev_x : x^\vee\otimes x\to 1_{\fC},\coev_b: 1_{\fC} \to x\otimes x^\vee)$.
Then, ${}^\vee x:=g^{-1}(x^\vee)$ is a left dual with the following evaluation and coevaluation morphism:
\begin{align*}
{}_x\ev 
&:= 
\ev_x\circ (\mu_{g,g^{-1}}^x \otimes \id_{x^\vee})\circ(\beta_{x,g^{-1}(x^\vee)}) : x\otimes g^{-1}(x^\vee) \to 1_\fC 
\\
{}_x\coev 
&:=
\beta^{-1}_{g^{-1}(x^\vee), x}\circ \psi^{g^{-1}}_{x,x^\vee}\circ g^{-1}(\coev_x)
:
1_\fC \to {}^\vee x \otimes x
\end{align*}
That these maps satisfy the zig-zag/snake equations is straightforward.
\end{proof}

\begin{rem}
\label{rem:OneAdjointImpliesBoth}
Similar to Lemma \ref{lem:GCrossedOneDualImpliesOther}, every 2-morphism between invertible 1-morphisms (or more generally, between fully dualizable 1-morphisms) in a 3-category has a right adjoint if and only if it has a left adjoint \cite[Prop.~A.2]{ReutterThesis}.
\end{rem}

\begin{prop}[Rigidity]
\label{prop:Rigidity}
Suppose $\cC$ is linear so that Proposition \ref{prop:Linearity} holds. If every $2$-morphism in $\cC(\pi^{\cC}(e) \to \pi^\cC(g))$ has either a right or a left adjoint (and thus necessarily both by Remark \ref{rem:OneAdjointImpliesBoth}) for all $g\in G$, then the $G$-crossed braided linear monoidal category $\fC$ is rigid.
\end{prop}
\begin{proof}
As the statement and the assumptions in this proposition are invariant under equivalences in $\TriCat_G$ and $ G\CrsBrd$ respectively, we may assume that $\cC$ is an object of $\TriCat_G^{\st}$, and hence the delooping $\rmB \cA$ of a $\Gray$-monoid $\cA$ whose set of $0$-cells is $\{g_\cA\}_{g\in G}$ with 0-composition $\xz$ the group multiplication, and that $\fC$ is the strict $G$-crossed braided category obtained from Constructions~\ref{const:UnderlyingCategories}--\ref{const:GCrossedBraiding}. 

By Lemma~\ref{lem:GCrossedOneDualImpliesOther} it suffices to prove that for every $g\in G$, every object $x\in \fC_g$ (given by a $1$-cell $x:1_{\cA}\to g_{\cA}$ in the strict $2$-category $\cA$) has either a right dual or a left dual. 
We assume the underlying $1$-cell $x:1_\cA \to g_\cA$ has a left adjoint $x^L: g_\cA \to 1_\cA$ in the 2-category $\cA$ and prove that the corresponding object $x\in \fC_g$ has a right dual in the monoidal category $\fC$.
We use the shorthand notation
$$
\tikzmath{
\draw (0,0) -- (0,.3);
\CircleMorphism{(0,0)}{\gColor}
}
:=
\tikzmath{
\node (g) at (0,.6) {$\scriptstyle g_\cC$};
\node[draw,rectangle, thick, rounded corners=5pt] (x) at (0,0) {$\scriptstyle x$};
\draw (x) to[in=-90,out=90] (g);
}
\qquad\qquad
\tikzmath{
\draw (0,0) -- (0,-.3);
\CircleMorphism{(0,0)}{\gColor}
}
:=
\tikzmath{
\node (g) at (0,-.6) {$\scriptstyle g_\cC$};
\node[draw,rectangle, thick, rounded corners=5pt] (x) at (0,0) {$\scriptstyle x^L$};
\draw (g) to[in=-90,out=90] (x);
}
$$

Setting
\[
x^\vee 
:= 
\tikzmath{
\phantom{\draw (0,-.1) -- (0,-.4) node[below] {$\scriptstyle g^{-1}_\cC$};};
\draw (0,-.1) arc (-180:0:.2cm) -- (.4,.4) node [above] {$\scriptstyle g^{-1}_\cC$};
\CircleMorphism{(0,0)}{\gColor}
} 
=
\tikzmath{
\draw (0,-.1) -- (0,-.4) node[below] {$\scriptstyle g_\cC$};
\draw (.3,-.4) -- (.3,.4) node[above] {$\scriptstyle g^{-1}_\cC$};
\CircleMorphism{(0,0)}{\gColor}
} 
\in \fC_{g^{-1}},
\]
it is a direct consequence of the adjunction between $x$ and $x^L$ (here denoted $\varepsilon: x^L \xo x \Rightarrow \id_{e_\cC}$ and $\eta: \id_{g_\cC} \Rightarrow x \xo x^L$) 
%together with Remark~\ref{rem:CoherenceForEvaluation} 
that the following evaluation and coevaluation morphisms exhibit $x^\vee$ as a right dual of $x$:
\begin{align*}
\ev_x&:
x^\vee \otimes x
=
\tikzmath{
\draw (-.4,-.1) -- (-.4,-.4) node[below] {$\scriptstyle g_\cC$};
\draw (0,-.4) -- (0,.4) node[above] {$\scriptstyle g^{-1}_\cC$};
\draw (.4,-.3) -- (.4,.4) node[above] {$\scriptstyle g_\cC$};
\CircleMorphism{(-.4,0)}{\gColor}
\CircleMorphism{(.4,-.3)}{\gColor}
} 
=
\tikzmath{
\draw (0,0) -- (0,-.4);
\path (0,0) -- (0,.2) node[above] {$\scriptstyle e_\cC$};
\CircleMorphism{(0,0)}{\gColor}
\CircleMorphism{(0,-.4)}{\gColor}
}
\,\,\,
\overset{\varepsilon}{\Longrightarrow}
\,\,\,
\tikzmath{
\draw[thick, dashed, rounded corners=5pt] (-.3,-.3) rectangle (.3,.3);
\path (0,0) -- (0,.4) node[above] {$\scriptstyle e_\cC$};
}
=1_\fC
\\
\coev_x &:
1_\fC = 
\tikzmath{
\draw[thick, dashed, rounded corners=5pt] (-.3,-.3) rectangle (.3,.3);
\path (0,0) -- (0,.4) node[above] {$\scriptstyle e_\cC$};
}
=
\tikzmath{
\draw (0,-.6) -- (0,.6) node[above] {$\scriptstyle g_\cC$};
\draw (.4,-.6) -- (.4,.6) node[above] {$\scriptstyle g^{-1}_\cC$};
\draw[thick, dashed, rounded corners=5pt] (-.2,-.3) rectangle (.2,.3);
}
\,\,\,
\overset{\eta}{\Longrightarrow}
\,\,\,
\tikzmath{
\draw (0,-.6) node[below] {$\scriptstyle g_\cC$} -- (0,-.3);
\draw (0,.3) -- (0,.6) node[above] {$\scriptstyle g_\cC$};
\draw (.4,-.6) -- (.4,.6) node[above] {$\scriptstyle g^{-1}_\cC$};
\CircleMorphism{(0,.2)}{\gColor}
\CircleMorphism{(0,-.2)}{\gColor}
}
=x\otimes x^\vee.
\end{align*}
We explicitly prove the relation $(\id_{x}\otimes \ev_x)\circ (\coev_{x} \otimes \id_x)= \id_x$; the other relation is left to the reader.
$$
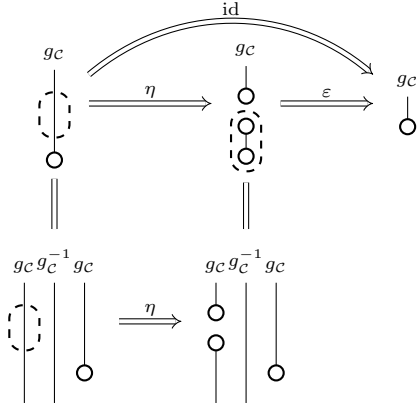
\begin{tikzcd}
\tikzmath{
\draw (0,-.6) -- (0,.6) node [above] {$\scriptstyle g_\cC$};
\CircleMorphism{(0,-.6)}{\gColor}
\draw[thick, dashed, rounded corners=5pt] (-.2,-.3) rectangle (.2,.3);
}
\arrow[d, equals]
\arrow[r, Rightarrow, "\eta"]
\arrow[rr, Rightarrow, bend left=40, "\id"]
&
\tikzmath{
\draw (0,.3) -- (0,.6) node[above] {$\scriptstyle g_\cC$};
\draw (0,-.6) -- (0,-.3);
\CircleMorphism{(0,.2)}{\gColor}
\CircleMorphism{(0,-.2)}{\gColor}
\CircleMorphism{(0,-.6)}{\gColor}
\draw[thick, dashed, rounded corners=5pt] (-.2,-.8) rectangle (.2,0);
}
\arrow[d, equals]
\arrow[r, Rightarrow, "\varepsilon"]
&
\tikzmath{
\draw (0,.3) -- (0,.6) node[above] {$\scriptstyle g_\cC$};
\CircleMorphism{(0,.2)}{\gColor}
}
\\
\tikzmath{
\draw (-.4,-1) -- (-.4,.6) node[above] {$\scriptstyle g_\cC$};
\draw (0,-1) -- (0,.6) node[above] {$\scriptstyle g^{-1}_\cC$};
\draw (.4,-.6) -- (.4,.6) node[above] {$\scriptstyle g_\cC$};
\CircleMorphism{(.4,-.6)}{\gColor}
\draw[thick, dashed, rounded corners=5pt] (-.6,-.3) rectangle (-.2,.3);
}
\arrow[r, Rightarrow, "\eta"]
&
\tikzmath{
\draw (-.4,.3) -- (-.4,.6) node[above] {$\scriptstyle g_\cC$};
\draw (-.4,-.1) -- (-.4,-1);
\draw (0,-1) -- (0,.6) node[above] {$\scriptstyle g^{-1}_\cC$};
\draw (.4,-.6) -- (.4,.6) node[above] {$\scriptstyle g_\cC$};
\CircleMorphism{(-.4,.2)}{\gColor}
\CircleMorphism{(-.4,-.2)}{\gColor}
\CircleMorphism{(.4,-.6)}{\gColor}
}
\end{tikzcd}
$$
In the diagram above, 
the composite
$(\id_{x}\otimes \ev_x)\circ \alpha\circ (\coev_{x} \otimes \id_x)$
is the path going down and then to the right.
The square commutes as both maps are identical.
The triangle commutes by the adjunction.
\end{proof}

\begin{rem}
There is a version of Proposition \ref{prop:Rigidity} that holds in the non-linear setting; one must be careful to define the correct notion of duals.
\end{rem}

The following propostion is also immediate.

\begin{prop}[Multifusion]
\label{prop:Multifusion}
Suppose $\cC$ is as in the hypotheses of Proposition \ref{prop:Rigidity} so that $\fC$ is rigid linear monoidal.
If each $2$-morphism category $\cC(\pi^{\cC}(e)\to \pi^{\cC}(g))$ is semisimple, then $\fC$ is multifusion.
If moreover the $2$-morphism $\id_{\pi^{\cC}(e)}: \pi^{\cC}(e)\to \pi^{\cC}(e)$ is simple, then $\fC$ is fusion.
\qed
\end{prop}

Since the fusion 2-categories of \cite{1812.11933} satisfy the hypotheses of Proposition \ref{prop:Multifusion}, we get the following corollary.

\begin{cor}
If $\cC$ is a fusion $2$-category in the sense of~\cite{1812.11933} and $\pi:G \to \cC$ is a monoidal $2$-functor which is essentially surjective on objects, then $\fC$ is a $G$-crossed braided fusion category. 
\qed
\end{cor}

% \begin{rem}[Pivotality]
% We expect that a spatial pivotal structure on a $\Gray$-monoid $\cC$ in the sense of~\cite{1211.0529} and an appropriately compatible $G$-action on $\cC$ induces a spherical pivotal structure on the $G$-crossed braided monoidal category $\fC$ constructed from $\cC$ and its $G$-action. 
% However, the notion of spatial pivotality from~\cite{1211.0529} is not compatible with weak monoidal 2-equivalences (see e.g.~\cite{1812.11933}) and in particular will not be compatible with $\Gray$-ification. 
% We believe this is a good reason to develop a fully weak notion of pivotality for monoidal 2-categories together with an appropriate $\Gray$-ification theorem.
% \end{rem}

\begin{rem}[Unitarity]
We define a dagger structure on a $\Gray$-monoid $\cC$ in terms of the unpacked Definition \ref{def:monoidal2cat} above.
We require the strict 2-category $\cC$ to be a dagger 2-category, all 2-functors to be dagger 2-functors, and all isomorphisms to be unitary.
Similarly, one can define the notion of a C$^*$ or W$^*$ $\Gray$-monoid.
Given a dagger $\Gray$-monoid $\cC$ and an appropriately compatible $G$ action on $\cC$ (all actions are by dagger functors and all isomorphisms are unitary), we expect our construction will yield a  $G$-crossed braided dagger category.
We expect analogous results in the C$^*$ and W$^*$ settings.
However, %as in the pivotal setting, 
the notion of dagger $\Gray$-monoid is not compatible with weak equivalences and $\Gray$-ification.
These notions merit further study.
\end{rem}

%%%%%%%%%%%%%%%%%%%%%%%%%%%%%%%%%%%%%%%%%%%%%%
\appendix
%auto-ignore
%this ensures the arxiv doesn't try to start TeXing here.
%!TEX root =../Equivalence.tex

%%%%%%%%%%%%%%%%%%%%%%%%%%%%%%%%%%%%%%%%%%%%%%%%%%
%%%%%%%%%%%%%%%%%%%%%%%%%%%%%%%%%%%%%%%%%%%%%%%%%%
%%%%%%%%%%%%%%%%%%%%%%%%%%%%%%%%%%%%%%%%%%%%%%%%%%
\section{Functors and higher morphisms between \texorpdfstring{$\Gray$}{Gray}-monoids}
\label{sec:Weak3CategoryCoherences}

In this section, we unpack the definitions of trihomomorphism, tritransformation, trimodiciation, and perturbation of~\cite[Def 4.10, 4.16, 4.18, 4.21]{MR3076451} between two $\Gray$-monoids in terms of the graphical calculus.

We remind the reader that as in Notation \ref{nota:CellsInGrayMonoids}, given a $\Gray$-monoid $\cC$, we refer to its objects, 1-morphisms, and 2-morphisms as 0-cells, 1-cells, and 2-cells respectively in order to distinguish these basic components of $\cC$ from morphisms in an ambient category in which $\cC$ lives.
The notion of adjoint equivalence in a $2$-category is well-known, so we will not unpack it further.
A \emph{biadjoint biequivalence} \cite{MR2972968} $0$-cell $\alpha$ in a $\Gray$-monoid
consists of 0-cells $\alpha_*,\alpha^{-1}_*$ 
which we depict in the graphical calculus
as oriented \textcolor{red}{red} strands:
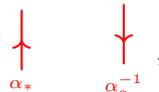
$$
\tikzmath{
\draw[thick, red, mid>] (0,-.4) node [below] {$\scriptstyle\alpha_*$} -- (0,.4) ;
}
\qquad
\tikzmath{
\draw[thick, red, mid<] (0,-.4) node [below] {$\scriptstyle\alpha_*^{-1}$} -- (0,.4) ;
}
\,,
$$
and cup and cap 1-morphisms
$$
\tikzmath{
\draw[thick, red, mid>] (.2,.4) -- (.2,.3) arc (0:-180:.2cm) -- (-.2,.4);
}
\qquad\qquad
\tikzmath{
\draw[thick, red, mid<] (.2,.4) -- (.2,.3) arc (0:-180:.2cm) -- (-.2,.4);
}
\qquad\qquad
\tikzmath{
\draw[thick, red, mid>] (-.2,-.4) -- (-.2,-.3) arc (180:0:.2cm) -- (.2,-.4);
}
\qquad\qquad
\tikzmath{
\draw[thick, red, mid<] (-.2,-.4)  -- (-.2,-.3) arc (180:0:.2cm) -- (.2,-.4); 
}\,
$$
together with 2-isomorphisms
\begin{equation}
\label{eq:BiadjointBiequivalence}
\tag{BB}
\begin{split}
&
\tikzmath{
\draw[thick, red, mid>] (-.2,-.4) -- (-.2,.4) ;
\draw[thick, red, mid<] (.2,-.4) -- (.2,.4);
}
\cong
\tikzmath{
\draw[thick, red, mid>] (.2,.4) -- (.2,.3) arc (0:-180:.2cm) -- (-.2,.4);
\draw[thick, red, mid>] (-.2,-.4) -- (-.2,-.3) arc (180:0:.2cm) -- (.2,-.4);
}
\qquad\qquad
\tikzmath{
 \draw[thick, red, mid<] (-.2,-.4)  -- (-.2,.4) ;
\draw[thick, red, mid>] (.2,-.4) -- (.2,.4);
}
\cong
\tikzmath{
\draw[thick, red, mid<] (.2,.4) -- (.2,.3) arc (0:-180:.2cm) -- (-.2,.4);
\draw[thick, red, mid<] (-.2,-.4)  -- (-.2,-.3) arc (180:0:.2cm) -- (.2,-.4); 
}
\qquad\qquad
\tikzmath{
\draw[thick, red, mid<] (0,0) node [left, xshift=-.3cm] {$\scriptstyle\alpha_*$} circle (.3cm) ;
}
\cong
\tikzmath{
\roundNbox{dashed}{(0,0)}{.4}{0}{0}{}
}
\cong
\tikzmath{
\draw[thick, red, mid>] (0,0) node [right, xshift=.25cm] {$\scriptstyle\alpha_*$} circle (.3cm) ;
}
\\&
\tikzmath{
\draw[thick, red, mid>] (-.2,-.4) -- (-.2,.1) arc (180:0:.1cm) -- (0,-.1) arc (-180:0:.1cm) -- (.2,.4);
}
\cong
\tikzmath{
\draw[thick, red, mid>] (0,-.4) -- (0,.4);
}
\cong
\tikzmath{
\draw[thick, red, mid>] (.2,-.4) -- (.2,.1) arc (0:180:.1cm) -- (0,-.1) arc (0:-180:.1cm) -- (-.2,.4);
}
\qquad\qquad
\tikzmath{
\draw[thick, red, mid<] (-.2,-.4) -- (-.2,.1) arc (180:0:.1cm) -- (0,-.1) arc (-180:0:.1cm) -- (.2,.4);
}
\cong
\tikzmath{
\draw[thick, red, mid<] (0,-.4) -- (0,.4);
}
\cong
\tikzmath{
\draw[thick, red, mid<] (.2,-.4) -- (.2,.1) arc (0:180:.1cm) -- (0,-.1) arc (0:-180:.1cm) -- (-.2,.4);
}
\end{split}
\end{equation}
fulfilling certain coherence conditions; see
\cite[Def.~2.1, Rem.~2.2, Def.~2.3]{MR2972968}.

%%%%%%%%%%%%%%%%%%%%%%%%%%%%%%%%%%%%
\subsection{3-functors between \texorpdfstring{$\Gray$}{Gray}-monoids}

\begin{defn}
\label{defn:3functor}
Suppose $\cC,\cD$ are $\Gray$-monoids.
A 3-functor $A: \rmB\cC \to \rmB\cD$ consists of:
\begin{enumerate}[label=(F-\Roman*)]
\item
\label{Functor:A}
A 2-functor $(A, A^1,A^2): \cC \to \cD$. That is, a function on globular sets $A$, 
an invertible $2$-cell $A^1_c: \id_{A(c)}\Rightarrow A(\id_c)$, and 
an invertible $2$-cell
$A^2_{y,x}:A(y)\xo A(x) \Rightarrow A(y\xo x)$, which satisfy the following coherence conditions:
\begin{enumerate}[label=\ref{Functor:A}.\roman*]
\item
\label{Functor:A.associative}
For all $x\in \cC(a\to b)$, $y\in \cC(b\to c)$, and $z\in \cC(c\to d)$, the following diagram commutes:
$$
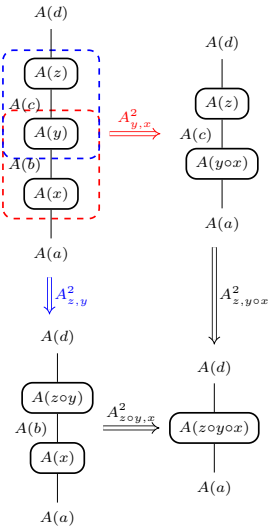
\begin{tikzpicture}[baseline= (a).base]
\node[scale=.8] (a) at (0,0){
\begin{tikzcd}
\tikzmath{
\node (a) at (0,0) {$\scriptstyle A(a)$};
\node[draw,rectangle, thick, rounded corners=5pt] (Ax) at (0,1) {$\scriptstyle A(x)$};
\node[draw,rectangle, thick, rounded corners=5pt] (Ay) at (0,2) {$\scriptstyle A(y)$};
\node[draw,rectangle, thick, rounded corners=5pt] (Az) at (0,3) {$\scriptstyle A(z)$};
\node (d) at (0,4) {$\scriptstyle A(d)$};
\draw[red, thick, dashed, rounded corners=5pt] (-.8,.65) rectangle (.8,2.45);
\draw[blue, thick, dashed, rounded corners=5pt] (-.8,1.65) rectangle (.8,3.45);
\draw (a) to[in=-90,out=90] (Ax);
\draw (Ax) to[in=-90,out=90] node[left] {$\scriptstyle A(b)$} (Ay);
\draw (Ay) to[in=-90,out=90] node[left] {$\scriptstyle A(c)$} (Az);
\draw (Az) to[in=-90,out=90] (d);
}
\arrow[Rightarrow,r,red,"A^2_{y,x}"]
\arrow[Rightarrow,d, blue,"A^2_{z,y}"]
&
\tikzmath{
\node (a) at (0,0) {$\scriptstyle A(a)$};
\node[draw,rectangle, thick, rounded corners=5pt] (Ax) at (0,1) {$\scriptstyle A(y\xo x)$};
\node[draw,rectangle, thick, rounded corners=5pt] (Az) at (0,2) {$\scriptstyle A(z)$};
\node (d) at (0,3) {$\scriptstyle A(d)$};
\draw (a) to[in=-90,out=90] (Ax);
\draw (Ax) to[in=-90,out=90] node[left] {$\scriptstyle A(c)$} (Az);
\draw (Az) to[in=-90,out=90] (d);
}
\arrow[Rightarrow,d,"A^2_{z,y\xo x}"]
\\
\tikzmath{
\node (a) at (0,0) {$\scriptstyle A(a)$};
\node[draw,rectangle, thick, rounded corners=5pt] (Ax) at (0,1) {$\scriptstyle A(x)$};
\node[draw,rectangle, thick, rounded corners=5pt] (Az) at (0,2) {$\scriptstyle A(z\xo y)$};
\node (d) at (0,3) {$\scriptstyle A(d)$};
\draw (a) to[in=-90,out=90] (Ax);
\draw (Ax) to[in=-90,out=90] node[left] {$\scriptstyle A(b)$} (Az);
\draw (Az) to[in=-90,out=90] (d);
}
\arrow[Rightarrow,r,"A^2_{z\xo y,x}"]
&
\tikzmath{
\node (a) at (0,0) {$\scriptstyle A(a)$};
\node[draw,rectangle, thick, rounded corners=5pt] (Ax) at (0,1) {$\scriptstyle A(z\xo y\xo x)$};
\node (d) at (0,2) {$\scriptstyle A(d)$};
\draw (a) to[in=-90,out=90] (Ax);
\draw (Ax) to[in=-90,out=90] (d);
}
\end{tikzcd}
};\end{tikzpicture}
$$
\item
\label{Functor:A.unital}
For all $x\in \cC(a\to b)$, the following two triangles commute:
$$
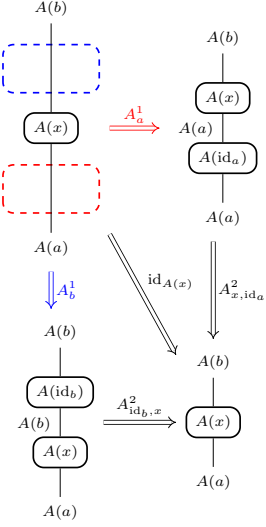
\begin{tikzpicture}[baseline= (a).base]
\node[scale=.8] (a) at (0,0){
\begin{tikzcd}
\tikzmath{
\node (a) at (0,0) {$\scriptstyle A(a)$};
\node[draw,rectangle, thick, rounded corners=5pt] (Ax) at (0,2) {$\scriptstyle A(x)$};
\node (c) at (0,4) {$\scriptstyle A(b)$};
\draw[red, thick, dashed, rounded corners=5pt] (-.8,.65) rectangle (.8,1.45);
\draw[blue, thick, dashed, rounded corners=5pt] (-.8,2.65) rectangle (.8,3.45);
\draw (a) to[in=-90,out=90] (Ax);
\draw (Ax) to[in=-90,out=90] (c);
}
\arrow[Rightarrow,r,red,"A^1_a"]
\arrow[Rightarrow,d, blue,"A^1_b"]
\arrow[Rightarrow,dr,"\id_{A(x)}"]
&
\tikzmath{
\node (a) at (0,0) {$\scriptstyle A(a)$};
\node[draw,rectangle, thick, rounded corners=5pt] (Ax) at (0,1) {$\scriptstyle A(\id_a)$};
\node[draw,rectangle, thick, rounded corners=5pt] (Az) at (0,2) {$\scriptstyle A(x)$};
\node (d) at (0,3) {$\scriptstyle A(b)$};
\draw (a) to[in=-90,out=90] (Ax);
\draw (Ax) to[in=-90,out=90] node[left] {$\scriptstyle A(a)$} (Az);
\draw (Az) to[in=-90,out=90] (d);
}
\arrow[Rightarrow,d,"A^2_{x,\id_a}"]
\\
\tikzmath{
\node (a) at (0,0) {$\scriptstyle A(a)$};
\node[draw,rectangle, thick, rounded corners=5pt] (Ax) at (0,1) {$\scriptstyle A(x)$};
\node[draw,rectangle, thick, rounded corners=5pt] (Az) at (0,2) {$\scriptstyle A(\id_b)$};
\node (d) at (0,3) {$\scriptstyle A(b)$};
\draw (a) to[in=-90,out=90] (Ax);
\draw (Ax) to[in=-90,out=90] node[left] {$\scriptstyle A(b)$} (Az);
\draw (Az) to[in=-90,out=90] (d);
}
\arrow[Rightarrow,r,"A^2_{\id_b,x}"]
&
\tikzmath{
\node (a) at (0,0) {$\scriptstyle A(a)$};
\node[draw,rectangle, thick, rounded corners=5pt] (Ax) at (0,1) {$\scriptstyle A(x)$};
\node (d) at (0,2) {$\scriptstyle A(b)$};
\draw (a) to[in=-90,out=90] (Ax);
\draw (Ax) to[in=-90,out=90] (d);
}
\end{tikzcd}
};\end{tikzpicture}
$$
\end{enumerate}

\item
\label{Functor:mu}
An adjoint equivalence $\mu^A: \xz_\cD \circ (A\times A) \Rightarrow A\circ \xz_\cC$
in the 2-category of 2-functors $\cC\times \cC\to \cD$. Explicitly, this is given by, for each pair of 0-cells $(a,b)\in \cC\times \cC$, an adjoint equivalence 1-cell $\mu^A_{a,b}: A(a) \xz A(b) \to A(a\xz b)$ and for each pair of 1-cells
$(x,y): (a,b) \to (c,d)$, an invertible 2-cell
$$
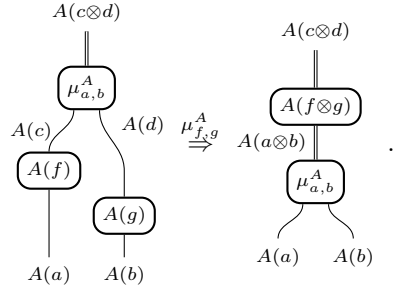
\tikzmath{
\node (a) at (0,-.5) {$\scriptstyle A(a)$};
\node (b) at (1,-.5) {$\scriptstyle A(b)$};
\node (d) at (.5,3) {$\scriptstyle A(c\xz d)$};
\node[draw,rectangle, thick, rounded corners=5pt] (f) at (0,.9) {$\scriptstyle A(f)$};
\node[draw,rectangle, thick, rounded corners=5pt] (g) at (1,.3) {$\scriptstyle A(g)$};
\coordinate (g2) at (1,.9);
\draw (a) to[in=-90,out=90] (f);
\draw (b) to[in=-90,out=90] (g);
\node[draw,rectangle, thick, rounded corners=5pt] (mu1) at (.5,2) {$\scriptstyle \mu^A_{a,b}$};
\draw (f) to[in=-90,out=90] node[left] {$\scriptstyle A(c)$} (mu1.-120);
\draw (g) to[in=-90,out=90] (g2) to[in=-90,out=90] node[right,yshift=.2cm] {$\scriptstyle A(d)$} (mu1.-60);
\draw[double] (mu1) to[in=-90,out=90] (d);
}
\overset{\mu^A_{f,g}}{\Rightarrow}
\tikzmath{
\node (a) at (0,0) {$\scriptstyle A(a)$};
\node (b) at (1,0) {$\scriptstyle A(b)$};
\node (d) at (.5,3) {$\scriptstyle A(c\xz d)$};
\node[draw,rectangle, thick, rounded corners=5pt] (mu1) at (.5,1) {$\scriptstyle \mu^A_{a,b}$};
\draw (a) to[in=-90,out=90] (mu1.-120);
\draw (b) to[in=-90,out=90] (mu1.-60);
\node[draw,rectangle, thick, rounded corners=5pt] (fg) at (.5,2) {$\scriptstyle A(f\xz g)$};
\draw[double] (mu1) to[in=-90,out=90] node[left] {$\scriptstyle A(a\xz b)$} (fg);
\draw[double] (fg) to[in=-90,out=90] (d);
}\,.
$$
That $\mu^A$ is a $2$-transformation means we have the following cohrences.
\begin{enumerate}[label=\ref{Functor:mu}.\roman*]
\item 
\label{Functor:mu.natural}
For all $x, x': a\to c$ and $y,y': b\to d$ and all $f: x \Rightarrow x'$ and $g: y \Rightarrow y$, the following square commutes:
$$
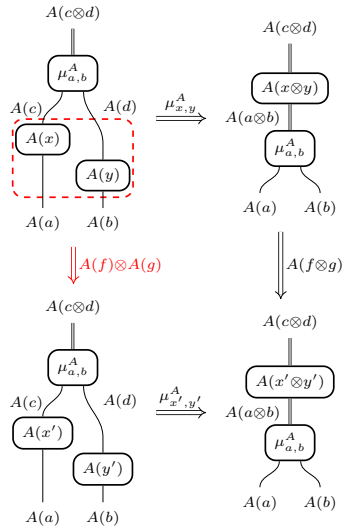
\begin{tikzpicture}[baseline= (a).base]
\node[scale=.8] (a) at (0,0){
\begin{tikzcd}
\tikzmath{
\node (a) at (0,-.5) {$\scriptstyle A(a)$};
\node (b) at (1,-.5) {$\scriptstyle A(b)$};
\node (d) at (.5,3) {$\scriptstyle A(c\xz d)$};
\node[draw,rectangle, thick, rounded corners=5pt] (f) at (0,.9) {$\scriptstyle A(x)$};
\draw[dashed, red, thick, rounded corners=5pt] (-.5,0) rectangle (1.5,1.3);
\node[draw,rectangle, thick, rounded corners=5pt] (g) at (1,.3) {$\scriptstyle A(y)$};
\coordinate (g2) at (1,.9);
\draw (a) to[in=-90,out=90] (f);
\draw (b) to[in=-90,out=90] (g);
\node[draw,rectangle, thick, rounded corners=5pt] (mu1) at (.5,2) {$\scriptstyle \mu^A_{a,b}$};
\draw (f) to[in=-90,out=90] node[left] {$\scriptstyle A(c)$} (mu1.-120);
\draw (g) to[in=-90,out=90] (g2) to[in=-90,out=90] node[right,yshift=.2cm] {$\scriptstyle A(d)$} (mu1.-60);
\draw[double] (mu1) to[in=-90,out=90] (d);
}
\arrow[Rightarrow,r,"\mu^A_{x,y}"]
\arrow[Rightarrow,d, red,"A(f)\xz A(g)"]
&
\tikzmath{
\node (a) at (0,0) {$\scriptstyle A(a)$};
\node (b) at (1,0) {$\scriptstyle A(b)$};
\node (d) at (.5,3) {$\scriptstyle A(c\xz d)$};
\node[draw,rectangle, thick, rounded corners=5pt] (mu1) at (.5,1) {$\scriptstyle \mu^A_{a,b}$};
\draw (a) to[in=-90,out=90] (mu1.-120);
\draw (b) to[in=-90,out=90] (mu1.-60);
\node[draw,rectangle, thick, rounded corners=5pt] (fg) at (.5,2) {$\scriptstyle A(x\xz y)$};
\draw[double] (mu1) to[in=-90,out=90] node[left] {$\scriptstyle A(a\xz b)$} (fg);
\draw[double] (fg) to[in=-90,out=90] (d);
}
\arrow[Rightarrow,d,"A(f\otimes g)"]
\\
\tikzmath{
\node (a) at (0,-.5) {$\scriptstyle A(a)$};
\node (b) at (1,-.5) {$\scriptstyle A(b)$};
\node (d) at (.5,3) {$\scriptstyle A(c\xz d)$};
\node[draw,rectangle, thick, rounded corners=5pt] (f) at (0,.9) {$\scriptstyle A(x')$};
\node[draw,rectangle, thick, rounded corners=5pt] (g) at (1,.3) {$\scriptstyle A(y')$};
\coordinate (g2) at (1,.9);
\draw (a) to[in=-90,out=90] (f);
\draw (b) to[in=-90,out=90] (g);
\node[draw,rectangle, thick, rounded corners=5pt] (mu1) at (.5,2) {$\scriptstyle \mu^A_{a,b}$};
\draw (f) to[in=-90,out=90] node[left] {$\scriptstyle A(c)$} (mu1.-120);
\draw (g) to[in=-90,out=90] (g2) to[in=-90,out=90] node[right,yshift=.2cm] {$\scriptstyle A(d)$} (mu1.-60);
\draw[double] (mu1) to[in=-90,out=90] (d);
}
\arrow[Rightarrow,r,"\mu^A_{x',y'}"]
&
\tikzmath{
\node (a) at (0,0) {$\scriptstyle A(a)$};
\node (b) at (1,0) {$\scriptstyle A(b)$};
\node (d) at (.5,3) {$\scriptstyle A(c\xz d)$};
\node[draw,rectangle, thick, rounded corners=5pt] (mu1) at (.5,1) {$\scriptstyle \mu^A_{a,b}$};
\draw (a) to[in=-90,out=90] (mu1.-120);
\draw (b) to[in=-90,out=90] (mu1.-60);
\node[draw,rectangle, thick, rounded corners=5pt] (fg) at (.5,2) {$\scriptstyle A(x'\xz y')$};
\draw[double] (mu1) to[in=-90,out=90] node[left] {$\scriptstyle A(a\xz b)$} (fg);
\draw[double] (fg) to[in=-90,out=90] (d);
}
\end{tikzcd}
};\end{tikzpicture}
$$
\item
\label{Functor:mu.monoidal}
For all 1-cells $x_1\in\cC(a_1\to a_2)$, $x_2\in\cC(a_2\to a_3)$, $y_1\in\cC(b_1\to b_2)$, and $y_2\in \cC(b_2\to b_3)$,
$$
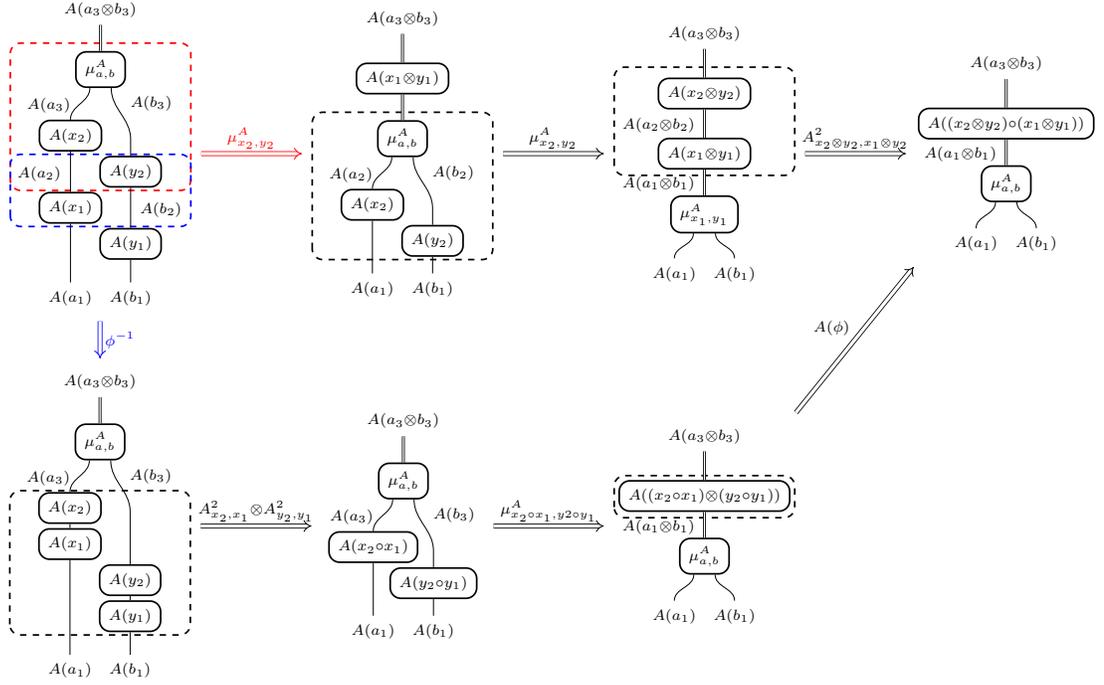
\begin{tikzpicture}[baseline= (a).base]
\node[scale=.8] (a) at (0,0){
\begin{tikzcd}
\tikzmath{
\node (a) at (0,-1.8) {$\scriptstyle A(a_1)$};
\node (b) at (1,-1.8) {$\scriptstyle A(b_1)$};
\node (d) at (.5,3) {$\scriptstyle A(a_3\xz b_3)$};
\node[draw,rectangle, thick, rounded corners=5pt] (x2) at (0,.9) {$\scriptstyle A(x_2)$};
\node[draw,rectangle, thick, rounded corners=5pt] (y2) at (1,.3) {$\scriptstyle A(y_2)$};
\node[draw,rectangle, thick, rounded corners=5pt] (x1) at (0,-.3) {$\scriptstyle A(x_1)$};
\node[draw,rectangle, thick, rounded corners=5pt] (y1) at (1,-.9) {$\scriptstyle A(y_1)$};
\draw[dashed, red, thick, rounded corners=5pt] (-1,0.05) rectangle (2,2.5);
\draw[dashed, blue, thick, rounded corners=5pt] (-1,-.55) rectangle (2,.65);
\coordinate (y22) at (1,.9);
\draw (a) to[in=-90,out=90] (x1);
\draw (x1) to[in=-90,out=90] node[left] {$\scriptstyle A(a_2)$} (x2);
\draw (b) to[in=-90,out=90] (y1);
\draw (y1) to[in=-90,out=90] node[right] {$\scriptstyle A(b_2)$} (y2);
\node[draw,rectangle, thick, rounded corners=5pt] (mu1) at (.5,2) {$\scriptstyle \mu^A_{a,b}$};
\draw (x2) to[in=-90,out=90] node[left] {$\scriptstyle A(a_3)$} (mu1.-120);
\draw (y2) to[in=-90,out=90] (y22) to[in=-90,out=90] node[right,yshift=.2cm] {$\scriptstyle A(b_3)$} (mu1.-60);
\draw[double] (mu1) to[in=-90,out=90] (d);
}
\arrow[Rightarrow,rr, red, "\mu^A_{x_2,y_2}"]
\arrow[Rightarrow,d, blue, "\phi^{-1}"]
&&
\tikzmath{
\node (a) at (0,-.5) {$\scriptstyle A(a_1)$};
\node (b) at (1,-.5) {$\scriptstyle A(b_1)$};
\node (d) at (.5,4) {$\scriptstyle A(a_3\xz b_3)$};
\node[draw,rectangle, thick, rounded corners=5pt] (f) at (0,.9) {$\scriptstyle A(x_2)$};
\node[draw,rectangle, thick, rounded corners=5pt] (g) at (1,.3) {$\scriptstyle A(y_2)$};
\node[draw,rectangle, thick, rounded corners=5pt] (Axy) at (.5,3) {$\scriptstyle A(x_1\xz y_1)$};
\draw[dashed, thick, rounded corners=5pt] (-1,0.05) rectangle (2,2.5);
\coordinate (g2) at (1,.9);
\draw (a) to[in=-90,out=90] (f);
\draw (b) to[in=-90,out=90] (g);
\node[draw,rectangle, thick, rounded corners=5pt] (mu1) at (.5,2) {$\scriptstyle \mu^A_{a,b}$};
\draw (f) to[in=-90,out=90] node[left] {$\scriptstyle A(a_2)$} (mu1.-120);
\draw (g) to[in=-90,out=90] (g2) to[in=-90,out=90] node[right,yshift=.2cm] {$\scriptstyle A(b_2)$} (mu1.-60);
\draw[double] (mu1) to[in=-90,out=90] (Axy);
\draw[double] (Axy) to[in=-90,out=90] (d);
}
\arrow[Rightarrow,rr,"\mu^A_{x_2,y_2}"]
&&
\tikzmath{
\node (a) at (0,0) {$\scriptstyle A(a_1)$};
\node (b) at (1,0) {$\scriptstyle A(b_1)$};
\node (d) at (.5,4) {$\scriptstyle A(a_3\xz b_3)$};
\node[draw,rectangle, thick, rounded corners=5pt] (mu1) at (.5,1) {$\scriptstyle \mu^A_{x_1,y_1}$};
\draw (a) to[in=-90,out=90] (mu1.-120);
\draw (b) to[in=-90,out=90] (mu1.-60);
\node[draw,rectangle, thick, rounded corners=5pt] (xy1) at (.5,2) {$\scriptstyle A(x_1\xz y_1)$};
\node[draw,rectangle, thick, rounded corners=5pt] (xy2) at (.5,3) {$\scriptstyle A(x_2\xz y_2)$};
\draw[dashed, thick, rounded corners=5pt] (-1,1.7) rectangle (2,3.5);
\draw[double] (mu1) to[in=-90,out=90] node[left] {$\scriptstyle A(a_1\xz b_1)$} (fg);
\draw[double] (xy1) to[in=-90,out=90] node[left] {$\scriptstyle A(a_2\xz b_2)$} (xy2);
\draw[double] (xy2) to[in=-90,out=90] (d);
}
\arrow[Rightarrow,rr,"A^2_{x_2\otimes y_2, x_1\otimes y_2}"]
&&
\tikzmath{
\node (a) at (0,0) {$\scriptstyle A(a_1)$};
\node (b) at (1,0) {$\scriptstyle A(b_1)$};
\node (d) at (.5,3) {$\scriptstyle A(a_3\xz b_3)$};
\node[draw,rectangle, thick, rounded corners=5pt] (mu1) at (.5,1) {$\scriptstyle \mu^A_{a,b}$};
\draw (a) to[in=-90,out=90] (mu1.-120);
\draw (b) to[in=-90,out=90] (mu1.-60);
\node[draw,rectangle, thick, rounded corners=5pt] (fg) at (.5,2) {$\scriptstyle A((x_2\xz y_2) \xo (x_1\xz y_1))$};
\draw[double] (mu1) to[in=-90,out=90] node[left] {$\scriptstyle A(a_1\xz b_1)$} (fg);
\draw[double] (fg) to[in=-90,out=90] (d);
}
%\arrow[Rightarrow,dll,"A(\phi)"]
\\
\tikzmath{
\node (a) at (0,-1.8) {$\scriptstyle A(a_1)$};
\node (b) at (1,-1.8) {$\scriptstyle A(b_1)$};
\node (d) at (.5,3) {$\scriptstyle A(a_3\xz b_3)$};
\node[draw,rectangle, thick, rounded corners=5pt] (x2) at (0,.9) {$\scriptstyle A(x_2)$};
\node[draw,rectangle, thick, rounded corners=5pt] (y2) at (1,-.3) {$\scriptstyle A(y_2)$};
\node[draw,rectangle, thick, rounded corners=5pt] (x1) at (0,.3) {$\scriptstyle A(x_1)$};
\node[draw,rectangle, thick, rounded corners=5pt] (y1) at (1,-.9) {$\scriptstyle A(y_1)$};
\draw[dashed, thick, rounded corners=5pt] (-1,-1.15) rectangle (2,1.25);
\coordinate (y22) at (1,.9);
\draw (a) to[in=-90,out=90] (x1);
\draw (x1) to[in=-90,out=90] (x2);
\draw (b) to[in=-90,out=90] (y1);
\draw (y1) to[in=-90,out=90] (y2);
\node[draw,rectangle, thick, rounded corners=5pt] (mu1) at (.5,2) {$\scriptstyle \mu^A_{a,b}$};
\draw (x2) to[in=-90,out=90] node[left] {$\scriptstyle A(a_3)$} (mu1.-120);
\draw (y2) to[in=-90,out=90] (y22) to[in=-90,out=90] node[right, yshift=.2cm] {$\scriptstyle A(b_3)$} (mu1.-60);
\draw[double] (mu1) to[in=-90,out=90] (d);
}
\arrow[Rightarrow,rr,"A^2_{x_2,x_1}\otimes A^2_{y_2,y_1}"]
&&
\tikzmath{
\node (a) at (0,-.5) {$\scriptstyle A(a_1)$};
\node (b) at (1,-.5) {$\scriptstyle A(b_1)$};
\node (d) at (.5,3) {$\scriptstyle A(a_3\xz b_3)$};
\node[draw,rectangle, thick, rounded corners=5pt] (f) at (0,.9) {$\scriptstyle A(x_2\xo x_1)$};
\node[draw,rectangle, thick, rounded corners=5pt] (g) at (1,.3) {$\scriptstyle A(y_2 \xo y_1)$};
\coordinate (g2) at (1,.9);
\draw (a) to[in=-90,out=90] (f);
\draw (b) to[in=-90,out=90] (g);
\node[draw,rectangle, thick, rounded corners=5pt] (mu1) at (.5,2) {$\scriptstyle \mu^A_{a,b}$};
\draw (f) to[in=-90,out=90] node[left] {$\scriptstyle A(a_3)$} (mu1.-120);
\draw (g) to[in=-90,out=90] (g2) to[in=-90,out=90] node[right,yshift=.2cm] {$\scriptstyle A(b_3)$} (mu1.-60);
\draw[double] (mu1) to[in=-90,out=90] (d);
}
\arrow[Rightarrow,rr,"\mu^A_{x_2\xo x_1,y2\xo y_1}"]
&&
\tikzmath{
\node (a) at (0,0) {$\scriptstyle A(a_1)$};
\node (b) at (1,0) {$\scriptstyle A(b_1)$};
\node (d) at (.5,3) {$\scriptstyle A(a_3\xz b_3)$};
\node[draw,rectangle, thick, rounded corners=5pt] (mu1) at (.5,1) {$\scriptstyle \mu^A_{a,b}$};
\draw (a) to[in=-90,out=90] (mu1.-120);
\draw (b) to[in=-90,out=90] (mu1.-60);
\node[draw,rectangle, thick, rounded corners=5pt] (fg) at (.5,2) {$\scriptstyle A((x_2\xo x_1) \xz (y_2\xo y_1))$};
\draw[dashed, thick, rounded corners=5pt] (-1,1.7) rectangle (2,2.4);
\draw[double] (mu1) to[in=-90,out=90] node[left] {$\scriptstyle A(a_1\xz b_1)$} (fg);
\draw[double] (fg) to[in=-90,out=90] (d);
}
\arrow[Rightarrow,urr,"A(\phi)"]
\end{tikzcd}
};\end{tikzpicture}
$$
\item
\label{Functor:mu.unital}
For all 0-cells $a,b\in\cC$, the following diagram commutes:
$$
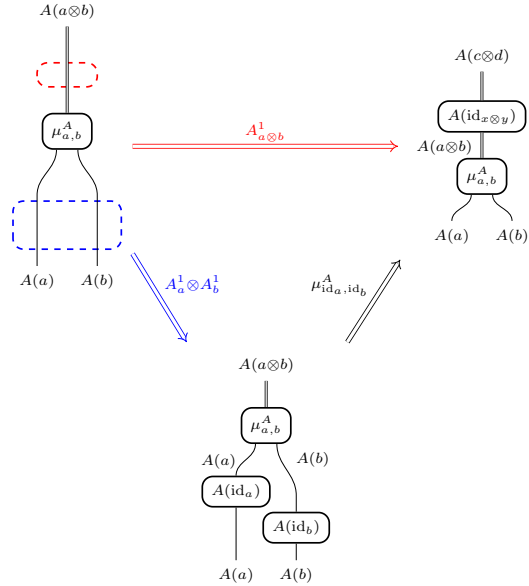
\begin{tikzpicture}[baseline= (a).base]
\node[scale=.8] (a) at (0,0){
\begin{tikzcd}
\tikzmath{
\node (a) at (0,-.5) {$\scriptstyle A(a)$};
\coordinate (a1) at (0,1);
\node (b) at (1,-.5) {$\scriptstyle A(b)$};
\coordinate (b1) at (1,1);
\node (d) at (.5,4) {$\scriptstyle A(a\xz b)$};
%\draw[dashed, blue, thick, rounded corners=5pt] (-.4,.6) rectangle (.4,1);
\draw[dashed, blue, thick, rounded corners=5pt] (-.4,.1) rectangle (1.4,.9);
\draw[dashed, red, thick, rounded corners=5pt] (0,2.8) rectangle (1,3.2);
\node[draw,rectangle, thick, rounded corners=5pt] (mu1) at (.5,2) {$\scriptstyle \mu^A_{a,b}$};
\draw (a) to[in=-90,out=90] (a1) to[in=-90,out=90] (mu1.-120);
\draw (b) to[in=-90,out=90] (b1) to[in=-90,out=90] (mu1.-60);
\draw[double] (mu1) to[in=-90,out=90] (d);
}
\arrow[Rightarrow,dr, blue,"A^1_a \otimes A^1_b"]
\arrow[Rightarrow,rr, red,"A^1_{a\otimes b}"]
&&
\tikzmath{
\node (a) at (0,0) {$\scriptstyle A(a)$};
\node (b) at (1,0) {$\scriptstyle A(b)$};
\node (d) at (.5,3) {$\scriptstyle A(c\xz d)$};
\node[draw,rectangle, thick, rounded corners=5pt] (mu1) at (.5,1) {$\scriptstyle \mu^A_{a,b}$};
\draw (a) to[in=-90,out=90] (mu1.-120);
\draw (b) to[in=-90,out=90] (mu1.-60);
\node[draw,rectangle, thick, rounded corners=5pt] (fg) at (.5,2) {$\scriptstyle A(\id_{x\otimes y})$};
\draw[double] (mu1) to[in=-90,out=90] node[left] {$\scriptstyle A(a\xz b)$} (fg);
\draw[double] (fg) to[in=-90,out=90] (d);
}
\\
&
\tikzmath{
\node (a) at (0,-.5) {$\scriptstyle A(a)$};
\node (b) at (1,-.5) {$\scriptstyle A(b)$};
\node (d) at (.5,3) {$\scriptstyle A(a\xz b)$};
\node[draw,rectangle, thick, rounded corners=5pt] (f) at (0,.9) {$\scriptstyle A(\id_a)$};
\node[draw,rectangle, thick, rounded corners=5pt] (g) at (1,.3) {$\scriptstyle A(\id_b)$};
\coordinate (g2) at (1,.9);
\draw (a) to[in=-90,out=90] (f);
\draw (b) to[in=-90,out=90] (g);
\node[draw,rectangle, thick, rounded corners=5pt] (mu1) at (.5,2) {$\scriptstyle \mu^A_{a,b}$};
\draw (f) to[in=-90,out=90] node[left] {$\scriptstyle A(a)$} (mu1.-120);
\draw (g) to[in=-90,out=90] (g2) to[in=-90,out=90] node[right,yshift=.2cm] {$\scriptstyle A(b)$} (mu1.-60);
\draw[double] (mu1) to[in=-90,out=90] (d);
}
\arrow[Rightarrow,ur,"\mu^A_{\id_a,\id_b}"]
\end{tikzcd}
};\end{tikzpicture}
$$
\end{enumerate}
\item
\label{Functor:iota}
An adjoint equivalence $\iota^A: I_\cD \Rightarrow A \circ I_\cC$
(in the 2-category of 2-functors $*\to \cD$)
where $I_\cC: * \to \cC$ is the inclusion of the trivial 2-category into $\cC$ which picks out $1_\cC, \id_{1_\cC}, \id_{\id_{1_\cC}}$, and similarly for $\cD$.
Explicitly, this is given by an adjoint equivalence 1-cell $\iota^A_*: 1_\cD \to A(1_\cC)$ and an invertible 2-cell 
$$
\left(
\tikzmath{
\node (d) at (0,3) {$\scriptstyle A(1_\cC)$};
\node[draw,rectangle, thick, rounded corners=5pt] (iota) at (0,1) {$\scriptstyle \iota^A_*$};
\draw (iota) to[in=-90,out=90] (d);
}
\overset{\iota^A_1}{\Rightarrow}
\tikzmath{
\node (d) at (0,3) {$\scriptstyle A(1_\cC)$};
\node[draw,rectangle, thick, rounded corners=5pt] (iota) at (0,1) {$\scriptstyle \iota^A_*$};
\node[draw,rectangle, thick, rounded corners=5pt] (Aid) at (0,2) {$\scriptstyle A(\id_{1_\cC})$};
\draw (iota) to[in=-90,out=90] node[left] {$\scriptstyle A(1_\cC)$} (Aid);
\draw (Aid) to[in=-90,out=90] (d);
}
\right)
=
\left(
\tikzmath{
\node (d) at (0,3) {$\scriptstyle A(1_\cC)$};
\node[draw,rectangle, thick, rounded corners=5pt] (iota) at (0,1) {$\scriptstyle \iota^A_*$};
\draw[dashed, thick, rounded corners=5pt] (-.4,1.8) rectangle (.4,2.2);
\draw (iota) to[in=-90,out=90] (d);
}
\overset{A^1_{1_\cC}}{\Rightarrow}
\tikzmath{
\node (d) at (0,3) {$\scriptstyle A(1_\cC)$};
\node[draw,rectangle, thick, rounded corners=5pt] (iota) at (0,1) {$\scriptstyle \iota^A_*$};
\node[draw,rectangle, thick, rounded corners=5pt] (Aid) at (0,2) {$\scriptstyle A(\id_{1_\cC})$};
\draw (iota) to[in=-90,out=90] node[left] {$\scriptstyle A(1_\cC)$} (Aid);
\draw (Aid) to[in=-90,out=90] (d);
}
\right).
$$
That $\iota^A$ is a $2$-transformation implies that $\iota^A_1$ equals the map on the right hand side above, which is a whiskering with $A^1_e$.
This means $\iota^A_1$ is automatically natural and compatible with $A^2$.

\item
\label{Functor:omega}
An invertible associator 2-modification $\omega^A$.
Explicitly, this is given by, for each $a,b,c\in \cC$, an invertible 2-cell
$$
\tikzmath{
\node (a) at (0,0) {$\scriptstyle A(a)$};
\node (b) at (.8,0) {$\scriptstyle A(b)$};
\node (c) at (1.6,0) {$\scriptstyle A(c)$};
\node (d) at (.8,3) {$\scriptstyle A(a\xz b \xz c)$};
\node[draw,rectangle, thick, rounded corners=5pt] (mu1) at (.4,1) {$\scriptstyle \mu^A_{a,b}$};
\draw (a) to[in=-90,out=90] (mu1.-120);
\draw (b) to[in=-90,out=90] (mu1.-60);
\node[draw,rectangle, thick, rounded corners=5pt] (mu2) at (.8,2) {$\scriptstyle \mu^A_{a\xz b,c}$};
\draw[double] (mu1) to[in=-90,out=90] node[left] {$\scriptstyle A(a\xz b)$} (mu2.-120);
\draw (c) to[in=-90,out=90] (mu2.-60);
\draw[triple] (mu2) to[in=-90,out=90] (d);
}
\overset{\omega^A_{a,b,c}}{\Rightarrow}
\tikzmath{
\node (a) at (0,0) {$\scriptstyle A(a)$};
\node (b) at (.8,0) {$\scriptstyle A(b)$};
\node (c) at (1.6,0) {$\scriptstyle A(c)$};
\node (d) at (.8,3) {$\scriptstyle A(a\xz b \xz c)$};
\node[draw,rectangle, thick, rounded corners=5pt] (mu1) at (1.2,1) {$\scriptstyle \mu^A_{b,c}$};
\draw (b) to[in=-90,out=90] (mu1.-120);
\draw (c) to[in=-90,out=90] (mu1.-60);
\node[draw,rectangle, thick, rounded corners=5pt] (mu2) at (.8,2) {$\scriptstyle \mu^A_{a, b\xz c}$};
\draw[double] (mu1) to[in=-90,out=90] node[right] {$\scriptstyle A(b\xz c)$} (mu2.-60);
\draw (a) to[in=-90,out=90] (mu2.-120);
\draw[triple] (mu2) to[in=-90,out=90] (d);
}
$$
and the fact that $\omega$ is a $2$-modification means that
for all $x\in\cC(a_1\to a_2)$, $y\in\cC(b_1\to b_2)$, and $z\in\cC(c_1\to c_2)$,
$$
\begin{tikzpicture}[baseline= (a).base]
\node[scale=.8] (a) at (0,0){
\begin{tikzcd}
\tikzmath{
\node (a) at (0,-2.5) {$\scriptstyle A(a_1)$};
\node (b) at (1,-2.5) {$\scriptstyle A(b_1)$};
\node (c) at (2,-2.5) {$\scriptstyle A(c_1)$};
\node (d) at (1.25,3.25) {$\scriptstyle A(c\xz b \xz c)$};
\node[draw,rectangle, thick, rounded corners=5pt] (Ax) at (0,0) {$\scriptstyle A(x)$};
\node[draw,rectangle, thick, rounded corners=5pt] (Ay) at (1,-.75) {$\scriptstyle A(y)$};
\node[draw,rectangle, thick, rounded corners=5pt] (Az) at (2,-1.5) {$\scriptstyle A(z)$};
\node[draw,rectangle, thick, rounded corners=5pt] (mu1) at (.5,1) {$\scriptstyle \mu^A_{a_2,b_2}$};
\draw[dashed, red, thick, rounded corners=5pt] (-.6,-1.1) rectangle (1.6,1.5);
\draw[dashed, blue, thick, rounded corners=5pt] (-.2,.5) rectangle (2.2,2.75);
\draw (a) to[in=-90,out=90] (Ax);
\draw (b) to[in=-90,out=90] (Ay);
\draw (c) to[in=-90,out=90] (Az);
\draw (Ax) to[in=-90,out=90] (mu1.-120);
\draw (Ay) to[in=-90,out=90] (mu1.-60);
\node[draw,rectangle, thick, rounded corners=5pt] (mu2) at (1.25,2.25) {$\scriptstyle \mu^A_{a_2\xz b_2,c_2}$};
\draw[double] (mu1) to[in=-90,out=90] (mu2.-120);
\draw (Az) to[in=-90,out=90] (mu2.-60);
\draw[triple] (mu2) to[in=-90,out=90] (d);
}
\arrow[Rightarrow,r,red,"\mu^A_{x,y}"]
\arrow[Rightarrow,d,blue,"\omega^A_{a_2,b_2,c_2}"]
&
\tikzmath{
\node (a) at (0,-2.5) {$\scriptstyle A(a_1)$};
\node (b) at (1,-2.5) {$\scriptstyle A(b_1)$};
\node (c) at (2,-2.5) {$\scriptstyle A(c_1)$};
\node (d) at (1.25,3.25) {$\scriptstyle A(a_2\xz b_2 \xz c_2)$};
\node[draw,rectangle, thick, rounded corners=5pt] (Axy) at (0.5,1) {$\scriptstyle A(x\otimes y)$};
\node[draw,rectangle, thick, rounded corners=5pt] (Az) at (2,-1.5) {$\scriptstyle A(z)$};
\node[draw,rectangle, thick, rounded corners=5pt] (mu1) at (.5,0) {$\scriptstyle \mu^A_{a_1,b_1}$};
\draw[dashed, thick, rounded corners=5pt] (-.2,-1.8) rectangle (2.6,.5);
\draw (a) to[in=-90,out=90] (mu1.-120);
\draw (b) to[in=-90,out=90] (mu1.-60);
\draw (c) to[in=-90,out=90] (Az);
\draw[double] (mu1) to[in=-90,out=90] (Axy);
\node[draw,rectangle, thick, rounded corners=5pt] (mu2) at (1.25,2.25) {$\scriptstyle \mu^A_{a_2\xz b_2,c_2}$};
\draw[double] (Axy) to[in=-90,out=90] (mu2.-120);
\draw (Az) to[in=-90,out=90] (mu2.-60);
\draw[triple] (mu2) to[in=-90,out=90] (d);
}
\arrow[Rightarrow,r,"\phi"]
&
\tikzmath{
\node (a) at (0,-2.5) {$\scriptstyle A(a_1)$};
\node (b) at (1,-2.5) {$\scriptstyle A(b_1)$};
\node (c) at (2,-2.5) {$\scriptstyle A(c_1)$};
\node (d) at (1.25,3.25) {$\scriptstyle A(a_2\xz b_2 \xz c_2)$};
\node[draw,rectangle, thick, rounded corners=5pt] (Axy) at (0.5,1) {$\scriptstyle A(x\otimes y)$};
\node[draw,rectangle, thick, rounded corners=5pt] (Az) at (2,0) {$\scriptstyle A(z)$};
\node[draw,rectangle, thick, rounded corners=5pt] (mu1) at (.5,-1) {$\scriptstyle \mu^A_{a_1,b_1}$};
\draw[dashed, thick, rounded corners=5pt] (-.3,-.4) rectangle (2.6,2.8);
\draw (a) to[in=-90,out=90] (mu1.-120);
\draw (b) to[in=-90,out=90] (mu1.-60);
\draw (c) to[in=-90,out=90] (Az);
\draw[double] (mu1) to[in=-90,out=90] (Axy);
\node[draw,rectangle, thick, rounded corners=5pt] (mu2) at (1.25,2.25) {$\scriptstyle \mu^A_{a_2\xz b_2,c_2}$};
\draw[double] (Axy) to[in=-90,out=90] (mu2.-120);
\draw (Az) to[in=-90,out=90] (mu2.-60);
\draw[triple] (mu2) to[in=-90,out=90] (d);
}
\arrow[Rightarrow,r,"\mu^A_{x\xz y, z}"]
&
\tikzmath{
\node (a) at (0,-1.5) {$\scriptstyle A(a_1)$};
\node (b) at (1,-1.5) {$\scriptstyle A(b_1)$};
\node (c) at (2,-1.5) {$\scriptstyle A(c_1)$};
\node (d) at (1.25,4.25) {$\scriptstyle A(a_2\xz b_2 \xz c_2)$};
\node[draw,rectangle, thick, rounded corners=5pt] (mu1) at (.5,0) {$\scriptstyle \mu^A_{a_1,b_1}$};
\draw[dashed, thick, rounded corners=5pt] (-.2,-.6) rectangle (2.2,2);
\draw (a) to[in=-90,out=90] (mu1.-120);
\draw (b) to[in=-90,out=90] (mu1.-60);
\node[draw,rectangle, thick, rounded corners=5pt] (mu2) at (1.25,1.5) {$\scriptstyle \mu^A_{a_1\xz b_1,c_1}$};
\node[draw,rectangle, thick, rounded corners=5pt] (Axyz) at (1.25,2.75) {$\scriptstyle A(x\otimes y\otimes z)$};
\draw[double] (mu1) to[in=-90,out=90] (mu2.-120);
\draw (c) to[in=-90,out=90] (mu2.-60);
\draw[triple] (mu2) to[in=-90,out=90] (Axyz);
\draw[triple] (Axyz) to[in=-90,out=90] (d);
}
\arrow[Rightarrow,d,"\omega^A_{a_1,b_1,c_1}"]
\\
\tikzmath{
\node (a) at (0,-2.5) {$\scriptstyle A(a_1)$};
\node (b) at (1,-2.5) {$\scriptstyle A(b_1)$};
\node (c) at (2,-2.5) {$\scriptstyle A(c_1)$};
\node (d) at (.75,3.25) {$\scriptstyle A(c\xz b \xz c)$};
\node[draw,rectangle, thick, rounded corners=5pt] (Ax) at (0,0) {$\scriptstyle A(x)$};
\node[draw,rectangle, thick, rounded corners=5pt] (Ay) at (1,-.75) {$\scriptstyle A(y)$};
\node[draw,rectangle, thick, rounded corners=5pt] (Az) at (2,-1.5) {$\scriptstyle A(z)$};
\node[draw,rectangle, thick, rounded corners=5pt] (mu1) at (1.5,1) {$\scriptstyle \mu^A_{b_2,c_2}$};
\draw[dashed, thick, rounded corners=5pt] (-.6,1.5) rectangle (2.2,-.3);
\draw (a) to[in=-90,out=90] (Ax);
\draw (b) to[in=-90,out=90] (Ay);
\draw (c) to[in=-90,out=90] (Az);
\node[draw,rectangle, thick, rounded corners=5pt] (mu2) at (.75,2.25) {$\scriptstyle \mu^A_{a_2,b_2 \xz c_2}$};
\draw (Ax) to[in=-90,out=90] (mu2.-120);
\draw (Ay) to[in=-90,out=90] (mu1.-120);
\draw[double] (mu1) to[in=-90,out=90] (mu2.-60);
\draw (Az) to[in=-90,out=90] (mu1.-60);
\draw[triple] (mu2) to[in=-90,out=90] (d);
}
\arrow[Rightarrow,r,"\phi"]
&
\tikzmath{
\node (a) at (0,-2.5) {$\scriptstyle A(a_1)$};
\node (b) at (1,-2.5) {$\scriptstyle A(b_1)$};
\node (c) at (2,-2.5) {$\scriptstyle A(c_1)$};
\node (d) at (.75,3.25) {$\scriptstyle A(c\xz b \xz c)$};
\node[draw,rectangle, thick, rounded corners=5pt] (Ax) at (0,1) {$\scriptstyle A(x)$};
\node[draw,rectangle, thick, rounded corners=5pt] (Ay) at (1,-.75) {$\scriptstyle A(y)$};
\node[draw,rectangle, thick, rounded corners=5pt] (Az) at (2,-1.5) {$\scriptstyle A(z)$};
\node[draw,rectangle, thick, rounded corners=5pt] (mu1) at (1.5,.3) {$\scriptstyle \mu^A_{b_2,c_2}$};
\draw[dashed, thick, rounded corners=5pt] (.4,.8) rectangle (2.6,-1.8);
\draw (a) to[in=-90,out=90] (Ax);
\draw (b) to[in=-90,out=90] (Ay);
\draw (c) to[in=-90,out=90] (Az);
\node[draw,rectangle, thick, rounded corners=5pt] (mu2) at (.75,2.25) {$\scriptstyle \mu^A_{a_2,b_2 \xz c_2}$};
\draw (Ax) to[in=-90,out=90] (mu2.-120);
\draw (Ay) to[in=-90,out=90] (mu1.-120);
\draw[double] (mu1) to[in=-90,out=90] (mu2.-60);
\draw (Az) to[in=-90,out=90] (mu1.-60);
\draw[triple] (mu2) to[in=-90,out=90] (d);
}
\arrow[Rightarrow,r,"\mu^A_{y, z}"]
&
\tikzmath{
\node (a) at (0,-2.5) {$\scriptstyle A(a_1)$};
\node (b) at (1,-2.5) {$\scriptstyle A(b_1)$};
\node (c) at (2,-2.5) {$\scriptstyle A(c_1)$};
\node (d) at (.75,3.25) {$\scriptstyle A(c\xz b \xz c)$};
\node[draw,rectangle, thick, rounded corners=5pt] (Ax) at (0,1) {$\scriptstyle A(x)$};
\node[draw,rectangle, thick, rounded corners=5pt] (Ayz) at (1.5,0) {$\scriptstyle A(y\xz z)$};
\node[draw,rectangle, thick, rounded corners=5pt] (mu1) at (1.5,-1) {$\scriptstyle \mu^A_{b_1,c_1}$};
\draw[dashed, thick, rounded corners=5pt] (-.6,-.4) rectangle (2.3,2.75);
\draw (a) to[in=-90,out=90] (Ax);
\draw (b) to[in=-90,out=90] (mu1.-120);
\draw (c) to[in=-90,out=90] (mu1.-60);
\node[draw,rectangle, thick, rounded corners=5pt] (mu2) at (.75,2.25) {$\scriptstyle \mu^A_{a_2,b_2 \xz c_2}$};
\draw (Ayz) to[in=-90,out=90] (mu2.-60);
\draw (Ax) to[in=-90,out=90] (mu2.-120);
\draw[double] (mu1) to[in=-90,out=90] (Ayz);
\draw[triple] (mu2) to[in=-90,out=90] (d);
}
\arrow[Rightarrow,r,"\mu^A_{x, y\xz z}"]
&
\tikzmath{
\node (a) at (0,-1.5) {$\scriptstyle A(a_1)$};
\node (b) at (1,-1.5) {$\scriptstyle A(b_1)$};
\node (c) at (2,-1.5) {$\scriptstyle A(c_1)$};
\node (d) at (.75,4.25) {$\scriptstyle A(a_2\xz b_2 \xz c_2)$};
\node[draw,rectangle, thick, rounded corners=5pt] (mu1) at (1.5,0) {$\scriptstyle \mu^A_{b_1,c_1}$};
\node[draw,rectangle, thick, rounded corners=5pt] (mu2) at (.75,1.5) {$\scriptstyle \mu^A_{a_1, b_1 \xz c_1}$};
\node[draw,rectangle, thick, rounded corners=5pt] (Axyz) at (.75,2.75) {$\scriptstyle A(x\otimes y\otimes z)$};
\draw[double] (mu1) to[in=-90,out=90] (mu2.-60);
\draw (a) to[in=-90,out=90] (mu2.-120);
\draw (b) to[in=-90,out=90] (mu1.-120);
\draw (c) to[in=-90,out=90] (mu1.-60);
\draw[triple] (mu2) to[in=-90,out=90] (Axyz);
\draw[triple] (Axyz) to[in=-90,out=90] (d);
}
\end{tikzcd}};\end{tikzpicture}
$$

\item
\label{Functor:lr}
invertible unitor 2-modifications $\ell^A$ and $r^A$, i.e., for each $c\in \cC$, invertible 2-cells
$$
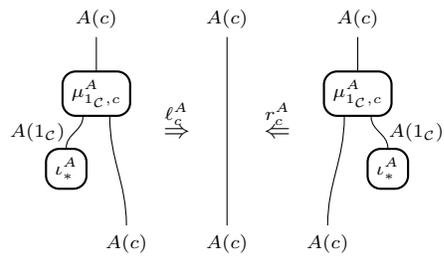
\tikzmath{
\node (c) at (.8,0) {$\scriptstyle A(c)$};
\node (d) at (.4,3) {$\scriptstyle A(c)$};
\node[draw,rectangle, thick, rounded corners=5pt] (iota) at (0,1) {$\scriptstyle \iota^A_* $};
\node[draw,rectangle, thick, rounded corners=5pt] (mu1) at (.4,2) {$\scriptstyle \mu^A_{1_\cC,c}$};
\draw (iota) to[in=-90,out=90] node[left] {$\scriptstyle A(1_\cC)$} (mu1.-120);
\draw (c) to[in=-90,out=90] (mu1.-60);
\draw (mu1) to[in=-90,out=90] (d);
}
\overset{\ell_c^A}{\Rightarrow}
\tikzmath{
\node (c) at (0,0) {$\scriptstyle A(c)$};
\node (d) at (0,3) {$\scriptstyle A(c)$};
\draw (c) to[in=-90,out=90] (d);
}
\overset{r_c^A}{\Leftarrow}
\tikzmath{
\node (c) at (0,0) {$\scriptstyle A(c)$};
\node (d) at (.4,3) {$\scriptstyle A(c)$};
\node[draw,rectangle, thick, rounded corners=5pt] (iota) at (.8,1) {$\scriptstyle \iota^A_* $};
\node[draw,rectangle, thick, rounded corners=5pt] (mu1) at (.4,2) {$\scriptstyle \mu^A_{1_\cC,c}$};
\draw (iota) to[in=-90,out=90] node[right] {$\scriptstyle A(1_\cC)$} (mu1.-60);
\draw (c) to[in=-90,out=90] (mu1.-120);
\draw (mu1) to[in=-90,out=90] (d);
}
$$
The fact that $\ell$ and $r$ are $2$-modifications means that
for all $x \in \cC(a\to b)$, the following diagram commutes:
$$
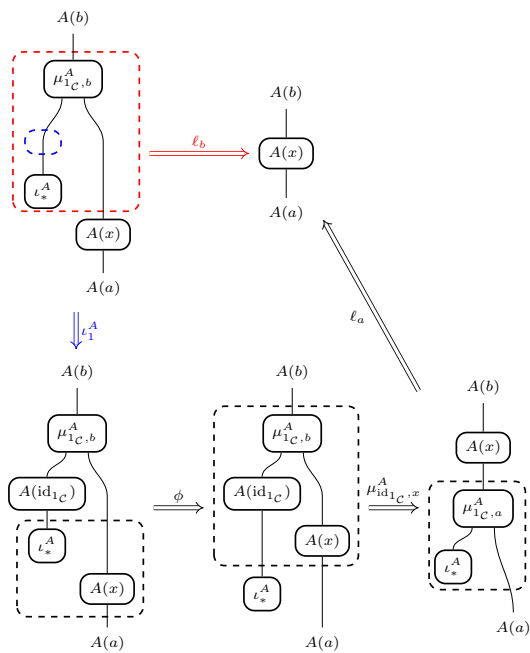
\begin{tikzpicture}[baseline= (a).base]\node[scale=.8] (a) at (0,0){\begin{tikzcd}
\tikzmath{
\coordinate (c2) at (0,1);
\node (a) at (1,-1.5) {$\scriptstyle A(a)$};
\node (b) at (.5,3) {$\scriptstyle A(b)$};
\coordinate (Ax2) at (1,1);
\node[draw,rectangle, thick, rounded corners=5pt] (iota) at (0,.1) {$\scriptstyle \iota^A_* $};
\node[draw,rectangle, thick, rounded corners=5pt] (Ax) at (1,-.6) {$\scriptstyle A(x) $};
\node[draw,rectangle, thick, rounded corners=5pt] (mu1) at (.5,2) {$\scriptstyle \mu^A_{1_\cC,b}$};
\draw (iota) to[in=-90,out=90] (c2)  to[in=-90,out=90] (mu1.-120);
\draw[dashed, blue, thick, rounded corners=5pt] (-.3,.8) rectangle (.3,1.2);
\draw[dashed, red, thick, rounded corners=5pt] (-.5,-.15) rectangle (1.6,2.5);
\draw (a) to[in=-90,out=90] (Ax);
\draw (Ax) to[in=-90,out=90] (Ax2) to[in=-90,out=90] (mu1.-60);
\draw (mu1) to[in=-90,out=90] (b);
}
\arrow[Rightarrow,r,red,"\ell_b"]
\arrow[Rightarrow,d,blue,"\iota^A_1"]
&
\tikzmath{
\node (a) at (0,0) {$\scriptstyle A(a)$};
\node (b) at (0,2) {$\scriptstyle A(b)$};
\node[draw,rectangle, thick, rounded corners=5pt] (Ax) at (0,1) {$\scriptstyle A(x) $};
\draw (a) to[in=-90,out=90] (Ax);
\draw (Ax) to[in=-90,out=90] (b);
}
\\
\tikzmath{
\node (a) at (1,-1.5) {$\scriptstyle A(a)$};
\node (b) at (.5,3) {$\scriptstyle A(b)$};
\coordinate (Ax2) at (1,1);
\node[draw,rectangle, thick, rounded corners=5pt] (iota) at (0,.1) {$\scriptstyle \iota^A_* $};
\node[draw,rectangle, thick, rounded corners=5pt] (c2) at (0,1) {$\scriptstyle A(\id_{1_\cC}) $};
\node[draw,rectangle, thick, rounded corners=5pt] (Ax) at (1,-.6) {$\scriptstyle A(x) $};
\node[draw,rectangle, thick, rounded corners=5pt] (mu1) at (.5,2) {$\scriptstyle \mu^A_{1_\cC,b}$};
\draw (iota) to[in=-90,out=90] (c2);
\draw (c2) to[in=-90,out=90] (mu1.-120);
\draw[dashed, thick, rounded corners=5pt] (-.5,-1) rectangle (1.6,.6);
\draw (a) to[in=-90,out=90] (Ax);
\draw (Ax) to[in=-90,out=90] (Ax2) to[in=-90,out=90] (mu1.-60);
\draw (mu1) to[in=-90,out=90] (b);
}
\arrow[Rightarrow,r,"\phi"]
&
\tikzmath{
\node (a) at (1,-1.5) {$\scriptstyle A(a)$};
\node (b) at (.5,3) {$\scriptstyle A(b)$};
\coordinate (Ax2) at (1,1);
\node[draw,rectangle, thick, rounded corners=5pt] (iota) at (0,-.7) {$\scriptstyle \iota^A_* $};
\node[draw,rectangle, thick, rounded corners=5pt] (c2) at (0,1) {$\scriptstyle A(\id_{1_\cC}) $};
\node[draw,rectangle, thick, rounded corners=5pt] (Ax) at (1,.2) {$\scriptstyle A(x) $};
\node[draw,rectangle, thick, rounded corners=5pt] (mu1) at (.5,2) {$\scriptstyle \mu^A_{1_\cC,b}$};
\draw (iota) to[in=-90,out=90] (c2);
\draw (c2) to[in=-90,out=90] (mu1.-120);
\draw[dashed, thick, rounded corners=5pt] (-.8,-.15) rectangle (1.6,2.5);
\draw (a) to[in=-90,out=90] (Ax);
\draw (Ax) to[in=-90,out=90] (Ax2) to[in=-90,out=90] (mu1.-60);
\draw (mu1) to[in=-90,out=90] (b);
}
\arrow[Rightarrow,r,"\mu^A_{\id_{1_\cC},x}"]
&
\tikzmath{
\node (a) at (1,0) {$\scriptstyle A(a)$};
\node (b) at (.5,4) {$\scriptstyle A(b)$};
\node[draw,rectangle, thick, rounded corners=5pt] (Ax) at (.5,3) {$\scriptstyle A(x)$};
\node[draw,rectangle, thick, rounded corners=5pt] (iota) at (0,1) {$\scriptstyle \iota^A_* $};
\node[draw,rectangle, thick, rounded corners=5pt] (mu1) at (.5,2) {$\scriptstyle \mu^A_{1_\cC,a}$};
\draw[dashed, thick, rounded corners=5pt] (-.4,.7) rectangle (1.2,2.5);
\draw (iota) to[in=-90,out=90] (mu1.-120);
\draw (a) to[in=-90,out=90] (mu1.-60);
\draw (mu1) to[in=-90,out=90] (Ax);
\draw (Ax) to[in=-90,out=90] (b);
}
\arrow[Rightarrow,ul,"\ell_a"]
\end{tikzcd}};\end{tikzpicture}
$$
and a similar condition for $r$.
\end{enumerate}

This data is subject to the additional two coherence conditions c.f.~\cite[Def.~4.10]{MR3076451}:
\begin{enumerate}[label=(F-\arabic*)]
\item
\label{Functor:PentagonCoherence}
For all $a,b,c,d\in\cC$, the following diagram commutes:
$$
\begin{tikzpicture}[baseline= (a).base]\node[scale=.8] (a) at (0,0){\begin{tikzcd}
\tikzmath{
\node (a) at (0,0) {$\scriptstyle A(a)$};
\node (b) at (1,0) {$\scriptstyle A(b)$};
\node (c) at (2,0) {$\scriptstyle A(c)$};
\node (d) at (3,0) {$\scriptstyle A(d)$};
\node (e) at (2,4) {$\scriptstyle A(a\xz b\xz c\xz d)$};
\node[draw,rectangle, thick, rounded corners=5pt] (mu1) at (.5,1) {$\scriptstyle \mu^A_{a,b}$};
\node[draw,rectangle, thick, rounded corners=5pt] (mu2) at (1.25,2) {$\scriptstyle \mu^A_{a\xz b, c}$};
\node[draw,rectangle, thick, rounded corners=5pt] (mu3) at (2,3) {$\scriptstyle \mu^A_{a\xz b\xz c,d}$};
\draw[dashed, red, thick, rounded corners=5pt] (0,.5) rectangle (2,2.5);
\draw[dashed, blue, thick, rounded corners=5pt] (.5,1.5) rectangle (3,3.5);
\draw (a) to[in=-90,out=90] (mu1.-120);
\draw (b) to[in=-90,out=90] (mu1.-60);
\draw[double] (mu1) to[in=-90,out=90] (mu2.-120);
\draw (c) to[in=-90,out=90] (mu2.-60);
\draw[triple] (mu2) to[in=-90,out=90] (mu3.-120);
\draw (d) to[in=-90,out=90] (mu3.-60);
\draw[quadruple] (mu3) to[in=-90,out=90] (e);
}
\arrow[Rightarrow,rr,red,"\omega^A_{a,b,c}"]
\arrow[Rightarrow,d,blue,"\omega^A_{a\xz b,c,d}"]
&&
\tikzmath{
\node (a) at (0,0) {$\scriptstyle A(a)$};
\node (b) at (1,0) {$\scriptstyle A(b)$};
\node (c) at (2,0) {$\scriptstyle A(c)$};
\node (d) at (3,0) {$\scriptstyle A(d)$};
\node (e) at (2,4) {$\scriptstyle A(a\xz b\xz c\xz d)$};
\node[draw,rectangle, thick, rounded corners=5pt] (mu1) at (1.5,1) {$\scriptstyle \mu^A_{b,c}$};
\node[draw,rectangle, thick, rounded corners=5pt] (mu2) at (1.25,2) {$\scriptstyle \mu^A_{a,b\xz c}$};
\node[draw,rectangle, thick, rounded corners=5pt] (mu3) at (2,3) {$\scriptstyle \mu^A_{a\xz b\xz c,d}$};
\draw[dashed, thick, rounded corners=5pt] (.5,1.5) rectangle (3,3.5);
\draw (a) to[in=-90,out=90] (mu2.-120);
\draw (b) to[in=-90,out=90] (mu1.-120);
\draw[double] (mu1) to[in=-90,out=90] (mu2.-60);
\draw (c) to[in=-90,out=90] (mu1.-60);
\draw[triple] (mu2) to[in=-90,out=90] (mu3.-120);
\draw (d) to[in=-90,out=90] (mu3.-60);
\draw[quadruple] (mu3) to[in=-90,out=90] (e);
}
\arrow[Rightarrow,rr,"\omega^A_{a,b\xz c, d}"]
&&
\tikzmath{
\node (a) at (0,0) {$\scriptstyle A(a)$};
\node (b) at (1,0) {$\scriptstyle A(b)$};
\node (c) at (2,0) {$\scriptstyle A(c)$};
\node (d) at (3,0) {$\scriptstyle A(d)$};
\node (e) at (1,4) {$\scriptstyle A(a\xz b\xz c\xz d)$};
\node[draw,rectangle, thick, rounded corners=5pt] (mu1) at (1.5,1) {$\scriptstyle \mu^A_{b,c}$};
\node[draw,rectangle, thick, rounded corners=5pt] (mu2) at (1.75,2) {$\scriptstyle \mu^A_{b\xz c, d}$};
\node[draw,rectangle, thick, rounded corners=5pt] (mu3) at (1,3) {$\scriptstyle \mu^A_{a,b\xz c\xz d}$};
\draw[dashed, thick, rounded corners=5pt] (1,.5) rectangle (3,2.5);
\draw (a) to[in=-90,out=90] (mu3.-120);
\draw (b) to[in=-90,out=90] (mu1.-120);
\draw[double] (mu1) to[in=-90,out=90] (mu2.-120);
\draw (c) to[in=-90,out=90] (mu1.-60);
\draw[triple] (mu2) to[in=-90,out=90] (mu3.-60);
\draw (d) to[in=-90,out=90] (mu2.-60);
\draw[quadruple] (mu3) to[in=-90,out=90] (e);
}
\arrow[Rightarrow,d,"\omega^A_{b,c,d}"]
\\
\tikzmath{
\node (a) at (0,0) {$\scriptstyle A(a)$};
\node (b) at (1,0) {$\scriptstyle A(b)$};
\node (c) at (2,0) {$\scriptstyle A(c)$};
\node (d) at (3,0) {$\scriptstyle A(d)$};
\node (e) at (1.5,4) {$\scriptstyle A(a\xz b\xz c\xz d)$};
\node[draw,rectangle, thick, rounded corners=5pt] (mu1) at (.5,1.1) {$\scriptstyle \mu^A_{a,b}$};
\node[draw,rectangle, thick, rounded corners=5pt] (mu2) at (2.5,1.9) {$\scriptstyle \mu^A_{c,d}$};
\node[draw,rectangle, thick, rounded corners=5pt] (mu3) at (1.5,3) {$\scriptstyle \mu^A_{a\xz b,c\xz d}$};
\draw[dashed, thick, rounded corners=5pt] (0,.5) rectangle (3,2.5);
\draw (a) to[in=-90,out=90] (mu1.-120);
\draw (b) to[in=-90,out=90] (mu1.-60);
\draw[double] (mu1) to[in=-90,out=90] (mu3.-120);
\draw (c) to[in=-90,out=90] (mu2.-120);
\draw[double] (mu2) to[in=-90,out=90] (mu3.-60);
\draw (d) to[in=-90,out=90] (mu2.-60);
\draw[quadruple] (mu3) to[in=-90,out=90] (e);
}
\arrow[Rightarrow,rr,"\phi^{-1}"]
&&
\tikzmath{
\node (a) at (0,0) {$\scriptstyle A(a)$};
\node (b) at (1,0) {$\scriptstyle A(b)$};
\node (c) at (2,0) {$\scriptstyle A(c)$};
\node (d) at (3,0) {$\scriptstyle A(d)$};
\node (e) at (1.5,4) {$\scriptstyle A(a\xz b\xz c\xz d)$};
\node[draw,rectangle, thick, rounded corners=5pt] (mu1) at (.5,1.9) {$\scriptstyle \mu^A_{a,b}$};
\node[draw,rectangle, thick, rounded corners=5pt] (mu2) at (2.5,1) {$\scriptstyle \mu^A_{c,d}$};
\node[draw,rectangle, thick, rounded corners=5pt] (mu3) at (1.5,3) {$\scriptstyle \mu^A_{a\xz b,c\xz d}$};
\draw[dashed, thick, rounded corners=5pt] (0,1.5) rectangle (2.7,3.5);
\draw (a) to[in=-90,out=90] (mu1.-120);
\draw (b) to[in=-90,out=90] (mu1.-60);
\draw[double] (mu1) to[in=-90,out=90] (mu3.-120);
\draw (c) to[in=-90,out=90] (mu2.-120);
\draw[double] (mu2) to[in=-90,out=90] (mu3.-60);
\draw (d) to[in=-90,out=90] (mu2.-60);
\draw[quadruple] (mu3) to[in=-90,out=90] (e);
}
\arrow[Rightarrow,rr,"\omega^A_{a,b,c\xz d}"]
&&
\tikzmath{
\node (a) at (0,0) {$\scriptstyle A(a)$};
\node (b) at (1,0) {$\scriptstyle A(b)$};
\node (c) at (2,0) {$\scriptstyle A(c)$};
\node (d) at (3,0) {$\scriptstyle A(d)$};
\node (e) at (1,4) {$\scriptstyle A(a\xz b\xz c\xz d)$};
\node[draw,rectangle, thick, rounded corners=5pt] (mu1) at (2.5,1) {$\scriptstyle \mu^A_{c,d}$};
\node[draw,rectangle, thick, rounded corners=5pt] (mu2) at (1.75,2) {$\scriptstyle \mu^A_{b,c\xz d}$};
\node[draw,rectangle, thick, rounded corners=5pt] (mu3) at (1,3) {$\scriptstyle \mu^A_{a,b\xz c\xz d}$};
\draw (a) to[in=-90,out=90] (mu3.-120);
\draw (b) to[in=-90,out=90] (mu2.-120);
\draw[double] (mu1) to[in=-90,out=90] (mu2.-60);
\draw (c) to[in=-90,out=90] (mu1.-120);
\draw[triple] (mu2) to[in=-90,out=90] (mu3.-60);
\draw (d) to[in=-90,out=90] (mu1.-60);
\draw[quadruple] (mu3) to[in=-90,out=90] (e);
}
\end{tikzcd}};\end{tikzpicture}
$$

\item 
\label{Functor:TriangleCoherence}
For all $a,b,c\in\cC$, the following diagram commutes:
$$
\begin{tikzpicture}[baseline= (a).base]\node[scale=.8] (a) at (0,0){\begin{tikzcd}
\tikzmath{
\node (a) at (0,-1) {$\scriptstyle A(a)$};
\node (c) at (2,-1) {$\scriptstyle A(b)$};
\node (d) at (1.25,3) {$\scriptstyle A(a\xz b)$};
\node[draw,rectangle, thick, rounded corners=5pt] (iota) at (1,0) {$\scriptstyle \iota^A_*$};
\node[draw,rectangle, thick, rounded corners=5pt] (mu1) at (.5,1) {$\scriptstyle \mu^A_{a,1_\cC}$};
\draw (a) to[in=-90,out=90] (mu1.-120);
\draw (iota) to[in=-90,out=90] (mu1.-60);
\node[draw,rectangle, thick, rounded corners=5pt] (mu2) at (1.25,2) {$\scriptstyle \mu^A_{a,b}$};
\draw[dashed, red, thick, rounded corners=5pt] (-.1,.5) rectangle (1.8,2.5);
\draw[dashed, blue, thick, rounded corners=5pt] (-.1,-.4) rectangle (1.4,1.5);
\draw[double] (mu1) to[in=-90,out=90] (mu2.-120);
\draw (c) to[in=-90,out=90] (mu2.-60);
\draw[triple] (mu2) to[in=-90,out=90] (d);
}
\arrow[Rightarrow,rr,red,"\omega^A_{a,1_\cC,b}"]
\arrow[Rightarrow,dr,blue,"r_a"]
&&
\tikzmath{
\node (a) at (0,-1) {$\scriptstyle A(a)$};
\node (c) at (2,-1) {$\scriptstyle A(b)$};
\node (d) at (.75,3) {$\scriptstyle A(a\xz b)$};
\node[draw,rectangle, thick, rounded corners=5pt] (mu1) at (1.5,1) {$\scriptstyle \mu^A_{1_\cC,b}$};
\node[draw,rectangle, thick, rounded corners=5pt] (iota) at (1,0) {$\scriptstyle \iota^A_*$};
\draw (iota) to[in=-90,out=90] (mu1.-120);
\draw (c) to[in=-90,out=90] (mu1.-60);
\node[draw,rectangle, thick, rounded corners=5pt] (mu2) at (.75,2) {$\scriptstyle \mu^A_{a, b}$};
\draw[dashed, thick, rounded corners=5pt] (.6,-.4) rectangle (2.1,1.5);
\draw[double] (mu1) to[in=-90,out=90] (mu2.-60);
\draw (a) to[in=-90,out=90] (mu2.-120);
\draw[triple] (mu2) to[in=-90,out=90] (d);
}
\arrow[Rightarrow,dl,"\ell_b"]
\\
&
\tikzmath{
\node (a) at (0,0) {$\scriptstyle A(a)$};
\node (c) at (1,0) {$\scriptstyle A(b)$};
\node (d) at (.5,2) {$\scriptstyle A(a\xz b)$};
\node[draw,rectangle, thick, rounded corners=5pt] (mu1) at (.5,1) {$\scriptstyle \mu^A_{a,b}$};
\draw (a) to[in=-90,out=90] (mu1.-120);
\draw (c) to[in=-90,out=90] (mu1.-60);
\draw[double] (mu1) to[in=-90,out=90] (d);
}
\end{tikzcd}};\end{tikzpicture}
$$
\end{enumerate}
\end{defn}

%%%%%%%%%%%%%%%%%%%%%%%%%%%%%%%%%%%%%%%%%%%%
\subsection{Transformations between functors of $\Gray$-monoids}
\begin{defn}
\label{defn:Transformation}
Suppose $\cC,\cD$ are $\Gray$-monoids and $A,B: \rmB\cC \to \rmB\cD$ are 3-functors.
A transformation $\eta: A \Rightarrow B$ consists of:
\begin{enumerate}[label=(T-\Roman*)]
\item
\label{Transformation:eta_*}
An object $\eta_* \in \cD$, which we depict by an oriented \textcolor{DarkGreen}{green} strand:
$
\tikzmath{
\node[DarkGreen] (eta1) at (0,0) {$\scriptstyle \eta_*$};
\node (eta2) at (0,1) {};
\draw[DarkGreen,thick,mid>] (eta1) to[in=-90,out=90] (eta2);
}
$
\item
\label{Transformation:eta_c}
An adjoint equivalence $\eta:\cD(\id,\eta_*)\circ A \Rightarrow \cD(\eta_*, \id) \circ B$ in the $2$-category of $2$-functors $\cC \to \cD$.
Explicitly, this is given by, for each $c\in \cC$, an adjoint equivalence $1$-cell $\eta_c : \eta_* \xz A(c) \Rightarrow B(c) \xz \eta_*$ which we depict as a crossing
$$
\eta_c
=
\tikzmath{
\node[DarkGreen] (eta1) at (0,0) {$\scriptstyle \eta_*$};
\node (Ac) at (1,0) {$\scriptstyle A(c)$};
\node[DarkGreen] (eta2) at (1,2) {$\scriptstyle \eta_*$};
\node (Bc) at (0,2) {$\scriptstyle B(c)$};
\node[draw,rectangle, thick, rounded corners=5pt] (etac) at (.5,1) {$\scriptstyle \eta_c$};
\draw[DarkGreen,thick,mid>] (eta1) to[in=-90,out=90] (etac.-120);
\draw (Ac) to[in=-90,out=90] (etac.-60);
\draw[DarkGreen,thick,mid>] (etac.60) to[in=-90,out=90] (eta2);
\draw (etac.120) to[in=-90,out=90] (Bc);
}
=:
\tikzmath{
\node[DarkGreen] (eta1) at (0,0) {$\scriptstyle \eta_*$};
\node (Ac) at (1,0) {$\scriptstyle A(c)$};
\node[DarkGreen] (eta2) at (1,2) {$\scriptstyle \eta_*$};
\node (Bc) at (0,2) {$\scriptstyle B(c)$};
\draw[DarkGreen,thick,mid>] (eta1) to[in=-90,out=90] (eta2);
\draw (Ac) to[in=-90,out=90] (Bc);
}
\,,
$$
together with, for each $x\in \cC(a\to b)$, invertible 2-cells
$$ 
\tikzmath{
\node[DarkGreen] (eta1) at (0,-1) {$\scriptstyle \eta_*$};
\node (Ac) at (1,-1) {$\scriptstyle A(a)$};
\coordinate (eta2) at (0,0);
\node (Bc) at (0,2) {$\scriptstyle B(b)$};
\node[DarkGreen] (eta3) at (1,2) {$\scriptstyle \eta_*$};
\node[draw,rectangle, thick, rounded corners=5pt] (Af) at (1,0) {$\scriptstyle A(x)$};
\draw (Ac) to[in=-90,out=90] (Af);
\draw[DarkGreen,thick,mid>] (eta1) to[in=-90,out=90] (eta2) to[in=-90,out=90] (eta3);
\draw (Af) node [right,yshift=.4cm] {$\scriptstyle A(b)$} to[in=-90,out=90] (Bc);
}
\overset{\eta_x}{\Rightarrow}
\tikzmath{
\node[DarkGreen] (eta1) at (0,0) {$\scriptstyle \eta_*$};
\node (Ac) at (1,0) {$\scriptstyle A(a)$};
\coordinate (eta2) at (1,2);
\node (Bc) at (0,3) {$\scriptstyle B(b)$};
\node[DarkGreen] (eta3) at (1,3) {$\scriptstyle \eta_*$};
\node[draw,rectangle, thick, rounded corners=5pt] (Bf) at (0,2) {$\scriptstyle B(x)$};
\draw (Ac) to[in=-90,out=90] (Bf);
\draw[DarkGreen,thick,mid>] (eta1) to[in=-90,out=90] (eta2) to[in=-90,out=90] (eta3);
\draw (Bf) node [left,yshift=-.4cm] {$\scriptstyle A(b)$} to[in=-90,out=90] (Bc);
}
\,.
$$
The fact that $\eta$ is a $2$-natural transformation means that:
\begin{enumerate}[label=\ref{Transformation:eta_c}.\roman*]
\item 
\label{Transformation:eta_c.natural}
For every $x,y\in \cC(a\to b)$ and $f\in\cC(x\Rightarrow y)$,
the following diagram commutes:
$$
\begin{tikzpicture}[baseline= (a).base]\node[scale=.8] (a) at (0,0){\begin{tikzcd}
\tikzmath{
\node[DarkGreen] (eta1) at (0,-1) {$\scriptstyle \eta_*$};
\node (Ac) at (1,-1) {$\scriptstyle A(a)$};
\coordinate (eta2) at (0,0);
\node (Bc) at (0,2) {$\scriptstyle B(b)$};
\node[DarkGreen] (eta3) at (1,2) {$\scriptstyle \eta_*$};
\node[draw,rectangle, thick, rounded corners=5pt] (Af) at (1,0) {$\scriptstyle A(x)$};
\draw[dashed, red, thick, rounded corners=5pt] (.4,.45) rectangle (1.6,-.3);
\draw (Ac) to[in=-90,out=90] (Af);
\draw[DarkGreen,thick,mid>] (eta1) to[in=-90,out=90] (eta2) to[in=-90,out=90] (eta3);
\draw (Af) to[in=-90,out=90] (Bc);
}
\arrow[Rightarrow,r,"\eta_x"]
\arrow[Rightarrow,d,red,"A(f)"]
&
\tikzmath{
\node[DarkGreen] (eta1) at (0,0) {$\scriptstyle \eta_*$};
\node (Ac) at (1,0) {$\scriptstyle A(a)$};
\coordinate (eta2) at (1,2);
\node (Bc) at (0,3) {$\scriptstyle B(b)$};
\node[DarkGreen] (eta3) at (1,3) {$\scriptstyle \eta_*$};
\node[draw,rectangle, thick, rounded corners=5pt] (Bf) at (0,2) {$\scriptstyle B(x)$};
\draw[dashed, thick, rounded corners=5pt] (-.6,2.45) rectangle (.6,1.7);
\draw (Ac) to[in=-90,out=90] (Bf);
\draw[DarkGreen,thick,mid>] (eta1) to[in=-90,out=90] (eta2) to[in=-90,out=90] (eta3);
\draw (Bf) to[in=-90,out=90] (Bc);
}
\arrow[Rightarrow,d,"B(f)"]
\\
\tikzmath{
\node[DarkGreen] (eta1) at (0,-1) {$\scriptstyle \eta_*$};
\node (Ac) at (1,-1) {$\scriptstyle A(a)$};
\coordinate (eta2) at (0,0);
\node (Bc) at (0,2) {$\scriptstyle B(b)$};
\node[DarkGreen] (eta3) at (1,2) {$\scriptstyle \eta_*$};
\node[draw,rectangle, thick, rounded corners=5pt] (Af) at (1,0) {$\scriptstyle A(y)$};
\draw (Ac) to[in=-90,out=90] (Af);
\draw[DarkGreen,thick,mid>] (eta1) to[in=-90,out=90] (eta2) to[in=-90,out=90] (eta3);
\draw (Af) to[in=-90,out=90] (Bc);
}
\arrow[Rightarrow,r,"\eta_y"]
&
\tikzmath{
\node[DarkGreen] (eta1) at (0,0) {$\scriptstyle \eta_*$};
\node (Ac) at (1,0) {$\scriptstyle A(a)$};
\coordinate (eta2) at (1,2);
\node (Bc) at (0,3) {$\scriptstyle B(b)$};
\node[DarkGreen] (eta3) at (1,3) {$\scriptstyle \eta_*$};
\node[draw,rectangle, thick, rounded corners=5pt] (Bf) at (0,2) {$\scriptstyle B(y)$};
\draw (Ac) to[in=-90,out=90] (Bf);
\draw[DarkGreen,thick,mid>] (eta1) to[in=-90,out=90] (eta2) to[in=-90,out=90] (eta3);
\draw (Bf) to[in=-90,out=90] (Bc);
}
\end{tikzcd}};\end{tikzpicture}
$$
\item
\label{Transformation:eta_c.monoidal}
For every $x\in \cC(a\to b)$ and $y\in\cC(b\Rightarrow c)$,
the following diagram commutes:
$$
\begin{tikzpicture}[baseline= (a).base]\node[scale=.8] (a) at (0,0){\begin{tikzcd}
\tikzmath{
\node[DarkGreen] (eta1) at (0,-2) {$\scriptstyle \eta_*$};
\node (Aa) at (1,-2) {$\scriptstyle A(a)$};
\coordinate (eta2) at (0,0);
\node (Bc) at (0,2) {$\scriptstyle B(c)$};
\node[DarkGreen] (eta3) at (1,2) {$\scriptstyle \eta_*$};
\node[draw,rectangle, thick, rounded corners=5pt] (Ay) at (1,0) {$\scriptstyle A(y)$};
\node[draw,rectangle, thick, rounded corners=5pt] (Ax) at (1,-1) {$\scriptstyle A(x)$};
\draw[dashed, red, thick, rounded corners=5pt] (-.4,1.4) rectangle (1.6,-.4);
\draw[dashed, blue, thick, rounded corners=5pt] (.4,-1.4) rectangle (1.6,.5);
\draw (Aa) to[in=-90,out=90] (Ax);
\draw (Ax) to[in=-90,out=90] (Ay);
\draw[DarkGreen,thick,mid>] (eta1) to[in=-90,out=90] (eta2) to[in=-90,out=90] (eta3);
\draw (Ay) to[in=-90,out=90] (Bc);
}
\arrow[Rightarrow,r,red,"\eta_y"]
\arrow[Rightarrow,d,blue,"A^2_{y,x}"]
&
\tikzmath{
\node[DarkGreen] (eta1) at (0,-1) {$\scriptstyle \eta_*$};
\node (Ac) at (1,-1) {$\scriptstyle A(a)$};
\coordinate (eta2) at (0,0);
\coordinate (eta3) at (1,2);
\node (Bc) at (0,3) {$\scriptstyle B(b)$};
\node[DarkGreen] (eta4) at (1,3) {$\scriptstyle \eta_*$};
\node[draw,rectangle, thick, rounded corners=5pt] (Bf) at (0,2) {$\scriptstyle B(y)$};
\node[draw,rectangle, thick, rounded corners=5pt] (Ax) at (1,0) {$\scriptstyle A(x)$};
\draw[dashed, thick, rounded corners=5pt] (-.3,-.4) rectangle (1.6,1.3);
\draw (Ac) to[in=-90,out=90] (Ax);
\draw (Ax) to[in=-90,out=90] (Bf);
\draw[DarkGreen,thick,mid>] (eta1) to[in=-90,out=90] (eta2) to[in=-90,out=90] (eta3) to[in=-90,out=90] (eta4);
\draw (Bf) to[in=-90,out=90] (Bc);
}
\arrow[Rightarrow,r,"\eta_x"]
&
\tikzmath{
\node[DarkGreen] (eta1) at (0,-1) {$\scriptstyle \eta_*$};
\node (Ac) at (1,-1) {$\scriptstyle A(a)$};
\coordinate (eta2) at (1,1);
\node (Bc) at (0,3) {$\scriptstyle B(b)$};
\node[DarkGreen] (eta3) at (1,3) {$\scriptstyle \eta_*$};
\node[draw,rectangle, thick, rounded corners=5pt] (Bf) at (0,2) {$\scriptstyle B(y)$};
\node[draw,rectangle, thick, rounded corners=5pt] (Ax) at (0,1) {$\scriptstyle B(x)$};
\draw[dashed, thick, rounded corners=5pt] (-.6,.6) rectangle (.6,2.5);
\draw (Ac) to[in=-90,out=90] (Ax);
\draw (Ax) to[in=-90,out=90] (Bf);
\draw[DarkGreen,thick,mid>] (eta1) to[in=-90,out=90] (eta2) to[in=-90,out=90] (eta3);
\draw (Bf) to[in=-90,out=90] (Bc);
}
\arrow[Rightarrow,dl,"B^2_{y,x}"]
\\
\tikzmath{
\node[DarkGreen] (eta1) at (0,-1) {$\scriptstyle \eta_*$};
\node (Ac) at (1,-1) {$\scriptstyle A(a)$};
\coordinate (eta2) at (0,0);
\node (Bc) at (0,2) {$\scriptstyle B(c)$};
\node[DarkGreen] (eta3) at (1,2) {$\scriptstyle \eta_*$};
\node[draw,rectangle, thick, rounded corners=5pt] (Af) at (1,0) {$\scriptstyle A(y\xo x)$};
\draw (Ac) to[in=-90,out=90] (Af);
\draw[DarkGreen,thick,mid>] (eta1) to[in=-90,out=90] (eta2) to[in=-90,out=90] (eta3);
\draw (Af) to[in=-90,out=90] (Bc);
}
\arrow[Rightarrow,r,"\eta_{y\xo x}"]
&
\tikzmath{
\node[DarkGreen] (eta1) at (0,0) {$\scriptstyle \eta_*$};
\node (Ac) at (1,0) {$\scriptstyle A(a)$};
\coordinate (eta2) at (1,2);
\node (Bc) at (0,3) {$\scriptstyle B(b)$};
\node[DarkGreen] (eta3) at (1,3) {$\scriptstyle \eta_*$};
\node[draw,rectangle, thick, rounded corners=5pt] (Bf) at (0,2) {$\scriptstyle B(y\xo x)$};
\draw (Ac) to[in=-90,out=90] (Bf);
\draw[DarkGreen,thick,mid>] (eta1) to[in=-90,out=90] (eta2) to[in=-90,out=90] (eta3);
\draw (Bf) to[in=-90,out=90] (Bc);
}
\end{tikzcd}};\end{tikzpicture}
$$
\item
\label{Transformation:eta_c.unital}
For every $c\in \cC$, the following diagram commutes:
$$
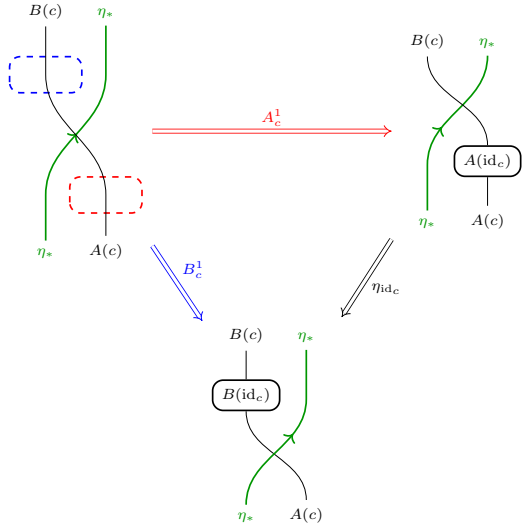
\begin{tikzpicture}[baseline= (a).base]\node[scale=.8] (a) at (0,0){\begin{tikzcd}
\tikzmath{
\node[DarkGreen] (eta1) at (0,-1) {$\scriptstyle \eta_*$};
\coordinate (eta2) at (0,0);
\coordinate (eta3) at (1,2);
\node[DarkGreen] (eta4) at (1,3) {$\scriptstyle \eta_*$};
\node (A1) at (1,-1) {$\scriptstyle A(c)$};
\coordinate (A2) at (1,0);
\coordinate (B1) at (0,2);
\node (B2) at (0,3) {$\scriptstyle B(c)$};
\draw[dashed, thick, blue, rounded corners=5pt] (-.6,1.7) rectangle (.6,2.3);
\draw[dashed, thick, red, rounded corners=5pt] (.4,-.3) rectangle (1.6,.3);
\draw[DarkGreen,thick,mid>] (eta1) to[in=-90,out=90] (eta2) to[in=-90,out=90] (eta3) to[in=-90,out=90] (eta4);
\draw (A1) to[in=-90,out=90] (A2) to[in=-90,out=90] (B1) to[in=-90,out=90] (B2);
}
\arrow[Rightarrow,rr,red,"A^1_c"]
\arrow[Rightarrow,dr,blue,"B^1_c"]
&&
\tikzmath{
\node[DarkGreen] (eta1) at (0,-1) {$\scriptstyle \eta_*$};
\node (Ac) at (1,-1) {$\scriptstyle A(c)$};
\coordinate (eta2) at (0,0);
\node (Bc) at (0,2) {$\scriptstyle B(c)$};
\node[DarkGreen] (eta3) at (1,2) {$\scriptstyle \eta_*$};
\node[draw,rectangle, thick, rounded corners=5pt] (Af) at (1,0) {$\scriptstyle A(\id_{c})$};
\draw (Ac) to[in=-90,out=90] (Af);
\draw[DarkGreen,thick,mid>] (eta1) to[in=-90,out=90] (eta2) to[in=-90,out=90] (eta3);
\draw (Af) to[in=-90,out=90] (Bc);
}
\arrow[Rightarrow,dl,"\eta_{\id_c}"]
\\
&
\tikzmath{
\node[DarkGreen] (eta1) at (0,-1) {$\scriptstyle \eta_*$};
\node (Ac) at (1,-1) {$\scriptstyle A(c)$};
\coordinate (eta2) at (1,1);
\node (Bc) at (0,2) {$\scriptstyle B(c)$};
\node[DarkGreen] (eta3) at (1,2) {$\scriptstyle \eta_*$};
\node[draw,rectangle, thick, rounded corners=5pt] (Af) at (0,1) {$\scriptstyle B(\id_{c})$};
\draw (Ac) to[in=-90,out=90] (Af);
\draw[DarkGreen,thick,mid>] (eta1) to[in=-90,out=90] (eta2) to[in=-90,out=90] (eta3);
\draw (Af) to[in=-90,out=90] (Bc);
}
\end{tikzcd}
};\end{tikzpicture}
$$
Observe this immediately implies that $\eta_{\id_c} = B^1_c \xt (A^1_c)^{-1}$ for all $c\in \cC$.
\end{enumerate}

\item
\label{Transformation:eta^1}
A unit coherence invertible $2$-modification
$$
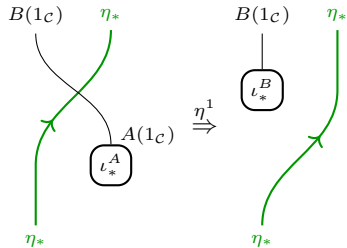
\tikzmath{
\node[DarkGreen] (eta1) at (0,-1) {$\scriptstyle \eta_*$};
\coordinate (eta2) at (0,0);
\node (Bc) at (0,2) {$\scriptstyle B(1_\cC)$};
\node[DarkGreen] (eta3) at (1,2) {$\scriptstyle \eta_*$};
\node[draw,rectangle, thick, rounded corners=5pt] (iA) at (1,0) {$\scriptstyle \iota^A_*$};
\draw[DarkGreen,thick,mid>] (eta1) to[in=-90,out=90] (eta2) to[in=-90,out=90] (eta3);
\draw (iA) node [right,yshift=.4cm] {$\scriptstyle A(1_\cC)$} to[in=-90,out=90] (Bc);
}
\overset{\eta^1}{\Rightarrow}
\tikzmath{
\node[DarkGreen] (eta1) at (0,-1) {$\scriptstyle \eta_*$};
\coordinate (eta2) at (1,1);
\coordinate (Bc1) at (0,1);
\node (Bc2) at (0,2) {$\scriptstyle B(1_\cC)$};
\node[DarkGreen] (eta3) at (1,2) {$\scriptstyle \eta_*$};
\node[draw,rectangle, thick, rounded corners=5pt] (iB) at (0,1) {$\scriptstyle \iota^B_*$};
\draw[DarkGreen,thick,mid>] (eta1) to[in=-90,out=90] (eta2) to[in=-90,out=90] (eta3);
\draw (iB) to[in=-90,out=90] (Bc);
}
$$
The 2-modification criterion for $\eta^1$ is automatically satisfied by \ref{Functor:iota} (which says $\iota^A_1=A^1_{1_\cC}$ and $\iota^B_1=B^1_{1_\cC}$) and \ref{Transformation:eta_c.unital} above.

\item
\label{Transformation:eta^2}
For every $a,b\in \cC$, a monoidality coherence invertible $2$-modification
$$
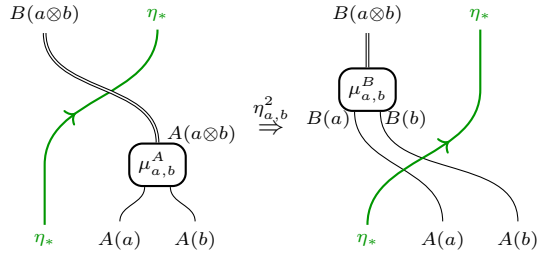
\tikzmath{
\node[DarkGreen] (eta1) at (0,-1) {$\scriptstyle \eta_*$};
\node (Aa) at (1,-1) {$\scriptstyle A(a)$};
\node (Ab) at (2,-1) {$\scriptstyle A(b)$};
\coordinate (eta2) at (0,0);
\node (Bc) at (0,2) {$\scriptstyle B(a\xz b)$};
\node[DarkGreen] (eta3) at (1.5,2) {$\scriptstyle \eta_*$};
\node[draw,rectangle, thick, rounded corners=5pt] (mu) at (1.5,0) {$\scriptstyle \mu^A_{a,b}$};
\draw (Aa) to[in=-90,out=90] (mu.-120);
\draw (Ab) to[in=-90,out=90] (mu.-60);
\draw[DarkGreen,thick,mid>] (eta1) to[in=-90,out=90] (eta2) to[in=-90,out=90] (eta3);
\draw[double] (mu) node [right,yshift=.4cm] {$\scriptstyle A(a\xz b)$} to[in=-90,out=90] (Bc);
}
\overset{\eta^2_{a,b}}{\Rightarrow}
\tikzmath{
\node[DarkGreen] (eta1) at (0,-1) {$\scriptstyle \eta_*$};
\node (Aa) at (1,-1) {$\scriptstyle A(a)$};
\node (Ab) at (2,-1) {$\scriptstyle A(b)$};
\coordinate (eta2) at (1.5,1);
\node (Bc) at (0,2) {$\scriptstyle B(a\xz b)$};
\node[DarkGreen] (eta3) at (1.5,2) {$\scriptstyle \eta_*$};
\node[draw,rectangle, thick, rounded corners=5pt] (mu) at (0,1) {$\scriptstyle \mu^B_{a,b}$};
\draw (Aa) to[in=-90,out=90] (mu.-120);
\draw (Ab) to[in=-90,out=90] (mu.-60);
\draw[DarkGreen,thick,mid>] (eta1) to[in=-90,out=90] (eta2) to[in=-90,out=90] (eta3);
\draw[double] (mu) node [left,yshift=-.4cm, xshift=-.1cm] {$\scriptstyle B(a)$} node [right, yshift=-.4cm, xshift=.1cm] {$\scriptstyle B(b)$} to[in=-90,out=90] (Bc);
}
$$
That $\eta^2$ is a 2-modification means that
for all $x\in \cC(a_1\to a_2)$ and $y\in \cC(b_1\to b_2)$,
$$
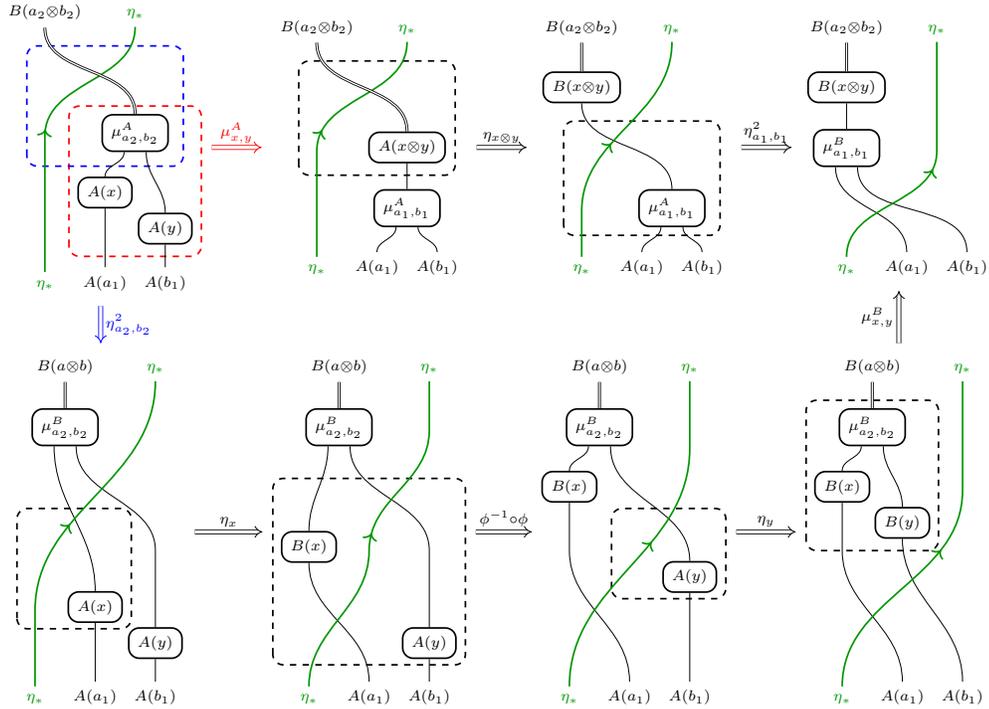
\begin{tikzpicture}[baseline= (a).base]\node[scale=.8] (a) at (0,0){\begin{tikzcd}
\tikzmath{
\node[DarkGreen] (eta0) at (0,-2.5) {$\scriptstyle \eta_*$};
\coordinate (eta1) at (0,-1);
\coordinate (eta2) at (0,0);
\node[DarkGreen] (eta3) at (1.5,2) {$\scriptstyle \eta_*$};
\node (Aa) at (1,-2.5) {$\scriptstyle A(a_1)$};
\node (Ab) at (2,-2.5) {$\scriptstyle A(b_1)$};
\node (Bc) at (0,2) {$\scriptstyle B(a_2\xz b_2)$};
\node[draw,rectangle, thick, rounded corners=5pt] (Ax) at (1,-1) {$\scriptstyle A(x)$};
\node[draw,rectangle, thick, rounded corners=5pt] (Ay) at (2,-1.6) {$\scriptstyle A(y)$};
\node[draw,rectangle, thick, rounded corners=5pt] (mu) at (1.5,0) {$\scriptstyle \mu^A_{a_2,b_2}$};
\draw[dashed, red, thick, rounded corners=5pt] (.4,-2) rectangle (2.6,.5);
\draw[dashed, blue, thick, rounded corners=5pt] (-.3,-.5) rectangle (2.3,1.5);
\draw (Aa) to[in=-90,out=90] (Ax);
\draw (Ab) to[in=-90,out=90] (Ay);
\draw (Ax) to[in=-90,out=90] (mu.-120);
\draw (Ay) to[in=-90,out=90] (mu.-60);
\draw[DarkGreen,thick,mid>] (eta0) to[in=-90,out=90] (eta1) to[in=-90,out=90] (eta2) to[in=-90,out=90] (eta3);
\draw[double] (mu) to[in=-90,out=90] (Bc);
}
\arrow[Rightarrow,d,blue,"\eta^2_{a_2,b_2}"]
\arrow[Rightarrow,r,red,"\mu^A_{x,y}"]
&
\tikzmath{
\node[DarkGreen] (eta0) at (0,-2) {$\scriptstyle \eta_*$};
\coordinate (eta1) at (0,-1);
\coordinate (eta2) at (0,0);
\node[DarkGreen] (eta3) at (1.5,2) {$\scriptstyle \eta_*$};
\node (Aa) at (1,-2) {$\scriptstyle A(a_1)$};
\node (Ab) at (2,-2) {$\scriptstyle A(b_1)$};
\node (Bc) at (0,2) {$\scriptstyle B(a_2\xz b_2)$};
\node[draw,rectangle, thick, rounded corners=5pt] (Axy) at (1.5,0) {$\scriptstyle A(x\xz y)$};
\node[draw,rectangle, thick, rounded corners=5pt] (mu) at (1.5,-1) {$\scriptstyle \mu^A_{a_1,b_1}$};
\draw[dashed, thick, rounded corners=5pt] (-.3,-.4) rectangle (2.3,1.5);
\draw (Aa) to[in=-90,out=90] (mu.-120);
\draw (Ab) to[in=-90,out=90] (mu.-60);
\draw (mu) to[in=-90,out=90] (Axy);
\draw[DarkGreen,thick,mid>] (eta0) to[in=-90,out=90] (eta1) to[in=-90,out=90] (eta2) to[in=-90,out=90] (eta3);
\draw[double] (Axy) to[in=-90,out=90] (Bc);
}
\arrow[Rightarrow,r,"\eta_{x\xz y}"]
&
\tikzmath{
\node[DarkGreen] (eta0) at (0,-2) {$\scriptstyle \eta_*$};
\coordinate (eta1) at (0,-1);
\node[DarkGreen] (eta2) at (1.5,2) {$\scriptstyle \eta_*$};
\node (Aa) at (1,-2) {$\scriptstyle A(a_1)$};
\node (Ab) at (2,-2) {$\scriptstyle A(b_1)$};
\node (Bc) at (0,2) {$\scriptstyle B(a_2\xz b_2)$};
\node[draw,rectangle, thick, rounded corners=5pt] (Axy) at (0,1) {$\scriptstyle B(x\xz y)$};
\node[draw,rectangle, thick, rounded corners=5pt] (mu) at (1.5,-1) {$\scriptstyle \mu^A_{a_1,b_1}$};
\draw[dashed, thick, rounded corners=5pt] (-.3,-1.4) rectangle (2.3,.5);
\draw (Aa) to[in=-90,out=90] (mu.-120);
\draw (Ab) to[in=-90,out=90] (mu.-60);
\draw (mu) to[in=-90,out=90] (Axy);
\draw[DarkGreen,thick,mid>] (eta0) to[in=-90,out=90] (eta1) to[in=-90,out=90] (eta2);
\draw[double] (Axy) to[in=-90,out=90] (Bc);
}
\arrow[Rightarrow,r,"\eta^2_{a_1,b_1}"]
&
\tikzmath{
\node[DarkGreen] (eta0) at (0,-2) {$\scriptstyle \eta_*$};
\coordinate (eta1) at (1.5,0);
\node[DarkGreen] (eta2) at (1.5,2) {$\scriptstyle \eta_*$};
\node (Aa) at (1,-2) {$\scriptstyle A(a_1)$};
\node (Ab) at (2,-2) {$\scriptstyle A(b_1)$};
\node (Bc) at (0,2) {$\scriptstyle B(a_2\xz b_2)$};
\node[draw,rectangle, thick, rounded corners=5pt] (Axy) at (0,1) {$\scriptstyle B(x\xz y)$};
\node[draw,rectangle, thick, rounded corners=5pt] (mu) at (0,0) {$\scriptstyle \mu^B_{a_1,b_1}$};
\draw (Aa) to[in=-90,out=90] (mu.-120);
\draw (Ab) to[in=-90,out=90] (mu.-60);
\draw (mu) to[in=-90,out=90] (Axy);
\draw[DarkGreen,thick,mid>] (eta0) to[in=-90,out=90] (eta1) to[in=-90,out=90] (eta2);
\draw[double] (Axy) to[in=-90,out=90] (Bc);
}
\\
\tikzmath{
\node[DarkGreen] (eta0) at (0,-2.5) {$\scriptstyle \eta_*$};
\coordinate (eta1) at (0,-1);
\node[DarkGreen] (eta2) at (2,3) {$\scriptstyle \eta_*$};
\node (Aa) at (1,-2.5) {$\scriptstyle A(a_1)$};
\node (Ab) at (2,-2.5) {$\scriptstyle A(b_1)$};
\node (Bc) at (.5,3) {$\scriptstyle B(a\xz b)$};
\node[draw,rectangle, thick, rounded corners=5pt] (mu) at (.5,2) {$\scriptstyle \mu^B_{a_2,b_2}$};
\node[draw,rectangle, thick, rounded corners=5pt] (Ax) at (1,-1) {$\scriptstyle A(x)$};
\node[draw,rectangle, thick, rounded corners=5pt] (Ay) at (2,-1.6) {$\scriptstyle A(y)$};
\coordinate (Ay2) at (2,0);
\draw[dashed, thick, rounded corners=5pt] (-.3,-1.3) rectangle (1.6,.7);
\draw (Aa) to[in=-90,out=90] (Ax);
\draw (Ab) to[in=-90,out=90] (Ay);
\draw (Ax) to[in=-90,out=90] (mu.-120);
\draw (Ay) to[in=-90,out=90] (Ay2) to[in=-90,out=90] (mu.-60);
\draw[DarkGreen,thick,mid>] (eta0) to[in=-90,out=90] (eta1) to[in=-90,out=90] (eta2);
\draw[double] (mu) to[in=-90,out=90] (Bc);
}
\arrow[Rightarrow,r,"\eta_x"]
&
\tikzmath{
\node[DarkGreen] (eta0) at (0,-2.5) {$\scriptstyle \eta_*$};
\coordinate (eta1) at (1,0);
\coordinate (eta2) at (2,2);
\node[DarkGreen] (eta3) at (2,3) {$\scriptstyle \eta_*$};
\node (Aa) at (1,-2.5) {$\scriptstyle A(a_1)$};
\node (Ab) at (2,-2.5) {$\scriptstyle A(b_1)$};
\node (Bc) at (.5,3) {$\scriptstyle B(a\xz b)$};
\node[draw,rectangle, thick, rounded corners=5pt] (mu) at (.5,2) {$\scriptstyle \mu^B_{a_2,b_2}$};
\node[draw,rectangle, thick, rounded corners=5pt] (Ax) at (0,0) {$\scriptstyle B(x)$};
\node[draw,rectangle, thick, rounded corners=5pt] (Ay) at (2,-1.6) {$\scriptstyle A(y)$};
\coordinate (Ay2) at (2,0);
\draw[dashed, thick, rounded corners=5pt] (-.6,-1.9) rectangle (2.6,1.2);
\draw (Aa) to[in=-90,out=90] (Ax);
\draw (Ab) to[in=-90,out=90] (Ay);
\draw (Ax) to[in=-90,out=90] (mu.-120);
\draw (Ay) to[in=-90,out=90] (Ay2) to[in=-90,out=90] (mu.-60);
\draw[DarkGreen,thick,mid>] (eta0) to[in=-90,out=90] (eta1) to[in=-90,out=90] (eta2) to[in=-90,out=90] (eta3);
\draw[double] (mu) to[in=-90,out=90] (Bc);
}
\arrow[Rightarrow,r, "\phi^{-1}\xo \phi "]
&
\tikzmath{
\node[DarkGreen] (eta0) at (0,-2.5) {$\scriptstyle \eta_*$};
\coordinate (eta2) at (2,1.7);
\node[DarkGreen] (eta3) at (2,3) {$\scriptstyle \eta_*$};
\node (Aa) at (1,-2.5) {$\scriptstyle A(a_1)$};
\node (Ab) at (2,-2.5) {$\scriptstyle A(b_1)$};
\node (Bc) at (.5,3) {$\scriptstyle B(a\xz b)$};
\node[draw,rectangle, thick, rounded corners=5pt] (mu) at (.5,2) {$\scriptstyle \mu^B_{a_2,b_2}$};
\node[draw,rectangle, thick, rounded corners=5pt] (Ax) at (0,1) {$\scriptstyle B(x)$};
\node[draw,rectangle, thick, rounded corners=5pt] (Ay) at (2,-.5) {$\scriptstyle A(y)$};
\coordinate (Ax2) at (0,0);
\draw[dashed, thick, rounded corners=5pt] (.7,-.8) rectangle (2.6,.7);
\draw (Aa) to[in=-90,out=90] (Ax2) to[in=-90,out=90] (Ax);
\draw (Ab) to[in=-90,out=90] (Ay);
\draw (Ax) to[in=-90,out=90] (mu.-120);
\draw (Ay) to[in=-90,out=90] (mu.-60);
\draw[DarkGreen,thick,mid>] (eta0) to[in=-90,out=90] (eta2) to[in=-90,out=90] (eta3);
\draw[double] (mu) to[in=-90,out=90] (Bc);
}
\arrow[Rightarrow,r,"\eta_y"]
&
\tikzmath{
\node[DarkGreen] (eta0) at (0,-2.5) {$\scriptstyle \eta_*$};
\coordinate (eta1) at (2,1);
\node[DarkGreen] (eta3) at (2,3) {$\scriptstyle \eta_*$};
\node (Aa) at (1,-2.5) {$\scriptstyle A(a_1)$};
\node (Ab) at (2,-2.5) {$\scriptstyle A(b_1)$};
\node (Bc) at (.5,3) {$\scriptstyle B(a\xz b)$};
\node[draw,rectangle, thick, rounded corners=5pt] (mu) at (.5,2) {$\scriptstyle \mu^B_{a_2,b_2}$};
\node[draw,rectangle, thick, rounded corners=5pt] (Ax) at (0,1) {$\scriptstyle B(x)$};
\node[draw,rectangle, thick, rounded corners=5pt] (Ay) at (1,.4) {$\scriptstyle B(y)$};
\coordinate (Ax2) at (0,0);
\draw[dashed, thick, rounded corners=5pt] (-.6,0) rectangle (1.6,2.5);
\draw (Aa) to[in=-90,out=90] (Ax2) to[in=-90,out=90] (Ax);
\draw (Ab) to[in=-90,out=90] (Ay);
\draw (Ax) to[in=-90,out=90] (mu.-120);
\draw (Ay) to[in=-90,out=90] (mu.-60);
\draw[DarkGreen,thick,mid>] (eta0) to[in=-90,out=90] (eta1) to[in=-90,out=90] (eta2) to[in=-90,out=90] (eta3);
\draw[double] (mu) to[in=-90,out=90] (Bc);
}
\arrow[Rightarrow,u,"\mu^B_{x,y}"]
\end{tikzcd}
};\end{tikzpicture}
$$
\end{enumerate}

This data is subject to the following additional three coherences c.f.~\cite[Def.~4.16]{MR3076451}
\begin{enumerate}[label=(T-\arabic*)]
\item
\label{Transformation:AssociatorCoherence}
For all $a,b,c\in\cC$, the following diagram commutes:
$$
\begin{tikzpicture}[baseline= (a).base]\node[scale=.8] (a) at (0,0){\begin{tikzcd}
\tikzmath{
\node (a) at (1,0) {$\scriptstyle A(a)$};
\node (b) at (2,0) {$\scriptstyle A(b)$};
\node (c) at (3,0) {$\scriptstyle A(c)$};
\node (d) at (1,4) {$\scriptstyle B(a\xz b \xz c)$};
\node[DarkGreen] (eta1) at (0,0) {$\scriptstyle \eta_*$};
\coordinate (eta2) at (0,2);
\node[DarkGreen] (eta3) at (2,4) {$\scriptstyle \eta_*$};
\node[draw,rectangle, thick, rounded corners=5pt] (mu1) at (1.5,1) {$\scriptstyle \mu^A_{a,b}$};
\node[draw,rectangle, thick, rounded corners=5pt] (mu2) at (2,2) {$\scriptstyle \mu^A_{a\xz b,c}$};
\draw[thick, dashed, red, rounded corners=5pt] (-.2,1.6) rectangle (2.7,3.5);
\draw[thick, dashed, blue, rounded corners=5pt] (1,.5) rectangle (3,2.5);
\draw (a) to[in=-90,out=90] (mu1.-120);
\draw (b) to[in=-90,out=90] (mu1.-60);
\draw[double] (mu1) to[in=-90,out=90] (mu2.-120);
\draw (c) to[in=-90,out=90] (mu2.-60);
\draw[triple] (mu2) to[in=-90,out=90] (d);
\draw[thick, DarkGreen, mid>] (eta1) to[in=-90,out=90] (eta2) to[in=-90,out=90] (eta3);
}
\arrow[Rightarrow,d,blue,"\omega^A_{a,b,c}"]
\arrow[Rightarrow,r,red,"\eta^2_{a\xz b, c}"]
&
\tikzmath{
\node (a) at (1,0) {$\scriptstyle A(a)$};
\node (b) at (2,0) {$\scriptstyle A(b)$};
\node (c) at (3,0) {$\scriptstyle A(c)$};
\coordinate (c2) at (3,1.5);
\node (d) at (1,4) {$\scriptstyle B(a\xz b \xz c)$};
\node[DarkGreen] (eta1) at (0,0) {$\scriptstyle \eta_*$};
\coordinate (eta2) at (0,1);
\coordinate (eta3) at (2,3);
\node[DarkGreen] (eta4) at (2,4) {$\scriptstyle \eta_*$};
\node[draw,rectangle, thick, rounded corners=5pt] (mu1) at (1.5,1) {$\scriptstyle \mu^A_{a,b}$};
\draw (a) to[in=-90,out=90] (mu1.-120);
\draw (b) to[in=-90,out=90] (mu1.-60);
\node[draw,rectangle, thick, rounded corners=5pt] (mu2) at (1,3) {$\scriptstyle \mu^B_{a\xz b,c}$};
\draw[thick, dashed, rounded corners=5pt] (-.2,.5) rectangle (2,2.2);
\draw[double] (mu1) to[in=-90,out=90] (mu2.-120);
\draw (c) to[in=-90,out=90] (c2) to[in=-90,out=90] (mu2.-60);
\draw[triple] (mu2) to[in=-90,out=90] (d);
\draw[thick, DarkGreen, mid>] (eta1) to[in=-90,out=90] (eta2) to[in=-90,out=90] (eta3) to[in=-90,out=90] (eta4);
}
\arrow[Rightarrow,r,"\eta^2_{a,b}"]
&
\tikzmath{
\node (a) at (1,0) {$\scriptstyle A(a)$};
\node (b) at (2,0) {$\scriptstyle A(b)$};
\node (c) at (3,0) {$\scriptstyle A(c)$};
\coordinate (c2) at (3,1.5);
\node (d) at (1,4) {$\scriptstyle B(a\xz b \xz c)$};
\node[DarkGreen] (eta1) at (0,0) {$\scriptstyle \eta_*$};
\coordinate (eta2) at (2,2);
\coordinate (eta3) at (2.5,3);
\node[DarkGreen] (eta4) at (2.5,4) {$\scriptstyle \eta_*$};
\node[draw,rectangle, thick, rounded corners=5pt] (mu1) at (.5,1.5) {$\scriptstyle \mu^B_{a,b}$};
\draw (a) to[in=-90,out=90] (mu1.-120);
\draw (b) to[in=-90,out=90] (mu1.-60);
\node[draw,rectangle, thick, rounded corners=5pt] (mu2) at (1,3) {$\scriptstyle \mu^B_{a\xz b,c}$};
\draw[thick, dashed, rounded corners=5pt] (-.2,1) rectangle (3.2,2.5);
\draw[double] (mu1) to[in=-90,out=90] (mu2.-120);
\draw (c) to[in=-90,out=90] (c2) to[in=-90,out=90] (mu2.-60);
\draw[triple] (mu2) to[in=-90,out=90] (d);
\draw[thick, DarkGreen, mid>] (eta1) to[in=-90,out=90] (eta2) to[in=-90,out=90] (eta3) to[in=-90,out=90] (eta4); 
}
\arrow[Rightarrow,r,"\phi^{-1}"]
&
\tikzmath{
\node (a) at (1,0) {$\scriptstyle A(a)$};
\node (b) at (2,0) {$\scriptstyle A(b)$};
\node (c) at (3,0) {$\scriptstyle A(c)$};
\node (d) at (1,4) {$\scriptstyle B(a\xz b \xz c)$};
\node[DarkGreen] (eta1) at (0,0) {$\scriptstyle \eta_*$};
\coordinate (eta2) at (2.5,2);
\node[DarkGreen] (eta3) at (2.5,4) {$\scriptstyle \eta_*$};
\node[draw,rectangle, thick, rounded corners=5pt] (mu1) at (.5,2) {$\scriptstyle \mu^B_{a,b}$};
\draw (a) to[in=-90,out=90] (mu1.-120);
\draw (b) to[in=-90,out=90] (mu1.-60);
\node[draw,rectangle, thick, rounded corners=5pt] (mu2) at (1,3) {$\scriptstyle \mu^B_{a\xz b,c}$};
\draw[thick, dashed, rounded corners=5pt] (-.2,1.5) rectangle (2,3.5);
\draw[double] (mu1) to[in=-90,out=90] (mu2.-120);
\draw (c) to[in=-90,out=90] (mu2.-60);
\draw[triple] (mu2) to[in=-90,out=90] (d);
\draw[thick, DarkGreen, mid>] (eta1) to[in=-90,out=90] (eta2) to[in=-90,out=90] (eta3); 
}
\arrow[Rightarrow,d,"\omega^B_{a,b,c}"]
\\
\tikzmath{
\node (a) at (1,0) {$\scriptstyle A(a)$};
\node (b) at (2,0) {$\scriptstyle A(b)$};
\node (c) at (3,0) {$\scriptstyle A(c)$};
\node (d) at (1,4) {$\scriptstyle B(a\xz b \xz c)$};
\node[DarkGreen] (eta1) at (0,0) {$\scriptstyle \eta_*$};
\coordinate (eta2) at (0,2);
\node[DarkGreen] (eta3) at (2,4) {$\scriptstyle \eta_*$};
\node[draw,rectangle, thick, rounded corners=5pt] (mu1) at (2.5,1) {$\scriptstyle \mu^A_{b,c}$};
\draw (b) to[in=-90,out=90] (mu1.-120);
\draw (c) to[in=-90,out=90] (mu1.-60);
\node[draw,rectangle, thick, rounded corners=5pt] (mu2) at (2,2) {$\scriptstyle \mu^A_{a, b\xz c}$};
\draw[thick, dashed, rounded corners=5pt] (-.2,1.6) rectangle (2.8,3.5);
\draw[double] (mu1) to[in=-90,out=90] (mu2.-60);
\draw (a) to[in=-90,out=90] (mu2.-120);
\draw[triple] (mu2) to[in=-90,out=90] (d);
\draw[thick, DarkGreen, mid>] (eta1) to[in=-90,out=90] (eta2) to[in=-90,out=90] (eta3);
}
\arrow[Rightarrow,r,"\eta^2_{a,b\xz c}"]
&
\tikzmath{
\node (a) at (1,0) {$\scriptstyle A(a)$};
\node (b) at (2,0) {$\scriptstyle A(b)$};
\node (c) at (3,0) {$\scriptstyle A(c)$};
\node (d) at (1,4) {$\scriptstyle B(a\xz b \xz c)$};
\node[DarkGreen] (eta1) at (0,0) {$\scriptstyle \eta_*$};
\coordinate (eta2) at (0,.5);
\coordinate (eta3) at (2,3);
\node[DarkGreen] (eta4) at (2,4) {$\scriptstyle \eta_*$};
\node[draw,rectangle, thick, rounded corners=5pt] (mu1) at (2.5,1) {$\scriptstyle \mu^A_{b,c}$};
\draw (b) to[in=-90,out=90] (mu1.-120);
\draw (c) to[in=-90,out=90] (mu1.-60);
\node[draw,rectangle, thick, rounded corners=5pt] (mu2) at (1,3) {$\scriptstyle \mu^B_{a, b\xz c}$};
\draw[thick, dashed, rounded corners=5pt] (-.2,.5) rectangle (3,1.9);
\draw[double] (mu1) to[in=-90,out=90] (mu2.-60);
\draw (a) to[in=-90,out=90] (mu2.-120);
\draw[triple] (mu2) to[in=-90,out=90] (d);
\draw[thick, DarkGreen, mid>] (eta1) to[in=-90,out=90] (eta2) to[in=-90,out=90] (eta3) to[in=-90,out=90] (eta4);
}
\arrow[Rightarrow,r,"\phi"]
&
\tikzmath{
\node (a) at (1,-1) {$\scriptstyle A(a)$};
\coordinate (a2) at (0,1);
\node (b) at (2,-1) {$\scriptstyle A(b)$};
\node (c) at (3,-1) {$\scriptstyle A(c)$};
\node (d) at (1,4) {$\scriptstyle B(a\xz b \xz c)$};
\node[DarkGreen] (eta1) at (0,-1) {$\scriptstyle \eta_*$};
\coordinate (eta2) at (1.5,1.5);
\node[DarkGreen] (eta3) at (2.5,4) {$\scriptstyle \eta_*$};
\node[draw,rectangle, thick, rounded corners=5pt] (mu1) at (2.5,1) {$\scriptstyle \mu^A_{b,c}$};
\draw (b) to[in=-90,out=90] (mu1.-120);
\draw (c) to[in=-90,out=90] (mu1.-60);
\node[draw,rectangle, thick, rounded corners=5pt] (mu2) at (1,3) {$\scriptstyle \mu^B_{a, b\xz c}$};
\draw[thick, dashed, rounded corners=5pt] (1,.5) rectangle (3,2.5);
\draw[double] (mu1) to[in=-90,out=90] (mu2.-60);
\draw (a) to[in=-90,out=90] (a2) to[in=-90,out=90] (mu2.-120);
\draw[triple] (mu2) to[in=-90,out=90] (d);
\draw[thick, DarkGreen, mid>] (eta1) to[in=-90,out=90] (eta2) to[in=-90,out=90] (eta3);
}
\arrow[Rightarrow,r,"\eta^2_{b,c}"]
&
\tikzmath{
\node (a) at (1,0) {$\scriptstyle A(a)$};
\node (b) at (2,0) {$\scriptstyle A(b)$};
\node (c) at (3,0) {$\scriptstyle A(c)$};
\node (d) at (1,4) {$\scriptstyle B(a\xz b \xz c)$};
\node[DarkGreen] (eta1) at (0,0) {$\scriptstyle \eta_*$};
\coordinate (eta2) at (0,.2);
\coordinate (eta3) at (2.5,2.3);
\node[DarkGreen] (eta4) at (2.5,4) {$\scriptstyle \eta_*$};
\node[draw,rectangle, thick, rounded corners=5pt] (mu1) at (1.5,2) {$\scriptstyle \mu^B_{b,c}$};
\draw (b) to[in=-90,out=90] (mu1.-120);
\draw (c) to[in=-90,out=90] (mu1.-60);
\node[draw,rectangle, thick, rounded corners=5pt] (mu2) at (1,3) {$\scriptstyle \mu^B_{a, b\xz c}$};
\draw[double] (mu1) to[in=-90,out=90] (mu2.-60);
\draw (a) to[in=-90,out=90] (mu2.-120);
\draw[triple] (mu2) to[in=-90,out=90] (d);
\draw[thick, DarkGreen, mid>] (eta1) to[in=-90,out=90] (eta2) to[in=-90,out=90] (eta3) to[in=-90,out=90] (eta4);
}
\end{tikzcd}
};\end{tikzpicture}
$$
\item
\label{Transformation:UnitCoherence}
For all $c\in\cC$, the following diagram commutes:
$$
\begin{tikzpicture}[baseline= (a).base]\node[scale=.8] (a) at (0,0){\begin{tikzcd}
\tikzmath{
\node[DarkGreen] (eta1) at (0,-2) {$\scriptstyle \eta_*$};
\node (Ac) at (2,-2) {$\scriptstyle A(c)$};
\coordinate (Ac2) at (2,-1);
\coordinate (eta2) at (0,0);
\node (Bc) at (.5,2) {$\scriptstyle B(c)$};
\node[DarkGreen] (eta3) at (1.5,2) {$\scriptstyle \eta_*$};
\node[draw,rectangle, thick, rounded corners=5pt] (mu) at (1.5,0) {$\scriptstyle \mu^A_{1_\cC,c}$};
\node[draw,rectangle, thick, rounded corners=5pt] (iota) at (1,-1) {$\scriptstyle \iota^A_*$};
\draw[dashed, thick, red, rounded corners=5pt] (-.3,-.5) rectangle (2.2,1.5);
\draw[dashed, thick, blue, rounded corners=5pt] (-.3,-1.4) rectangle (2.2,.5);
\draw (iota) to[in=-90,out=90] (mu.-120);
\draw (Ac) to[in=-90,out=90] (Ac2) to[in=-90,out=90] (mu.-60);
\draw[DarkGreen,thick,mid>] (eta1) to[in=-90,out=90] (eta2) to[in=-90,out=90] (eta3);
\draw[double] (mu) to[in=-90,out=90] (Bc);
}
\arrow[Rightarrow,d,blue,"\ell^A_c"]
\arrow[Rightarrow,r,red,"\eta^2_{1_\cC,c}"]
&
\tikzmath{
\node[DarkGreen] (eta1) at (0,-2) {$\scriptstyle \eta_*$};
\node (Ac) at (2,-2) {$\scriptstyle A(c)$};
\coordinate (Ac2) at (2,-.75);
\coordinate (eta2) at (1.5,1);
\node (Bc) at (.5,2) {$\scriptstyle B(a\xz b)$};
\node[DarkGreen] (eta3) at (1.5,2) {$\scriptstyle \eta_*$};
\node[draw,rectangle, thick, rounded corners=5pt] (mu) at (.5,1) {$\scriptstyle \mu^B_{1_\cC,c}$};
\node[draw,rectangle, thick, rounded corners=5pt] (iota) at (1,-1) {$\scriptstyle \iota^A_*$};
\draw[dashed, thick, rounded corners=5pt] (-.2,-1.4) rectangle (1.4,0);
\draw (iota) to[in=-90,out=90] (mu.-120);
\draw (Ac) to[in=-90,out=90] (Ac2) to[in=-90,out=90] (mu.-60);
\draw[DarkGreen,thick,mid>] (eta1) to[in=-90,out=90] (eta2) to[in=-90,out=90] (eta3);
\draw[double] (mu) to[in=-90,out=90] (Bc);
}
\arrow[Rightarrow,r,"\eta^1"]
&
\tikzmath{
\node[DarkGreen] (eta1) at (0,-2) {$\scriptstyle \eta_*$};
\node (Ac) at (2,-2) {$\scriptstyle A(c)$};
\coordinate (Ac2) at (2,-.75);
\coordinate (eta2) at (1.5,1);
\node (Bc) at (.5,2) {$\scriptstyle B(a\xz b)$};
\node[DarkGreen] (eta3) at (1.5,2) {$\scriptstyle \eta_*$};
\node[draw,rectangle, thick, rounded corners=5pt] (mu) at (.5,1) {$\scriptstyle \mu^B_{1_\cC,c}$};
\node[draw,rectangle, thick, rounded corners=5pt] (iota) at (0,-.5) {$\scriptstyle \iota^B_*$};
\draw[dashed, thick, rounded corners=5pt] (-.5,-.9) rectangle (2.2,.3);
\draw (iota) to[in=-90,out=90] (mu.-120);
\draw (Ac) to[in=-90,out=90] (Ac2) to[in=-90,out=90] (mu.-60);
\draw[DarkGreen,thick,mid>] (eta1) to[in=-90,out=90] (eta2) to[in=-90,out=90] (eta3);
\draw[double] (mu) to[in=-90,out=90] (Bc);
}
\arrow[Rightarrow,dl,"\phi^{-1}"]
\\
\tikzmath{
\node[DarkGreen] (eta1) at (0,0) {$\scriptstyle \eta_*$};
\node (Ac) at (1,0) {$\scriptstyle A(c)$};
\node[DarkGreen] (eta2) at (1,2) {$\scriptstyle \eta_*$};
\node (Bc) at (0,2) {$\scriptstyle B(c)$};
\draw[DarkGreen,thick,mid>] (eta1) to[in=-90,out=90] (eta2);
\draw (Ac) to[in=-90,out=90] (Bc);
}
&
\tikzmath{
\node[DarkGreen] (eta1) at (0,-2) {$\scriptstyle \eta_*$};
\node (Ac) at (2,-2) {$\scriptstyle A(c)$};
\coordinate (eta2) at (1.5,.75);
\node (Bc) at (.5,2) {$\scriptstyle B(a\xz b)$};
\node[DarkGreen] (eta3) at (1.5,2) {$\scriptstyle \eta_*$};
\node[draw,rectangle, thick, rounded corners=5pt] (mu) at (.5,1) {$\scriptstyle \mu^B_{1_\cC,c}$};
\node[draw,rectangle, thick, rounded corners=5pt] (iota) at (0,0) {$\scriptstyle \iota^B_*$};
\draw[dashed, thick, rounded corners=5pt] (-.5,-.3) rectangle (1.2,1.5);
\draw (iota) to[in=-90,out=90] (mu.-120);
\draw (Ac) to[in=-90,out=90] (mu.-60);
\draw[DarkGreen,thick,mid>] (eta1) to[in=-90,out=90] (eta2) to[in=-90,out=90] (eta3);
\draw[double] (mu) to[in=-90,out=90] (Bc);
}
\arrow[Rightarrow,l,"\ell^B_c"]
\end{tikzcd}
};\end{tikzpicture}
$$
and a similar coherence equation holds for $r^A_c$ as well.
\end{enumerate}
Observe that \ref{Transformation:UnitCoherence}
completely determines $\eta^1$ in terms of lower data.
This means that one needs only verify the existence of some $\eta^1$ satisfying \ref{Transformation:UnitCoherence} to verify \ref{Transformation:eta^1} above.
\end{defn}

%%%%%%%%%%%%%%%%%%%%%%%%%%%%%%%%%%%%%%%%%%%%
\subsection{Modifications between transformations of $\Gray$-monoids}
\begin{defn}
\label{defn:Modification}
Suppose $\cC,\cD$ are $\Gray$-monoids, $A,B: \rmB\cC \to \rmB\cD$ are 3-functors, and $\eta,\zeta: A \Rightarrow B$ are transformations.
A modification $m: \eta \Rrightarrow \zeta$ consists of:
\begin{enumerate}[label=(M-\Roman*)]
\item
\label{Modification:m_*}
a 1-cell $m_*: \eta_* \to \zeta_*$ depicted
$
\tikzmath{
\node[DarkGreen] (eta) at (0,0) {$\scriptstyle \eta_*$};
\node[orange] (zeta) at (0,2) {$\scriptstyle \zeta_*$};
\node[draw,rectangle, thick, rounded corners=5pt] (m) at (0,1) {$\scriptstyle m_*$};
\draw[DarkGreen, thick] (eta) to[in=-90,out=90] (m);
\draw[orange, thick] (m) to[in=-90,out=90] (zeta);
}
$
\item
\label{Modification:m_c}
For each $c\in \cC$, an invertible 2-modification
$$
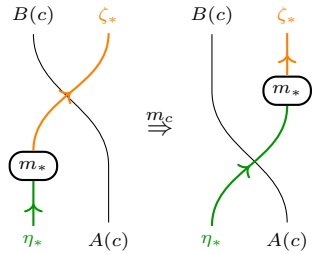
\tikzmath{
\node[DarkGreen] (eta) at (0,0) {$\scriptstyle \eta_*$};
\node (Ac1) at (1,0) {$\scriptstyle A(c)$};
\coordinate (Ac2) at (1,1);
\node (Bc) at (0,3) {$\scriptstyle B(c)$};
\node[orange] (zeta) at (1,3) {$\scriptstyle \zeta_*$};
\node[draw,rectangle, thick, rounded corners=5pt] (m) at (0,1) {$\scriptstyle m_*$};
\draw (Ac1) to[in=-90,out=90] (Ac2) to[in=-90,out=90] (Bc);
\draw[DarkGreen,thick,mid>] (eta) to[in=-90,out=90] (m);
\draw[orange, thick, mid>] (m) to[in=-90,out=90] (zeta);
}
\overset{m_c}{\Rightarrow}
\tikzmath{
\node[DarkGreen] (eta) at (0,0) {$\scriptstyle \eta_*$};
\node (Ac) at (1,0) {$\scriptstyle A(c)$};
\coordinate (Bc1) at (0,2);
\node (Bc2) at (0,3) {$\scriptstyle B(c)$};
\node[orange] (zeta) at (1,3) {$\scriptstyle \zeta_*$};
\node[draw,rectangle, thick, rounded corners=5pt] (m) at (1,2) {$\scriptstyle m_*$};
\draw (Ac) to[in=-90,out=90] (Bc1) to[in=-90,out=90] (Bc2);
\draw[DarkGreen,thick,mid>] (eta) to[in=-90,out=90] (m);
\draw[orange, thick, mid>] (m) to[in=-90,out=90] (zeta);
}
$$
Explicitly, this means that $m_c$ satisfies the following coherence condition for all $x:a\to b$:
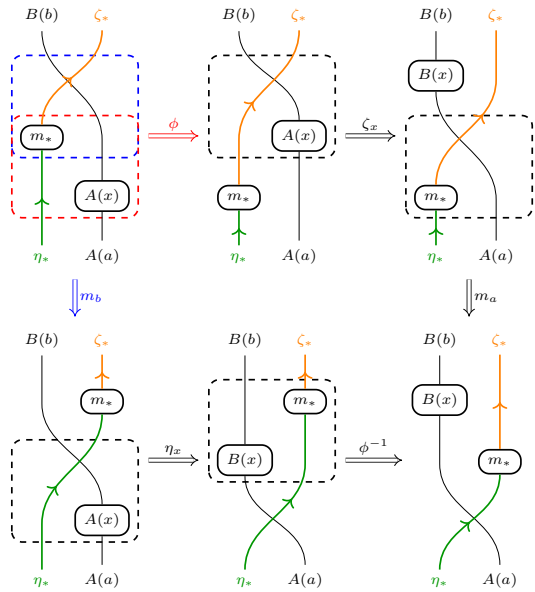
$$
\begin{tikzpicture}[baseline= (a).base]\node[scale=.8] (a) at (0,0){\begin{tikzcd}
\tikzmath{
\node[DarkGreen] (eta1) at (0,-1) {$\scriptstyle \eta_*$};
\coordinate (eta2) at (0,0);
\node (g) at (1,-1) {$\scriptstyle A(a)$};
\coordinate (h1) at (1,1);
\node (h2) at (0,3) {$\scriptstyle B(b)$};
\node[orange] (zeta) at (1,3) {$\scriptstyle \zeta_*$};
\node[draw,rectangle, thick, rounded corners=5pt] (Ax) at (1,0) {$\scriptstyle A(x)$};
\node[draw,rectangle, thick, rounded corners=5pt] (m) at (0,1) {$\scriptstyle m_*$};
\draw[thick, dashed, blue, rounded corners=5pt] (-.5,.7) rectangle (1.6,2.4);
\draw[thick, dashed, red, rounded corners=5pt] (-.5,-.3) rectangle (1.6,1.4);
\draw (g) to[in=-90,out=90] (Ax);
\draw (Ax) to[in=-90,out=90] (h1) to[in=-90,out=90] (h2);
\draw[DarkGreen,thick,mid>] (eta1) to[in=-90,out=90] (eta2) to[in=-90,out=90] (m);
\draw[orange,thick, mid>] (m) to[in=-90,out=90] (zeta);
}
\arrow[d, blue, Rightarrow, "m_b"]
\arrow[r, red, Rightarrow, "\phi"]
&
\tikzmath{
\node[DarkGreen] (eta) at (0,-1) {$\scriptstyle \eta_*$};
\node (g) at (1,-1) {$\scriptstyle A(a)$};
\node (h) at (0,3) {$\scriptstyle B(b)$};
\coordinate (zeta1) at (0,1);
\node[orange] (zeta2) at (1,3) {$\scriptstyle \zeta_*$};
\node[draw,rectangle, thick, rounded corners=5pt] (Ax) at (1,1) {$\scriptstyle A(x)$};
\node[draw,rectangle, thick, rounded corners=5pt] (m) at (0,0) {$\scriptstyle m_*$};
\draw[thick, dashed, rounded corners=5pt] (-.5,.7) rectangle (1.6,2.4);
\draw (g) to[in=-90,out=90] (Ax);
\draw (Ax) to[in=-90,out=90] (h);
\draw[DarkGreen,thick,mid>] (eta) to[in=-90,out=90] (m);
\draw[orange,thick, mid>] (m) to[in=-90,out=90] (zeta1) to[in=-90,out=90] (zeta2);
}
\arrow[r, Rightarrow, "\zeta_x"]
&
\tikzmath{
\node[DarkGreen] (eta) at (0,-1) {$\scriptstyle \eta_*$};
\node (g1) at (1,-1) {$\scriptstyle A(a)$};
\coordinate (g2) at (1,0);
\node (h) at (0,3) {$\scriptstyle B(b)$};
\coordinate (zeta1) at (1,2);
\node[orange] (zeta2) at (1,3) {$\scriptstyle \zeta_*$};
\node[draw,rectangle, thick, rounded corners=5pt] (Ax) at (0,2) {$\scriptstyle B(x)$};
\node[draw,rectangle, thick, rounded corners=5pt] (m) at (0,0) {$\scriptstyle m_*$};
\draw[thick, dashed, rounded corners=5pt] (-.5,-.3) rectangle (1.6,1.4);
\draw (g1) to[in=-90,out=90] (g2) to[in=-90,out=90] (Ax);
\draw (Ax) to[in=-90,out=90] (h);
\draw[DarkGreen,thick,mid>] (eta) to[in=-90,out=90] (m);
\draw[orange,thick, mid>] (m) to[in=-90,out=90] (zeta1) to[in=-90,out=90] (zeta2);
}
\arrow[d, Rightarrow, "m_a"]
\\
\tikzmath{
\node[DarkGreen] (eta1) at (0,-1) {$\scriptstyle \eta_*$};
\coordinate (eta2) at (0,0);
\node (g) at (1,-1) {$\scriptstyle A(a)$};
\coordinate (h1) at (0,2);
\node (h2) at (0,3) {$\scriptstyle B(b)$};
\node[orange] (zeta) at (1,3) {$\scriptstyle \zeta_*$};
\node[draw,rectangle, thick, rounded corners=5pt] (Ax) at (1,0) {$\scriptstyle A(x)$};
\node[draw,rectangle, thick, rounded corners=5pt] (m) at (1,2) {$\scriptstyle m_*$};
\draw[thick, dashed, rounded corners=5pt] (-.5,-.3) rectangle (1.6,1.4);
\draw (g) to[in=-90,out=90] (Ax);
\draw (Ax) to[in=-90,out=90] (h1) to[in=-90,out=90] (h2);
\draw[DarkGreen,thick,mid>] (eta1) to[in=-90,out=90] (eta2) to[in=-90,out=90] (m);
\draw[orange,thick, mid>] (m) to[in=-90,out=90] (zeta);
}
\arrow[r, Rightarrow, "\eta_x"]
&
\tikzmath{
\node[DarkGreen] (eta1) at (0,-1) {$\scriptstyle \eta_*$};
\coordinate (eta2) at (1,1);
\node (g) at (1,-1) {$\scriptstyle A(a)$};
\coordinate (h1) at (0,2);
\node (h2) at (0,3) {$\scriptstyle B(b)$};
\node[orange] (zeta) at (1,3) {$\scriptstyle \zeta_*$};
\node[draw,rectangle, thick, rounded corners=5pt] (Ax) at (0,1) {$\scriptstyle B(x)$};
\node[draw,rectangle, thick, rounded corners=5pt] (m) at (1,2) {$\scriptstyle m_*$};
\draw[thick, dashed, rounded corners=5pt] (-.6,.7) rectangle (1.5,2.4);
\draw (g) to[in=-90,out=90] (Ax);
\draw (Ax) to[in=-90,out=90] (h1) to[in=-90,out=90] (h2);
\draw[DarkGreen,thick,mid>] (eta1) to[in=-90,out=90] (eta2) to[in=-90,out=90] (m);
\draw[orange,thick, mid>] (m) to[in=-90,out=90] (zeta);
}
\arrow[r, Rightarrow, "\phi^{-1}"]
&
\tikzmath{
\node[DarkGreen] (eta) at (0,-1) {$\scriptstyle \eta_*$};
\node (g1) at (1,-1) {$\scriptstyle A(a)$};
\coordinate (g2) at (0,1);
\node (h) at (0,3) {$\scriptstyle B(b)$};
\node[orange] (zeta) at (1,3) {$\scriptstyle \zeta_*$};
\node[draw,rectangle, thick, rounded corners=5pt] (Ax) at (0,2) {$\scriptstyle B(x)$};
\node[draw,rectangle, thick, rounded corners=5pt] (m) at (1,1) {$\scriptstyle m_*$};
\draw (g1) to[in=-90,out=90] (g2) to[in=-90,out=90] (Ax);
\draw (Ax) to[in=-90,out=90] (h);
\draw[DarkGreen,thick,mid>] (eta) to[in=-90,out=90] (m);
\draw[orange,thick, mid>] (m) to[in=-90,out=90] (zeta);
}
\end{tikzcd}
};\end{tikzpicture}
$$
\end{enumerate}
This data is subject to the following two conditions c.f.~\cite[Def.~4.18]{MR3076451}:
\begin{enumerate}[label=(M-\arabic*)]
\item
\label{Modification:MonoidalCoherence}
For all $a,b\in\cC$,
$$
\begin{tikzpicture}[baseline= (a).base]\node[scale=.8] (a) at (0,0){\begin{tikzcd}
\tikzmath{
\node[DarkGreen] (eta) at (0,-2) {$\scriptstyle \eta_*$};
\node[orange] (zeta) at (1.5,2) {$\scriptstyle \zeta_*$};
\node (Aa) at (1,-2) {$\scriptstyle A(a)$};
\node (Ab) at (2,-2) {$\scriptstyle A(b)$};
\node (Bab) at (0,2) {$\scriptstyle B(a\xz b)$};
\node[draw,rectangle, thick, rounded corners=5pt] (mu) at (1.5,-.8) {$\scriptstyle \mu^A_{a,b}$};
\node[draw,rectangle, thick, rounded corners=5pt] (m) at (0,0) {$\scriptstyle m_*$};
\coordinate (mu2) at (1.5,0);
\draw[dashed, red, thick, rounded corners=5pt] (-.5,-.3) rectangle (1.7,1.5);
\draw[dashed, blue, thick, rounded corners=5pt] (-.5,-1.5) rectangle (2.1,.5);
\draw (Aa) to[in=-90,out=90] (mu.-120);
\draw (Ab) to[in=-90,out=90] (mu.-60);
\draw[DarkGreen,thick,mid>] (eta) to[in=-90,out=90] (m);
\draw[orange,thick,mid>] (m) to[in=-90,out=90] (zeta);
\draw[double] (mu) to[in=-90,out=90] (mu2) to[in=-90,out=90] (Bab);
}
\arrow[Rightarrow,d,blue,"\phi"]
\arrow[Rightarrow,r,red,"m_{a\xz b}"]
&
\tikzmath{
\node[DarkGreen] (eta) at (0,-2) {$\scriptstyle \eta_*$};
\coordinate (eta1) at (0,-1);
\node[orange] (zeta) at (1.5,2) {$\scriptstyle \zeta_*$};
\node (Aa) at (1,-2) {$\scriptstyle A(a)$};
\node (Ab) at (2,-2) {$\scriptstyle A(b)$};
\node (Bab) at (0,2) {$\scriptstyle B(a\xz b)$};
\node[draw,rectangle, thick, rounded corners=5pt] (mu) at (1.5,-.8) {$\scriptstyle \mu^A_{a,b}$};
\node[draw,rectangle, thick, rounded corners=5pt] (m) at (1.5,1) {$\scriptstyle m_*$};
\draw[dashed, thick, rounded corners=5pt] (-.5,-1.5) rectangle (2.1,.5);
\draw (Aa) to[in=-90,out=90] (mu.-120);
\draw (Ab) to[in=-90,out=90] (mu.-60);
\draw[DarkGreen,thick,mid>] (eta) to[in=-90,out=90] (eta1) to[in=-90,out=90] (m);
\draw[orange,thick,mid>] (m) to[in=-90,out=90] (zeta);
\draw[double] (mu) to[in=-90,out=90] (Bab);
}
\arrow[Rightarrow,r,"\eta^2_{a,b}"]
&
\tikzmath{
\node[DarkGreen] (eta) at (0,-2) {$\scriptstyle \eta_*$};
\coordinate (zeta1) at (1.5,1.5);
\node[orange] (zeta2) at (1.5,2) {$\scriptstyle \zeta_*$};
\node (Aa) at (1,-2) {$\scriptstyle A(a)$};
\node (Ab) at (2,-2) {$\scriptstyle A(b)$};
\node (Bab) at (.5,2) {$\scriptstyle B(a\xz b)$};
\node[draw,rectangle, thick, rounded corners=5pt] (mu) at (.5,.3) {$\scriptstyle \mu^B_{a,b}$};
\node[draw,rectangle, thick, rounded corners=5pt] (m) at (1.5,1) {$\scriptstyle m_*$};
\draw[dashed, thick, rounded corners=5pt] (-.1,-.1) rectangle (2,1.4);
\draw (Aa) to[in=-90,out=90] (mu.-120);
\draw (Ab) to[in=-90,out=90] (mu.-60);
\draw[DarkGreen,thick,mid>] (eta) to[in=-90,out=90] (m);
\draw[orange,thick,mid>] (m) to[in=-90,out=90] (zeta1) to[in=-90,out=90] (zeta2);
\draw[double] (mu) to[in=-90,out=90] (Bab);
}
\arrow[dr, Rightarrow, "\phi^{-1}"]
\\
\tikzmath{
\node[DarkGreen] (eta) at (0,-2) {$\scriptstyle \eta_*$};
\coordinate (zeta1) at (0,0);
\node[orange] (zeta2) at (1.5,2) {$\scriptstyle \zeta_*$};
\node (Aa) at (1,-2) {$\scriptstyle A(a)$};
\node (Ab) at (2,-2) {$\scriptstyle A(b)$};
\node (Bab) at (0,2) {$\scriptstyle B(a\xz b)$};
\node[draw,rectangle, thick, rounded corners=5pt] (mu) at (1.5,0) {$\scriptstyle \mu^A_{a,b}$};
\node[draw,rectangle, thick, rounded corners=5pt] (m) at (0,-.8) {$\scriptstyle m_*$};
\draw[dashed, thick, rounded corners=5pt] (-.5,-.4) rectangle (2.1,1.5);
\draw (Aa) to[in=-90,out=90] (mu.-120);
\draw (Ab) to[in=-90,out=90] (mu.-60);
\draw[DarkGreen,thick,mid>] (eta) to[in=-90,out=90] (m);
\draw[orange,thick,mid>] (m) to[in=-90,out=90] (zeta1) to[in=-90,out=90] (zeta2);
\draw[double] (mu) to[in=-90,out=90] (Bab);
}
\arrow[Rightarrow,r,"\zeta^2_{a,b}"]
&
\tikzmath{
\node[DarkGreen] (eta) at (0,-2) {$\scriptstyle \eta_*$};
\coordinate (zeta1) at (1.5,1.5);
\node[orange] (zeta2) at (1.5,2) {$\scriptstyle \zeta_*$};
\node (Aa) at (1,-2) {$\scriptstyle A(a)$};
\node (Ab) at (2,-2) {$\scriptstyle A(b)$};
\node (Bab) at (.5,2) {$\scriptstyle B(a\xz b)$};
\node[draw,rectangle, thick, rounded corners=5pt] (mu) at (.5,1) {$\scriptstyle \mu^B_{a,b}$};
\node[draw,rectangle, thick, rounded corners=5pt] (m) at (0,-.8) {$\scriptstyle m_*$};
\draw[dashed, thick, rounded corners=5pt] (-.5,-1.2) rectangle (.8,.5);
\draw (Aa) to[in=-90,out=90] (mu.-120);
\draw (Ab) to[in=-90,out=90] (mu.-60);
\draw[DarkGreen,thick,mid>] (eta) to[in=-90,out=90] (m);
\draw[orange,thick,mid>] (m) to[in=-90,out=90] (zeta1) to[in=-90,out=90] (zeta2);
\draw[double] (mu) to[in=-90,out=90] (Bab);
}
\arrow[Rightarrow,r,"m_a"]
&
\tikzmath{
\node[DarkGreen] (eta) at (0,-1.5) {$\scriptstyle \eta_*$};
\node[orange] (zeta) at (1.5,2.5) {$\scriptstyle \zeta_*$};
\node (Aa) at (1,-1.5) {$\scriptstyle A(a)$};
\coordinate (Aa1) at (0,0);
\node (Ab) at (2,-1.5) {$\scriptstyle A(b)$};
\coordinate (Ab1) at (2,0);
\node (Bab) at (.5,2.5) {$\scriptstyle B(a\xz b)$};
\node[draw,rectangle, thick, rounded corners=5pt] (mu) at (.5,1.5) {$\scriptstyle \mu^B_{a,b}$};
\node[draw,rectangle, thick, rounded corners=5pt] (m) at (1,0) {$\scriptstyle m_*$};
\draw[dashed, thick, rounded corners=5pt] (.5,-.5) rectangle (2.2,1);
\draw (Aa) to[in=-90,out=90] (Aa1) to[in=-90,out=90] (mu.-120);
\draw (Ab) to[in=-90,out=90] (Ab1) to[in=-90,out=90] (mu.-60);
\draw[DarkGreen,thick,mid>] (eta) to[in=-90,out=90] (m);
\draw[orange,thick,mid>] (m) to[in=-90,out=90] (zeta);
\draw[double] (mu) to[in=-90,out=90] (Bab);
}
\arrow[Rightarrow,r,"m_b"]
&
\tikzmath{
\node[DarkGreen] (eta) at (0,-2) {$\scriptstyle \eta_*$};
\coordinate (zeta1) at (1.5,1.5);
\node[orange] (zeta2) at (1.5,2) {$\scriptstyle \zeta_*$};
\node (Aa) at (1,-2) {$\scriptstyle A(a)$};
\node (Ab) at (2,-2) {$\scriptstyle A(b)$};
\node (Bab) at (.5,2) {$\scriptstyle B(a\xz b)$};
\node[draw,rectangle, thick, rounded corners=5pt] (mu) at (.5,1) {$\scriptstyle \mu^B_{a,b}$};
\node[draw,rectangle, thick, rounded corners=5pt] (m) at (1.5,.3) {$\scriptstyle m_*$};
\draw (Aa) to[in=-90,out=90] (mu.-120);
\draw (Ab) to[in=-90,out=90] (mu.-60);
\draw[DarkGreen,thick,mid>] (eta) to[in=-90,out=90] (m);
\draw[orange,thick,mid>] (m) to[in=-90,out=90] (zeta1) to[in=-90,out=90] (zeta2);
\draw[double] (mu) to[in=-90,out=90] (Bab);
}
\end{tikzcd}
};\end{tikzpicture}
$$
\item
\label{Modification:UnitCoherence}
The following diagram commutes:
$$
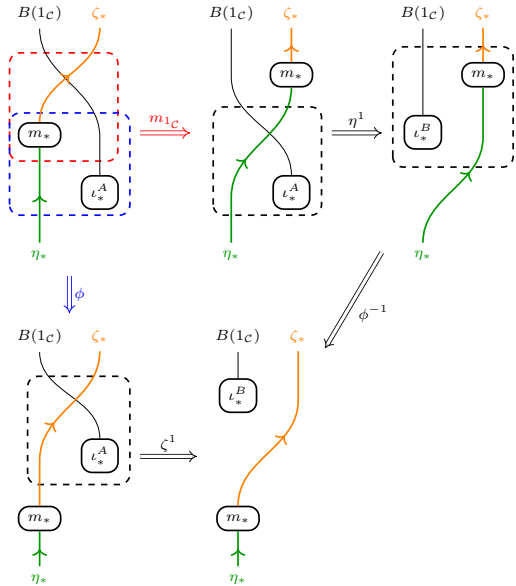
\begin{tikzpicture}[baseline= (a).base]\node[scale=.8] (a) at (0,0){\begin{tikzcd}
\tikzmath{
\node[DarkGreen] (eta) at (0,-1) {$\scriptstyle \eta_*$};
\node[orange] (zeta) at (1,3) {$\scriptstyle \zeta_*$};
\coordinate (A2) at (1,1);
\node (B1) at (0,3) {$\scriptstyle B(1_\cC)$};
\node[draw,rectangle, thick, rounded corners=5pt] (m) at (0,1) {$\scriptstyle m_*$};
\node[draw,rectangle, thick, rounded corners=5pt] (iota) at (1,0) {$\scriptstyle \iota^A_*$};
\draw[dashed, thick, red, rounded corners=5pt] (-.5,.6) rectangle (1.3,2.4);
\draw[dashed, thick, blue, rounded corners=5pt] (-.5,-.3) rectangle (1.5,1.4);
\draw (iota) to[in=-90,out=90] (A2) to[in=-90,out=90] (B1);
\draw[DarkGreen,thick,mid>] (eta) to[in=-90,out=90] (m);
\draw[orange, thick, mid>] (m) to[in=-90,out=90] (zeta);
}
\arrow[Rightarrow,d,blue,"\phi"]
\arrow[Rightarrow,r,red,"m_{1_\cC}"]
&
\tikzmath{
\node[DarkGreen] (eta1) at (0,-1) {$\scriptstyle \eta_*$};
\coordinate (eta2) at (0,0);
\node[orange] (zeta) at (1,3) {$\scriptstyle \zeta_*$};
\coordinate (A2) at (0,2);
\node (B1) at (0,3) {$\scriptstyle B(1_\cC)$};
\node[draw,rectangle, thick, rounded corners=5pt] (m) at (1,2) {$\scriptstyle m_*$};
\node[draw,rectangle, thick, rounded corners=5pt] (iota) at (1,0) {$\scriptstyle \iota^A_*$};
\draw[dashed, thick, rounded corners=5pt] (-.3,-.3) rectangle (1.5,1.5);
\draw (iota) to[in=-90,out=90] (A2) to[in=-90,out=90] (B1);
\draw[DarkGreen,thick,mid>] (eta1) to[in=-90,out=90] (eta2) to[in=-90,out=90] (m);
\draw[orange, thick, mid>] (m) to[in=-90,out=90] (zeta);
}
\arrow[Rightarrow,r,"\eta^1"]
&
\tikzmath{
\node[DarkGreen] (eta1) at (0,-1) {$\scriptstyle \eta_*$};
\coordinate (eta2) at (1,1);
\node[orange] (zeta) at (1,3) {$\scriptstyle \zeta_*$};
\node (B1) at (0,3) {$\scriptstyle B(1_\cC)$};
\node[draw,rectangle, thick, rounded corners=5pt] (m) at (1,2) {$\scriptstyle m_*$};
\node[draw,rectangle, thick, rounded corners=5pt] (iota) at (0,1) {$\scriptstyle \iota^B_*$};
\draw[dashed, thick, rounded corners=5pt] (-.5,.5) rectangle (1.5,2.5);
\draw (iota) to[in=-90,out=90] (B1);
\draw[DarkGreen,thick,mid>] (eta1) to[in=-90,out=90] (eta2) to[in=-90,out=90] (m);
\draw[orange, thick, mid>] (m) to[in=-90,out=90] (zeta);
}
\arrow[dl,Rightarrow,"\phi^{-1}"]
\\
\tikzmath{
\node[DarkGreen] (eta) at (0,-1) {$\scriptstyle \eta_*$};
\coordinate (zeta1) at (0,1);
\node[orange] (zeta2) at (1,3) {$\scriptstyle \zeta_*$};
\node (B1) at (0,3) {$\scriptstyle B(1_\cC)$};
\node[draw,rectangle, thick, rounded corners=5pt] (m) at (0,0) {$\scriptstyle m_*$};
\node[draw,rectangle, thick, rounded corners=5pt] (iota) at (1,1) {$\scriptstyle \iota^A_*$};
\draw[dashed, thick, rounded corners=5pt] (-.2,.6) rectangle (1.5,2.4);
\draw (iota) to[in=-90,out=90] (B1);
\draw[DarkGreen,thick,mid>] (eta) to[in=-90,out=90] (m);
\draw[orange, thick, mid>] (m) to[in=-90,out=90] (zeta1) to[in=-90,out=90] (zeta2);
}
\arrow[Rightarrow,r,"\zeta^1"]
&
\tikzmath{
\node[DarkGreen] (eta) at (0,-1) {$\scriptstyle \eta_*$};
\coordinate (zeta1) at (1,2);
\node[orange] (zeta2) at (1,3) {$\scriptstyle \zeta_*$};
\node (A2) at (1,0) {};
\node (B1) at (0,3) {$\scriptstyle B(1_\cC)$};
\node[draw,rectangle, thick, rounded corners=5pt] (m) at (0,0) {$\scriptstyle m_*$};
\node[draw,rectangle, thick, rounded corners=5pt] (iota) at (0,2) {$\scriptstyle \iota^B_*$};
\draw (iota) to[in=-90,out=90] (B1);
\draw[DarkGreen,thick,mid>] (eta) to[in=-90,out=90] (m);
\draw[orange, thick, mid>] (m) to[in=-90,out=90] (zeta1) to[in=-90,out=90] (zeta2);
}
\end{tikzcd}
};\end{tikzpicture}
$$
Observe this coherence completely determines $m_{1_\cC}$ in terms of $\eta^1$ and $\zeta^1$.
\end{enumerate}
\end{defn}

%%%%%%%%%%%%%%%%%%%%%%%%%%%%%%%%%%%%%%%%%%%%
\subsection{Perturbations between modifications of $\Gray$-monoids}
\begin{defn}
\label{defn:Perturbation}
Suppose $\cC,\cD$ are $\Gray$-monoids, $A,B: \rmB\cC \to \rmB\cD$ are 3-functors, $\eta,\zeta: A \Rightarrow B$ are transformations, and $m,n: \eta \Rrightarrow \zeta$ are modifications.
A perturbation $\rho: m \RRightarrow n$ consists of a 2-cell $\rho_* : m_* \Rightarrow n_*$
satisfying the following coherence condition c.f.~\cite[Def.~4.21]{MR3076451}:
\begin{enumerate}[label=(P-\arabic*)]
\item
\label{Petrubtation:Condition}
For each $c\in \cC$, the following square commutes:
$$
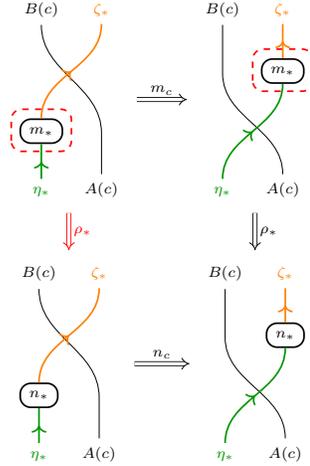
\begin{tikzpicture}[baseline= (a).base]\node[scale=.8] (a) at (0,0){\begin{tikzcd}
\tikzmath{
\node[DarkGreen] (eta) at (0,0) {$\scriptstyle \eta_*$};
\node (Ac1) at (1,0) {$\scriptstyle A(c)$};
\coordinate (Ac2) at (1,1);
\node (Bc) at (0,3) {$\scriptstyle B(c)$};
\node[orange] (zeta) at (1,3) {$\scriptstyle \zeta_*$};
\node[draw,rectangle, thick, rounded corners=5pt] (m) at (0,1) {$\scriptstyle m_*$};
\draw[thick, dashed, red, rounded corners=5pt] (-.5,.7) rectangle (.5,1.4);
\draw (Ac1) to[in=-90,out=90] (Ac2) to[in=-90,out=90] (Bc);
\draw[DarkGreen,thick,mid>] (eta) to[in=-90,out=90] (m);
\draw[orange, thick, mid>] (m) to[in=-90,out=90] (zeta);
}
\arrow[Rightarrow,r,"m_c"]
\arrow[Rightarrow,d,red,"\rho_*"]
&
\tikzmath{
\node[DarkGreen] (eta) at (0,0) {$\scriptstyle \eta_*$};
\node (Ac) at (1,0) {$\scriptstyle A(c)$};
\coordinate (Bc1) at (0,2);
\node (Bc2) at (0,3) {$\scriptstyle B(c)$};
\node[orange] (zeta) at (1,3) {$\scriptstyle \zeta_*$};
\node[draw,rectangle, thick, rounded corners=5pt] (m) at (1,2) {$\scriptstyle m_*$};
\draw[thick, dashed, red, rounded corners=5pt] (.5,1.7) rectangle (1.5,2.4);
\draw (Ac) to[in=-90,out=90] (Bc1) to[in=-90,out=90] (Bc2);
\draw[DarkGreen,thick,mid>] (eta) to[in=-90,out=90] (m);
\draw[orange, thick, mid>] (m) to[in=-90,out=90] (zeta);
}
\arrow[Rightarrow,d,"\rho_*"]
\\
\tikzmath{
\node[DarkGreen] (eta) at (0,0) {$\scriptstyle \eta_*$};
\node (Ac1) at (1,0) {$\scriptstyle A(c)$};
\coordinate (Ac2) at (1,1);
\node (Bc) at (0,3) {$\scriptstyle B(c)$};
\node[orange] (zeta) at (1,3) {$\scriptstyle \zeta_*$};
\node[draw,rectangle, thick, rounded corners=5pt] (m) at (0,1) {$\scriptstyle n_*$};
\draw (Ac1) to[in=-90,out=90] (Ac2) to[in=-90,out=90] (Bc);
\draw[DarkGreen,thick,mid>] (eta) to[in=-90,out=90] (m);
\draw[orange, thick, mid>] (m) to[in=-90,out=90] (zeta);
}
\arrow[Rightarrow,r,"n_c"]
&
\tikzmath{
\node[DarkGreen] (eta) at (0,0) {$\scriptstyle \eta_*$};
\node (Ac) at (1,0) {$\scriptstyle A(c)$};
\coordinate (Bc1) at (0,2);
\node (Bc2) at (0,3) {$\scriptstyle B(c)$};
\node[orange] (zeta) at (1,3) {$\scriptstyle \zeta_*$};
\node[draw,rectangle, thick, rounded corners=5pt] (m) at (1,2) {$\scriptstyle n_*$};
\draw (Ac) to[in=-90,out=90] (Bc1) to[in=-90,out=90] (Bc2);
\draw[DarkGreen,thick,mid>] (eta) to[in=-90,out=90] (m);
\draw[orange, thick, mid>] (m) to[in=-90,out=90] (zeta);
}
\end{tikzcd}};\end{tikzpicture}
$$
\end{enumerate}
\end{defn}

%auto-ignore
%this ensures the arxiv doesn't try to start TeXing here.
%!TEX root =../Equivalence.tex

%%%%%%%%%%%%%%%%%%%%%%%%%%%%%%%%%%%%%%%%%%%%%%%%%%
%%%%%%%%%%%%%%%%%%%%%%%%%%%%%%%%%%%%%%%%%%%%%%%%%%
%%%%%%%%%%%%%%%%%%%%%%%%%%%%%%%%%%%%%%%%%%%%%%%%%%

%%%%%%%%%%%%%%%%%%%%%%%%%%%%%%%%%%%%%%%%%%%%%%%%%%
\section{Coherence proofs for strictification}
\label{sec:CoherenceProofs}

This appendix contains all proofs from \S\ref{sec:Truncation} which amount to checking/using various coherence conditions from Appendix \ref{sec:Weak3CategoryCoherences}.
As most of the proofs in this section are similar, we provide full detail for one part of each coherence proof below, and we explain the components of the proof in other parts whose details are left to the reader.
To make the commutative diagrams more readable, we suppress all whiskering notation, including Notation \ref{nota:DashedBoxForWhiskering}.

%%%%%%%%%%%%%%%%%%%%%%%%%%%%%%%%%%%%%%%%%%%%%%%%%%
\subsection{Coherence proofs for Strictifying 1-morphisms \S\ref{sec:Strictifying1Morphisms}}
\label{sec:CoherenceProofs1Morphisms}

We remind the reader that $(\rmB\cC,\pi^\cC), (\rmB\cD,\pi^\cD)$ are two objects in $\TriCat_G^{\pt}$, so that $\cC,\cD$ are $\Gray$-monoids, and $(A,\alpha)\in \TriCat_G( (\rmB\cC,\pi^\cC)\to (\rmB\cD,\pi^\cD))$.
We specified data for $(B,\beta): (\rmB\cC,\pi^\cC)\to (\rmB\cD,\pi^\cD)$ above in \S\ref{sec:Strictifying1Morphisms}, together with data for $(\gamma,\id): (B,\beta) \Rightarrow (A, \alpha)$.

\begin{nota}
\label{nota:ShadedBoxes}
In this section, we will use a shorthand notation for 1-cells in $\cD$ for proofs using commutative diagrams.
For $x\in \cC(a\to b)$, $y\in \cC(b\to c)$, and $z\in \cC(c\to d)$, we will denote their corresponding image in $\cD$ under $A$ as a small shaded square, e.g.,
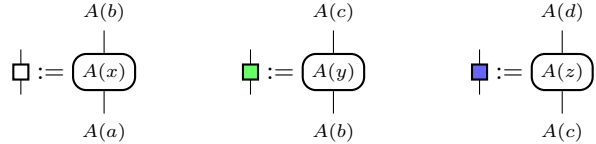
$$
\tikzmath{
\draw (0,-.3) -- (0,.3);
\filldraw[fill=\gColor, thick] (-.1,-.1) rectangle (.1,.1);
}
:=
\tikzmath{
\node (Aa) at (0,-.8) {$\scriptstyle A(a)$};
\node (Ab) at (0,.8) {$\scriptstyle A(b)$};
\node[draw,rectangle, thick, rounded corners=5pt] (Ax) at (0,0) {$\scriptstyle A(x)$};
\draw (Aa) to[in=-90,out=90] (Ax);
\draw (Ax) to[in=-90,out=90] (Ab);
}
\qquad\qquad
\tikzmath{
\draw (0,-.3) -- (0,.3);
\filldraw[fill=\hColor, thick] (-.1,-.1) rectangle (.1,.1);
}
:=
\tikzmath{
\node (Aa) at (0,-.8) {$\scriptstyle A(b)$};
\node (Ab) at (0,.8) {$\scriptstyle A(c)$};
\node[draw,rectangle, thick, rounded corners=5pt] (Ax) at (0,0) {$\scriptstyle A(y)$};
\draw (Aa) to[in=-90,out=90] (Ax);
\draw (Ax) to[in=-90,out=90] (Ab);
}
\qquad\qquad
\tikzmath{
\draw (0,-.3) -- (0,.3);
\filldraw[fill=\kColor, thick] (-.1,-.1) rectangle (.1,.1);
}
:=
\tikzmath{
\node (Aa) at (0,-.8) {$\scriptstyle A(c)$};
\node (Ab) at (0,.8) {$\scriptstyle A(d)$};
\node[draw,rectangle, thick, rounded corners=5pt] (Ax) at (0,0) {$\scriptstyle A(z)$};
\draw (Aa) to[in=-90,out=90] (Ax);
\draw (Ax) to[in=-90,out=90] (Ab);
}
$$
While the 1-composition $\xo$ in $\cD$ is stacking of diagrams, we denote $A$ applied to a 1-composite in $\cC$ by vertically joining the shaded squares:
$$
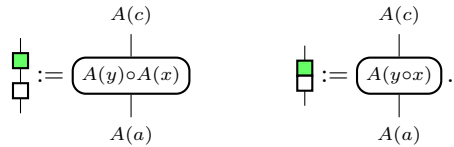
\tikzmath{
\draw (0,-.5) -- (0,.5);
\filldraw[fill=\hColor, thick] (-.1,.1) rectangle (.1,.3);
\filldraw[fill=\gColor, thick] (-.1,-.3) rectangle (.1,-.1);
}
:=
\tikzmath{
\node (Aa) at (0,-.8) {$\scriptstyle A(a)$};
\node (Ab) at (0,.8) {$\scriptstyle A(c)$};
\node[draw,rectangle, thick, rounded corners=5pt] (Ax) at (0,0) {$\scriptstyle A(y)\xo A(x)$};
\draw (Aa) to[in=-90,out=90] (Ax);
\draw (Ax) to[in=-90,out=90] (Ab);
}
\qquad\qquad
\tikzmath{
\draw (0,-.4) -- (0,.4);
\filldraw[fill=\hColor, thick] (-.1,0) rectangle (.1,.2);
\filldraw[fill=\gColor, thick] (-.1,-.2) rectangle (.1,0);
}
:=
\tikzmath{
\node (Aa) at (0,-.8) {$\scriptstyle A(a)$};
\node (Ab) at (0,.8) {$\scriptstyle A(c)$};
\node[draw,rectangle, thick, rounded corners=5pt] (Ax) at (0,0) {$\scriptstyle A(y\xo x)$};
\draw (Aa) to[in=-90,out=90] (Ax);
\draw (Ax) to[in=-90,out=90] (Ab);
}\,.
$$
Since 1-composition in $\cC$ and $\cD$ are both strict, we denote a triple 1-composite by stacking three boxes, and we denote $A$ applied to a triple 1-composite by vertically joining three boxes:
$$
\qquad
\tikzmath{
\draw (0,-.7) -- (0,.7);
\filldraw[fill=\kColor, thick] (-.1,.3) rectangle (.1,.5);
\filldraw[fill=\hColor, thick] (-.1,-.1) rectangle (.1,.1);
\filldraw[fill=\gColor, thick] (-.1,-.5) rectangle (.1,-.3);
}
:=
\tikzmath{
\node (Aa) at (0,-.8) {$\scriptstyle A(b)$};
\node (Ab) at (0,.8) {$\scriptstyle A(d)$};
\node[draw,rectangle, thick, rounded corners=5pt] (Ax) at (0,0) {$\scriptstyle A(z)\xo A(y)\xo A(x)$};
\draw (Aa) to[in=-90,out=90] (Ax);
\draw (Ax) to[in=-90,out=90] (Ab);
}
\qquad\qquad
\tikzmath{
\draw (0,-.5) -- (0,.5);
\filldraw[fill=\kColor, thick] (-.1,.1) rectangle (.1,.3);
\filldraw[fill=\hColor, thick] (-.1,-.1) rectangle (.1,.1);
\filldraw[fill=\gColor, thick] (-.1,-.3) rectangle (.1,-.1);
}
:=
\tikzmath{
\node (Aa) at (0,-.8) {$\scriptstyle A(a)$};
\node (Ab) at (0,.8) {$\scriptstyle A(d)$};
\node[draw,rectangle, thick, rounded corners=5pt] (Ax) at (0,0) {$\scriptstyle A(z\xo y\xo x)$};
\draw (Aa) to[in=-90,out=90] (Ax);
\draw (Ax) to[in=-90,out=90] (Ab);
}\,.
$$
For the $\xz$ composite of 1-cells, we use the nudging convention as in \eqref{eq:Nudging}.
We denote $A$ applied to a $\xz$ composite of 1-cells in $\cC$ by joining the shaded boxes along corners.
For the following example, given $x_1 \in \cC(a_1 \to b_1)$, $y_1\in \cC(b_1\to c_1)$, $x_2\in \cC(a_2\to b_2)$, and $y_2\in \cC(b_2\to c_2)$, we write
$$
\tikzmath{
\draw (0,-.3) -- (0,.3);
\filldraw[fill=\gColor, thick] (-.1,-.1) rectangle (.1,.1);
}
:=
\tikzmath{
\node (Aa) at (0,-.8) {$\scriptstyle A(a_1)$};
\node (Ab) at (0,.8) {$\scriptstyle A(b_1)$};
\node[draw,rectangle, thick, rounded corners=5pt] (Ax) at (0,0) {$\scriptstyle A(x_1)$};
\draw (Aa) to[in=-90,out=90] (Ax);
\draw (Ax) to[in=-90,out=90] (Ab);
}
\qquad\qquad
\tikzmath{
\draw (0,-.3) -- (0,.3);
\filldraw[fill=\hColor, thick] (-.1,-.1) rectangle (.1,.1);
}
:=
\tikzmath{
\node (Aa) at (0,-.8) {$\scriptstyle A(b_1)$};
\node (Ab) at (0,.8) {$\scriptstyle A(c_1)$};
\node[draw,rectangle, thick, rounded corners=5pt] (Ax) at (0,0) {$\scriptstyle A(y_1)$};
\draw (Aa) to[in=-90,out=90] (Ax);
\draw (Ax) to[in=-90,out=90] (Ab);
}
\qquad\qquad
\tikzmath{
\draw (0,-.3) -- (0,.3);
\filldraw[fill=\kColor, thick] (-.1,-.1) rectangle (.1,.1);
}
:=
\tikzmath{
\node (Aa) at (0,-.8) {$\scriptstyle A(a_2)$};
\node (Ab) at (0,.8) {$\scriptstyle A(b_2)$};
\node[draw,rectangle, thick, rounded corners=5pt] (Ax) at (0,0) {$\scriptstyle A(x_2)$};
\draw (Aa) to[in=-90,out=90] (Ax);
\draw (Ax) to[in=-90,out=90] (Ab);
}
\qquad\qquad
\tikzmath{
\draw (0,-.3) -- (0,.3);
\filldraw[fill, thick] (-.1,-.1) rectangle (.1,.1);
}
:=
\tikzmath{
\node (Aa) at (0,-.8) {$\scriptstyle A(b_2)$};
\node (Ab) at (0,.8) {$\scriptstyle A(c_2)$};
\node[draw,rectangle, thick, rounded corners=5pt] (Ax) at (0,0) {$\scriptstyle A(y_2)$};
\draw (Aa) to[in=-90,out=90] (Ax);
\draw (Ax) to[in=-90,out=90] (Ab);
}
$$
We then write
$$
\tikzmath{
\draw (-.2,-.5) -- (-.2,.5);
\draw (.2,-.5) -- (.2,.5);
\filldraw[fill=\gColor, thick] (-.3,.1) rectangle (-.1,.3);
\filldraw[fill=\kColor, thick] (.1,-.3) rectangle (.3,-.1);
}
:=
\tikzmath{
\node (Aa1) at (-.4,-.8) {$\scriptstyle A(a_1)$};
\node (Ab1) at (-.4,.8) {$\scriptstyle A(b_1)$};
\node (Aa2) at (.4,-.8) {$\scriptstyle A(a_2)$};
\node (Ab2) at (.4,.8) {$\scriptstyle A(b_2)$};
\node[draw,rectangle, thick, rounded corners=5pt] (Ax) at (0,0) {$\scriptstyle A(x_1)\xz A(x_2)$};
\draw (Aa1) to[in=-90,out=90] (-.4,-.23);
\draw (-.4,.23) to[in=-90,out=90] (Ab1);
\draw (Aa2) to[in=-90,out=90] (.4,-.23);
\draw (.4,.23) to[in=-90,out=90] (Ab2);
}
\qquad\qquad
\tikzmath{
\draw[double] (0,-.4) -- (0,.4);
\filldraw[fill=\gColor, thick] (-.2,0) rectangle (0,.2);
\filldraw[fill=\kColor, thick] (0,-.2) rectangle (.2,0);
}
:=
\tikzmath{
\node (Aa) at (0,-.8) {$\scriptstyle A(a_1\xz a_2)$};
\node (Ab) at (0,.8) {$\scriptstyle A(b_1\xz b_2)$};
\node[draw,rectangle, thick, rounded corners=5pt] (Ax) at (0,0) {$\scriptstyle A(x_1\xz x_2)$};
\draw[double] (Aa) to[in=-90,out=90] (Ax);
\draw[double] (Ax) to[in=-90,out=90] (Ab);
}\,.
$$
In this notation, we would write the following diagram for $A$ applied to the following composites:
$$
\tikzmath{
\draw[double] (0,-.6) -- (0,.6);
\filldraw[fill=\hColor, thick] (-.2,.2) rectangle (0,.4);
\filldraw[fill, thick] (0,0) rectangle (.2,.2);
\filldraw[fill=\gColor, thick] (-.2,-.2) rectangle (0,0);
\filldraw[fill=\kColor, thick] (0,-.4) rectangle (.2,-.2);
}
:=
\tikzmath{
\node (Aa) at (0,-.8) {$\scriptstyle A(a_1\xz a_2)$};
\node (Ab) at (0,.8) {$\scriptstyle A(c_1\xz c_2)$};
\node[draw,rectangle, thick, rounded corners=5pt] (Ax) at (0,0) {$\scriptstyle A(y_1\xz y_2) \xo A(x_1\xz x_2)$};
\draw[double] (Aa) to[in=-90,out=90] (Ax);
\draw[double] (Ax) to[in=-90,out=90] (Ab);
}
\qquad\qquad
\tikzmath{
\draw[double] (0,-.6) -- (0,.6);
\filldraw[fill=\hColor, thick] (-.2,.2) rectangle (0,.4);
\filldraw[fill=\gColor, thick] (-.2,0) rectangle (0,.2);
\filldraw[fill, thick] (0,-.2) rectangle (.2,0);
\filldraw[fill=\kColor, thick] (0,-.4) rectangle (.2,-.2);
}
:=
\tikzmath{
\node (Aa) at (0,-.8) {$\scriptstyle A(a_1\xz a_2)$};
\node (Ab) at (0,.8) {$\scriptstyle A(c_1\xz c_2)$};
\node[draw,rectangle, thick, rounded corners=5pt] (Ax) at (0,0) {$\scriptstyle A(y_1\xo x_1) \xz A(y_2\xo x_2)$};
\draw[double] (Aa) to[in=-90,out=90] (Ax);
\draw[double] (Ax) to[in=-90,out=90] (Ab);
}
\,.
$$
\end{nota}

\begin{proof}[Proof of Lem.~\ref{lem:Truncation2Functor}: $(B,B^1,B^2)$ is a 2-functor]
We must check \ref{Functor:A} for $B$.
We provide a complete proof for \ref{Functor:A.associative}, and leave most of \ref{Functor:A.unital} as an exercise for the reader.

\item[\ref{Functor:A.associative}]
For $x\in \cC(g_\cC\to h_\cC)$, $y\in \cC(h_\cC\to k_\cC)$, and $z\in \cC(k_\cC\to \ell_\cC)$, we use the following shorthand as in Notation \ref{nota:ShadedBoxes}:
$$
\tikzmath{
\draw (0,-.3) -- (0,.3);
\filldraw[fill=\gColor, thick] (-.1,-.1) rectangle (.1,.1);
}
:=
\tikzmath{
\node (Aa) at (0,-.8) {$\scriptstyle A(g_\cC)$};
\node (Ab) at (0,.8) {$\scriptstyle A(h_\cC)$};
\node[draw,rectangle, thick, rounded corners=5pt] (Ax) at (0,0) {$\scriptstyle A(x)$};
\draw (Aa) to[in=-90,out=90] (Ax);
\draw (Ax) to[in=-90,out=90] (Ab);
}
\qquad\qquad
\tikzmath{
\draw (0,-.3) -- (0,.3);
\filldraw[fill=\hColor, thick] (-.1,-.1) rectangle (.1,.1);
}
:=
\tikzmath{
\node (Aa) at (0,-.8) {$\scriptstyle A(h_\cC)$};
\node (Ab) at (0,.8) {$\scriptstyle A(k_\cC)$};
\node[draw,rectangle, thick, rounded corners=5pt] (Ax) at (0,0) {$\scriptstyle A(y)$};
\draw (Aa) to[in=-90,out=90] (Ax);
\draw (Ax) to[in=-90,out=90] (Ab);
}
\qquad\qquad
\tikzmath{
\draw (0,-.3) -- (0,.3);
\filldraw[fill=\kColor, thick] (-.1,-.1) rectangle (.1,.1);
}
:=
\tikzmath{
\node (Aa) at (0,-.8) {$\scriptstyle A(k_\cC)$};
\node (Ab) at (0,.8) {$\scriptstyle A(\ell_\cC)$};
\node[draw,rectangle, thick, rounded corners=5pt] (Ax) at (0,0) {$\scriptstyle A(z)$};
\draw (Aa) to[in=-90,out=90] (Ax);
\draw (Ax) to[in=-90,out=90] (Ab);
}
$$

$$
\begin{tikzpicture}[baseline= (a).base]\node[scale=.8] (a) at (0,0){\begin{tikzcd}
\tikzmath{
\coordinate (x) at (0,-.7);
\coordinate (y) at (0,0);
\coordinate (z) at (0,.7);
\draw (0,-1.1) -- (0,1.1);
\RectangleMorphism{(x)}{\gColor}
\RectangleMorphism{(y)}{\hColor}
\RectangleMorphism{(z)}{\kColor}
\AlphaAction{(x)}{.25}{.15}
\AlphaAction{(y)}{.25}{.15}
\AlphaAction{(z)}{.25}{.15}
}
\arrow[r, Rightarrow]
\arrow[d, Rightarrow]
&
\tikzmath{
\coordinate (x) at (0,-.7);
\coordinate (y) at (0,0);
\coordinate (z) at (0,.7);
\draw (0,-1.1) -- (0,1.1);
\RectangleMorphism{(x)}{\gColor}
\RectangleMorphism{(y)}{\hColor}
\RectangleMorphism{(z)}{\kColor}
\DoubleAlphaAction{(x)}{.25}{.15}{.35}
\AlphaAction{(z)}{.25}{.15}
}
\arrow[r, Rightarrow]
\arrow[d, Rightarrow]
&
\tikzmath{
\coordinate (x) at (0,-.7);
\coordinate (y) at (0,0);
\coordinate (z) at (0,.7);
\draw (0,-1.1) -- (0,1.1);
\RectangleMorphism{(x)}{\gColor}
\RectangleMorphism{(y)}{\hColor}
\RectangleMorphism{(z)}{\kColor}
\LongAlphaAction{($ (x) + (0,.35) $)}{.25}{.5}{.35}
\AlphaAction{(z)}{.25}{.15}
}
\arrow[r, Rightarrow]
\arrow[d, Rightarrow]
&
\tikzmath{
\coordinate (x) at (0,-.2);
\coordinate (y) at (0,0);
\coordinate (z) at (0,.7);
\draw (0,-.6) -- (0,1.1);
\RectangleMorphism{(x)}{\gColor}
\RectangleMorphism{(y)}{\hColor}
\RectangleMorphism{(z)}{\kColor}
\LongAlphaAction{($ (x) + (0,.1) $)}{.25}{.25}{.1}
\AlphaAction{(z)}{.25}{.15}
}
\arrow[d, Rightarrow]
\\
\tikzmath{
\coordinate (x) at (0,-.7);
\coordinate (y) at (0,0);
\coordinate (z) at (0,.7);
\draw (0,-1.1) -- (0,1.1);
\RectangleMorphism{(x)}{\gColor}
\RectangleMorphism{(y)}{\hColor}
\RectangleMorphism{(z)}{\kColor}
\AlphaAction{(x)}{.25}{.15}
\DoubleAlphaAction{(y)}{.25}{.15}{.35}
}
\arrow[r, Rightarrow]
\arrow[d, Rightarrow]
&
\tikzmath{
\coordinate (x) at (0,-.7);
\coordinate (y) at (0,0);
\coordinate (z) at (0,.7);
\draw (0,-1.1) -- (0,1.1);
\RectangleMorphism{(x)}{\gColor}
\RectangleMorphism{(y)}{\hColor}
\RectangleMorphism{(z)}{\kColor}
\TripleAlphaAction{(y)}{.25}{.15}
}
\arrow[r, Rightarrow]
\arrow[d, Rightarrow]
&
\tikzmath{
\coordinate (x) at (0,-.7);
\coordinate (y) at (0,0);
\coordinate (z) at (0,.7);
\draw (0,-1.1) -- (0,1.1);
\RectangleMorphism{(x)}{\gColor}
\RectangleMorphism{(y)}{\hColor}
\RectangleMorphism{(z)}{\kColor}
\draw[thick, red, mid>] 
(-.125,.35)
 .. controls ++(90:{.5*.25}) and ++(270:{.25}) .. 
(.25,.7) 
.. controls ++(90:.25) and ++(90:.25) .. 
(-.25,.85) 
-- 
(-.25,-.85) 
.. controls ++(270:.25) and ++(270:.25) .. 
(.25,-.7) 
--
(.25,0) 
.. controls ++(90:{.25}) and ++(270:{.5*.25}) .. 
(-.125,.35);
}
\arrow[r, Rightarrow]
\arrow[d, Rightarrow]
&
\tikzmath{
\coordinate (x) at (0,-.2);
\coordinate (y) at (0,0);
\coordinate (z) at (0,.7);
\draw (0,-.6) -- (0,1.1);
\RectangleMorphism{(x)}{\gColor}
\RectangleMorphism{(y)}{\hColor}
\RectangleMorphism{(z)}{\kColor}
\draw[thick, red, mid>] 
(-.125,.35)
 .. controls ++(90:{.5*.25}) and ++(270:{.25}) .. 
(.25,.7) 
.. controls ++(90:.25) and ++(90:.25) .. 
(-.25,.85) 
-- 
(-.25,-.35) 
.. controls ++(270:.25) and ++(270:.25) .. 
(.25,-.2) 
--
(.25,0) 
.. controls ++(90:{.25}) and ++(270:{.5*.25}) .. 
(-.125,.35);
}
\arrow[d, Rightarrow]
\\
\tikzmath{
\coordinate (x) at (0,-.7);
\coordinate (y) at (0,0);
\coordinate (z) at (0,.7);
\draw (0,-1.1) -- (0,1.1);
\RectangleMorphism{(x)}{\gColor}
\RectangleMorphism{(y)}{\hColor}
\RectangleMorphism{(z)}{\kColor}
\AlphaAction{(x)}{.25}{.15}
\LongAlphaAction{($ (y) + (0,.35) $)}{.25}{.5}{.35}
}
\arrow[r, Rightarrow]
\arrow[d, Rightarrow]
&
\tikzmath{
\coordinate (x) at (0,-.7);
\coordinate (y) at (0,0);
\coordinate (z) at (0,.7);
\draw (0,-1.1) -- (0,1.1);
\RectangleMorphism{(x)}{\gColor}
\RectangleMorphism{(y)}{\hColor}
\RectangleMorphism{(z)}{\kColor}
\draw[thick, red, mid>] 
(-.125,-.35)
 .. controls ++(90:{.5*.25}) and ++(270:{.25}) .. 
(.25,0) 
--
(.25,.7) 
.. controls ++(90:.25) and ++(90:.25) .. 
(-.25,.85) 
-- 
(-.25,-.85) 
.. controls ++(270:.25) and ++(270:.25) .. 
(.25,-.7) 
.. controls ++(90:{.25}) and ++(270:{.5*.25}) .. 
(-.125,-.35);
}
\arrow[r, Rightarrow]
\arrow[d, Rightarrow]
&
\tikzmath{
\coordinate (x) at (0,-.7);
\coordinate (y) at (0,0);
\coordinate (z) at (0,.7);
\draw (0,-1.1) -- (0,1.1);
\RectangleMorphism{(x)}{\gColor}
\RectangleMorphism{(y)}{\hColor}
\RectangleMorphism{(z)}{\kColor}
\LongAlphaAction{(y)}{.25}{.85}{.7}
}
\arrow[r, Rightarrow, "A^2_{y,x}"]
\arrow[d, Rightarrow, swap, "A^2_{z,y}"]
\arrow[dr, phantom, "\text{\scriptsize \ref{Functor:A.associative}}"]
&
\tikzmath{
\coordinate (x) at (0,-.2);
\coordinate (y) at (0,0);
\coordinate (z) at (0,.7);
\draw (0,-.6) -- (0,1.1);
\RectangleMorphism{(x)}{\gColor}
\RectangleMorphism{(y)}{\hColor}
\RectangleMorphism{(z)}{\kColor}
\LongAlphaAction{($ (y) + (0,.25) $)}{.25}{.65}{.45}
}
\arrow[d, Rightarrow, "A^2_{z,y\xo x}"]
\\
\tikzmath{
\coordinate (x) at (0,-.7);
\coordinate (y) at (0,0);
\coordinate (z) at (0,.2);
\draw (0,-1.1) -- (0,.6);
\RectangleMorphism{(x)}{\gColor}
\RectangleMorphism{(y)}{\hColor}
\RectangleMorphism{(z)}{\kColor}
\AlphaAction{(x)}{.25}{.15}
\LongAlphaAction{($ (y) + (0,.1) $)}{.25}{.25}{.1}
}
\arrow[r, Rightarrow]
&
\tikzmath{
\coordinate (x) at (0,-.7);
\coordinate (y) at (0,0);
\coordinate (z) at (0,.2);
\draw (0,-1.1) -- (0,.6);
\RectangleMorphism{(x)}{\gColor}
\RectangleMorphism{(y)}{\hColor}
\RectangleMorphism{(z)}{\kColor}
\draw[thick, red, mid>] 
(-.125,-.35)
 .. controls ++(90:{.5*.25}) and ++(270:{.25}) .. 
(.25,0)
--
(.25,.2) 
.. controls ++(90:.25) and ++(90:.25) .. 
(-.25,.35) 
-- 
(-.25,-.85) 
.. controls ++(270:.25) and ++(270:.25) .. 
(.25,-.7) 
.. controls ++(90:{.25}) and ++(270:{.5*.25}) .. 
(-.125,-.35);
}
\arrow[r, Rightarrow]
&
\tikzmath{
\coordinate (x) at (0,-.7);
\coordinate (y) at (0,0);
\coordinate (z) at (0,.2);
\draw (0,-1.1) -- (0,.6);
\RectangleMorphism{(x)}{\gColor}
\RectangleMorphism{(y)}{\hColor}
\RectangleMorphism{(z)}{\kColor}
\LongAlphaAction{($ (y) - (0,.25) $)}{.25}{.65}{.45}
}
\arrow[r, Rightarrow, "A^2_{z\xo y, x}"]
&
\tikzmath{
\coordinate (x) at (0,-.2);
\coordinate (y) at (0,0);
\coordinate (z) at (0,.2);
\draw (0,-.6) -- (0,.6);
\RectangleMorphism{(x)}{\gColor}
\RectangleMorphism{(y)}{\hColor}
\RectangleMorphism{(z)}{\kColor}
\LongAlphaAction{($ (y) $)}{.25}{.35}{.2}
}
\end{tikzcd}};\end{tikzpicture}
$$
Every square except for the bottom right square commutes by functoriality of 1-cell composition $\xo$ in a $\Gray$-monoid, that is, applying two 2-cells locally to non-overlapping regions in a 1-cell commutes.
The bottom right square commutes by \ref{Functor:A.associative} applied to the underlying 2-functor of $A$.

\item[\ref{Functor:A.unital}]
Follows from the properties of the adjoint equivalence $\alpha$ (see Remark \ref{rem:CombineAlphaCommutationSquares} below) together with \ref{Functor:A.unital} for the underlying 2-functor of $A$.
\end{proof}

\begin{rem}
\label{rem:CombineAlphaCommutationSquares}
In subsequent proofs, we will freely combine squares that commute by functoriality of 1-cell composition $\xo$ when the involved 2-cells applied locally are part of the biadjoint biequivalence $\alpha_*$ \eqref{eq:BiadjointBiequivalence} or the adjoint equivalences $\alpha_g$.
We will then simply state this larger face commutes by the properties of the adjoint equivalence $\alpha$, i.e., the properties of the biadjoint biequivalence $\alpha_*$ \eqref{eq:BiadjointBiequivalence} and the properties of the adjoint equivalences $\alpha_g$.
\end{rem}

\begin{proof}[Proof of Lem.~\ref{lem:Truncation3Functor}: $(B,\mu^B, \iota^B, \omega, \ell,r)$ is a weak 3-functor $\rmB \cC \to \rmB \cD$]

\item[\ref{Functor:mu.natural}] 
Every component which makes up $\mu^B$, 
in Construction \ref{const:B3Functor}, 
especially $\mu^A$, is natural.

\item[\ref{Functor:mu.monoidal}]
For $x_1 \in \cC(a_1 \to b_1)$, $y_1\in \cC(b_1\to c_1)$, $x_2\in \cC(a_2\to b_2)$, and $y_2\in \cC(b_2\to c_2)$, we use the following shorthand as in Notation \ref{nota:ShadedBoxes}:
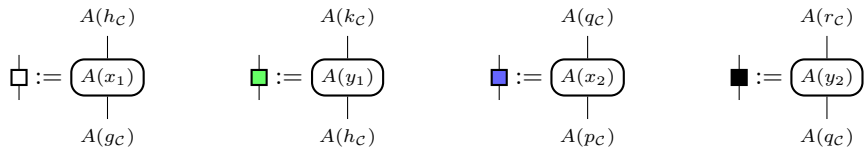
$$
\tikzmath{
\draw (0,-.3) -- (0,.3);
\filldraw[fill=\gColor, thick] (-.1,-.1) rectangle (.1,.1);
}
:=
\tikzmath{
\node (Aa) at (0,-.8) {$\scriptstyle A(g_\cC)$};
\node (Ab) at (0,.8) {$\scriptstyle A(h_\cC)$};
\node[draw,rectangle, thick, rounded corners=5pt] (Ax) at (0,0) {$\scriptstyle A(x_1)$};
\draw (Aa) to[in=-90,out=90] (Ax);
\draw (Ax) to[in=-90,out=90] (Ab);
}
\qquad\qquad
\tikzmath{
\draw (0,-.3) -- (0,.3);
\filldraw[fill=\hColor, thick] (-.1,-.1) rectangle (.1,.1);
}
:=
\tikzmath{
\node (Aa) at (0,-.8) {$\scriptstyle A(h_\cC)$};
\node (Ab) at (0,.8) {$\scriptstyle A(k_\cC)$};
\node[draw,rectangle, thick, rounded corners=5pt] (Ax) at (0,0) {$\scriptstyle A(y_1)$};
\draw (Aa) to[in=-90,out=90] (Ax);
\draw (Ax) to[in=-90,out=90] (Ab);
}
\qquad\qquad
\tikzmath{
\draw (0,-.3) -- (0,.3);
\filldraw[fill=\kColor, thick] (-.1,-.1) rectangle (.1,.1);
}
:=
\tikzmath{
\node (Aa) at (0,-.8) {$\scriptstyle A(p_\cC)$};
\node (Ab) at (0,.8) {$\scriptstyle A(q_\cC)$};
\node[draw,rectangle, thick, rounded corners=5pt] (Ax) at (0,0) {$\scriptstyle A(x_2)$};
\draw (Aa) to[in=-90,out=90] (Ax);
\draw (Ax) to[in=-90,out=90] (Ab);
}
\qquad\qquad
\tikzmath{
\draw (0,-.3) -- (0,.3);
\filldraw[fill, thick] (-.1,-.1) rectangle (.1,.1);
}
:=
\tikzmath{
\node (Aa) at (0,-.8) {$\scriptstyle A(q_\cC)$};
\node (Ab) at (0,.8) {$\scriptstyle A(r_\cC)$};
\node[draw,rectangle, thick, rounded corners=5pt] (Ax) at (0,0) {$\scriptstyle A(y_2)$};
\draw (Aa) to[in=-90,out=90] (Ax);
\draw (Ax) to[in=-90,out=90] (Ab);
}
$$
For the following diagram to fit on one page, we compress the definition of $\mu^B_{x,y}$ from Construction \ref{const:B3Functor} into four steps
\begin{equation}
\label{eq:CompressedMuB}
\mu^B_{x,y}
:=
\left(
\tikzmath{
\coordinate (x) at (-.3,0);
\coordinate (y) at (.3,-.7);
\draw (-.3,-1.2) -- (-.3,.1) .. controls ++(90:.3) and ++(270:.3) .. (-.015,.6) -- (-.015,1);
\draw (.3,-1.2) -- (.3,.1) .. controls ++(90:.3) and ++(270:.3) .. (.015,.6) -- (.015,1);
\RectangleMorphism{(x)}{\gColor}
\RectangleMorphism{(y)}{\kColor}
\AlphaAction{(x)}{.25}{.15}
\AlphaAction{(y)}{.25}{.15}
}
\Rightarrow
\tikzmath{
\coordinate (x) at (-.2,-.1);
\coordinate (y) at (.2,-.5);
\draw (-.2,-1.2) -- (-.2,.1) .. controls ++(90:.3) and ++(270:.3) .. (-.015,.6) -- (-.015,1);
\draw (.2,-1.2) -- (.2,.1) .. controls ++(90:.3) and ++(270:.3) .. (.015,.6) -- (.015,1);
\RectangleMorphism{(x)}{\gColor}
\RectangleMorphism{(y)}{\kColor}
\AlphaAction{(0,-.35)}{.5}{.3}
}
\overset{(\alpha^2_{h,q})^{-1}}{\Rightarrow}
\tikzmath{
\coordinate (x) at (-.2,-.6);
\coordinate (y) at (.2,-.9);
\node[draw,rectangle, thick, rounded corners=5pt] (m) at (0,0) {$\scriptstyle \mu^A_{h,q}$};
\draw (-.2,-1.6) -- (x) to[in=-110,out=90] (m);
\draw (.2,-1.6) -- (y) to[in=-70,out=90] (m);
\draw[double] (m) to[in=-90,out=90] (0,.8);
\RectangleMorphism{(x)}{\gColor}
\RectangleMorphism{(y)}{\kColor}
\LongAlphaAction{(0,-.4)}{.7}{.6}{.3}
}
\overset{\mu^A}{\Rightarrow}
\tikzmath{
\coordinate (x) at (-.1,.2);
\coordinate (y) at (.1,0);
\node[draw,rectangle, thick, rounded corners=5pt] (m) at (0,-.7) {$\scriptstyle \mu^A_{g,p}$};
\draw (-.2,-1.6) to[in=-110,out=90] (m);
\draw (.2,-1.6) to[in=-70,out=90] (m);
\draw[double] (m) to[in=-90,out=90] (0,.8);
\RectangleMorphism{(x)}{\gColor}
\RectangleMorphism{(y)}{\kColor}
\LongAlphaAction{(0,-.4)}{.7}{.6}{.3}
}
\overset{(\alpha^2_{g,p})^{-1}}{\Rightarrow}
\tikzmath{
\coordinate (x) at (-.1,.5);
\coordinate (y) at (.1,.3);
\draw (-.3,-.8) .. controls ++(90:.3) and ++(270:.3) .. (-.015,-.2) -- (-.015,1);
\draw (.3,-.8) .. controls ++(90:.3) and ++(270:.3) .. (.015,-.2) -- (.015,1);
\RectangleMorphism{(x)}{\gColor}
\RectangleMorphism{(y)}{\kColor}
\AlphaAction{(0,.4)}{.4}{.2}
}
\right),
\end{equation}
we suppress as many interchangers as possible, and we combine commuting squares involving only $\alpha_*$ and the $\alpha_g$ as in Remark \ref{rem:CombineAlphaCommutationSquares}.
$$
\begin{tikzpicture}[baseline= (a).base]\node[scale=.8] (a) at (0,0){\begin{tikzcd}[column sep=4em]
\tikzmath{
\coordinate (y1) at (-.3,0);
\coordinate (y2) at (.3,-.7);
\coordinate (x1) at (-.3,-1.4);
\coordinate (x2) at (.3,-2.1);
\draw (-.3,-2.6) -- (-.3,.1) .. controls ++(90:.3) and ++(270:.3) .. (-.015,.6) -- (-.015,1);
\draw (.3,-2.6) -- (.3,.1) .. controls ++(90:.3) and ++(270:.3) .. (.015,.6) -- (.015,1);
\RectangleMorphism{(x1)}{\gColor}
\RectangleMorphism{(x2)}{\kColor}
\RectangleMorphism{(y1)}{\hColor}
\RectangleMorphism{(y2)}{black}
\AlphaAction{(x1)}{.25}{.15}
\AlphaAction{(x2)}{.25}{.15}
\AlphaAction{(y1)}{.25}{.15}
\AlphaAction{(y2)}{.25}{.15}
}
\arrow[r, Rightarrow]
\arrow[d, Rightarrow, "\phi^{-1}"]
\arrow[dr, Rightarrow]
&
\tikzmath{
\coordinate (y1) at (-.2,-.1);
\coordinate (y2) at (.2,-.5);
\coordinate (x1) at (-.2,-1.4);
\coordinate (x2) at (.2,-2.1);
\draw (-.2,-2.6) -- (-.2,.1) .. controls ++(90:.3) and ++(270:.3) .. (-.015,.6) -- (-.015,1);
\draw (.2,-2.6) -- (.2,.1) .. controls ++(90:.3) and ++(270:.3) .. (.015,.6) -- (.015,1);
\RectangleMorphism{(x1)}{\gColor}
\RectangleMorphism{(x2)}{\kColor}
\RectangleMorphism{(y1)}{\hColor}
\RectangleMorphism{(y2)}{black}
\AlphaAction{(x1)}{.25}{.15}
\AlphaAction{(x2)}{.25}{.15}
\AlphaAction{(0,-.35)}{.5}{.3}
}
\arrow[r, Rightarrow, "(\alpha^2_{k,r})^{-1}"]
\arrow[dr, Rightarrow]
&
\tikzmath{
\coordinate (y1) at (-.2,-.6);
\coordinate (y2) at (.2,-.9);
\coordinate (x1) at (-.2,-1.8);
\coordinate (x2) at (.2,-2.5);
\node[draw,rectangle, thick, rounded corners=5pt] (m) at (0,0) {$\scriptstyle \mu^A_{k,r}$};
\draw (-.2,-3) -- (y1) to[in=-110,out=90] (m);
\draw (.2,-3) -- (y2) to[in=-70,out=90] (m);
\draw[double] (m) to[in=-90,out=90] (0,1);
\RectangleMorphism{(x1)}{\gColor}
\RectangleMorphism{(x2)}{\kColor}
\RectangleMorphism{(y1)}{\hColor}
\RectangleMorphism{(y2)}{black}
\LongAlphaAction{(0,-.3)}{.7}{.6}{.3}
\AlphaAction{(x1)}{.25}{.15}
\AlphaAction{(x2)}{.25}{.15}
}
\arrow[r, Rightarrow, "\mu^A_{y_1,y_2}"]
\arrow[dr, Rightarrow]
&
\tikzmath{
\coordinate (y1) at (-.1,.2);
\coordinate (y2) at (.1,0);
\coordinate (x1) at (-.2,-1.8);
\coordinate (x2) at (.2,-2.5);
\node[draw,rectangle, thick, rounded corners=5pt] (m) at (0,-.7) {$\scriptstyle \mu^A_{h,q}$};
\draw (-.2,-3) -- (x1) to[in=-100,out=90] (m);
\draw (.2,-3) -- (x2) to[in=-80,out=90] (m);
\draw[double] (m) to[in=-90,out=90] (0,.8);
\RectangleMorphism{(x1)}{\gColor}
\RectangleMorphism{(x2)}{\kColor}
\RectangleMorphism{(y1)}{\hColor}
\RectangleMorphism{(y2)}{black}
\LongAlphaAction{(0,-.4)}{.7}{.6}{.3}
\AlphaAction{(x1)}{.25}{.15}
\AlphaAction{(x2)}{.25}{.15}
}
\arrow[r, Rightarrow, "(\alpha^2_{h,q})^{-1}"]
\arrow[dr, Rightarrow]
&
\tikzmath{
\coordinate (y1) at (-.1,.5);
\coordinate (y2) at (.1,.3);
\coordinate (x1) at (-.3,-.9);
\coordinate (x2) at (.3,-1.4);
\draw (-.3,-1.9) -- (-.3,-.8) .. controls ++(90:.3) and ++(270:.3) .. (-.015,-.2) -- (-.015,1);
\draw (.3,-1.9) -- (.3,-.8) .. controls ++(90:.3) and ++(270:.3) .. (.015,-.2) -- (.015,1);
\RectangleMorphism{(x1)}{\gColor}
\RectangleMorphism{(x2)}{\kColor}
\RectangleMorphism{(y1)}{\hColor}
\RectangleMorphism{(y2)}{black}
\AlphaAction{(0,.4)}{.4}{.2}
\AlphaAction{(x1)}{.25}{.15}
\AlphaAction{(x2)}{.25}{.15}
}
\arrow[r, Rightarrow]
&
\tikzmath{
\coordinate (y1) at (-.1,1.2);
\coordinate (y2) at (.1,1);
\coordinate (x1) at (-.2,-.1);
\coordinate (x2) at (.2,-.5);
\draw (-.2,-1.2) -- (-.2,.1) .. controls ++(90:.3) and ++(270:.3) .. (-.015,.6) -- (-.015,1.8);
\draw (.2,-1.2) -- (.2,.1) .. controls ++(90:.3) and ++(270:.3) .. (.015,.6) -- (.015,1.8);
\RectangleMorphism{(x1)}{\gColor}
\RectangleMorphism{(x2)}{\kColor}
\RectangleMorphism{(y1)}{\hColor}
\RectangleMorphism{(y2)}{black}
\AlphaAction{(0,-.35)}{.5}{.3}
\AlphaAction{(0,1.1)}{.4}{.2}
}
\arrow[d, Rightarrow, "(\alpha^2_{h,q})^{-1}"]
\\
\tikzmath{
\coordinate (y1) at (-.3,0);
\coordinate (y2) at (.3,-1.4);
\coordinate (x1) at (-.3,-.7);
\coordinate (x2) at (.3,-2.1);
\draw (-.3,-2.6) -- (-.3,.1) .. controls ++(90:.3) and ++(270:.3) .. (-.015,.6) -- (-.015,1);
\draw (.3,-2.6) -- (.3,.1) .. controls ++(90:.3) and ++(270:.3) .. (.015,.6) -- (.015,1);
\RectangleMorphism{(x1)}{\gColor}
\RectangleMorphism{(x2)}{\kColor}
\RectangleMorphism{(y1)}{\hColor}
\RectangleMorphism{(y2)}{black}
\AlphaAction{(x1)}{.25}{.15}
\AlphaAction{(x2)}{.25}{.15}
\AlphaAction{(y1)}{.25}{.15}
\AlphaAction{(y2)}{.25}{.15}
}
\arrow[dr, Rightarrow]
\arrow[d, Rightarrow]
&
\tikzmath{
\coordinate (y1) at (-.2,-.1);
\coordinate (y2) at (.2,-.5);
\coordinate (x1) at (-.2,-.9);
\coordinate (x2) at (.2,-1.3);
\draw (-.2,-2.4) -- (-.2,.1) .. controls ++(90:.3) and ++(270:.3) .. (-.015,.6) -- (-.015,1);
\draw (.2,-2.4) -- (.2,.1) .. controls ++(90:.3) and ++(270:.3) .. (.015,.6) -- (.015,1);
\RectangleMorphism{(x1)}{\gColor}
\RectangleMorphism{(x2)}{\kColor}
\RectangleMorphism{(y1)}{\hColor}
\RectangleMorphism{(y2)}{black}
\LongAlphaAction{(0,-.8)}{.5}{.7}{.4}
}
\arrow[d, Rightarrow, "\phi^{-1}"]
\arrow[dr, Rightarrow, "(\alpha^2_{k,r})^{-1}"]
&
\tikzmath{
\coordinate (y1) at (-.2,-.1);
\coordinate (y2) at (.2,-.5);
\coordinate (x1) at (-.2,-1.4);
\coordinate (x2) at (.2,-1.9);
\draw (-.2,-2.5) -- (-.2,.1) .. controls ++(90:.3) and ++(270:.3) .. (-.015,.6) -- (-.015,1);
\draw (.2,-2.5) -- (.2,.1) .. controls ++(90:.3) and ++(270:.3) .. (.015,.6) -- (.015,1);
\RectangleMorphism{(x1)}{\gColor}
\RectangleMorphism{(x2)}{\kColor}
\RectangleMorphism{(y1)}{\hColor}
\RectangleMorphism{(y2)}{black}
\AlphaAction{(0,-.35)}{.5}{.3}
\AlphaAction{(0,-1.75)}{.5}{.3}
}
\arrow[l, Rightarrow]
\arrow[r, Rightarrow, "(\alpha^2_{k,r})^{-1}"]
&
\tikzmath{
\coordinate (y1) at (-.2,-.6);
\coordinate (y2) at (.2,-.9);
\coordinate (x1) at (-.2,-1.8);
\coordinate (x2) at (.2,-2.3);
\node[draw,rectangle, thick, rounded corners=5pt] (m) at (0,0) {$\scriptstyle \mu^A_{k,r}$};
\draw (-.2,-2.9) -- (y1) to[in=-110,out=90] (m);
\draw (.2,-2.9) -- (y2) to[in=-70,out=90] (m);
\draw[double] (m) to[in=-90,out=90] (0,1);
\RectangleMorphism{(x1)}{\gColor}
\RectangleMorphism{(x2)}{\kColor}
\RectangleMorphism{(y1)}{\hColor}
\RectangleMorphism{(y2)}{black}
\LongAlphaAction{(0,-.3)}{.7}{.6}{.3}
\AlphaAction{(0,-2.15)}{.5}{.3}
}
\arrow[r, Rightarrow, "\mu^A_{y_1,y_2}"]
\arrow[dl, Rightarrow]
&
\tikzmath{
\coordinate (y1) at (-.1,.2);
\coordinate (y2) at (.1,0);
\coordinate (x1) at (-.2,-1.8);
\coordinate (x2) at (.2,-2.3);
\node[draw,rectangle, thick, rounded corners=5pt] (m) at (0,-.7) {$\scriptstyle \mu^A_{h,q}$};
\draw (-.2,-3) -- (x1) to[in=-100,out=90] (m);
\draw (.2,-3) -- (x2) to[in=-80,out=90] (m);
\draw[double] (m) to[in=-90,out=90] (0,.8);
\RectangleMorphism{(x1)}{\gColor}
\RectangleMorphism{(x2)}{\kColor}
\RectangleMorphism{(y1)}{\hColor}
\RectangleMorphism{(y2)}{black}
\LongAlphaAction{(0,-.4)}{.7}{.6}{.3}
\AlphaAction{(0,-2.15)}{.5}{.3}
}
\arrow[d, Rightarrow]
\arrow[ur, Rightarrow, "(\alpha^2_{h,q})^{-1}"]
&
\tikzmath{
\coordinate (x1) at (-.2,-.6);
\coordinate (x2) at (.2,-.9);
\coordinate (y1) at (-.1,1.3);
\coordinate (y2) at (.1,1.1);
\node[draw,rectangle, thick, rounded corners=5pt] (m) at (0,0) {$\scriptstyle \mu^A_{h,q}$};
\draw (-.2,-1.6) -- (x1) to[in=-110,out=90] (m);
\draw (.2,-1.6) -- (x2) to[in=-70,out=90] (m);
\draw[double] (m) to[in=-90,out=90] (0,1.8);
\RectangleMorphism{(x1)}{\gColor}
\RectangleMorphism{(x2)}{\kColor}
\RectangleMorphism{(y1)}{\hColor}
\RectangleMorphism{(y2)}{black}
\LongAlphaAction{(0,-.3)}{.7}{.6}{.3}
\AlphaAction{(0,1.2)}{.4}{.2}
}
\arrow[d, Rightarrow, "\mu^A_{x_1,x_2}"]
\arrow[dl, Rightarrow]
\\
\tikzmath{
\coordinate (y1) at (-.3,0);
\coordinate (y2) at (.3,-1.2);
\coordinate (x1) at (-.3,-.4);
\coordinate (x2) at (.3,-1.6);
\draw (-.3,-2.1) -- (-.3,.1) .. controls ++(90:.3) and ++(270:.3) .. (-.015,.6) -- (-.015,1);
\draw (.3,-2.1) -- (.3,.1) .. controls ++(90:.3) and ++(270:.3) .. (.015,.6) -- (.015,1);
\RectangleMorphism{(x1)}{\gColor}
\RectangleMorphism{(x2)}{\kColor}
\RectangleMorphism{(y1)}{\hColor}
\RectangleMorphism{(y2)}{black}
\LongAlphaAction{(-.3,-.2)}{.25}{.35}{.15}
\LongAlphaAction{(.3,-1.4)}{.25}{.35}{.15}
}
\arrow[r, Rightarrow]
\arrow[d, Rightarrow, "A^2_{y_1, x_1}\xz A^2_{y_2, x_2}"]
&
\tikzmath{
\coordinate (y1) at (-.2,-.1);
\coordinate (y2) at (.2,-.9);
\coordinate (x1) at (-.2,-.5);
\coordinate (x2) at (.2,-1.3);
\draw (-.2,-2.4) -- (-.2,.1) .. controls ++(90:.3) and ++(270:.3) .. (-.015,.6) -- (-.015,1);
\draw (.2,-2.4) -- (.2,.1) .. controls ++(90:.3) and ++(270:.3) .. (.015,.6) -- (.015,1);
\RectangleMorphism{(x1)}{\gColor}
\RectangleMorphism{(x2)}{\kColor}
\RectangleMorphism{(y1)}{\hColor}
\RectangleMorphism{(y2)}{black}
\LongAlphaAction{(0,-.8)}{.5}{.7}{.4}
}
\arrow[ddl, Rightarrow, "A^2_{y_1, x_1}\xz A^2_{y_2, x_2}"]
\arrow[dr, Rightarrow, "(\alpha^2_{k,r})^{-1}"]
&
\tikzmath{
\coordinate (y1) at (-.2,-.6);
\coordinate (y2) at (.2,-.9);
\coordinate (x1) at (-.2,-1.2);
\coordinate (x2) at (.2,-1.5);
\node[draw,rectangle, thick, rounded corners=5pt] (m) at (0,0) {$\scriptstyle \mu^A_{k,r}$};
\draw (-.2,-2.1) -- (y1) to[in=-110,out=90] (m);
\draw (.2,-2.1) -- (y2) to[in=-70,out=90] (m);
\draw[double] (m) to[in=-90,out=90] (0,1);
\RectangleMorphism{(x1)}{\gColor}
\RectangleMorphism{(x2)}{\kColor}
\RectangleMorphism{(y1)}{\hColor}
\RectangleMorphism{(y2)}{black}
\LongAlphaAction{(0,-.6)}{.7}{.9}{.6}
}
\arrow[d, Rightarrow, "\phi^{-1}"]
&
\text{\ref{Functor:mu.monoidal} for $A$}
&
\tikzmath{
\coordinate (y1) at (-.1,.2);
\coordinate (y2) at (.1,0);
\coordinate (x1) at (-.2,-1.3);
\coordinate (x2) at (.2,-1.6);
\node[draw,rectangle, thick, rounded corners=5pt] (m) at (0,-.7) {$\scriptstyle \mu^A_{h,q}$};
\draw (-.2,-2.2) -- (x1) to[in=-110,out=90] (m);
\draw (.2,-2.2) -- (x2) to[in=-70,out=90] (m);
\draw[double] (m) to[in=-90,out=90] (0,.8);
\RectangleMorphism{(x1)}{\gColor}
\RectangleMorphism{(x2)}{\kColor}
\RectangleMorphism{(y1)}{\hColor}
\RectangleMorphism{(y2)}{black}
\LongAlphaAction{(0,-.7)}{.7}{.9}{.6}
}
\arrow[d, Rightarrow, "\mu^A_{x_1,x_2}"]
&
\tikzmath{
\coordinate (y1) at (-.1,1.4);
\coordinate (y2) at (.1,1.2);
\coordinate (x1) at (-.1,.2);
\coordinate (x2) at (.1,0);
\node[draw,rectangle, thick, rounded corners=5pt] (m) at (0,-.7) {$\scriptstyle \mu^A_{g,p}$};
\draw (-.2,-1.6) to[in=-110,out=90] (m);
\draw (.2,-1.6) to[in=-70,out=90] (m);
\draw[double] (m) to[in=-90,out=90] (0,1.9);
\RectangleMorphism{(x1)}{\gColor}
\RectangleMorphism{(x2)}{\kColor}
\RectangleMorphism{(y1)}{\hColor}
\RectangleMorphism{(y2)}{black}
\LongAlphaAction{(0,-.4)}{.7}{.6}{.3}
\AlphaAction{(0,1.3)}{.4}{.2}
}
\arrow[d, Rightarrow, "(\alpha^2_{g,p})^{-1}"]
\arrow[dl, Rightarrow]
\\
\tikzmath{
\coordinate (x2) at (-.3,0);
\coordinate (y2) at (.3,-1);
\coordinate (x1) at (-.3,-.2);
\coordinate (y1) at (.3,-1.2);
\draw (-.3,-1.7) -- (-.3,.1) .. controls ++(90:.3) and ++(270:.3) .. (-.015,.6) -- (-.015,1);
\draw (.3,-1.7) -- (.3,.1) .. controls ++(90:.3) and ++(270:.3) .. (.015,.6) -- (.015,1);
\RectangleMorphism{(x1)}{\gColor}
\RectangleMorphism{(x2)}{\hColor}
\RectangleMorphism{(y1)}{\kColor}
\RectangleMorphism{(y2)}{black}
\LongAlphaAction{($ (x1) + (0,.1) $)}{.25}{.25}{.1}
\LongAlphaAction{($ (y1) + (0,.1) $)}{.25}{.25}{.1}
}
\arrow[d, Rightarrow]
&&
\tikzmath{
\coordinate (y1) at (-.2,-.6);
\coordinate (y2) at (.2,-1.3);
\coordinate (x1) at (-.2,-1);
\coordinate (x2) at (.2,-1.7);
\node[draw,rectangle, thick, rounded corners=5pt] (m) at (0,0) {$\scriptstyle \mu^A_{k,r}$};
\draw (-.2,-2.3) -- (y1) to[in=-110,out=90] (m);
\draw (.2,-2.3) -- (y2) to[in=-70,out=90] (m);
\draw[double] (m) to[in=-90,out=90] (0,1);
\RectangleMorphism{(x1)}{\gColor}
\RectangleMorphism{(x2)}{\kColor}
\RectangleMorphism{(y1)}{\hColor}
\RectangleMorphism{(y2)}{black}
\LongAlphaAction{(0,-.7)}{.7}{1}{.7}
}
\arrow[dl, Rightarrow, "A^2_{y_1, x_1}\xz A^2_{y_2, x_2}"]
&
\tikzmath{
\coordinate (y1) at (-.1,.7);
\coordinate (y2) at (.1,.5);
\coordinate (x1) at (-.1,.3);
\coordinate (x2) at (.1,.1);
\node[draw,rectangle, thick, rounded corners=5pt] (m) at (0,-.7) {$\scriptstyle \mu^A_{g,p}$};
\draw (-.2,-1.6) to[in=-110,out=90] (m);
\draw (.2,-1.6) to[in=-70,out=90] (m);
\draw[double] (m) to[in=-90,out=90] (0,1.3);
\RectangleMorphism{(x1)}{\gColor}
\RectangleMorphism{(x2)}{\kColor}
\RectangleMorphism{(y1)}{\hColor}
\RectangleMorphism{(y2)}{black}
\LongAlphaAction{(0,-.2)}{.7}{.8}{.5}
}
\arrow[dr, Rightarrow, "(\alpha^2_{g,p})^{-1}"]
&
\tikzmath{
\coordinate (y1) at (-.1,.9);
\coordinate (y2) at (.1,.7);
\coordinate (x1) at (-.1,.3);
\coordinate (x2) at (.1,.1);
\node[draw,rectangle, thick, rounded corners=5pt] (m) at (0,-.7) {$\scriptstyle \mu^A_{g,p}$};
\draw (-.2,-1.6) to[in=-110,out=90] (m);
\draw (.2,-1.6) to[in=-70,out=90] (m);
\draw[double] (m) to[in=-90,out=90] (0,1.3);
\RectangleMorphism{(x1)}{\gColor}
\RectangleMorphism{(x2)}{\kColor}
\RectangleMorphism{(y1)}{\hColor}
\RectangleMorphism{(y2)}{black}
\LongAlphaAction{(0,-.1)}{.7}{.8}{.5}
}
\arrow[l, swap, Rightarrow, "A^2_{y_1\xz y_2, x_1\xz x_2}"]
\arrow[dr, Rightarrow, "(\alpha^2_{g,p})^{-1}"]
&
\tikzmath{
\coordinate (y1) at (-.1,1.5);
\coordinate (y2) at (.1,1.3);
\coordinate (x1) at (-.1,.5);
\coordinate (x2) at (.1,.3);
\draw (-.3,-.8) .. controls ++(90:.3) and ++(270:.3) .. (-.015,-.2) -- (-.015,2);
\draw (.3,-.8) .. controls ++(90:.3) and ++(270:.3) .. (.015,-.2) -- (.015,2);
\RectangleMorphism{(x1)}{\gColor}
\RectangleMorphism{(x2)}{\kColor}
\RectangleMorphism{(y1)}{\hColor}
\RectangleMorphism{(y2)}{black}
\AlphaAction{(0,.4)}{.4}{.2}
\AlphaAction{(0,1.4)}{.4}{.2}
}
\arrow[d, Rightarrow]
\\
\tikzmath{
\coordinate (y1) at (-.2,-.1);
\coordinate (y2) at (.2,-.7);
\coordinate (x1) at (-.2,-.3);
\coordinate (x2) at (.2,-.9);
\draw (-.2,-1.6) -- (-.2,.1) .. controls ++(90:.3) and ++(270:.3) .. (-.015,.6) -- (-.015,1);
\draw (.2,-1.6) -- (.2,.1) .. controls ++(90:.3) and ++(270:.3) .. (.015,.6) -- (.015,1);
\RectangleMorphism{(x1)}{\gColor}
\RectangleMorphism{(x2)}{\kColor}
\RectangleMorphism{(y1)}{\hColor}
\RectangleMorphism{(y2)}{black}
\LongAlphaAction{(0,-.6)}{.5}{.5}{.2}
}
\arrow[r, Rightarrow, "(\alpha^2_{k,r})^{-1}"]
&
\tikzmath{
\coordinate (y1) at (-.2,-.6);
\coordinate (y2) at (.2,-1.1);
\coordinate (x1) at (-.2,-.8);
\coordinate (x2) at (.2,-1.3);
\node[draw,rectangle, thick, rounded corners=5pt] (m) at (0,0) {$\scriptstyle \mu^A_{k,r}$};
\draw (-.2,-2.3) -- (y1) to[in=-110,out=90] (m);
\draw (.2,-2.3) -- (y2) to[in=-70,out=90] (m);
\draw[double] (m) to[in=-90,out=90] (0,1);
\RectangleMorphism{(x1)}{\gColor}
\RectangleMorphism{(x2)}{\kColor}
\RectangleMorphism{(y1)}{\hColor}
\RectangleMorphism{(y2)}{black}
\LongAlphaAction{(0,-.5)}{.7}{.8}{.5}
}
\arrow[r, Rightarrow, "\mu^A_{y_1\xo x_1, y_2\xo x_2}"]
&
\tikzmath{
\coordinate (y1) at (-.1,.7);
\coordinate (y2) at (.1,.3);
\coordinate (x1) at (-.1,.5);
\coordinate (x2) at (.1,.1);
\node[draw,rectangle, thick, rounded corners=5pt] (m) at (0,-.7) {$\scriptstyle \mu^A_{g,k}$};
\draw (-.2,-1.6) to[in=-110,out=90] (m);
\draw (.2,-1.6) to[in=-70,out=90] (m);
\draw[double] (m) to[in=-90,out=90] (0,1.3);
\RectangleMorphism{(x1)}{\gColor}
\RectangleMorphism{(x2)}{\kColor}
\RectangleMorphism{(y1)}{\hColor}
\RectangleMorphism{(y2)}{black}
\LongAlphaAction{(0,-.2)}{.7}{.8}{.5}
}
\arrow[ur, Rightarrow, "A(\phi)"]
\arrow[r, Rightarrow, "(\alpha^2_{g,p})^{-1}"]
&
\tikzmath{
\coordinate (y1) at (-.1,.7);
\coordinate (y2) at (.1,.3);
\coordinate (x1) at (-.1,.5);
\coordinate (x2) at (.1,.1);
\draw (-.3,-.8) .. controls ++(90:.3) and ++(270:.3) .. (-.015,-.2) -- (-.015,1.3);
\draw (.3,-.8) .. controls ++(90:.3) and ++(270:.3) .. (.015,-.2) -- (.015,1.3);
\RectangleMorphism{(x1)}{\gColor}
\RectangleMorphism{(x2)}{\kColor}
\RectangleMorphism{(y1)}{\hColor}
\RectangleMorphism{(y2)}{black}
\LongAlphaAction{(0,.4)}{.4}{.4}{.2}
}
\arrow[r, Rightarrow, "A(\phi)"]
&
\tikzmath{
\coordinate (y1) at (-.1,.7);
\coordinate (y2) at (.1,.5);
\coordinate (x1) at (-.1,.3);
\coordinate (x2) at (.1,.1);
\draw (-.3,-.8) .. controls ++(90:.3) and ++(270:.3) .. (-.015,-.2) -- (-.015,1.3);
\draw (.3,-.8) .. controls ++(90:.3) and ++(270:.3) .. (.015,-.2) -- (.015,1.3);
\RectangleMorphism{(x1)}{\gColor}
\RectangleMorphism{(x2)}{\kColor}
\RectangleMorphism{(y1)}{\hColor}
\RectangleMorphism{(y2)}{black}
\LongAlphaAction{(0,.4)}{.4}{.4}{.2}
}
&
\tikzmath{
\coordinate (y1) at (-.1,1.1);
\coordinate (y2) at (.1,.9);
\coordinate (x1) at (-.1,.5);
\coordinate (x2) at (.1,.3);
\draw (-.3,-.8) .. controls ++(90:.3) and ++(270:.3) .. (-.015,-.2) -- (-.015,1.7);
\draw (.3,-.8) .. controls ++(90:.3) and ++(270:.3) .. (.015,-.2) -- (.015,1.7);
\RectangleMorphism{(x1)}{\gColor}
\RectangleMorphism{(x2)}{\kColor}
\RectangleMorphism{(y1)}{\hColor}
\RectangleMorphism{(y2)}{black}
\LongAlphaAction{(0,.7)}{.4}{.5}{.3}
}
\arrow[l, Rightarrow, swap, "A^2_{y_1\xz y_2, x_1\xz x_2}"]
\end{tikzcd}};\end{tikzpicture}
$$
Non-labelled faces either commute by functoriality of 1-cell composition $\xo$,
axioms \ref{Interchanger:Composition} and \ref{Interchanger:Natural} of the interchanger,
or Remark \ref{rem:CombineAlphaCommutationSquares}.

\item[\ref{Functor:mu.unital}]
This follows by Remark \ref{rem:CombineAlphaCommutationSquares} and functoriality of 1-cell composition $\xo$, together with \ref{Functor:mu.unital} applied to $A$.

\item[\ref{Functor:iota}]
This part is automatic as $\iota^B_1 := B^1_e$.

\item[\ref{Functor:omega}]
This follows by Remark \ref{rem:CombineAlphaCommutationSquares} and functoriality of 1-cell composition $\xo$, together with \ref{Functor:omega} applied to $A$ and two instances of \ref{Transformation:AssociatorCoherence} for the transformation $\alpha: \pi^\cD \Rightarrow A \circ \pi^\cC$.

\item[\ref{Functor:lr}]
This follows by Remark \ref{rem:CombineAlphaCommutationSquares} and functoriality of 1-cell composition $\xo$, together with \ref{Functor:lr} applied to $A$ and two instances of \ref{Transformation:UnitCoherence} for the transformation $\alpha: \pi^\cD \Rightarrow A \circ \pi^\cC$.

\item[\ref{Functor:PentagonCoherence}] Every map is the identity map.

\item[\ref{Functor:TriangleCoherence}] Every map is the identity map.
\end{proof}

\begin{proof}[Proof of Lem.~\ref{lem:Truncation3FunctorGPointed}: $(\beta_*,\beta_g,\beta_{\id_g},\beta^1,\beta^2)$ is a 2-natural transformation $\pi^\cD\Rightarrow B\circ \pi^\cC$]

\item[\ref{Transformation:eta_c.natural}]
This condition is immediate as the only 1-cells and 2-cells in $\rmB G$ are identities.

\item[\ref{Transformation:eta_c.monoidal}]
This step amounts to checking $B^2_{\id_g,\id_g}\xt (B^1_g\xo B^1_g) = B^1_g$.
Using Remark \ref{rem:CombineAlphaCommutationSquares} and functoriality of 1-cell composition $\xo$, this reduces to the identity
$A^2_{\id_g,\id_g}\xt (A^1_g\xo A^1_g) = A^1_g$.

\item[\ref{Transformation:eta_c.unital}]
This condition is immediate as $(\pi^\cD)^1_g = \id_{\id_{g_\cD}}$ and $\beta^1_g = B^1_g$.

\item[\ref{Transformation:eta^1}] 
This condition is automatically satisfied.

\item[\ref{Transformation:eta^2}] 
This condition is immediate as $\mu^{\pi^\cD}_{g,h} = \id_{gh_\cD} = \mu^{B\circ \pi^\cC}_{g,h}$ and $\beta_g = \id_g$ for all $g\in G$.

\item[\ref{Transformation:AssociatorCoherence}]
Every map is the identity map.

\item[\ref{Transformation:UnitCoherence}]
Every map is the identity map.
\end{proof}

\begin{proof}[Proof of Thm.~\ref{thm:TruncationStrictifying1Morphisms}: $(\gamma,\id): (B,\beta) \Rightarrow (A,\alpha)$ is an invertible 2-morphism in $\TriCat_G$]
\mbox{}

It suffices to prove that $\gamma$ defines a 2-transformation $\gamma: B \Rightarrow A$, as it clearly invertible.

\item[\ref{Transformation:eta_c.natural}]
Every component which makes up $\gamma_x$ in \eqref{eq:DefinitionOfGamma_x} is natural in $x$.

\item[\ref{Transformation:eta_c.monoidal}]
This follows by Remark \ref{rem:CombineAlphaCommutationSquares}.

\item[\ref{Transformation:eta_c.unital}]
This follows by Remark \ref{rem:CombineAlphaCommutationSquares}.

\item[\ref{Transformation:eta^1}] 
This condition is automatically satisfied.

\item[\ref{Transformation:eta^2}] 
For $x\in \cC(g_\cC \to h_\cC)$ and $y\in \cC(k_\cC \to \ell_\cC)$, we use the following shorthand as in Notation \ref{nota:ShadedBoxes}:
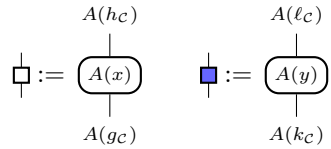
$$
\tikzmath{
\draw (0,-.3) -- (0,.3);
\filldraw[fill=\gColor, thick] (-.1,-.1) rectangle (.1,.1);
}
:=
\tikzmath{
\node (Aa) at (0,-.8) {$\scriptstyle A(g_\cC)$};
\node (Ab) at (0,.8) {$\scriptstyle A(h_\cC)$};
\node[draw,rectangle, thick, rounded corners=5pt] (Ax) at (0,0) {$\scriptstyle A(x)$};
\draw (Aa) to[in=-90,out=90] (Ax);
\draw (Ax) to[in=-90,out=90] (Ab);
}
\qquad
\tikzmath{
\draw (0,-.3) -- (0,.3);
\filldraw[fill=\kColor, thick] (-.1,-.1) rectangle (.1,.1);
}
:=
\tikzmath{
\node (Aa) at (0,-.8) {$\scriptstyle A(k_\cC)$};
\node (Ab) at (0,.8) {$\scriptstyle A(\ell_\cC)$};
\node[draw,rectangle, thick, rounded corners=5pt] (Ax) at (0,0) {$\scriptstyle A(y)$};
\draw (Aa) to[in=-90,out=90] (Ax);
\draw (Ax) to[in=-90,out=90] (Ab);
}
$$
For the following diagram to fit on one page, we compress the definition of $\mu^B_{x,y}$ from Construction \ref{const:B3Functor} into four steps as in \eqref{eq:CompressedMuB},
we suppress as many interchangers as possible, and we combine commuting squares involving only $\alpha_*$ and the $\alpha_g$ as in Remark \ref{rem:CombineAlphaCommutationSquares}.
$$
\begin{tikzpicture}[baseline= (a).base]\node[scale=.8] (a) at (0,0){\begin{tikzcd}
\tikzmath{
\coordinate (x) at (-.3,0);
\coordinate (y) at (.3,-.7);
\draw (-.3,-1.2) -- (-.3,.1) .. controls ++(90:.3) and ++(270:.3) .. (-.015,.6) -- (-.015,1);
\draw (.3,-1.2) -- (.3,.1) .. controls ++(90:.3) and ++(270:.3) .. (.015,.6) -- (.015,1);
\draw[thick, red, mid>] (-.8,-1.2) -- (-.8,.2) .. controls ++(90:.5) and ++(270:.5) .. (.4,1);
\RectangleMorphism{(x)}{\gColor}
\RectangleMorphism{(y)}{\kColor}
\AlphaAction{(x)}{.25}{.15}
\AlphaAction{(y)}{.25}{.15}
}
\arrow[r, Rightarrow]
\arrow[d, Rightarrow, "\alpha^2_{h,\ell}"]
&
\tikzmath{
\coordinate (x) at (-.2,-.1);
\coordinate (y) at (.2,-.5);
\draw (-.2,-1.2) -- (-.2,.1) .. controls ++(90:.3) and ++(270:.3) .. (-.015,.6) -- (-.015,1);
\draw (.2,-1.2) -- (.2,.1) .. controls ++(90:.3) and ++(270:.3) .. (.015,.6) -- (.015,1);
\draw[thick, red, mid>] (-.8,-1.2) -- (-.8,.2) .. controls ++(90:.5) and ++(270:.5) .. (.4,1);
\RectangleMorphism{(x)}{\gColor}
\RectangleMorphism{(y)}{\kColor}
\AlphaAction{(0,-.35)}{.5}{.3}
}
\arrow[r, Rightarrow, "(\alpha^2_{h,\ell})^{-1}"]
\arrow[d, Rightarrow, "\alpha^2_{h,\ell}"]
&
\tikzmath{
\coordinate (x) at (-.2,-.6);
\coordinate (y) at (.2,-.9);
\node[draw,rectangle, thick, rounded corners=5pt] (m) at (0,0) {$\scriptstyle \mu^A_{h,\ell}$};
\draw (-.2,-1.6) -- (x) to[in=-110,out=90] (m);
\draw (.2,-1.6) -- (y) to[in=-70,out=90] (m);
\draw[thick, red, mid>] (-1,-1.6) -- (-1,.6) .. controls ++(90:.5) and ++(270:.5) .. (.4,1.4);
\draw[double] (m) to[in=-90,out=90] (0,1.4);
\RectangleMorphism{(x)}{\gColor}
\RectangleMorphism{(y)}{\kColor}
\LongAlphaAction{(0,-.3)}{.7}{.6}{.3}
}
\arrow[r, Rightarrow, "\mu^A_{x,y}"]
\arrow[ddl, Rightarrow, "\gamma"]%_{\mu^A_{h,\ell} \xo (A(x)\xz A(y))}"]
&
\tikzmath{
\coordinate (x) at (-.1,.2);
\coordinate (y) at (.1,0);
\node[draw,rectangle, thick, rounded corners=5pt] (m) at (0,-.7) {$\scriptstyle \mu^A_{g,k}$};
\draw (-.2,-1.6) to[in=-110,out=90] (m);
\draw (.2,-1.6) to[in=-70,out=90] (m);
\draw[thick, red, mid>] (-1,-1.6) -- (-1,.6) .. controls ++(90:.5) and ++(270:.5) .. (.4,1.4);
\draw[double] (m) to[in=-90,out=90] (0,1.4);
\RectangleMorphism{(x)}{\gColor}
\RectangleMorphism{(y)}{\kColor}
\LongAlphaAction{(0,-.4)}{.7}{.6}{.3}
}
\arrow[d, Rightarrow, "(\alpha^2_{g,k})^{_1}"]
\arrow[ddl, Rightarrow, swap, "\gamma"]%_{A(x\xz y) \xo \mu^A_{g,k}}"]
\\
\tikzmath{
\coordinate (x) at (-.3,-.4);
\coordinate (y) at (.3,-1.1);
\node[draw,rectangle, thick, rounded corners=5pt] (m) at (0,1) {$\scriptstyle \mu^A_{h,\ell}$};
\draw (-.3,-1.6) -- (x) to[in=-105,out=90] (m);
\draw (.3,-1.6) -- (y) to[in=-75,out=90] (m);
\draw[thick, red, mid>] (-.8,-1.6) -- (-.8,-.1) .. controls ++(90:.5) and ++(270:.5) .. (.7,.7) -- (.7,1.8);
\draw[double] (m) to[in=-90,out=90] (0,1.8);
\RectangleMorphism{(x)}{\gColor}
\RectangleMorphism{(y)}{\kColor}
\AlphaAction{(x)}{.25}{.15}
\AlphaAction{(y)}{.25}{.15}
}
\arrow[r, Rightarrow]
\arrow[d, Rightarrow, "\gamma_x"]
&
\tikzmath{
\coordinate (x) at (-.2,-.4);
\coordinate (y) at (.2,-.8);
\node[draw,rectangle, thick, rounded corners=5pt] (m) at (0,1) {$\scriptstyle \mu^A_{h,\ell}$};
\draw (-.2,-1.5) -- (x) to[in=-100,out=90] (m);
\draw (.2,-1.5) -- (y) to[in=-80,out=90] (m);
\draw[thick, red, mid>] (-.8,-1.5) -- (-.8,-.1) .. controls ++(90:.5) and ++(270:.5) .. (.7,.7) -- (.7,1.8);
\draw[double] (m) to[in=-90,out=90] (0,1.8);
\RectangleMorphism{(x)}{\gColor}
\RectangleMorphism{(y)}{\kColor}
\AlphaAction{(0,-.65)}{.5}{.3}
}
\arrow[d, Rightarrow, "\gamma"]
&&
\tikzmath{
\coordinate (x) at (-.1,.5);
\coordinate (y) at (.1,.3);
\draw (-.3,-.8) .. controls ++(90:.3) and ++(270:.3) .. (-.015,-.2) -- (-.015,1.4);
\draw (.3,-.8) .. controls ++(90:.3) and ++(270:.3) .. (.015,-.2) -- (.015,1.4);
\draw[thick, red, mid>] (-.7,-.8) -- (-.7,.6) .. controls ++(90:.5) and ++(270:.5) .. (.4,1.4);
\RectangleMorphism{(x)}{\gColor}
\RectangleMorphism{(y)}{\kColor}
\AlphaAction{(0,.4)}{.4}{.2}
}
\arrow[d, Rightarrow, "\gamma_{x \xz y}"]
\\
\tikzmath{
\coordinate (x) at (-.3,.2);
\coordinate (y) at (.3,-.7);
\node[draw,rectangle, thick, rounded corners=5pt] (m) at (0,1) {$\scriptstyle \mu^A_{h,\ell}$};
\draw (-.3,-1.2) -- (x) to[in=-105,out=90] (m);
\draw (.3,-1.2) -- (y) to[in=-75,out=90] (m);
\draw[thick, red, mid>] (-.8,-1.2) -- (-.8,-.6) .. controls ++(90:.5) and ++(270:.5) .. (.7,.2) -- (.7,1.8);
\draw[double] (m) to[in=-90,out=90] (0,1.8);
\RectangleMorphism{(x)}{\gColor}
\RectangleMorphism{(y)}{\kColor}
\AlphaAction{(y)}{.25}{.15}
}
\arrow[r, Rightarrow, "\gamma_y"]
&
\tikzmath{
\coordinate (x) at (-.2,-.6);
\coordinate (y) at (.2,-.9);
\node[draw,rectangle, thick, rounded corners=5pt] (m) at (0,0) {$\scriptstyle \mu^A_{h,\ell}$};
\draw (-.2,-1.6) -- (x) to[in=-110,out=90] (m);
\draw (.2,-1.6) -- (y) to[in=-70,out=90] (m);
\draw[thick, red, mid>] (-.7,-1.6) .. controls ++(90:.5) and ++(270:.5) .. (.7,-.9) -- (.7,.8);
\draw[double] (m) to[in=-90,out=90] (0,.8);
\RectangleMorphism{(x)}{\gColor}
\RectangleMorphism{(y)}{\kColor}
}
\arrow[r,Rightarrow, "\mu^A_{x,y}"]
&
\tikzmath{
\coordinate (x) at (-.1,.3);
\coordinate (y) at (.1,.1);
\node[draw,rectangle, thick, rounded corners=5pt] (m) at (0,-.7) {$\scriptstyle \mu^A_{g,k}$};
\draw (-.2,-1.6) to[in=-110,out=90] (m);
\draw (.2,-1.6) to[in=-70,out=90] (m);
\draw[thick, red, mid>] (-.7,-1.6) .. controls ++(90:.5) and ++(270:.5) .. (.7,-.9) -- (.7,.8);
\draw[double] (m) to[in=-90,out=90] (0,.8);
\RectangleMorphism{(x)}{\gColor}
\RectangleMorphism{(y)}{\kColor}
}
&
\tikzmath{
\coordinate (x) at (-.1,.5);
\coordinate (y) at (.1,.3);
\draw (-.3,-.8) .. controls ++(90:.3) and ++(270:.3) .. (-.015,-.2) -- (-.015,1);
\draw (.3,-.8) .. controls ++(90:.3) and ++(270:.3) .. (.015,-.2) -- (.015,1);
\draw[thick, red, mid>] (-.7,-.8) -- (-.7,-.6) .. controls ++(90:.5) and ++(270:.5) .. (.4,.2) -- (.4,1);
\RectangleMorphism{(x)}{\gColor}
\RectangleMorphism{(y)}{\kColor}
}
\arrow[l,Rightarrow, swap, "\alpha^2_{g,k}"]
\end{tikzcd}};\end{tikzpicture}
$$
Every square here commutes by 
properties of the biadjoint biequivalence $\alpha_*$ \eqref{eq:BiadjointBiequivalence}, the adjoint equivalences $\alpha_g$, and functoriality of 1-cell composition $\xo$.
%Here, many steps have been compressed as discussed in Remark \ref{rem:CombineAlphaCommutationSquares}.

\item[\ref{Transformation:AssociatorCoherence}]
Since $\omega^B_{g,h,k}$ is the identity, it is equal to $\omega^{\pi^{\cD}}_{g,h,k}$.
Since  \ref{Transformation:AssociatorCoherence} holds for $\alpha: \pi^\cD \Rightarrow A\circ \pi^\cC$,
we conclude  \ref{Transformation:AssociatorCoherence} holds for $\zeta: B\Rightarrow A$.

\item[\ref{Transformation:UnitCoherence}]
Since $\ell^B_g,r^B_g$ are identities, they are equal to $ \ell^{\pi^{\cD}}_g, r^{\pi^{\cD}}_g$ respectively.
Since \ref{Transformation:UnitCoherence} holds for $\alpha: \pi^\cD \Rightarrow A\circ \pi^\cC$,
we conclude \ref{Transformation:UnitCoherence} holds for $\zeta: B\Rightarrow A$.
\end{proof}

%%%%%%%%%%%%%%%%%%%%%%%%%%%%%%%%%%%%%%%%%%%%%%%%%%
\subsection{Coherence proofs for Strictifying 2-morphisms \S\ref{sec:Strictifying2Morphisms}}
\label{sec:CoherenceProofs2Morphisms}

We remind the reader that 
$(\rmB\cC,\pi^\cC), (\rmB\cD,\pi^\cD)$ are two objects in $\TriCat_G^{\pt}$,
$(A,\alpha),(B,\beta) \in \TriCat_G^{\pt}((\rmB\cC,\pi^\cC)\to (\rmB\cD,\pi^\cD))$,
and $(\eta,m)\in \TriCat_G((A,\alpha) \Rightarrow (B,\beta))$.
We specified data for $(\zeta,\id): (A,\alpha)\Rightarrow (B,\beta)$ above in \S\ref{sec:Strictifying2Morphisms}.

\begin{proof}[Proof of Lem.~\ref{lem:TruncationDefine2Morphism}: $(\zeta,\id):(A,\alpha)\Rightarrow (B,\beta)$ is a 2-morphism in $\TriCat_G$]
It suffices to check that $\zeta: A \Rightarrow B$ is a 2-natural transformation.

\item[\ref{Transformation:eta_c.natural}]
Every component which makes up $\zeta_x$ in \eqref{eq:DefinitionOfZeta_x} is natural in $x$.

\item[\ref{Transformation:eta_c.monoidal}]
This follows by Remark \ref{rem:CombineAlphaCommutationSquares} and functoriality of 1-cell composition $\xo$, together with \ref{Transformation:eta_c.monoidal} applied to $\alpha: \pi^\cD \Rightarrow A \circ \pi^\cC$.

\item[\ref{Transformation:eta_c.unital}]
This follows by Remark \ref{rem:CombineAlphaCommutationSquares} and functoriality of 1-cell composition $\xo$, together with \ref{Transformation:eta_c.unital} applied to $\alpha: \pi^\cD \Rightarrow A \circ \pi^\cC$.

\item[\ref{Transformation:eta^1}] 
This condition is automatically satisfied.

\item[\ref{Transformation:eta^2}] 
For $x\in \cC(g_\cC \to h_\cC)$ and $y\in \cC(k_\cC \to \ell_\cC)$, we use the following shorthand as in Notation \ref{nota:ShadedBoxes}:
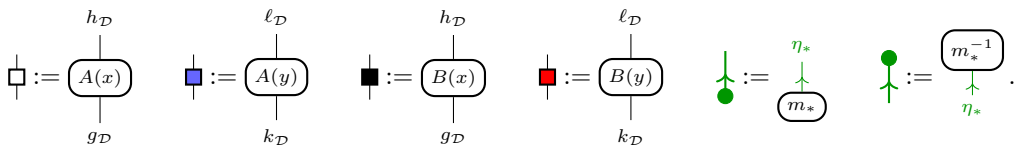
$$
\tikzmath{
\draw (0,-.3) -- (0,.3);
\filldraw[fill=\gColor, thick] (-.1,-.1) rectangle (.1,.1);
}
:=
\tikzmath{
\node (Aa) at (0,-.8) {$\scriptstyle g_\cD$};
\node (Ab) at (0,.8) {$\scriptstyle h_\cD$};
\node[draw,rectangle, thick, rounded corners=5pt] (Ax) at (0,0) {$\scriptstyle A(x)$};
\draw (Aa) to[in=-90,out=90] (Ax);
\draw (Ax) to[in=-90,out=90] (Ab);
}
\qquad
\tikzmath{
\draw (0,-.3) -- (0,.3);
\filldraw[fill=\kColor, thick] (-.1,-.1) rectangle (.1,.1);
}
:=
\tikzmath{
\node (Aa) at (0,-.8) {$\scriptstyle k_\cD$};
\node (Ab) at (0,.8) {$\scriptstyle \ell_\cD$};
\node[draw,rectangle, thick, rounded corners=5pt] (Ax) at (0,0) {$\scriptstyle A(y)$};
\draw (Aa) to[in=-90,out=90] (Ax);
\draw (Ax) to[in=-90,out=90] (Ab);
}
\qquad
\tikzmath{
\draw (0,-.3) -- (0,.3);
\filldraw[fill=black, thick] (-.1,-.1) rectangle (.1,.1);
}
:=
\tikzmath{
\node (Aa) at (0,-.8) {$\scriptstyle g_\cD$};
\node (Ab) at (0,.8) {$\scriptstyle h_\cD$};
\node[draw,rectangle, thick, rounded corners=5pt] (Ax) at (0,0) {$\scriptstyle B(x)$};
\draw (Aa) to[in=-90,out=90] (Ax);
\draw (Ax) to[in=-90,out=90] (Ab);
}
\qquad
\tikzmath{
\draw (0,-.3) -- (0,.3);
\filldraw[fill=red, thick] (-.1,-.1) rectangle (.1,.1);
}
:=
\tikzmath{
\node (Aa) at (0,-.8) {$\scriptstyle k_\cD$};
\node (Ab) at (0,.8) {$\scriptstyle \ell_\cD$};
\node[draw,rectangle, thick, rounded corners=5pt] (Ax) at (0,0) {$\scriptstyle B(y)$};
\draw (Aa) to[in=-90,out=90] (Ax);
\draw (Ax) to[in=-90,out=90] (Ab);
}
\qquad
\tikzmath{
\draw[thick, DarkGreen, mid>] (0,0) -- (0,.6);
\filldraw[DarkGreen] (0,0) circle (.1cm);
}
:=
\tikzmath{
\node[DarkGreen] (Ab) at (0,.8) {$\scriptstyle \eta_*$};
\node[draw,rectangle, thick, rounded corners=5pt] (Ax) at (0,0) {$\scriptstyle m_*$};
\draw[DarkGreen,mid>] (Ax) to[in=-90,out=90] (Ab);
}
\qquad
\tikzmath{
\draw[thick, DarkGreen, mid>] (0,-.6) -- (0,0);
\filldraw[DarkGreen] (0,0) circle (.1cm);
}
:=
\tikzmath{
\node[DarkGreen] (Aa) at (0,-.8) {$\scriptstyle \eta_*$};
\node[draw,rectangle, thick, rounded corners=5pt] (Ax) at (0,0) {$\scriptstyle m_*^{-1}$};
\draw[DarkGreen,mid>] (Aa) to[in=-90,out=90] (Ax);
}\,.
$$
For the following diagram to fit on one page, we compress the definition of $\zeta_x$ from \eqref{eq:DefinitionOfZeta_x} into three steps
$$
\zeta_x
:=
\left(
\tikzmath{
\coordinate (x) at (0,-.3);
\draw (0,-.7) -- (0,.7);
\RectangleMorphism{(x)}{\gColor}
}
\Rightarrow
\tikzmath{
\coordinate (x) at (0,-.3);
\draw (0,-.7) -- (0,.7);
\coordinate (m1) at (-.3,-.5);
\coordinate (m2) at (.3,.5);
\draw[thick, DarkGreen, mid>] ($ (m1) + (0,.1) $) to[in=-90,out=90] ($ (m2) - (0,.1) $);
\RectangleMorphism{(x)}{\gColor}
\filldraw[DarkGreen] (m1) circle (.1cm);
\filldraw[DarkGreen] (m2) circle (.1cm);
}
\overset{\eta_x}{\Rightarrow}
\tikzmath{
\coordinate (x) at (0,.3);
\draw (0,-.7) -- (0,.7);
\coordinate (m1) at (-.3,-.5);
\coordinate (m2) at (.3,.5);
\draw[thick, DarkGreen, mid>] ($ (m1) + (0,.1) $) to[in=-90,out=90] ($ (m2) - (0,.1) $);
\RectangleMorphism{(x)}{black}
\filldraw[DarkGreen] (m1) circle (.1cm);
\filldraw[DarkGreen] (m2) circle (.1cm);
}
\Rightarrow
\tikzmath{
\coordinate (x) at (0,.3);
\draw (0,-.7) -- (0,.7);
\RectangleMorphism{(x)}{black}
}
\right),
$$
and we combine and suppress as many interchangers as possible, simply writing $\phi$.
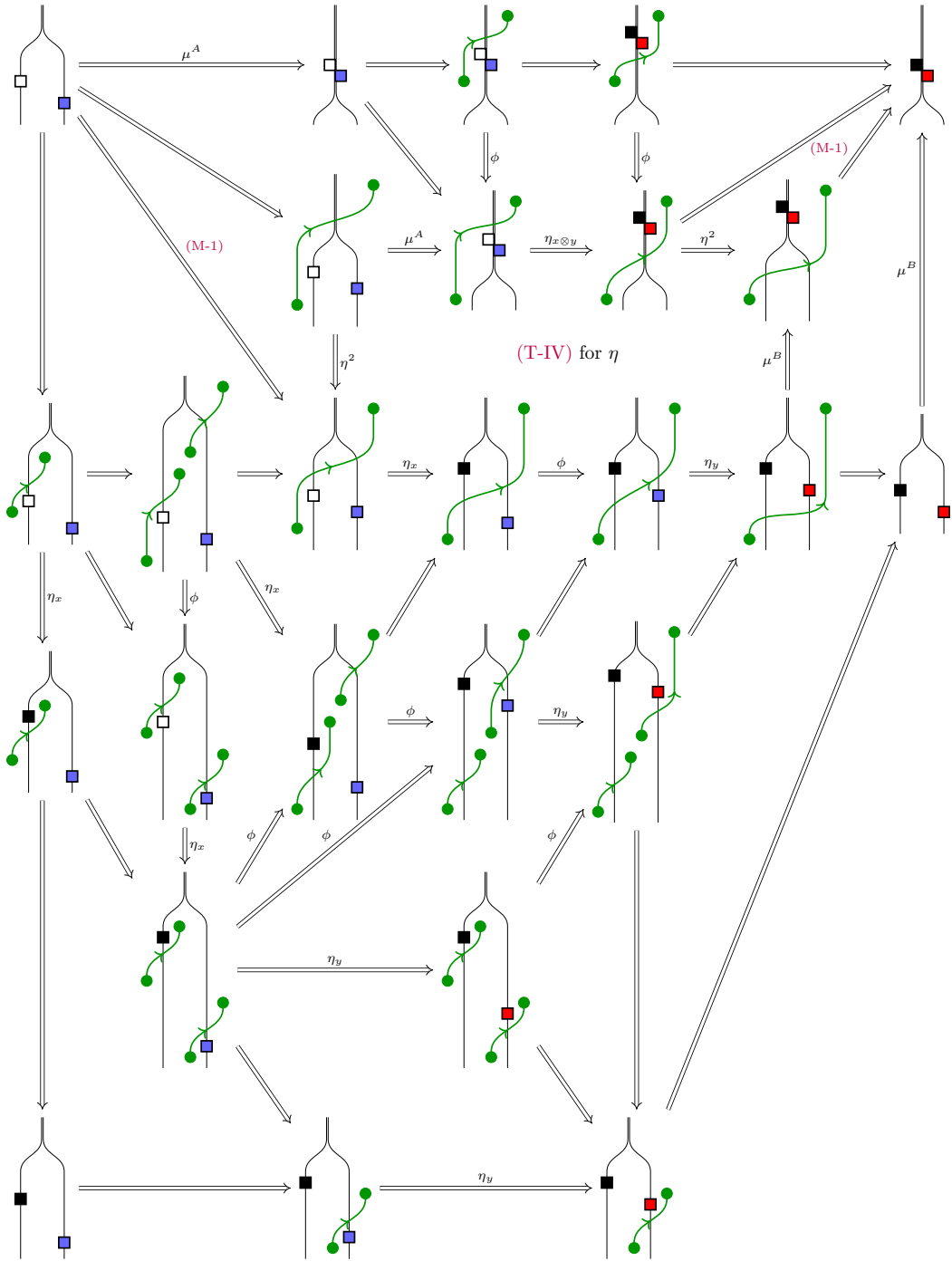
$$
\begin{tikzpicture}[baseline= (a).base]\node[scale=.8] (a) at (0,0){\begin{tikzcd}
\tikzmath{
\coordinate (x) at (-.4,-.4);
\coordinate (y) at (.4,-.8);
\draw (-.4,-1.2) -- (-.4,0) .. controls ++(90:.3) and ++(270:.3) .. (-.015,.6) -- (-.015,1);
\draw (.4,-1.2) -- (.4,0) .. controls ++(90:.3) and ++(270:.3) .. (.015,.6) -- (.015,1);
\RectangleMorphism{(x)}{\gColor}
\RectangleMorphism{(y)}{\kColor}
}
\arrow[rr,Rightarrow, "\mu^A"]
\arrow[drr,Rightarrow]
\arrow[dd,Rightarrow]
\arrow[ddrr, Rightarrow,"\text{\ref{Modification:MonoidalCoherence}}"]
&&
\tikzmath{
\coordinate (x) at (-.1,1.1);
\coordinate (y) at (.1,.9);
\draw (-.4,0) .. controls ++(90:.3) and ++(270:.3) .. (-.015,.6) -- (-.015,2.2);
\draw (.4,0) .. controls ++(90:.3) and ++(270:.3) .. (.015,.6) -- (.015,2.2);
\RectangleMorphism{(x)}{\gColor}
\RectangleMorphism{(y)}{\kColor}
}
\arrow[r,Rightarrow]
\arrow[dr,Rightarrow]
&
\tikzmath{
\coordinate (x) at (-.1,1.3);
\coordinate (y) at (.1,1.1);
\coordinate (m1) at (-.4,.8);
\coordinate (m2) at (.4,2);
\draw (-.4,0) .. controls ++(90:.3) and ++(270:.3) .. (-.015,.6) -- (-.015,2.2);
\draw (.4,0) .. controls ++(90:.3) and ++(270:.3) .. (.015,.6) -- (.015,2.2);
\draw[thick, DarkGreen, mid>] ($ (m1) + (0,.1) $) -- ($ (m1) + (0,.5) $) to[in=-90,out=90] ($ (m2) - (0,.1) $);
\RectangleMorphism{(x)}{\gColor}
\RectangleMorphism{(y)}{\kColor}
\filldraw[DarkGreen] (m1) circle (.1cm);
\filldraw[DarkGreen] (m2) circle (.1cm);
}
\arrow[r,Rightarrow]
\arrow[d,Rightarrow, "\phi"]
&
\tikzmath{
\coordinate (x) at (-.1,1.7);
\coordinate (y) at (.1,1.5);
\coordinate (m1) at (-.4,.8);
\coordinate (m2) at (.4,2);
\draw (-.4,0) .. controls ++(90:.3) and ++(270:.3) .. (-.015,.6) -- (-.015,2.2);
\draw (.4,0) .. controls ++(90:.3) and ++(270:.3) .. (.015,.6) -- (.015,2.2);
\draw[thick, DarkGreen, mid>] ($ (m1) + (0,.1) $) to[in=-90,out=90] ($ (m2) - (0,.5) $)  -- ($ (m2) - (0,.1) $);
\RectangleMorphism{(x)}{black}
\RectangleMorphism{(y)}{red}
\filldraw[DarkGreen] (m1) circle (.1cm);
\filldraw[DarkGreen] (m2) circle (.1cm);
}
\arrow[rr,Rightarrow]
\arrow[d,Rightarrow, "\phi"]
&&
\tikzmath{
\coordinate (x) at (-.1,1.1);
\coordinate (y) at (.1,.9);
\draw (-.4,0) .. controls ++(90:.3) and ++(270:.3) .. (-.015,.6) -- (-.015,2.2);
\draw (.4,0) .. controls ++(90:.3) and ++(270:.3) .. (.015,.6) -- (.015,2.2);
\RectangleMorphism{(x)}{black}
\RectangleMorphism{(y)}{red}
}
\\
&&
\tikzmath{
\coordinate (x) at (-.4,-.2);
\coordinate (y) at (.4,-.5);
\coordinate (m1) at (-.7,-.8);
\coordinate (m2) at (.7,1.4);
\draw (-.4,-1.2) -- (-.4,0) .. controls ++(90:.3) and ++(270:.3) .. (-.015,.6) -- (-.015,1.6);
\draw (.4,-1.2) -- (.4,0) .. controls ++(90:.3) and ++(270:.3) .. (.015,.6) -- (.015,1.6);
\draw[thick, DarkGreen, mid>] ($ (m1) + (0,.1) $) -- ($ (m1) + (0,1) $) to[in=-90,out=90] ($ (m2) - (0,.1) $);
\RectangleMorphism{(x)}{\gColor}
\RectangleMorphism{(y)}{\kColor}
\filldraw[DarkGreen] (m1) circle (.1cm);
\filldraw[DarkGreen] (m2) circle (.1cm);
}
\arrow[r, Rightarrow, "\mu^A"]
\arrow[d, Rightarrow, "\eta^2"]
&
\tikzmath{
\coordinate (x) at (-.1,1.3);
\coordinate (y) at (.1,1.1);
\coordinate (m1) at (-.7,.2);
\coordinate (m2) at (.4,2);
\draw (-.4,0) -- (-.4,.2) .. controls ++(90:.3) and ++(270:.3) .. (-.015,.8) -- (-.015,2.2);
\draw (.4,0) -- (.4,.2) .. controls ++(90:.3) and ++(270:.3) .. (.015,.8) -- (.015,2.2);
\draw[thick, DarkGreen, mid>] ($ (m1) + (0,.1) $) -- ($ (m1) + (0,1) $) to[in=-90,out=90] ($ (m2) - (0,.1) $);
\RectangleMorphism{(x)}{\gColor}
\RectangleMorphism{(y)}{\kColor}
\filldraw[DarkGreen] (m1) circle (.1cm);
\filldraw[DarkGreen] (m2) circle (.1cm);
}
\arrow[r,Rightarrow, "\eta_{x\xz y}"]
\arrow[d, phantom, "\text{\hspace{3cm}\ref{Transformation:eta^2} for $\eta$}"]
&
\tikzmath{
\coordinate (x) at (-.1,1.7);
\coordinate (y) at (.1,1.5);
\coordinate (m1) at (-.7,.2);
\coordinate (m2) at (.4,2);
\draw (-.4,0) -- (-.4,.2) .. controls ++(90:.3) and ++(270:.3) .. (-.015,.8) -- (-.015,2.2);
\draw (.4,0) -- (.4,.2) .. controls ++(90:.3) and ++(270:.3) .. (.015,.8) -- (.015,2.2);
\draw[thick, DarkGreen, mid>] ($ (m1) + (0,.1) $) to[in=-90,out=90] ($ (m2) - (0,.5) $)  -- ($ (m2) - (0,.1) $);
\RectangleMorphism{(x)}{black}
\RectangleMorphism{(y)}{red}
\filldraw[DarkGreen] (m1) circle (.1cm);
\filldraw[DarkGreen] (m2) circle (.1cm);
}
\arrow[urr, Rightarrow]
\arrow[r, Rightarrow, "\eta^2"]
&
\tikzmath{
\coordinate (x) at (-.1,1.1);
\coordinate (y) at (.1,.9);
\coordinate (m1) at (-.7,-.6);
\coordinate (m2) at (.7,1.4);
\draw (-.4,-1) -- (-.4,0) .. controls ++(90:.3) and ++(270:.3) .. (-.015,.6) -- (-.015,1.6);
\draw (.4,-1) -- (.4,0) .. controls ++(90:.3) and ++(270:.3) .. (.015,.6) -- (.015,1.6);
\draw[thick, DarkGreen, mid>] ($ (m1) + (0,.1) $) to[in=-90,out=90] ($ (m2) - (0,1) $) -- ($ (m2) - (0,.1) $);
\RectangleMorphism{(x)}{black}
\RectangleMorphism{(y)}{red}
\filldraw[DarkGreen] (m1) circle (.1cm);
\filldraw[DarkGreen] (m2) circle (.1cm);
}
\arrow[ur,Rightarrow, near start, "\text{\ref{Modification:MonoidalCoherence}}"]
\\
\tikzmath{
\coordinate (x) at (-.4,-.8);
\coordinate (y) at (.4,-1.3);
\coordinate (m1) at (-.7,-1);
\coordinate (m2) at (-.1,0);
\draw (-.4,-1.6) -- (-.4,0) .. controls ++(90:.3) and ++(270:.3) .. (-.015,.6) -- (-.015,1);
\draw (.4,-1.6) -- (.4,0) .. controls ++(90:.3) and ++(270:.3) .. (.015,.6) -- (.015,1);
\draw[thick, DarkGreen, mid>] ($ (m1) + (0,.1) $) to[in=-90,out=90] ($ (m2) - (0,.1) $);
\RectangleMorphism{(x)}{\gColor}
\RectangleMorphism{(y)}{\kColor}
\filldraw[DarkGreen] (m1) circle (.1cm);
\filldraw[DarkGreen] (m2) circle (.1cm);
}
\arrow[r,Rightarrow]
\arrow[d,Rightarrow, "\eta_x"]
\arrow[dr, Rightarrow]
&
\tikzmath{
\coordinate (x) at (-.4,-1.6);
\coordinate (y) at (.4,-2);
\coordinate (m1) at (-.7,-2.4);
\coordinate (m2) at (-.1,-.8);
\coordinate (m3) at (.1,-.4);
\coordinate (m4) at (.7,.8);
\draw (-.4,-2.6) -- (-.4,0) .. controls ++(90:.3) and ++(270:.3) .. (-.015,.6) -- (-.015,1);
\draw (.4,-2.6) -- (.4,0) .. controls ++(90:.3) and ++(270:.3) .. (.015,.6) -- (.015,1);
\draw[thick, DarkGreen, mid>] ($ (m1) + (0,.1) $) -- ($ (m1) + (0,.7) $) to[in=-90,out=90] ($ (m2) - (0,.1) $);
\draw[thick, DarkGreen, mid>] ($ (m3) + (0,.1) $) to[in=-90,out=90] ($ (m4) - (0,.1) $);
\RectangleMorphism{(x)}{\gColor}
\RectangleMorphism{(y)}{\kColor}
\filldraw[DarkGreen] (m1) circle (.1cm);
\filldraw[DarkGreen] (m2) circle (.1cm);
\filldraw[DarkGreen] (m3) circle (.1cm);
\filldraw[DarkGreen] (m4) circle (.1cm);
}
\arrow[r,Rightarrow]
\arrow[d,Rightarrow, "\phi"]
\arrow[dr,Rightarrow, "\eta_x"]
&
\tikzmath{
\coordinate (x) at (-.4,-.2);
\coordinate (y) at (.4,-.5);
\coordinate (m1) at (-.7,-.8);
\coordinate (m2) at (.7,1.4);
\draw (-.4,-1.2) -- (-.4,.4) .. controls ++(90:.3) and ++(270:.3) .. (-.015,1) -- (-.015,1.6);
\draw (.4,-1.2) -- (.4,.4) .. controls ++(90:.3) and ++(270:.3) ..(.015,1) -- (.015,1.6);
\draw[thick, DarkGreen, mid>] ($ (m1) + (0,.1) $) -- ($ (m1) + (0,.6) $) to[in=-90,out=90] ($ (m2) - (0,.5) $) -- ($ (m2) - (0,.1) $);
\RectangleMorphism{(x)}{\gColor}
\RectangleMorphism{(y)}{\kColor}
\filldraw[DarkGreen] (m1) circle (.1cm);
\filldraw[DarkGreen] (m2) circle (.1cm);
}
\arrow[r, Rightarrow, "\eta_x"]
&
\tikzmath{
\coordinate (x) at (-.4,.3);
\coordinate (y) at (.4,-.7);
\coordinate (m1) at (-.7,-1);
\coordinate (m2) at (.7,1.4);
\draw (-.4,-1.2) -- (-.4,.4) .. controls ++(90:.3) and ++(270:.3) .. (-.015,1) -- (-.015,1.6);
\draw (.4,-1.2) -- (.4,.4) .. controls ++(90:.3) and ++(270:.3) ..(.015,1) -- (.015,1.6);
\draw[thick, DarkGreen, mid>] ($ (m1) + (0,.1) $) -- ($ (m1) + (0,.3) $) to[in=-90,out=90] ($ (m2) - (0,1) $) -- ($ (m2) - (0,.1) $);
\RectangleMorphism{(x)}{black}
\RectangleMorphism{(y)}{\kColor}
\filldraw[DarkGreen] (m1) circle (.1cm);
\filldraw[DarkGreen] (m2) circle (.1cm);
}
\arrow[r, Rightarrow, "\phi"]
&
\tikzmath{
\coordinate (x) at (-.4,.3);
\coordinate (y) at (.4,-.2);
\coordinate (m1) at (-.7,-1);
\coordinate (m2) at (.7,1.4);
\draw (-.4,-1.2) -- (-.4,.4) .. controls ++(90:.3) and ++(270:.3) .. (-.015,1) -- (-.015,1.6);
\draw (.4,-1.2) -- (.4,.4) .. controls ++(90:.3) and ++(270:.3) ..(.015,1) -- (.015,1.6);
\draw[thick, DarkGreen, mid>] ($ (m1) + (0,.1) $) to[in=-90,out=90] ($ (m2) - (0,.7) $) -- ($ (m2) - (0,.1) $);
\RectangleMorphism{(x)}{black}
\RectangleMorphism{(y)}{\kColor}
\filldraw[DarkGreen] (m1) circle (.1cm);
\filldraw[DarkGreen] (m2) circle (.1cm);
}
\arrow[r, Rightarrow, "\eta_y"]
&
\tikzmath{
\coordinate (x) at (-.4,.3);
\coordinate (y) at (.4,-.1);
\coordinate (m1) at (-.7,-1);
\coordinate (m2) at (.7,1.4);
\draw (-.4,-1.2) -- (-.4,.4) .. controls ++(90:.3) and ++(270:.3) .. (-.015,1) -- (-.015,1.6);
\draw (.4,-1.2) -- (.4,.4) .. controls ++(90:.3) and ++(270:.3) .. (.015,1) -- (.015,1.6);
\draw[thick, DarkGreen, mid>] ($ (m1) + (0,.1) $) to[in=-90,out=90] ($ (m2) - (0,1.6) $) -- ($ (m2) - (0,.1) $);
\RectangleMorphism{(x)}{black}
\RectangleMorphism{(y)}{red}
\filldraw[DarkGreen] (m1) circle (.1cm);
\filldraw[DarkGreen] (m2) circle (.1cm);
}
\arrow[u,Rightarrow, "\mu^B"]
\arrow[r, Rightarrow]
&
\tikzmath{
\coordinate (x) at (-.4,-.4);
\coordinate (y) at (.4,-.8);
\draw (-.4,-1.2) -- (-.4,0) .. controls ++(90:.3) and ++(270:.3) .. (-.015,.6) -- (-.015,1);
\draw (.4,-1.2) -- (.4,0) .. controls ++(90:.3) and ++(270:.3) .. (.015,.6) -- (.015,1);
\RectangleMorphism{(x)}{black}
\RectangleMorphism{(y)}{red}
}
\arrow[uu, Rightarrow, "\mu^B"]
\\
\tikzmath{
\coordinate (x) at (-.4,-.2);
\coordinate (y) at (.4,-1.3);
\coordinate (m1) at (-.7,-1);
\coordinate (m2) at (-.1,0);
\draw (-.4,-1.6) -- (-.4,0) .. controls ++(90:.3) and ++(270:.3) .. (-.015,.6) -- (-.015,1);
\draw (.4,-1.6) -- (.4,0) .. controls ++(90:.3) and ++(270:.3) .. (.015,.6) -- (.015,1);
\draw[thick, DarkGreen, mid>] ($ (m1) + (0,.1) $) to[in=-90,out=90] ($ (m2) - (0,.1) $);
\RectangleMorphism{(x)}{black}
\RectangleMorphism{(y)}{\kColor}
\filldraw[DarkGreen] (m1) circle (.1cm);
\filldraw[DarkGreen] (m2) circle (.1cm);
}
\arrow[dd, Rightarrow]
\arrow[dr, Rightarrow]
&
\tikzmath{
\coordinate (x) at (-.4,-.8);
\coordinate (y) at (.4,-2.2);
\coordinate (m1) at (-.7,-1);
\coordinate (m2) at (-.1,0);
\coordinate (m3) at (.1,-2.4);
\coordinate (m4) at (.7,-1.4);
\draw (-.4,-2.6) -- (-.4,0) .. controls ++(90:.3) and ++(270:.3) .. (-.015,.6) -- (-.015,1);
\draw (.4,-2.6) -- (.4,0) .. controls ++(90:.3) and ++(270:.3) .. (.015,.6) -- (.015,1);
\draw[thick, DarkGreen, mid>] ($ (m1) + (0,.1) $) to[in=-90,out=90] ($ (m2) - (0,.1) $);
\draw[thick, DarkGreen, mid>] ($ (m3) + (0,.1) $) to[in=-90,out=90] ($ (m4) - (0,.1) $);
\RectangleMorphism{(x)}{\gColor}
\RectangleMorphism{(y)}{\kColor}
\filldraw[DarkGreen] (m1) circle (.1cm);
\filldraw[DarkGreen] (m2) circle (.1cm);
\filldraw[DarkGreen] (m3) circle (.1cm);
\filldraw[DarkGreen] (m4) circle (.1cm);
}
\arrow[d,Rightarrow, "\eta_x"]
&
\tikzmath{
\coordinate (x) at (-.4,-1.2);
\coordinate (y) at (.4,-2);
\coordinate (m1) at (-.7,-2.4);
\coordinate (m2) at (-.1,-.8);
\coordinate (m3) at (.1,-.4);
\coordinate (m4) at (.7,.8);
\draw (-.4,-2.6) -- (-.4,0) .. controls ++(90:.3) and ++(270:.3) .. (-.015,.6) -- (-.015,1);
\draw (.4,-2.6) -- (.4,0) .. controls ++(90:.3) and ++(270:.3) .. (.015,.6) -- (.015,1);
\draw[thick, DarkGreen, mid>] ($ (m1) + (0,.1) $) to[in=-90,out=90] ($ (m2) - (0,.5) $) -- ($ (m2) - (0,.1) $);
\draw[thick, DarkGreen, mid>] ($ (m3) + (0,.1) $) to[in=-90,out=90] ($ (m4) - (0,.1) $);
\RectangleMorphism{(x)}{black}
\RectangleMorphism{(y)}{\kColor}
\filldraw[DarkGreen] (m1) circle (.1cm);
\filldraw[DarkGreen] (m2) circle (.1cm);
\filldraw[DarkGreen] (m3) circle (.1cm);
\filldraw[DarkGreen] (m4) circle (.1cm);
}
\arrow[ur, Rightarrow]
\arrow[r, Rightarrow,"\phi"]
&
\tikzmath{
\coordinate (x) at (-.4,-.1);
\coordinate (y) at (.4,-.5);
\coordinate (m1) at (-.7,-2.4);
\coordinate (m2) at (-.1,-1.4);
\coordinate (m3) at (.1,-1);
\coordinate (m4) at (.7,.8);
\draw (-.4,-2.6) -- (-.4,0) .. controls ++(90:.3) and ++(270:.3) .. (-.015,.6) -- (-.015,1);
\draw (.4,-2.6) -- (.4,0) .. controls ++(90:.3) and ++(270:.3) .. (.015,.6) -- (.015,1);
\draw[thick, DarkGreen, mid>] ($ (m1) + (0,.1) $) to[in=-90,out=90] ($ (m2) - (0,.1) $);
\draw[thick, DarkGreen, mid>] ($ (m3) + (0,.1) $) -- ($ (m3) + (0,.4) $) to[in=-90,out=90] ($ (m4) - (0,.1) $);
\RectangleMorphism{(x)}{black}
\RectangleMorphism{(y)}{\kColor}
\filldraw[DarkGreen] (m1) circle (.1cm);
\filldraw[DarkGreen] (m2) circle (.1cm);
\filldraw[DarkGreen] (m3) circle (.1cm);
\filldraw[DarkGreen] (m4) circle (.1cm);
}
\arrow[ur, Rightarrow]
\arrow[r, Rightarrow, "\eta_y"]
&
\tikzmath{
\coordinate (x) at (-.4,.1);
\coordinate (y) at (.4,-.2);
\coordinate (m1) at (-.7,-2.4);
\coordinate (m2) at (-.1,-1.4);
\coordinate (m3) at (.1,-1);
\coordinate (m4) at (.7,.9);
\draw (-.4,-2.6) -- (-.4,.1) .. controls ++(90:.3) and ++(270:.3) .. (-.015,.7) -- (-.015,1.1);
\draw (.4,-2.6) -- (.4,.1) .. controls ++(90:.3) and ++(270:.3) .. (.015,.7) -- (.015,1.1);
\draw[thick, DarkGreen, mid>] ($ (m1) + (0,.1) $) to[in=-90,out=90] ($ (m2) - (0,.1) $);
\draw[thick, DarkGreen, mid>] ($ (m3) + (0,.1) $) to[in=-90,out=90] ($ (m4) - (0,1.1) $) -- ($ (m4) - (0,.1) $);
\RectangleMorphism{(x)}{black}
\RectangleMorphism{(y)}{red}
\filldraw[DarkGreen] (m1) circle (.1cm);
\filldraw[DarkGreen] (m2) circle (.1cm);
\filldraw[DarkGreen] (m3) circle (.1cm);
\filldraw[DarkGreen] (m4) circle (.1cm);
}
\arrow[ur, Rightarrow]
\arrow[dd, Rightarrow]
\\
&
\tikzmath{
\coordinate (x) at (-.4,-.2);
\coordinate (y) at (.4,-2.2);
\coordinate (m1) at (-.7,-1);
\coordinate (m2) at (-.1,0);
\coordinate (m3) at (.1,-2.4);
\coordinate (m4) at (.7,-1.4);
\draw (-.4,-2.6) -- (-.4,0) .. controls ++(90:.3) and ++(270:.3) .. (-.015,.6) -- (-.015,1);
\draw (.4,-2.6) -- (.4,0) .. controls ++(90:.3) and ++(270:.3) .. (.015,.6) -- (.015,1);
\draw[thick, DarkGreen, mid>] ($ (m1) + (0,.1) $) to[in=-90,out=90] ($ (m2) - (0,.1) $);
\draw[thick, DarkGreen, mid>] ($ (m3) + (0,.1) $) to[in=-90,out=90] ($ (m4) - (0,.1) $);
\RectangleMorphism{(x)}{black}
\RectangleMorphism{(y)}{\kColor}
\filldraw[DarkGreen] (m1) circle (.1cm);
\filldraw[DarkGreen] (m2) circle (.1cm);
\filldraw[DarkGreen] (m3) circle (.1cm);
\filldraw[DarkGreen] (m4) circle (.1cm);
}
\arrow[dr, Rightarrow]
\arrow[ur, Rightarrow, "\phi"]
\arrow[urr, Rightarrow, "\phi"]
\arrow[rr, Rightarrow, "\eta_y"]
&&
\tikzmath{
\coordinate (x) at (-.4,-.2);
\coordinate (y) at (.4,-1.6);
\coordinate (m1) at (-.7,-1);
\coordinate (m2) at (-.1,0);
\coordinate (m3) at (.1,-2.4);
\coordinate (m4) at (.7,-1.4);
\draw (-.4,-2.6) -- (-.4,0) .. controls ++(90:.3) and ++(270:.3) .. (-.015,.6) -- (-.015,1);
\draw (.4,-2.6) -- (.4,0) .. controls ++(90:.3) and ++(270:.3) .. (.015,.6) -- (.015,1);
\draw[thick, DarkGreen, mid>] ($ (m1) + (0,.1) $) to[in=-90,out=90] ($ (m2) - (0,.1) $);
\draw[thick, DarkGreen, mid>] ($ (m3) + (0,.1) $) to[in=-90,out=90] ($ (m4) - (0,.1) $);
\RectangleMorphism{(x)}{black}
\RectangleMorphism{(y)}{red}
\filldraw[DarkGreen] (m1) circle (.1cm);
\filldraw[DarkGreen] (m2) circle (.1cm);
\filldraw[DarkGreen] (m3) circle (.1cm);
\filldraw[DarkGreen] (m4) circle (.1cm);
}
\arrow[dr, Rightarrow]
\arrow[ur, Rightarrow, "\phi"]
\\
\tikzmath{
\coordinate (x) at (-.4,-.5);
\coordinate (y) at (.4,-1.3);
\draw (-.4,-1.6) -- (-.4,0) .. controls ++(90:.3) and ++(270:.3) .. (-.015,.6) -- (-.015,1);
\draw (.4,-1.6) -- (.4,0) .. controls ++(90:.3) and ++(270:.3) .. (.015,.6) -- (.015,1);
\RectangleMorphism{(x)}{black}
\RectangleMorphism{(y)}{\kColor}
}
\arrow[rr, Rightarrow]
&&
\tikzmath{
\coordinate (x) at (-.4,-.2);
\coordinate (y) at (.4,-1.2);
\coordinate (m1) at (.1,-1.4);
\coordinate (m2) at (.7,-.4);
\draw (-.4,-1.6) -- (-.4,0) .. controls ++(90:.3) and ++(270:.3) .. (-.015,.6) -- (-.015,1);
\draw (.4,-1.6) -- (.4,0) .. controls ++(90:.3) and ++(270:.3) .. (.015,.6) -- (.015,1);
\draw[thick, DarkGreen, mid>] ($ (m1) + (0,.1) $) to[in=-90,out=90] ($ (m2) - (0,.1) $);
\RectangleMorphism{(x)}{black}
\RectangleMorphism{(y)}{\kColor}
\filldraw[DarkGreen] (m1) circle (.1cm);
\filldraw[DarkGreen] (m2) circle (.1cm);
}
\arrow[rr, Rightarrow, "\eta_y"]
&&
\tikzmath{
\coordinate (x) at (-.4,-.2);
\coordinate (y) at (.4,-.6);
\coordinate (m1) at (.1,-1.4);
\coordinate (m2) at (.7,-.4);
\draw (-.4,-1.6) -- (-.4,0) .. controls ++(90:.3) and ++(270:.3) .. (-.015,.6) -- (-.015,1);
\draw (.4,-1.6) -- (.4,0) .. controls ++(90:.3) and ++(270:.3) .. (.015,.6) -- (.015,1);
\draw[thick, DarkGreen, mid>] ($ (m1) + (0,.1) $) to[in=-90,out=90] ($ (m2) - (0,.1) $);
\RectangleMorphism{(x)}{black}
\RectangleMorphism{(y)}{red}
\filldraw[DarkGreen] (m1) circle (.1cm);
\filldraw[DarkGreen] (m2) circle (.1cm);
}
\arrow[uuurr, Rightarrow]
&&
\end{tikzcd}};\end{tikzpicture}
$$
Non-labelled faces commute by either functoriality of 1-cell composition $\xo$ or by properties of a (bi)adjoint (bi)equivalence.

\item[\ref{Transformation:AssociatorCoherence}]
Every map is the identity map.

\item[\ref{Transformation:UnitCoherence}]
Every map is the identity map.
\end{proof}

We remind the reader that by the strictness properties for $\alpha, \beta$ as components of 1-morphisms in $\TriCat_G^{\pt}$, 
$m_* : e_\cD \Rightarrow \eta_*$, and $m_g$ is an invertible 2-cell
\begin{equation}
\tag{\ref{eq:mgIso}}
\tikzmath{
\node (g1) at (0,0) {$\scriptstyle g_\cD$};
\node (g2) at (0,1) {};
\node[draw,rectangle, thick, rounded corners=5pt] (p) at (-.5,.5) {$\scriptstyle m_*$};
\draw (g1) to[in=-90,out=90] (g2);
}
\overset{m_g}{\Rightarrow}
\tikzmath{
\node (g1) at (0,0) {$\scriptstyle g_\cD$};
\node (g2) at (0,1) {};
\node[draw,rectangle, thick, rounded corners=5pt] (p) at (.5,.5) {$\scriptstyle m_*$};
\draw (g1) to[in=-90,out=90] (g2);
}
\,.
\end{equation}

\begin{proof}[Proof of Thm.~\ref{thm:TruncationStrictify2Morphisms}: $(m,\id): (\zeta,\id)\Rrightarrow (\eta,m)$ is an invertible 3-morphism in $\TriCat_G$]
It suffices to check that $m: \zeta\Rrightarrow \eta$ is an invertible 3-modification.

\item[\ref{Modification:m_c}]
This condition corresponds for $m: \zeta \Rrightarrow \eta$ corresponds to the outside of the following commutative diagram, where we use the following shorthand notation for $x\in \cC(g_\cC \to h_\cC)$ and $m_*, m_*^{-1}$:
$$
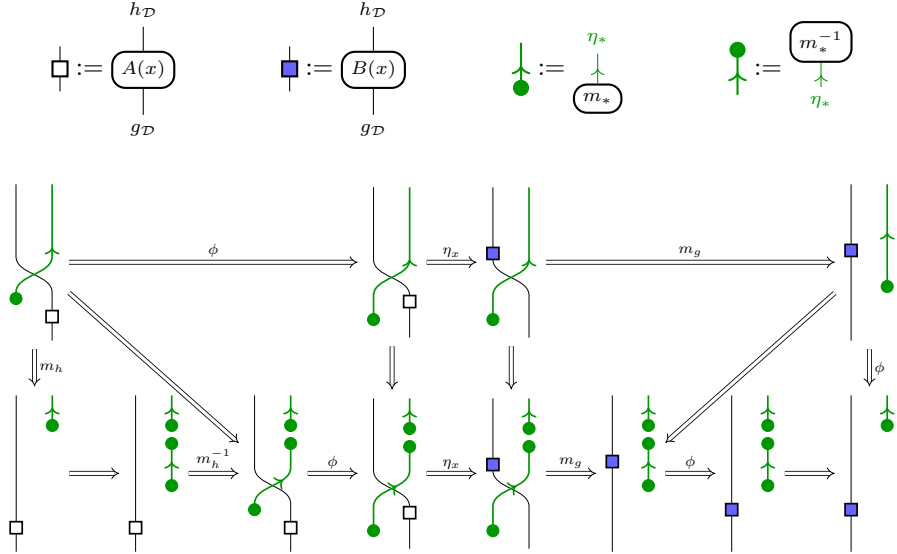
\tikzmath{
\draw (0,-.3) -- (0,.3);
\filldraw[fill=\gColor, thick] (-.1,-.1) rectangle (.1,.1);
}
:=
\tikzmath{
\node (Aa) at (0,-.8) {$\scriptstyle g_\cD$};
\node (Ab) at (0,.8) {$\scriptstyle h_\cD$};
\node[draw,rectangle, thick, rounded corners=5pt] (Ax) at (0,0) {$\scriptstyle A(x)$};
\draw (Aa) to[in=-90,out=90] (Ax);
\draw (Ax) to[in=-90,out=90] (Ab);
}
\qquad\qquad
\tikzmath{
\draw (0,-.3) -- (0,.3);
\filldraw[fill=\kColor, thick] (-.1,-.1) rectangle (.1,.1);
}
:=
\tikzmath{
\node (Aa) at (0,-.8) {$\scriptstyle g_\cD$};
\node (Ab) at (0,.8) {$\scriptstyle h_\cD$};
\node[draw,rectangle, thick, rounded corners=5pt] (Ax) at (0,0) {$\scriptstyle B(x)$};
\draw (Aa) to[in=-90,out=90] (Ax);
\draw (Ax) to[in=-90,out=90] (Ab);
}
\qquad\qquad
\tikzmath{
\draw[thick, DarkGreen, mid>] (0,0) -- (0,.6);
\filldraw[DarkGreen] (0,0) circle (.1cm);
}
:=
\tikzmath{
\node[DarkGreen] (Ab) at (0,.8) {$\scriptstyle \eta_*$};
\node[draw,rectangle, thick, rounded corners=5pt] (Ax) at (0,0) {$\scriptstyle m_*$};
\draw[DarkGreen,mid>] (Ax) to[in=-90,out=90] (Ab);
}
\qquad\qquad
\tikzmath{
\draw[thick, DarkGreen, mid>] (0,-.6) -- (0,0);
\filldraw[DarkGreen] (0,0) circle (.1cm);
}
:=
\tikzmath{
\node[DarkGreen] (Aa) at (0,-.8) {$\scriptstyle \eta_*$};
\node[draw,rectangle, thick, rounded corners=5pt] (Ax) at (0,0) {$\scriptstyle m_*^{-1}$};
\draw[DarkGreen,mid>] (Aa) to[in=-90,out=90] (Ax);
}
$$
$$
\begin{tikzpicture}[baseline= (a).base]\node[scale=.8] (a) at (0,0){\begin{tikzcd}
\tikzmath{
\pgfmathsetmacro{\extra}{.2}
\coordinate (Ax) at (.3,-.7);
\coordinate (Bx) at (-.3,.4);
\coordinate (m1) at (-.3,-.4);
\coordinate (m2) at (.3,.4);
\draw ($ (Ax) - 2*(0,\extra) $) -- ($ (Ax) + 2*(0,\extra) $) .. controls ++(90:.3) and ++(270:.3) .. ($ (Bx) + (0,-.1) $) -- (-.3,1.5);
\draw[thick, DarkGreen, mid>] ($ (m1) + (0,.1) $) .. controls ++(90:.3) and ++(270:.3) .. ($ (m2) + (0,-.1) $) -- (.3,1.5);
\RectangleMorphism{(Ax)}{\gColor}
\filldraw[DarkGreen] (m1) circle (.1cm);
}
\arrow[d,Rightarrow,"m_h"]
\arrow[rrr,Rightarrow,"\phi"]
\arrow[drr,Rightarrow]
&&&
\tikzmath{
\pgfmathsetmacro{\extra}{.2}
\coordinate (Ax) at (.3,-.4);
\coordinate (Bx) at (-.3,.4);
\coordinate (m1) at (-.3,-.7);
\coordinate (m2) at (.3,.4);
\draw ($ (Ax) - 3*(0,\extra) $) -- ($ (Ax) + (0,.1) $) .. controls ++(90:.3) and ++(270:.3) .. ($ (Bx) + (0,-.1) $) -- (-.3,1.5);
\draw[thick, DarkGreen, mid>] ($ (m1) + (0,.1) $) -- ($ (m1) + 2*(0,\extra) $) .. controls ++(90:.3) and ++(270:.3) .. ($ (m2) + (0,-.1) $) -- (.3,1.5);
\RectangleMorphism{(Ax)}{\gColor}
\filldraw[DarkGreen] (m1) circle (.1cm);
}
\arrow[r,Rightarrow,"\eta_x"]
\arrow[d,Rightarrow]
&
\tikzmath{
\pgfmathsetmacro{\extra}{.2}
\coordinate (Ax) at (.3,-.4);
\coordinate (Bx) at (-.3,.4);
\coordinate (m1) at (-.3,-.7);
\coordinate (m2) at (.3,.4);
\draw ($ (Ax) - 3*(0,\extra) $) -- ($ (Ax) + (0,.1) $) .. controls ++(90:.3) and ++(270:.3) .. ($ (Bx) + (0,-.1) $) -- (-.3,1.5);
\draw[thick, DarkGreen, mid>] ($ (m1) + (0,.1) $) -- ($ (m1) + 2*(0,\extra) $) .. controls ++(90:.3) and ++(270:.3) .. ($ (m2) + (0,-.1) $) -- (.3,1.5);
\RectangleMorphism{(Bx)}{\kColor}
\filldraw[DarkGreen] (m1) circle (.1cm);
}
\arrow[rrr,Rightarrow,"m_g"]
\arrow[d,Rightarrow]
&&&
\tikzmath{
\pgfmathsetmacro{\extra}{.2}
\coordinate (Ax) at (-.3,.4);
\coordinate (m1) at (.3,-.2);
\draw (-.3,-1.1) -- (-.3,1.5);
\draw[thick, DarkGreen, mid>] ($ (m1) + (0,.1) $) -- (.3,1.5);
\RectangleMorphism{(Ax)}{\kColor}
\filldraw[DarkGreen] (m1) circle (.1cm);
}
\arrow[d,Rightarrow,"\phi"]
\arrow[dll,Rightarrow]
\\
\tikzmath{
\pgfmathsetmacro{\extra}{.2}
\coordinate (Ax) at (-.3,-.7);
\coordinate (Bx) at (-.3,.4);
\coordinate (m1) at (.3,1);
\draw ($ (Ax) - 2*(0,\extra) $) -- (-.3,1.5);
\draw[thick, DarkGreen, mid>] ($ (m1) + (0,.1) $) -- (.3,1.5);
\RectangleMorphism{(Ax)}{\gColor}
\filldraw[DarkGreen] (m1) circle (.1cm);
}
\arrow[r,Rightarrow]
&
\tikzmath{
\pgfmathsetmacro{\extra}{.2}
\coordinate (Ax) at (-.3,-.7);
\coordinate (Bx) at (-.3,.4);
\coordinate (m3) at (.3,1);
\coordinate (m2) at (.3,.7);
\coordinate (m1) at (.3,0);
\draw ($ (Ax) - 2*(0,\extra) $) -- (-.3,1.5);
\draw[thick, DarkGreen, mid>] ($ (m3) + (0,.1) $) -- (.3,1.5);
\draw[thick, DarkGreen, mid>] ($ (m1) + (0,.1) $) -- (m2);
\RectangleMorphism{(Ax)}{\gColor}
\filldraw[DarkGreen] (m1) circle (.1cm);
\filldraw[DarkGreen] (m2) circle (.1cm);
\filldraw[DarkGreen] (m3) circle (.1cm);
}
\arrow[r,Rightarrow, "m_h^{-1}"]
&
\tikzmath{
\pgfmathsetmacro{\extra}{.2}
\coordinate (Ax) at (.3,-.7);
\coordinate (Bx) at (-.3,.4);
\coordinate (m1) at (-.3,-.4);
\coordinate (m3) at (.3,1);
\coordinate (m2) at (.3,.7);
\draw ($ (Ax) - 2*(0,\extra) $) -- ($ (Ax) + 2*(0,\extra) $) .. controls ++(90:.3) and ++(270:.3) .. ($ (Bx) + (0,-.1) $) -- (-.3,1.5);
\draw[thick, DarkGreen, mid>] ($ (m1) + (0,.1) $) .. controls ++(90:.3) and ++(270:.3) .. ($ (m2) + (0,-.4) $) -- (m2); 
\draw[thick, DarkGreen, mid>] ($ (m3) + (0,.1) $) -- (.3,1.5);
\RectangleMorphism{(Ax)}{\gColor}
\filldraw[DarkGreen] (m1) circle (.1cm);
\filldraw[DarkGreen] (m2) circle (.1cm);
\filldraw[DarkGreen] (m3) circle (.1cm);
}
\arrow[r,Rightarrow, "\phi"]
&
\tikzmath{
\pgfmathsetmacro{\extra}{.2}
\coordinate (Ax) at (.3,-.4);
\coordinate (Bx) at (-.3,.4);
\coordinate (m1) at (-.3,-.7);
\coordinate (m2) at (.3,.4);
\draw ($ (Ax) - 3*(0,\extra) $) -- ($ (Ax) + (0,.1) $) .. controls ++(90:.3) and ++(270:.3) .. ($ (Bx) + (0,-.1) $) -- (-.3,1.5);
\draw[thick, DarkGreen, mid>] ($ (m1) + (0,.1) $) -- ($ (m1) + 2*(0,\extra) $) .. controls ++(90:.3) and ++(270:.3) .. ($ (m2) + (0,-.1) $) -- ($ (m2) + 2*(0,\extra) $);
\RectangleMorphism{(Ax)}{\gColor}
\draw[thick, DarkGreen, mid>] ($ (m2) + (0,.7) $) -- ($ (m2) + (0,1.1) $);
\filldraw[DarkGreen] (m1) circle (.1cm);
\filldraw[DarkGreen] ($ (m2) + (0,.3) $) circle (.1cm);
\filldraw[DarkGreen] ($ (m2) + (0,.6) $) circle (.1cm);
}
\arrow[r,Rightarrow, "\eta_x"]
&
\tikzmath{
\pgfmathsetmacro{\extra}{.2}
\coordinate (Ax) at (.3,-.4);
\coordinate (Bx) at (-.3,.4);
\coordinate (m1) at (-.3,-.7);
\coordinate (m2) at (.3,.4);
\draw ($ (Ax) - 3*(0,\extra) $) -- ($ (Ax) + (0,.1) $) .. controls ++(90:.3) and ++(270:.3) .. ($ (Bx) + (0,-.1) $) -- (-.3,1.5);
\draw[thick, DarkGreen, mid>] ($ (m1) + (0,.1) $) -- ($ (m1) + 2*(0,\extra) $) .. controls ++(90:.3) and ++(270:.3) .. ($ (m2) + (0,-.1) $) -- ($ (m2) + 2*(0,\extra) $);
\RectangleMorphism{(Bx)}{\kColor}
\draw[thick, DarkGreen, mid>] ($ (m2) + (0,.7) $) -- ($ (m2) + (0,1.1) $);
\filldraw[DarkGreen] (m1) circle (.1cm);
\filldraw[DarkGreen] ($ (m2) + (0,.3) $) circle (.1cm);
\filldraw[DarkGreen] ($ (m2) + (0,.6) $) circle (.1cm);
}
\arrow[r,Rightarrow, "m_g"]
&
\tikzmath{
\pgfmathsetmacro{\extra}{.2}
\coordinate (Ax) at (-.3,.4);
\coordinate (Bx) at (-.3,.4);
\coordinate (m3) at (.3,1);
\coordinate (m2) at (.3,.7);
\coordinate (m1) at (.3,0);
\draw (-.3,-1.1) -- (-.3,1.5);
\draw[thick, DarkGreen, mid>] ($ (m3) + (0,.1) $) -- (.3,1.5);
\draw[thick, DarkGreen, mid>] ($ (m1) + (0,.1) $) -- (m2);
\RectangleMorphism{(Ax)}{\kColor}
\filldraw[DarkGreen] (m1) circle (.1cm);
\filldraw[DarkGreen] (m2) circle (.1cm);
\filldraw[DarkGreen] (m3) circle (.1cm);
}
\arrow[r,Rightarrow,"\phi"]
&
\tikzmath{
\pgfmathsetmacro{\extra}{.2}
\coordinate (Ax) at (-.3,-.4);
\coordinate (m3) at (.3,1);
\coordinate (m2) at (.3,.7);
\coordinate (m1) at (.3,0);
\draw (-.3,-1.1) -- (-.3,1.5);
\draw[thick, DarkGreen, mid>] ($ (m3) + (0,.1) $) -- (.3,1.5);
\draw[thick, DarkGreen, mid>] ($ (m1) + (0,.1) $) -- (m2);
\RectangleMorphism{(Ax)}{\kColor}
\filldraw[DarkGreen] (m1) circle (.1cm);
\filldraw[DarkGreen] (m2) circle (.1cm);
\filldraw[DarkGreen] (m3) circle (.1cm);
}
\arrow[r,Rightarrow]
&
\tikzmath{
\pgfmathsetmacro{\extra}{.2}
\coordinate (Ax) at (-.3,-.4);
\coordinate (m1) at (.3,1);
\draw (-.3,-1.1) -- (-.3,1.5);
\draw[thick, DarkGreen, mid>] ($ (m1) + (0,.1) $) -- (.3,1.5);
\RectangleMorphism{(Ax)}{\kColor}
\filldraw[DarkGreen] (m1) circle (.1cm);
}
\end{tikzcd}};\end{tikzpicture}
$$
All inner faces in the above diagram are squares which commute by functoriality of 1-cell composition $\xo$.

\item[\ref{Modification:MonoidalCoherence}]
This is exactly \ref{Modification:MonoidalCoherence} applied to $m$ viewed as a modification $m: \beta \Rrightarrow (\eta \xo \id_{\pi^\cC})\xt \alpha$ as in \eqref{eq:ModificationConditionFor2Morphisms} above
    
\item[\ref{Modification:UnitCoherence}]
By the strictness properties of $(A,\alpha)$ and $(B,\beta)$, \ref{Modification:UnitCoherence} for the modification $m: \beta \Rrightarrow (\eta \xo \id_{\pi^\cC})\xt \alpha$ as in \eqref{eq:ModificationConditionFor2Morphisms} above tells us that $m_{1_\cC}=\eta^1$ on the nose.
This exactly gives the coherence \ref{Modification:UnitCoherence} for $\gamma$.
\end{proof}

%%%%%%%%%%%%%%%%%%%%%%%%%%%%%%%%%%%%%%%%%%%%%%%%%%
\subsection{Coherence proofs for Strictifying 3-morphisms \S\ref{sec:Strictifying3Morphisms}}
\label{sec:CoherenceProofs3Morphisms}

We remind the reader that in this section,
$(\eta,m=\id),(\zeta, n=\id): (A,\alpha) \Rightarrow (B,\beta)$ are two 2-morphisms in $\TriCat_G^{\pt}$ and $(p,\rho):(\eta, \id) \Rrightarrow (\zeta, \id)$ is a 3-morphism in $\TriCat_G$.
This means
$\rho_*$ is an invertible 2-cell $\id_{\id_{e_\cD}} \Rightarrow p_*$ satisfying
the coherence
\begin{equation}
\tag{\ref{eq:PerturbationEquationForRho}}
\left(
\tikzmath{
\node (g1) at (0,0) {$\scriptstyle g_\cD$};
\node (g2) at (0,1) {};
\draw[dashed, thick, rounded corners=5pt] (-.2,.3) rectangle (-.8,.7);
\draw (g1) to[in=-90,out=90] (g2);
}
\overset{\rho_*}{\Rightarrow}
\tikzmath{
\node (g1) at (0,0) {$\scriptstyle g_\cD$};
\node (g2) at (0,1) {};
\node[draw,rectangle, thick, rounded corners=5pt] (p) at (-.5,.5) {$\scriptstyle p_*$};
\draw (g1) to[in=-90,out=90] (g2);
}
\overset{p_g}{\Rightarrow}
\tikzmath{
\node (g1) at (0,0) {$\scriptstyle g_\cD$};
\node (g2) at (0,1) {};
\node[draw,rectangle, thick, rounded corners=5pt] (p) at (.5,.5) {$\scriptstyle p_*$};
\draw (g1) to[in=-90,out=90] (g2);
}
\right)
=
\left(
\tikzmath{
\node (g1) at (0,0) {$\scriptstyle g_\cD$};
\node (g2) at (0,1) {};
\draw[dashed, thick, rounded corners=5pt] (.2,.3) rectangle (.8,.7);
\draw (g1) to[in=-90,out=90] (g2);
}
\overset{\rho_*}{\Rightarrow}
\tikzmath{
\node (g1) at (0,0) {$\scriptstyle g_\cD$};
\node (g2) at (0,1) {};
\node[draw,rectangle, thick, rounded corners=5pt] (p) at (.5,.5) {$\scriptstyle p_*$};
\draw (g1) to[in=-90,out=90] (g2);
}
\right)
\qquad\qquad \forall\,g\in G.
\end{equation}

\begin{proof}[Proof of Lem.~\ref{lem:Only3Endos}: $\eta_x=\zeta_x$ for all $x\in \cC(g_\cC\to h_\cC)$]
For $x\in\cC(g_\cC \to h_\cC)$, 
we use the following shorthand as in Notation \ref{nota:ShadedBoxes}.
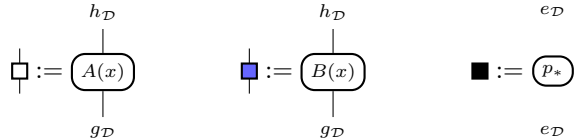
$$
\tikzmath{
\draw (0,-.3) -- (0,.3);
\filldraw[fill=\gColor, thick] (-.1,-.1) rectangle (.1,.1);
}
:=
\tikzmath{
\node (Aa) at (0,-.8) {$\scriptstyle g_\cD$};
\node (Ab) at (0,.8) {$\scriptstyle h_\cD$};
\node[draw,rectangle, thick, rounded corners=5pt] (Ax) at (0,0) {$\scriptstyle A(x)$};
\draw (Aa) to[in=-90,out=90] (Ax);
\draw (Ax) to[in=-90,out=90] (Ab);
}
\qquad\qquad
\tikzmath{
\draw (0,-.3) -- (0,.3);
\filldraw[fill=\kColor, thick] (-.1,-.1) rectangle (.1,.1);
}
:=
\tikzmath{
\node (Aa) at (0,-.8) {$\scriptstyle g_\cD$};
\node (Ab) at (0,.8) {$\scriptstyle h_\cD$};
\node[draw,rectangle, thick, rounded corners=5pt] (Ax) at (0,0) {$\scriptstyle B(x)$};
\draw (Aa) to[in=-90,out=90] (Ax);
\draw (Ax) to[in=-90,out=90] (Ab);
}
\qquad\qquad
\tikzmath{
\filldraw[fill=black, thick] (-.1,-.1) rectangle (.1,.1);
}
:=
\tikzmath{
\node (Aa) at (0,-.8) {$\scriptstyle e_\cD$};
\node (Ab) at (0,.8) {$\scriptstyle e_\cD$};
\node[draw,rectangle, thick, rounded corners=5pt] (Ax) at (0,0) {$\scriptstyle p_*$};
}
$$
The outside of the following commutative diagram is a bigon with one arrow $\eta_x$ and one arrow $\zeta_x$; hence $\eta_x=\zeta_x$:
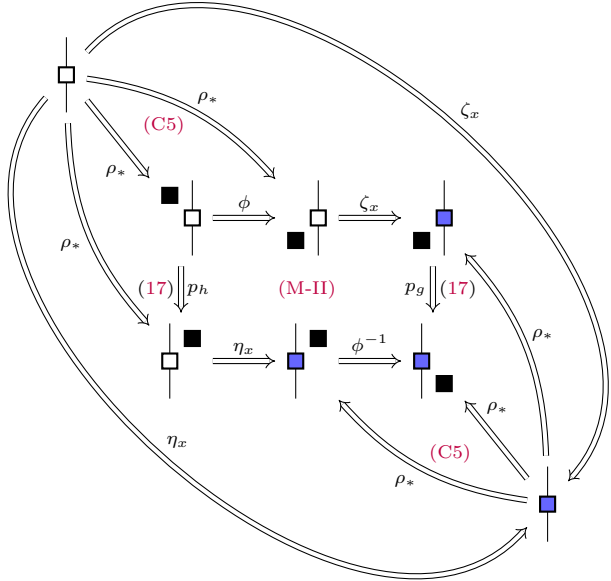
$$
\begin{tikzpicture}[baseline= (a).base]\clip (-4.5,-4.5) rectangle (4.5,4.5);\node[scale=1] (a) at (0,0){\begin{tikzcd}
\tikzmath{
\draw (0,-.5) -- (0,.5);
\RectangleMorphism{(0,0)}{\gColor}
}
\arrow[dddrrrr, Rightarrow, bend right=90, "\eta_x"]
\arrow[dddrrrr, Rightarrow, bend left=90, "\zeta_x"]
\arrow[dr, Rightarrow, swap, near end, "\rho_*"]
\arrow[drr, Rightarrow, bend left=20, "\rho_*", "\text{\scriptsize \ref{Interchanger:Natural}}"']
\arrow[ddr, Rightarrow, bend right=20,"\rho_*"']
\\
&
\tikzmath{
\draw (0,-.5) -- (0,.5);
\RectangleMorphism{(0,0)}{\gColor}
\RectangleMorphism{(-.3,.3)}{black}
}
\arrow[d, Rightarrow, "p_h", "\text{\scriptsize \eqref{eq:PerturbationEquationForRho}}"']
\arrow[r, Rightarrow, "\phi"]
&
\tikzmath{
\draw (0,-.5) -- (0,.5);
\RectangleMorphism{(0,0)}{\gColor}
\RectangleMorphism{(-.3,-.3)}{black}
}
\arrow[r, Rightarrow, "\zeta_x"]
\arrow[d, phantom, "\text{\scriptsize \ref{Modification:m_c}}"]
&
\tikzmath{
\draw (0,-.5) -- (0,.5);
\RectangleMorphism{(0,0)}{\kColor}
\RectangleMorphism{(-.3,-.3)}{black}
}
\arrow[d, Rightarrow, "\text{\scriptsize \eqref{eq:PerturbationEquationForRho}}", "p_g"']
\\
&
\tikzmath{
\draw (0,-.5) -- (0,.5);
\RectangleMorphism{(0,0)}{\gColor}
\RectangleMorphism{(.3,.3)}{black}
}
\arrow[r, Rightarrow, "\eta_x"]
&
\tikzmath{
\draw (0,-.5) -- (0,.5);
\RectangleMorphism{(0,0)}{\kColor}
\RectangleMorphism{(.3,.3)}{black}
}
\arrow[r, Rightarrow, "\phi^{-1}"]
&
\tikzmath{
\draw (0,-.5) -- (0,.5);
\RectangleMorphism{(0,0)}{\kColor}
\RectangleMorphism{(.3,-.3)}{black}
}
\\
&&&&
\tikzmath{
\draw (0,-.5) -- (0,.5);
\RectangleMorphism{(0,0)}{\kColor}
}
\arrow[ul, Rightarrow, swap, near end, "\rho_*"]
\arrow[ull, Rightarrow, bend left=20, "\rho_*","\text{\scriptsize \ref{Interchanger:Natural}}"']
\arrow[uul, Rightarrow, bend right=20, "\rho_*"']
\end{tikzcd}};\end{tikzpicture}
$$
The unlabeled faces commute by functoriality of 1-cell composition $\xo$.
\end{proof}

%auto-ignore
%this ensures the arxiv doesn't try to start TeXing here.
%!TEX root =../Equivalence.tex

%%%%%%%%%%%%%%%%%%%%%%%%%%%%%%%%%%%%%%%%%%%%%%%%%%
%%%%%%%%%%%%%%%%%%%%%%%%%%%%%%%%%%%%%%%%%%%%%%%%%%
%%%%%%%%%%%%%%%%%%%%%%%%%%%%%%%%%%%%%%%%%%%%%%%%%%
\section{Coherence proofs for \texorpdfstring{$G$}{G}-crossed braided categories}
\label{sec:CoherenceProofsGCrossed}

This appendix contains all proofs from \S\ref{sec:GCrossed} which amount to checking/using various coherence conditions using the properties listed in Appendix \ref{sec:Weak3CategoryCoherences}.
To make the commutative diagrams more readable, we suppress all whiskering notation, including Notation \ref{nota:DashedBoxForWhiskering}.

%%%%%%%%%%%%%%%%%%%%%%%%%%%%%%%%%%%%%%%%%%%%%%%%%%
\subsection{Coherence proofs for the 2-functor 
\texorpdfstring{$\TriCat_G^{\st}$}{TriCatGst}
to
\texorpdfstring{$G\CrsBrd^{\st}$}{GCrsBrdst} from \S\ref{sec:FromGBoring3CatsToGCrossedBriadedCats}}
\label{sec:CoherenceProofsForGCrossedObjects}

We now supply the proofs for statements in \S\ref{sec:FromGBoring3CatsToGCrossedBriadedCats}.
We remind the reader that $(\fC, \otimes_{g,h}, F_g, \beta_{g,h})$ is the data constructed from $\cC\in \TriCat_G^{\st}$ in Constructions \ref{const:UnderlyingCategories}, \ref{const:GAction}, and \ref{const:GCrossedBraiding}.

\begin{proof}[Proof of Thm.~\ref{thm:ConstructionOfGCrossedBraidedCategory}: 
$(\fC, \xz_{g,h}, F_g, \beta^{g,h})$ forms a strict $G$-crossed braided category]

We remind the reader that we use the shorthand notation that white, green, and blue shaded disks correspond to 1-morphisms into $g_\cC, h_\cC,$ and $k_\cC$, respectively:
$$
\tikzmath{
    \draw (0,0)  -- (0,.3) node [above] {\scriptsize{$g_\cC$}};
    \filldraw[fill=\gColor, thick] (0,0) circle (.1cm);
}
\qquad
\tikzmath{
    \draw (0,0)  -- (0,.3) node [above] {\scriptsize{$h_\cC$}};
    \filldraw[fill=\hColor, thick] (0,0) circle (.1cm);
}
\qquad
\tikzmath{
    \draw (0,0)  -- (0,.3) node [above] {\scriptsize{$k_\cC$}};
    \filldraw[fill=\kColor, thick] (0,0) circle (.1cm);
}\,.
$$
It remains to check the commutativity of \eqref{eq:Hexagon}, \eqref{eq:Heptagon1}, and \eqref{eq:Heptagon2}.
We treat \eqref{eq:Hexagon} in detail.
Going around the outside of the diagram below corresponds to \eqref{eq:Hexagon}.
The large face consists of only equalities, so it manifestly commutes.
\begin{equation*}
\tag{\ref{eq:Hexagon}}
\begin{tikzcd}
\tikzmath{
    \coordinate (a) at (-.4,0);
    \coordinate (b) at (.4,-.4);
    \draw (a) -- ($ (a) + (0,.4) $) node [above,xshift=.15cm] {$\scriptstyle ghg^{-1}_\cC$};
    \draw[red] ($ (a) + (-.2,.4) $) -- ($ (a) + (-.2,0) $) node [left,xshift=.1cm] {$\scriptstyle g$} arc (-180:0:.2cm) -- ($ (a) + (.2,.4) $); 
    \draw (b) -- ($ (b) + (0,.8) $) node [above,xshift=.15cm] {$\scriptstyle gkg^{-1}_\cC$};
    \draw[red] ($ (b) + (-.2,.8) $) -- ($ (b) + (-.2,0) $) node [left,xshift=.1cm] {$\scriptstyle g$} arc (-180:0:.2cm) -- ($ (b) + (.2,.8) $); 
    \filldraw[fill=\hColor, thick] (a) circle (.1cm);
    \filldraw[fill=\kColor, thick] (b) circle (.1cm);
}
\arrow[r, "\cong"]
&
\tikzmath{
    \coordinate (a) at (-.4,-.4);
    \coordinate (b) at (.4,0);
    \draw (a) -- ($ (a) + (0,.8) $);    
    \draw[red] ($ (a) + (-.2,.8) $) -- ($ (a) + (-.2,0) $) node [left,xshift=.1cm] {$\scriptstyle g$} arc (-180:0:.2cm) -- ($ (a) + (.2,.8) $); 
    \draw (b) -- ($ (b) + (0,.4) $);
    \draw[red] ($ (b) + (-.2,.4) $) -- ($ (b) + (-.2,0) $) node [left,xshift=.1cm] {$\scriptstyle g$} arc (-180:0:.2cm) -- ($ (b) + (.2,.4) $); 
    \filldraw[fill=\hColor, thick] (a) circle (.1cm);
    \filldraw[fill=\kColor, thick] (b) circle (.1cm);
}
\arrow[r,"="]
&
\tikzmath{
    \coordinate (a) at (-.4,0);
    \coordinate (b) at (.4,-.4);
    \draw (a) -- ($ (a) + (0,.4) $);
    \draw[red,double] ($ (a) + (-.2,.4) $) -- ($ (a) + (-.2,0) $) node [left,xshift=.1cm] {$\scriptstyle ghg^{-1},g$} arc (-180:0:.2cm) -- ($ (a) + (.2,.4) $); 
    \draw (b) -- ($ (b) + (0,.8) $);
    \draw[red] ($ (b) + (-.2,.8) $) -- ($ (b) + (-.2,0) $) node [left,xshift=.1cm] {$\scriptstyle g$} arc (-180:0:.2cm) -- ($ (b) + (.2,.8) $); 
    \filldraw[fill=\kColor, thick] (a) circle (.1cm);
    \filldraw[fill=\hColor, thick] (b) circle (.1cm);
}
\arrow[d,"="]
\\
\tikzmath{
    \coordinate (a) at (-.2,0);
    \coordinate (b) at (.2,-.2);
    \draw (a) -- ($ (a) + (0,.4) $);
    \draw[red] ($ (a) + (-.2,.4) $) -- ($ (a) + (-.2,-.2) $) node [left,xshift=.1cm] {$\scriptstyle g$} arc (-180:0:.4cm) -- ($ (b) + (.2,.6) $); 
    \draw (b) -- ($ (b) + (0,.6) $);
    \filldraw[fill=\hColor, thick] (a) circle (.1cm);
    \filldraw[fill=\kColor, thick] (b) circle (.1cm);
}
\arrow[u,"="]
\arrow[r, "\cong"]
&
\tikzmath{
    \coordinate (a) at (-.2,-.2);
    \coordinate (b) at (.2,0);
    \draw (a) -- ($ (a) + (0,.6) $);
    \draw[red] ($ (a) + (-.2,.6) $) -- ($ (a) + (-.2,0) $) node [left,xshift=.1cm] {$\scriptstyle g$} arc (-180:0:.4cm) -- ($ (b) + (.2,.4) $); 
    \draw (b) -- ($ (b) + (0,.4) $);
    \filldraw[fill=\hColor, thick] (a) circle (.1cm);
    \filldraw[fill=\kColor, thick] (b) circle (.1cm);
}
\arrow[dl,"="]
\arrow[u,"="]
&
\tikzmath{
    \coordinate (a) at (-.4,0);
    \coordinate (b) at (.4,-.4);
    \draw (a) -- ($ (a) + (0,.4) $);
    \draw[red] ($ (a) + (-.2,.4) $) -- ($ (a) + (-.2,0) $) node [left,xshift=.1cm] {$\scriptstyle gh$} arc (-180:0:.2cm) -- ($ (a) + (.2,.4) $); 
    \draw (b) -- ($ (b) + (0,.8) $);
    \draw[red] ($ (b) + (-.2,.8) $) -- ($ (b) + (-.2,0) $) node [left,xshift=.1cm] {$\scriptstyle g$} arc (-180:0:.2cm) -- ($ (b) + (.2,.8) $); 
    \filldraw[fill=\kColor, thick] (a) circle (.1cm);
    \filldraw[fill=\hColor, thick] (b) circle (.1cm);
}
\\
\tikzmath{
	\coordinate (a) at (-.2,0);
	\coordinate (b) at (.2,-.4);
	\draw (a) -- ($ (a) + (0,.4) $);
	\draw[red] ($ (a) + (-.5,.4) $) -- ($ (a) + (-.5,-.4) $) node [left,xshift=.1cm] {$\scriptstyle g$} arc (-180:0:.55cm) -- ($ (b) + (.2,.8) $); 
	\draw[red] ($ (a) + (-.2,.4) $) -- ($ (a) + (-.2,0) $) node [left,xshift=.1cm] {$\scriptstyle h$} arc (-180:0:.2cm) -- ($ (a) + (.2,.4) $); 
	\draw (b) -- ($ (b) + (0,.8) $);
	\filldraw[fill=\kColor, thick] (a) circle (.1cm);
	\filldraw[fill=\hColor, thick] (b) circle (.1cm);
}
\arrow[rr,"="]
&&
\tikzmath{
    \coordinate (a) at (-.4,0);
    \coordinate (b) at (.4,-.4);
    \draw (a) -- ($ (a) + (0,.4) $);
    \draw[red,double] ($ (a) + (-.2,.4) $) -- ($ (a) + (-.2,0) $) node [left,xshift=.1cm] {$\scriptstyle g,h$} arc (-180:0:.2cm) -- ($ (a) + (.2,.4) $); 
    \draw (b) -- ($ (b) + (0,.8) $);
    \draw[red] ($ (b) + (-.2,.8) $) -- ($ (b) + (-.2,0) $) node [left,xshift=.1cm] {$\scriptstyle g$} arc (-180:0:.2cm) -- ($ (b) + (.2,.8) $); 
    \filldraw[fill=\kColor, thick] (a) circle (.1cm);
    \filldraw[fill=\hColor, thick] (b) circle (.1cm);
}
\arrow[u,"="]
\end{tikzcd}
\end{equation*}
The top left square commutes as the only two non-trivial maps are the same interchanger.

The equations \eqref{eq:Heptagon1} and \eqref{eq:Heptagon2} are similar.
In the two diagrams below, the outside 7 diagrams are the vertices in the heptagons \eqref{eq:Heptagon1} and \eqref{eq:Heptagon2} respectively.
There is only one non-trivial face in each the two diagrams below corresponding to these two coherences, and this face commutes by the axiom \ref{Interchanger:Composition} of the interchanger in a $\Gray$-monoid.
\begin{equation*}
\tag{\ref{eq:Heptagon1}}
\begin{tikzpicture}[baseline= (a).base]
\node (a) at (0,0){
\begin{tikzcd}
&
\tikzmath{
    \coordinate (a) at (-.2,0);
    \coordinate (b) at (.2,-.2);
    \coordinate (c) at (.6,-.4);
    \draw (a) -- ($ (a) + (0,.4) $) node [above] {$\scriptstyle g_\cC$};
    \draw (b) -- ($ (b) + (0,.6) $) node [above] {$\scriptstyle h_\cC$};
    \draw (c) -- ($ (c) + (0,.8) $) node [above] {$\scriptstyle k_\cC$};
    \filldraw[fill=\gColor, thick] (a) circle (.1cm);
    \filldraw[fill=\hColor, thick] (b) circle (.1cm);
    \filldraw[fill=\kColor, thick] (c) circle (.1cm);
}
\arrow[dl, "="]
\arrow[dr, "\cong"]
\arrow[d, "\cong"]
\arrow[d, phantom, bend left=40, "\text{\scriptsize \ref{Interchanger:Composition}}"]
\\
\tikzmath{
    \coordinate (a) at (-.2,0);
    \coordinate (b) at (.2,-.2);
    \coordinate (c) at (.6,-.4);
    \draw (a) -- ($ (a) + (0,.4) $);
    \draw (b) -- ($ (b) + (0,.6) $);
    \draw (c) -- ($ (c) + (0,.8) $);
    \filldraw[fill=\gColor, thick] (a) circle (.1cm);
    \filldraw[fill=\hColor, thick] (b) circle (.1cm);
    \filldraw[fill=\kColor, thick] (c) circle (.1cm);
}
\arrow[r, "\cong"]
&
\tikzmath{
    \coordinate (a) at (-.2,-.4);
    \coordinate (b) at (.2,0);
    \coordinate (c) at (.6,-.2);
    \draw (a) -- ($ (a) + (0,.8) $);
    \draw (b) -- ($ (b) + (0,.4) $);
    \draw (c) -- ($ (c) + (0,.6) $);
    \filldraw[fill=\gColor, thick] (a) circle (.1cm);
    \filldraw[fill=\hColor, thick] (b) circle (.1cm);
    \filldraw[fill=\kColor, thick] (c) circle (.1cm);
}
\arrow[dl,"="]
\arrow[dr,"="]
&
\tikzmath{
    \coordinate (a) at (-.2,-.2);
    \coordinate (b) at (.2,0);
    \coordinate (c) at (.6,-.4);
    \draw (a) -- ($ (a) + (0,.6) $);
    \draw (b) -- ($ (b) + (0,.4) $);
    \draw (c) -- ($ (c) + (0,.8) $);
    \filldraw[fill=\gColor, thick] (a) circle (.1cm);
    \filldraw[fill=\hColor, thick] (b) circle (.1cm);
    \filldraw[fill=\kColor, thick] (c) circle (.1cm);
}
\arrow[l, "\cong"]
\arrow[r, "="]
& 
\tikzmath{
    \coordinate (a) at (0,0);
    \coordinate (b) at (.4,-.4);
    \coordinate (c) at (.8,-.6);
    \draw (a) -- ($ (a) + (0,.4) $);
    \draw[red] ($ (a) + (-.2,.4) $) -- ($ (a) + (-.2,0) $) node [left,xshift=.1cm] {$\scriptstyle g$} arc (-180:0:.2cm) -- ($ (a) + (.2,.4) $); 
    \draw (b) -- ($ (b) + (0,.8) $);
    \draw (c) -- ($ (c) + (0,1) $);
    \filldraw[fill=\hColor, thick] (a) circle (.1cm);
    \filldraw[fill=\gColor, thick] (b) circle (.1cm);
    \filldraw[fill=\kColor, thick] (c) circle (.1cm);
}
\arrow[d, "="]
\arrow[dl, "\cong"]
\\
\tikzmath{
    \coordinate (a) at (-.2,0);
    \coordinate (b) at (.2,-.2);
    \coordinate (c) at (.6,-.8);
    \draw (a) -- ($ (a) + (0,.4) $);
    \draw[red] ($ (a) + (-.2,.4) $) -- ($ (a) + (-.2,-.2) $) node [left,xshift=.1cm] {$\scriptstyle g$} arc (-180:0:.4cm) -- ($ (b) + (.2,.6) $); 
    \draw (b) -- ($ (b) + (0,.6) $);
    \draw (c) -- ($ (c) + (0,1.2) $);
    \filldraw[fill=\hColor, thick] (a) circle (.1cm);
    \filldraw[fill=\kColor, thick] (b) circle (.1cm);
    \filldraw[fill=\gColor, thick] (c) circle (.1cm);
}
\arrow[d, "="]
&&
\tikzmath{
    \coordinate (a) at (0,0);
    \coordinate (b) at (.4,-.6);
    \coordinate (c) at (.8,-.4);
    \draw (a) -- ($ (a) + (0,.4) $);
    \draw[red] ($ (a) + (-.2,.4) $) -- ($ (a) + (-.2,0) $) node [left,xshift=.1cm] {$\scriptstyle g$} arc (-180:0:.2cm) -- ($ (a) + (.2,.4) $); 
    \draw (b) -- ($ (b) + (0,1) $);
    \draw (c) -- ($ (c) + (0,.8) $);
    \filldraw[fill=\hColor, thick] (a) circle (.1cm);
    \filldraw[fill=\gColor, thick] (b) circle (.1cm);
    \filldraw[fill=\kColor, thick] (c) circle (.1cm);
}
\arrow[dr, "="]
&
\tikzmath{
    \coordinate (a) at (0,0);
    \coordinate (b) at (.4,-.4);
    \coordinate (c) at (.8,-.6);
    \draw (a) -- ($ (a) + (0,.4) $);
    \draw[red] ($ (a) + (-.2,.4) $) -- ($ (a) + (-.2,0) $) node [left,xshift=.1cm] {$\scriptstyle g$} arc (-180:0:.2cm) -- ($ (a) + (.2,.4) $); 
    \draw (b) -- ($ (b) + (0,.8) $);
    \draw (c) -- ($ (c) + (0,1) $);
    \filldraw[fill=\hColor, thick] (a) circle (.1cm);
    \filldraw[fill=\gColor, thick] (b) circle (.1cm);
    \filldraw[fill=\kColor, thick] (c) circle (.1cm);
}
\arrow[l, "\cong"]
\\
\tikzmath{
    \coordinate (a) at (-.2,0);
    \coordinate (b) at (.4,-.4);
    \coordinate (c) at (.8,-.8);
    \draw (a) -- ($ (a) + (0,.4) $);
    \draw[red] ($ (a) + (-.2,.4) $) -- ($ (a) + (-.2,0) $) node [left,xshift=.1cm] {$\scriptstyle g$} arc (-180:0:.2cm) -- ($ (a) + (.2,.4) $); 
    \draw (b) -- ($ (b) + (0,.8) $);
    \draw[red] ($ (b) + (-.2,.8) $) -- ($ (b) + (-.2,0) $) node [left,xshift=.1cm] {$\scriptstyle g$} arc (-180:0:.2cm) -- ($ (b) + (.2,.8) $); 
    \draw (c) -- ($ (c) + (0,1.2) $);
    \filldraw[fill=\hColor, thick] (a) circle (.1cm);
    \filldraw[fill=\kColor, thick] (b) circle (.1cm);
    \filldraw[fill=\gColor, thick] (c) circle (.1cm);
}
\arrow[rrr, "="]
&&&
\tikzmath{
    \coordinate (a) at (-.2,0);
    \coordinate (b) at (.4,-.4);
    \coordinate (c) at (.8,-.8);
    \draw (a) -- ($ (a) + (0,.4) $);
    \draw[red] ($ (a) + (-.2,.4) $) -- ($ (a) + (-.2,0) $) node [left,xshift=.1cm] {$\scriptstyle g$} arc (-180:0:.2cm) -- ($ (a) + (.2,.4) $); 
    \draw (b) -- ($ (b) + (0,.8) $);
    \draw[red] ($ (b) + (-.2,.8) $) -- ($ (b) + (-.2,0) $) node [left,xshift=.1cm] {$\scriptstyle g$} arc (-180:0:.2cm) -- ($ (b) + (.2,.8) $); 
    \draw (c) -- ($ (c) + (0,1.2) $);
    \filldraw[fill=\hColor, thick] (a) circle (.1cm);
    \filldraw[fill=\kColor, thick] (b) circle (.1cm);
    \filldraw[fill=\gColor, thick] (c) circle (.1cm);
}
\end{tikzcd}
};
\end{tikzpicture}
\end{equation*}

\begin{equation*}
\tag{\ref{eq:Heptagon2}}
\begin{tikzpicture}[baseline= (a).base]
\node (a) at (0,0){
\begin{tikzcd}
&
\tikzmath{
    \coordinate (a) at (-.2,0);
    \coordinate (b) at (.2,-.2);
    \coordinate (c) at (.6,-.4);
    \draw (a) -- ($ (a) + (0,.4) $) node [above] {$\scriptstyle g_\cC$};
    \draw (b) -- ($ (b) + (0,.6) $) node [above] {$\scriptstyle h_\cC$};
    \draw (c) -- ($ (c) + (0,.8) $) node [above] {$\scriptstyle k_\cC$};
    \filldraw[fill=\gColor, thick] (a) circle (.1cm);
    \filldraw[fill=\hColor, thick] (b) circle (.1cm);
    \filldraw[fill=\kColor, thick] (c) circle (.1cm);
}
\arrow[dl, "="]
\arrow[dr, "\cong"]
\arrow[d, "\cong"]
\arrow[d, phantom, bend left=40, "\text{\scriptsize \ref{Interchanger:Composition}}"]
\\
\tikzmath{
    \coordinate (a) at (-.2,0);
    \coordinate (b) at (.2,-.2);
    \coordinate (c) at (.6,-.4);
    \draw (a) -- ($ (a) + (0,.4) $);
    \draw (b) -- ($ (b) + (0,.6) $);
    \draw (c) -- ($ (c) + (0,.8) $);
    \filldraw[fill=\gColor, thick] (a) circle (.1cm);
    \filldraw[fill=\hColor, thick] (b) circle (.1cm);
    \filldraw[fill=\kColor, thick] (c) circle (.1cm);
}
\arrow[r, "\cong"]
&
\tikzmath{
    \coordinate (a) at (-.2,-.2);
    \coordinate (b) at (.2,-.4);
    \coordinate (c) at (.6,0);
    \draw (a) -- ($ (a) + (0,.6) $);
    \draw (b) -- ($ (b) + (0,.8) $);
    \draw (c) -- ($ (c) + (0,.4) $);
    \filldraw[fill=\gColor, thick] (a) circle (.1cm);
    \filldraw[fill=\hColor, thick] (b) circle (.1cm);
    \filldraw[fill=\kColor, thick] (c) circle (.1cm);
}
\arrow[dr, "="]
\arrow[dl, "="]
&
\tikzmath{
    \coordinate (a) at (-.2,0);
    \coordinate (b) at (.2,-.4);
    \coordinate (c) at (.6,-.2);
    \draw (a) -- ($ (a) + (0,.4) $);
    \draw (b) -- ($ (b) + (0,.8) $);
    \draw (c) -- ($ (c) + (0,.6) $);
    \filldraw[fill=\gColor, thick] (a) circle (.1cm);
    \filldraw[fill=\hColor, thick] (b) circle (.1cm);
    \filldraw[fill=\kColor, thick] (c) circle (.1cm);
}
\arrow[r, "="]
\arrow[l, "\cong"]
& 
\tikzmath{
    \coordinate (a) at (-.1,0);
    \coordinate (b) at (.4,-.2);
    \coordinate (c) at (.8,-.6);
    \draw (a) -- ($ (a) + (0,.4) $);
    \draw (b) -- ($ (b) + (0,.6) $);
    \draw[red] ($ (b) + (-.2,.6) $) -- ($ (b) + (-.2,0) $) node [left,xshift=.1cm] {$\scriptstyle h$} arc (-180:0:.2cm) -- ($ (b) + (.2,.6) $); 
    \draw (c) -- ($ (c) + (0,1) $);
    \filldraw[fill=\gColor, thick] (a) circle (.1cm);
    \filldraw[fill=\kColor, thick] (b) circle (.1cm);
    \filldraw[fill=\hColor, thick] (c) circle (.1cm);
}
\arrow[d, "="]
\\
\tikzmath{
    \coordinate (a) at (0,0);
    \coordinate (b) at (.4,-.4);
    \coordinate (c) at (.8,-.6);
    \draw (a) -- ($ (a) + (0,.4) $);
    \draw[red] ($ (a) + (-.2,.4) $) -- ($ (a) + (-.2,0) $) node [left,xshift=.1cm] {$\scriptstyle gh$} arc (-180:0:.2cm) -- ($ (a) + (.2,.4) $); 
    \draw (b) -- ($ (b) + (0,.8) $);
    \draw (c) -- ($ (c) + (0,1) $);
    \filldraw[fill=\kColor, thick] (a) circle (.1cm);
    \filldraw[fill=\gColor, thick] (b) circle (.1cm);
    \filldraw[fill=\hColor, thick] (c) circle (.1cm);
}
\arrow[d, "="]
&
&
\tikzmath{
    \coordinate (a) at (-.1,-.4);
    \coordinate (b) at (.4,0);
    \coordinate (c) at (.8,-.6);
    \draw (a) -- ($ (a) + (0,.8) $);
    \draw (b) -- ($ (b) + (0,.4) $);
    \draw[red] ($ (b) + (-.2,.4) $) -- ($ (b) + (-.2,0) $) node [left,xshift=.1cm] {$\scriptstyle h$} arc (-180:0:.2cm) -- ($ (b) + (.2,.4) $); 
    \draw (c) -- ($ (c) + (0,1) $);
    \filldraw[fill=\gColor, thick] (a) circle (.1cm);
    \filldraw[fill=\kColor, thick] (b) circle (.1cm);
    \filldraw[fill=\hColor, thick] (c) circle (.1cm);
}
\arrow[dr, "="]
&
\tikzmath{
    \coordinate (a) at (-.1,0);
    \coordinate (b) at (.4,-.2);
    \coordinate (c) at (.8,-.6);
    \draw (a) -- ($ (a) + (0,.4) $);
    \draw (b) -- ($ (b) + (0,.6) $);
    \draw[red] ($ (b) + (-.2,.6) $) -- ($ (b) + (-.2,0) $) node [left,xshift=.1cm] {$\scriptstyle h$} arc (-180:0:.2cm) -- ($ (b) + (.2,.6) $); 
    \draw (c) -- ($ (c) + (0,1) $);
    \filldraw[fill=\gColor, thick] (a) circle (.1cm);
    \filldraw[fill=\kColor, thick] (b) circle (.1cm);
    \filldraw[fill=\hColor, thick] (c) circle (.1cm);
}
\arrow[l, "\cong"]
\\
\tikzmath{
    \coordinate (a) at (0,0);
    \coordinate (b) at (.4,-.4);
    \coordinate (c) at (.8,-.6);
    \draw (a) -- ($ (a) + (0,.4) $);
    \draw[red,double] ($ (a) + (-.2,.4) $) -- ($ (a) + (-.2,0) $) node [left,xshift=.1cm] {$\scriptstyle g,h$} arc (-180:0:.2cm) -- ($ (a) + (.2,.4) $); 
    \draw (b) -- ($ (b) + (0,.8) $);
    \draw (c) -- ($ (c) + (0,1) $);
    \filldraw[fill=\kColor, thick] (a) circle (.1cm);
    \filldraw[fill=\gColor, thick] (b) circle (.1cm);
    \filldraw[fill=\hColor, thick] (c) circle (.1cm);
}
\arrow[rrr, "="]
&&&
\tikzmath{
    \coordinate (a) at (0,0);
    \coordinate (b) at (.4,-.4);
    \coordinate (c) at (.8,-.6);
    \draw (a) -- ($ (a) + (0,.4) $);
    \draw[red,double] ($ (a) + (-.2,.4) $) -- ($ (a) + (-.2,0) $) node [left,xshift=.1cm] {$\scriptstyle g,h$} arc (-180:0:.2cm) -- ($ (a) + (.2,.4) $); 
    \draw (b) -- ($ (b) + (0,.8) $);
    \draw (c) -- ($ (c) + (0,1) $);
    \filldraw[fill=\kColor, thick] (a) circle (.1cm);
    \filldraw[fill=\gColor, thick] (b) circle (.1cm);
    \filldraw[fill=\hColor, thick] (c) circle (.1cm);
}
\end{tikzcd}
};
\end{tikzpicture}
\end{equation*}
This completes the proof.
\end{proof}

%%%%%%%%%%%%%%%%%%%%%%%%%%%%%%%%%%%%%%%%%%%%%%%%%%

For $\cC, \cD \in \TriCat_G^{\st}$, $A, B\in \TriCat_G^{\st}(\cC \to \cD)$, and $\eta\in \TriCat_G^{\st}(A \Rightarrow B)$,
let $\fC, \fD$ be the $G$-crossed braided categories obtained from $\cC,\cD$ respectively
from Theorem \ref{thm:ConstructionOfGCrossedBraidedCategory}.
In Construction \ref{const:From3FunctorToGCrossedBraidedFunctor} and \eqref{eq:ActionatorForGCrossedFunctor}, we defined the data $(\bfA, \bfa): \fC \to \fD$.

\begin{proof}[Proof of Lem.~\ref{lem:FromUnderlying2FunctorToGMonoidalFunctor}: $(\bfA, \bfA^1, \bfA^2): \fC \to \fD$ is a $G$-graded monoidal functor]

That each $\bfA_g$ is a functor follows immediately from the fact that $A$ is a functor.
The data $\bfA^2$ satisfies associativity by property \ref{Functor:omega} of $(A,\mu^A,\iota^A)$, and the data $\bfA^2$ and $\bfA^1$ satisfies unitality by property \ref{Functor:lr} of $(A,\mu^A,\iota^A)$.
(Observe that in \ref{Functor:omega} and \ref{Functor:lr}, all instances of $\phi$, $\omega^A$, $\ell^A$, and $r^A$ are identities, so these reduce to the usual associativity and unitality conditions for a monoidal functor.)
\end{proof}

\begin{proof}[Proof of Thm.~\ref{thm:From3FunctorToGCrossedBraidedFunctor}:  $(\bfA, \bfA^1,\bfA^2, \bfa):\fC \to \fD$ is a $G$-crossed braided monoidal functor.]

Naturality of $\bfa$ follows by naturality of $\bfA^1$ and \ref{Functor:mu.natural} of $\mu^A$.
It remains to prove the coherences 
\eqref{eq:GCrossedFunctor-Gamma1} and \eqref{eq:GCrossedFunctor-Gamma2}.
\item[\eqref{eq:GCrossedFunctor-Gamma1}]
Observe that since $\fC$ and $\fD$ are strict, the coherence condition \eqref{eq:GCrossedFunctor-Gamma1} is actually a triangle.
For $a\in \fC_k$, we use the shorthand notation a small shaded box for $A(a)$.
For $g,h\in G$ and $a\in \fC_k = \cC(1_\cC \to k_\cC)$, we use the following shorthand as in Notation \ref{nota:ShadedBoxes}:
$$
\tikzmath{
\draw[red] (-.1,-.4) -- (-.1,.4);
}
:=
\id_{g_\cD}
\qquad\qquad
\tikzmath{
\draw[red] (0,-.3) -- (0,.3);
\filldraw[fill=\gColor, thick] (-.1,-.1) rectangle (.1,.1);
}
:=
\tikzmath{
\node (Aa) at (0,-.8) {$\scriptstyle g_\cD$};
\node (Ab) at (0,.8) {$\scriptstyle g_\cD$};
\node[draw,rectangle, thick, rounded corners=5pt] (Ax) at (0,0) {$\scriptstyle A(\id_g)$};
\draw (Aa) to[in=-90,out=90] (Ax);
\draw (Ax) to[in=-90,out=90] (Ab);
}
\qquad\qquad
\tikzmath{
\draw[DarkGreen] (-.1,-.4) -- (-.1,.4);
}
:=
\id_{h_\cD}
\qquad\qquad
\tikzmath{
\draw[DarkGreen] (0,-.3) -- (0,.3);
\filldraw[fill=\gColor, thick] (-.1,-.1) rectangle (.1,.1);
}
:=
\tikzmath{
\node (Aa) at (0,-.8) {$\scriptstyle h_\cD$};
\node (Ab) at (0,.8) {$\scriptstyle h_\cD$};
\node[draw,rectangle, thick, rounded corners=5pt] (Ax) at (0,0) {$\scriptstyle A(\id_h)$};
\draw (Aa) to[in=-90,out=90] (Ax);
\draw (Ax) to[in=-90,out=90] (Ab);
}
\qquad\qquad
\tikzmath{
\draw (0,0) -- (0,.3);
\filldraw[fill=\kColor, thick] (-.1,-.1) rectangle (.1,.1);
}
:=
\tikzmath{
\node (Ab) at (0,.8) {$\scriptstyle k_\cD$};
\node[draw,rectangle, thick, rounded corners=5pt] (Ax) at (0,0) {$\scriptstyle A(a)$};
\draw (Ax) to[in=-90,out=90] (Ab);
}
$$
Observe that since $\cC$ is $\Gray$, we have an equality $\id_{g_\cC} \xz \id_{h_\cC} = \id_{gh_\cC}$:
$$
\tikzmath{
\draw[red] (-.1,-.4) -- (-.1,.4);
\draw[DarkGreen] (.1,-.4) -- (.1,.4);
\filldraw[fill=\gColor, thick] (0,.2) rectangle (-.2,0);
\filldraw[fill=\gColor, thick] (.2,-.2) rectangle (0,0);
}
=
\tikzmath{
\draw[red] (-.1,-.4) -- (-.1,.4);
\draw[DarkGreen] (.1,-.4) -- (.1,.4);
\filldraw[fill=\gColor, thick] (-.2,.1) rectangle (.2,-.1);
}
=
\tikzmath{
\node (gh1) at (0,-.8) {$\scriptstyle gh_\cD$};
\node (gh2) at (0,.8) {$\scriptstyle gh_\cD$};
\node[draw,rectangle, thick, rounded corners=5pt] (Ax) at (0,0) {$\scriptstyle A(\id_{gh_\cC})$};
\draw (gh1) to[in=-90,out=90] (Ax);
\draw (Ax) to[in=-90,out=90] (gh2);
}
$$
Expanding \eqref{eq:ActionatorForGCrossedFunctor}, we see that \eqref{eq:GCrossedFunctor-Gamma1} follows from the following commuting diagram.
(Recall that the cups on the bottom in \eqref{eq:ActionatorForGCrossedFunctor} are really identity maps, and do not need to be drawn.)
$$
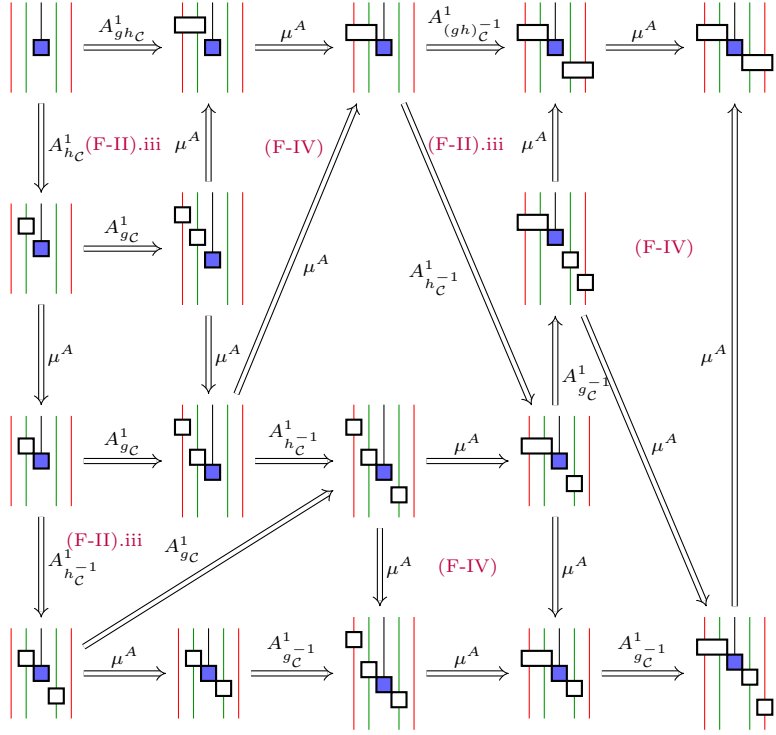
\begin{tikzpicture}[baseline= (a).base]\node[scale=1] (a) at (0,0){\begin{tikzcd}[row sep=3em, column sep=3em]
\tikzmath{
\coordinate (Aa) at (0,0);
\draw (Aa) to[in=-90,out=90] (0,.6);
\draw[DarkGreen] (-.2,-.6) -- (-.2,.6);
\draw[DarkGreen] (.2,-.6) -- (.2,.6);
\draw[red] (-.4,-.6) -- (-.4,.6);
\draw[red] (.4,-.6) -- (.4,.6);
\filldraw[fill=\kColor, thick] ($ (Aa) - (.1,.1) $) rectangle ($ (Aa) + (.1,.1) $);
}
\arrow[r, Rightarrow, "A^1_{gh_\cC}"]
\arrow[d,Rightarrow, "A^1_{h_\cC}"]
\arrow[rd, phantom, "\text{\scriptsize \ref{Functor:mu.unital}}"]
&
\tikzmath{
\coordinate (Aa) at (0,0);
\draw (Aa) to[in=-90,out=90] (0,.6);
\draw[DarkGreen] (-.2,-.6) -- (-.2,.6);
\draw[DarkGreen] (.2,-.6) -- (.2,.6);
\draw[red] (-.4,-.6) -- (-.4,.6);
\draw[red] (.4,-.6) -- (.4,.6);
\filldraw[fill=\gColor, thick] (-.5,.2) rectangle (-.1,.4);
\filldraw[fill=\kColor, thick] ($ (Aa) - (.1,.1) $) rectangle ($ (Aa) + (.1,.1) $);
}
\arrow[r, Rightarrow, "\mu^A"]
&
\tikzmath{
\coordinate (Aa) at (0,0);
\draw (Aa) to[in=-90,out=90] (0,.6);
\draw[DarkGreen] (-.2,-.6) -- (-.2,.6);
\draw[DarkGreen] (.2,-.6) -- (.2,.6);
\draw[red] (-.4,-.6) -- (-.4,.6);
\draw[red] (.4,-.6) -- (.4,.6);
\filldraw[fill=\gColor, thick] (-.5,.1) rectangle (-.1,.3);
\filldraw[fill=\kColor, thick] ($ (Aa) - (.1,.1) $) rectangle ($ (Aa) + (.1,.1) $);
}
\arrow[rd, phantom, "\text{\scriptsize \ref{Functor:mu.unital}}"]
\arrow[rdd, Rightarrow, swap, "A^1_{h^{-1}_\cC}"]
\arrow[r, Rightarrow, "A^1_{(gh)^{-1}_\cC}"]
&
\tikzmath{
\coordinate (Aa) at (0,0);
\draw (Aa) to[in=-90,out=90] (0,.6);
\draw[DarkGreen] (-.2,-.6) -- (-.2,.6);
\draw[DarkGreen] (.2,-.6) -- (.2,.6);
\draw[red] (-.4,-.6) -- (-.4,.6);
\draw[red] (.4,-.6) -- (.4,.6);
\filldraw[fill=\gColor, thick] (-.5,.1) rectangle (-.1,.3);
\filldraw[fill=\gColor, thick] (.5,-.2) rectangle (.1,-.4);
\filldraw[fill=\kColor, thick] ($ (Aa) - (.1,.1) $) rectangle ($ (Aa) + (.1,.1) $);
}
\arrow[r, Rightarrow, "\mu^A"]
&
\tikzmath{
\coordinate (Aa) at (0,0);
\draw (Aa) to[in=-90,out=90] (0,.6);
\draw[DarkGreen] (-.2,-.6) -- (-.2,.6);
\draw[DarkGreen] (.2,-.6) -- (.2,.6);
\draw[red] (-.4,-.6) -- (-.4,.6);
\draw[red] (.4,-.6) -- (.4,.6);
\filldraw[fill=\gColor, thick] (-.5,.1) rectangle (-.1,.3);
\filldraw[fill=\gColor, thick] (.5,-.1) rectangle (.1,-.3);
\filldraw[fill=\kColor, thick] ($ (Aa) - (.1,.1) $) rectangle ($ (Aa) + (.1,.1) $);
}
\\
\tikzmath{
\coordinate (Aa) at (0,0);
\draw (Aa) to[in=-90,out=90] (0,.6);
\draw[DarkGreen] (-.2,-.6) -- (-.2,.6);
\draw[DarkGreen] (.2,-.6) -- (.2,.6);
\draw[red] (-.4,-.6) -- (-.4,.6);
\draw[red] (.4,-.6) -- (.4,.6);
\filldraw[fill=\gColor, thick] (-.3,.2) rectangle (-.1,.4);
\filldraw[fill=\kColor, thick] ($ (Aa) - (.1,.1) $) rectangle ($ (Aa) + (.1,.1) $);
}
\arrow[r, Rightarrow, "A^1_{g_\cC}"]
\arrow[d, Rightarrow, "\mu^A"]
&
\tikzmath{
\coordinate (Aa) at (0,0);
\draw (Aa) to[in=-90,out=90] (0,.9);
\draw[DarkGreen] (-.2,-.6) -- (-.2,.9);
\draw[DarkGreen] (.2,-.6) -- (.2,.9);
\draw[red] (-.4,-.6) -- (-.4,.9);
\draw[red] (.4,-.6) -- (.4,.9);
\filldraw[fill=\gColor, thick] (-.5,.5) rectangle (-.3,.7);
\filldraw[fill=\gColor, thick] (-.3,.2) rectangle (-.1,.4);
\filldraw[fill=\kColor, thick] ($ (Aa) - (.1,.1) $) rectangle ($ (Aa) + (.1,.1) $);
}
\arrow[d, Rightarrow, "\mu^A"]
\arrow[u, Rightarrow, "\mu^A"]
\arrow[ur, phantom, "\text{\scriptsize \ref{Functor:omega}}"]
&\mbox{}&
\tikzmath{
\coordinate (Aa) at (0,0);
\draw (Aa) to[in=-90,out=90] (0,.6);
\draw[DarkGreen] (-.2,-.9) -- (-.2,.6);
\draw[DarkGreen] (.2,-.9) -- (.2,.6);
\draw[red] (-.4,-.9) -- (-.4,.6);
\draw[red] (.4,-.9) -- (.4,.6);
\filldraw[fill=\gColor, thick] (-.5,.1) rectangle (-.1,.3);
\filldraw[fill=\gColor, thick] (.3,-.2) rectangle (.1,-.4);
\filldraw[fill=\gColor, thick] (.3,-.5) rectangle (.5,-.7);
\filldraw[fill=\kColor, thick] ($ (Aa) - (.1,.1) $) rectangle ($ (Aa) + (.1,.1) $);
}
\arrow[u, Rightarrow, "\mu^A"]
\arrow[ddr, Rightarrow, "\mu^A"]
\arrow[r, phantom, "\text{\scriptsize \ref{Functor:omega}}"]
&
\mbox{}
\\
\tikzmath{
\coordinate (Aa) at (0,0);
\draw (Aa) to[in=-90,out=90] (0,.6);
\draw[DarkGreen] (-.2,-.6) -- (-.2,.6);
\draw[DarkGreen] (.2,-.6) -- (.2,.6);
\draw[red] (-.4,-.6) -- (-.4,.6);
\draw[red] (.4,-.6) -- (.4,.6);
\filldraw[fill=\gColor, thick] (-.3,.1) rectangle (-.1,.3);
\filldraw[fill=\kColor, thick] ($ (Aa) - (.1,.1) $) rectangle ($ (Aa) + (.1,.1) $);
}
\arrow[r, Rightarrow, "A^1_{g_\cC}"]
\arrow[d, Rightarrow, "A^1_{h^{-1}_\cC}"]
\arrow[rd, phantom, near start, "\text{\scriptsize \ref{Functor:mu.unital}}"]
&
\tikzmath{
\coordinate (Aa) at (0,0);
\draw (Aa) to[in=-90,out=90] (0,.9);
\draw[DarkGreen] (-.2,-.6) -- (-.2,.9);
\draw[DarkGreen] (.2,-.6) -- (.2,.9);
\draw[red] (-.4,-.6) -- (-.4,.9);
\draw[red] (.4,-.6) -- (.4,.9);
\filldraw[fill=\gColor, thick] (-.5,.5) rectangle (-.3,.7);
\filldraw[fill=\gColor, thick] (-.3,.1) rectangle (-.1,.3);
\filldraw[fill=\kColor, thick] ($ (Aa) - (.1,.1) $) rectangle ($ (Aa) + (.1,.1) $);
}
\arrow[uur, Rightarrow, swap, "\mu^A"]
\arrow[r, Rightarrow, "A^1_{h^{-1}_\cC}"]
&
\tikzmath{
\coordinate (Aa) at (0,0);
\draw (Aa) to[in=-90,out=90] (0,.9);
\draw[DarkGreen] (-.2,-.6) -- (-.2,.9);
\draw[DarkGreen] (.2,-.6) -- (.2,.9);
\draw[red] (-.4,-.6) -- (-.4,.9);
\draw[red] (.4,-.6) -- (.4,.9);
\filldraw[fill=\gColor, thick] (-.5,.5) rectangle (-.3,.7);
\filldraw[fill=\gColor, thick] (-.3,.1) rectangle (-.1,.3);
\filldraw[fill=\gColor, thick] (.3,-.2) rectangle (.1,-.4);
\filldraw[fill=\kColor, thick] ($ (Aa) - (.1,.1) $) rectangle ($ (Aa) + (.1,.1) $);
}
\arrow[r, Rightarrow, "\mu^A"]
\arrow[d, Rightarrow, "\mu^A\text{\hspace{.5cm}}"]
\arrow[dr, phantom, "\text{\scriptsize \ref{Functor:omega}}"]
&
\tikzmath{
\coordinate (Aa) at (0,0);
\draw (Aa) to[in=-90,out=90] (0,.6);
\draw[DarkGreen] (-.2,-.6) -- (-.2,.6);
\draw[DarkGreen] (.2,-.6) -- (.2,.6);
\draw[red] (-.4,-.6) -- (-.4,.6);
\draw[red] (.4,-.6) -- (.4,.6);
\filldraw[fill=\gColor, thick] (-.5,.1) rectangle (-.1,.3);
\filldraw[fill=\gColor, thick] (.3,-.2) rectangle (.1,-.4);
\filldraw[fill=\kColor, thick] ($ (Aa) - (.1,.1) $) rectangle ($ (Aa) + (.1,.1) $);
}
\arrow[u, Rightarrow, swap, near start, "A^1_{g^{-1}_\cC}"]
\arrow[d, Rightarrow, "\mu^A"]
\\
\tikzmath{
\coordinate (Aa) at (0,0);
\draw (Aa) to[in=-90,out=90] (0,.6);
\draw[DarkGreen] (-.2,-.6) -- (-.2,.6);
\draw[DarkGreen] (.2,-.6) -- (.2,.6);
\draw[red] (-.4,-.6) -- (-.4,.6);
\draw[red] (.4,-.6) -- (.4,.6);
\filldraw[fill=\gColor, thick] (-.3,.1) rectangle (-.1,.3);
\filldraw[fill=\gColor, thick] (.3,-.2) rectangle (.1,-.4);
\filldraw[fill=\kColor, thick] ($ (Aa) - (.1,.1) $) rectangle ($ (Aa) + (.1,.1) $);
}
\arrow[r, Rightarrow, "\mu^A"]
\arrow[rru, Rightarrow, "A^1_{g_\cC}"]
&
\tikzmath{
\coordinate (Aa) at (0,0);
\draw (Aa) to[in=-90,out=90] (0,.6);
\draw[DarkGreen] (-.2,-.6) -- (-.2,.6);
\draw[DarkGreen] (.2,-.6) -- (.2,.6);
\draw[red] (-.4,-.6) -- (-.4,.6);
\draw[red] (.4,-.6) -- (.4,.6);
\filldraw[fill=\gColor, thick] (-.3,.1) rectangle (-.1,.3);
\filldraw[fill=\gColor, thick] (.3,-.1) rectangle (.1,-.3);
\filldraw[fill=\kColor, thick] ($ (Aa) - (.1,.1) $) rectangle ($ (Aa) + (.1,.1) $);
}
\arrow[r, Rightarrow, "A^1_{g^{-1}_\cC}"]
&
\tikzmath{
\coordinate (Aa) at (0,0);
\draw (Aa) to[in=-90,out=90] (0,.9);
\draw[DarkGreen] (-.2,-.6) -- (-.2,.9);
\draw[DarkGreen] (.2,-.6) -- (.2,.9);
\draw[red] (-.4,-.6) -- (-.4,.9);
\draw[red] (.4,-.6) -- (.4,.9);
\filldraw[fill=\gColor, thick] (-.5,.5) rectangle (-.3,.7);
\filldraw[fill=\gColor, thick] (-.3,.1) rectangle (-.1,.3);
\filldraw[fill=\gColor, thick] (.3,-.1) rectangle (.1,-.3);
\filldraw[fill=\kColor, thick] ($ (Aa) - (.1,.1) $) rectangle ($ (Aa) + (.1,.1) $);
}
\arrow[r, Rightarrow, "\mu^A"]
&
\tikzmath{
\coordinate (Aa) at (0,0);
\draw (Aa) to[in=-90,out=90] (0,.6);
\draw[DarkGreen] (-.2,-.6) -- (-.2,.6);
\draw[DarkGreen] (.2,-.6) -- (.2,.6);
\draw[red] (-.4,-.6) -- (-.4,.6);
\draw[red] (.4,-.6) -- (.4,.6);
\filldraw[fill=\gColor, thick] (-.5,.1) rectangle (-.1,.3);
\filldraw[fill=\gColor, thick] (.3,-.1) rectangle (.1,-.3);
\filldraw[fill=\kColor, thick] ($ (Aa) - (.1,.1) $) rectangle ($ (Aa) + (.1,.1) $);
}
\arrow[r, Rightarrow, "A^1_{g^{-1}_\cC}"]
&
\tikzmath{
\coordinate (Aa) at (0,0);
\draw (Aa) to[in=-90,out=90] (0,.6);
\draw[DarkGreen] (-.2,-.9) -- (-.2,.6);
\draw[DarkGreen] (.2,-.9) -- (.2,.6);
\draw[red] (-.4,-.9) -- (-.4,.6);
\draw[red] (.4,-.9) -- (.4,.6);
\filldraw[fill=\gColor, thick] (-.5,.1) rectangle (-.1,.3);
\filldraw[fill=\gColor, thick] (.3,-.1) rectangle (.1,-.3);
\filldraw[fill=\gColor, thick] (.3,-.5) rectangle (.5,-.7);
\filldraw[fill=\kColor, thick] ($ (Aa) - (.1,.1) $) rectangle ($ (Aa) + (.1,.1) $);
}
\arrow[uuu, Rightarrow, "\mu^A"]
\end{tikzcd}};\end{tikzpicture}
$$
Each square above is labelled by the property for $A$ which makes it commute.
Unlabelled squares commute by functoriality of 1-cell composition $\xo$.

\item[\eqref{eq:GCrossedFunctor-Gamma2}]
For $g,h\in G$, $a\in \fC_g = \cC(1_\cC \to g_\cC)$, and $b\in \fC_h = \cC(1_\cC \to h_\cC)$, we use the following shorthand as in Notation \ref{nota:ShadedBoxes}:
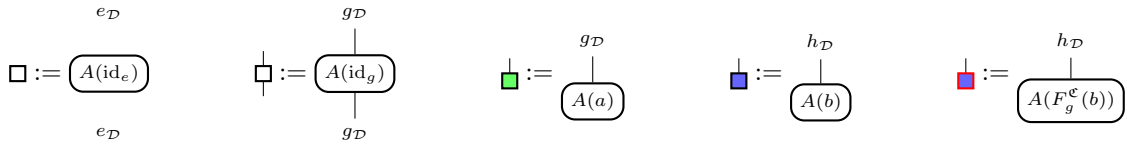
$$
\tikzmath{
\filldraw[fill=\gColor, thick] (-.1,-.1) rectangle (.1,.1);
}
:=
\tikzmath{
\node (Aa) at (0,-.8) {$\scriptstyle e_\cD$};
\node (Ab) at (0,.8) {$\scriptstyle e_\cD$};
\node[draw,rectangle, thick, rounded corners=5pt] (Ax) at (0,0) {$\scriptstyle A(\id_e)$};
}
\qquad\qquad\tikzmath{
\draw (0,-.3) -- (0,.3);
\filldraw[fill=\gColor, thick] (-.1,-.1) rectangle (.1,.1);
}
:=
\tikzmath{
\node (Aa) at (0,-.8) {$\scriptstyle g_\cD$};
\node (Ab) at (0,.8) {$\scriptstyle g_\cD$};
\node[draw,rectangle, thick, rounded corners=5pt] (Ax) at (0,0) {$\scriptstyle A(\id_g)$};
\draw (Aa) to[in=-90,out=90] (Ax);
\draw (Ax) to[in=-90,out=90] (Ab);
}
\qquad\qquad
\tikzmath{
\draw (0,0) -- (0,.3);
\filldraw[fill=\hColor, thick] (-.1,-.1) rectangle (.1,.1);
}
:=
\tikzmath{
\node (Ab) at (0,.8) {$\scriptstyle g_\cD$};
\node[draw,rectangle, thick, rounded corners=5pt] (Ax) at (0,0) {$\scriptstyle A(a)$};
\draw (Ax) to[in=-90,out=90] (Ab);
}
\qquad\qquad
\tikzmath{
\draw (0,0) -- (0,.3);
\filldraw[fill=\kColor, thick] (-.1,-.1) rectangle (.1,.1);
}
:=
\tikzmath{
\node (Ab) at (0,.8) {$\scriptstyle h_\cD$};
\node[draw,rectangle, thick, rounded corners=5pt] (Ax) at (0,0) {$\scriptstyle A(b)$};
\draw (Ax) to[in=-90,out=90] (Ab);
}
\qquad\qquad
\tikzmath{
\draw (0,0) -- (0,.3);
\RedRectangleMorphism{(0,0)}{\kColor}
}
:=
\tikzmath{
\node (Ab) at (0,.8) {$\scriptstyle h_\cD$};
\node[draw,rectangle, thick, rounded corners=5pt] (Ax) at (0,0) {$\scriptstyle A(F^\fC_g(b))$};
\draw (Ax) to[in=-90,out=90] (Ab);
}
$$
Recall that by Construction \ref{const:GCrossedBraiding} of the $G$-crossed braiding in $\fC$, we have the identities
$$
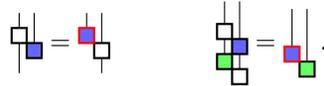
\tikzmath{
\coordinate (Ab) at (.1,.1);
\coordinate (i1) at (-.1,.3);
\draw (.1,.2) -- (.1,.6);
\draw (-.1,-.2) -- (-.1,.6);
\RectangleMorphism{(i1)}{white}
\RectangleMorphism{(Ab)}{\kColor}
}
=
\tikzmath{
\coordinate (Ab) at (-.1,.3);
\coordinate (i1) at (.1,.1);
\draw (-.1,.2) -- (-.1,.6);
\draw (.1,-.2) -- (.1,.6);
\RectangleMorphism{(i1)}{white}
\RedRectangleMorphism{(Ab)}{\kColor}
}
\qquad\qquad
\tikzmath{
\coordinate (Aa) at (-.1,-.2);
\coordinate (Ab) at (.1,0);
\coordinate (i1) at (-.1,.2);
\coordinate (i2) at (.1,-.4);
\draw (-.1,-.2) -- (-.1,.6);
\draw (.1,-.4) -- (.1,.6);
\RectangleMorphism{(i1)}{white}
\RectangleMorphism{(i2)}{white}
\RectangleMorphism{(Aa)}{\hColor}
\RectangleMorphism{(Ab)}{\kColor}
}
=
\tikzmath{
\coordinate (Aa) at (-.1,0);
\coordinate (Ab) at (.1,-.2);
\draw (-.1,0) -- (-.1,.6);
\draw (.1,-.2) -- (.1,.6);
\RectangleMorphism{(Ab)}{\hColor}
\RedRectangleMorphism{(Aa)}{\kColor}
}\,.
$$
Going around the outside of the diagram below corresponds to \eqref{eq:GCrossedFunctor-Gamma2}.
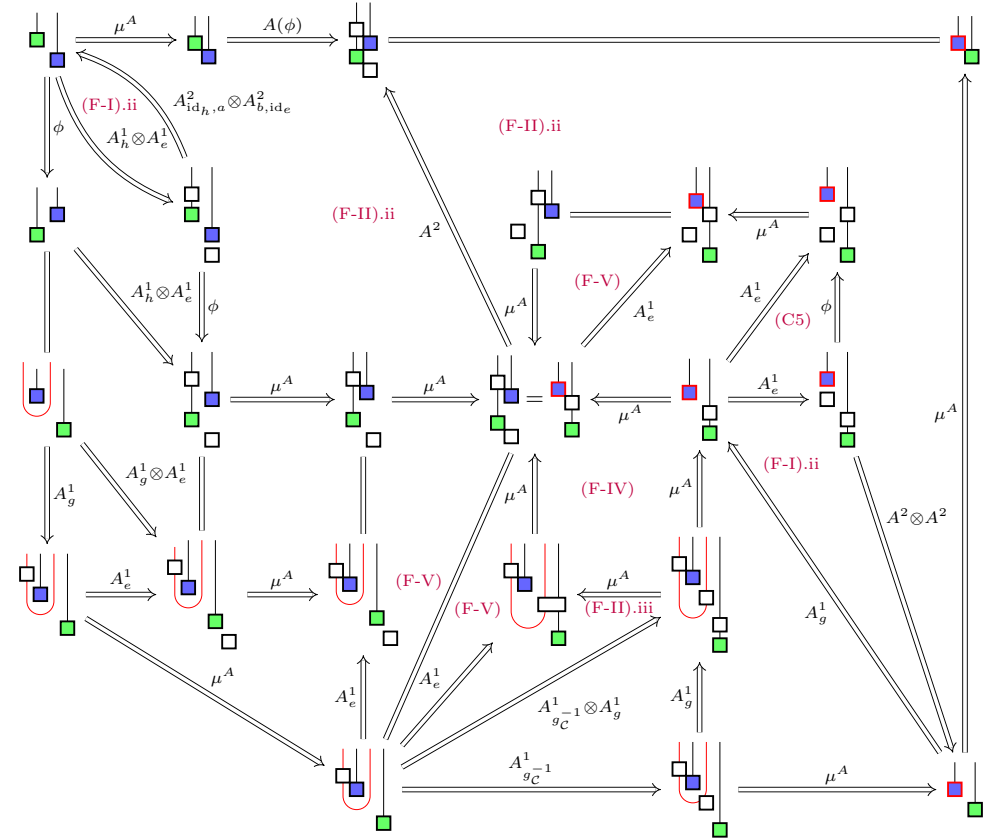
$$
\begin{tikzpicture}[baseline= (a).base]\node[scale=.9] (a) at (0,0){\begin{tikzcd}[row sep=3em, column sep=3em]
\tikzmath{
\coordinate (Aa) at (0,0);
\coordinate (Ab) at (.3,-.3);
\draw (Aa) to[in=-90,out=90] (0,.4);
\draw (Ab) to[in=-90,out=90] (.3,.4);
\RectangleMorphism{(Aa)}{\hColor}
\RectangleMorphism{(Ab)}{\kColor}
}
\arrow[r, Rightarrow, "\mu^A"]
\arrow[d, Rightarrow, "\phi"]
\arrow[dr, bend right = 25, Rightarrow, "A^1_h\xz A^1_e"]
&
\tikzmath{
\coordinate (Aa) at (-.1,0);
\coordinate (Ab) at (.1,-.2);
\draw (.1,-.2) to[in=-90,out=90] (.1,.4);
\draw (-.1,0) to[in=-90,out=90] (-.1,.4);
\RectangleMorphism{(Aa)}{\hColor}
\RectangleMorphism{(Ab)}{\kColor}
}
\arrow[r, Rightarrow, "A(\phi)"]
&
\tikzmath{
\coordinate (Aa) at (-.1,-.2);
\coordinate (Ab) at (.1,0);
\coordinate (i1) at (-.1,.2);
\coordinate (i2) at (.1,-.4);
\draw (-.1,-.2) to[in=-90,out=90] (-.1,.6);
\draw (.1,0) to[in=-90,out=90] (.1,.6);
\RectangleMorphism{(i1)}{white}
\RectangleMorphism{(Aa)}{\hColor}
\RectangleMorphism{(Ab)}{\kColor}
\RectangleMorphism{(i2)}{white}
}
\arrow[drr, phantom, "\text{\scriptsize \ref{Functor:mu.monoidal}}"]
\arrow[rrrr, equals]
&&&&
\tikzmath{
\coordinate (Aa) at (-.1,0);
\coordinate (Ab) at (.1,-.2);
\draw (.1,-.2) to[in=-90,out=90] (.1,.4);
\draw (-.1,0) to[in=-90,out=90] (-.1,.4);
\RectangleMorphism{(Ab)}{\hColor}
\RedRectangleMorphism{(Aa)}{\kColor}
}
\\
\tikzmath{
\coordinate (Aa) at (0,-.3);
\coordinate (Ab) at (.3,0);
\draw (Aa) to[in=-90,out=90] (0,.4);
\draw (Ab) to[in=-90,out=90] (.3,.4);
\RectangleMorphism{(Aa)}{\hColor}
\RectangleMorphism{(Ab)}{\kColor}
}
\arrow[d, equals]
\arrow[dr, Rightarrow, "A^1_h\xz A^1_e"]
&
\tikzmath{
\coordinate (Aa) at (0,0);
\coordinate (Ab) at (.3,-.3);
\coordinate (i1) at (0,.3);
\coordinate (i2) at (.3,-.6);
\draw (Aa) to[in=-90,out=90] (0,.7);
\draw (Ab) to[in=-90,out=90] (.3,.7);
\RectangleMorphism{(i1)}{white}
\RectangleMorphism{(i2)}{white}
\RectangleMorphism{(Aa)}{\hColor}
\RectangleMorphism{(Ab)}{\kColor}
}
\arrow[ul, bend right = 25, Rightarrow, "\text{\scriptsize \ref{Functor:A.unital}}", "A^2_{\id_h, a}\xz A^2_{b, \id_e}" ' near start]
\arrow[d, Rightarrow, "\phi"]
&&
\tikzmath{
\coordinate (Aa) at (-.1,-.5);
\coordinate (Ab) at (.1,.1);
\coordinate (i1) at (-.1,.3);
\coordinate (i2) at (-.4,-.2);
\draw (.1,.2) to[in=-90,out=90] (.1,.7);
\draw (Aa) -- (-.1,.7);
\RectangleMorphism{(i1)}{white}
\RectangleMorphism{(i2)}{white}
\RectangleMorphism{(Aa)}{\hColor}
\RectangleMorphism{(Ab)}{\kColor}
}
\arrow[r, equals]
\arrow[d, Rightarrow, swap, "\mu^A"]
&
\tikzmath{
\coordinate (Aa) at (.1,-.1);
\coordinate (Ab) at (.3,-.9);
\coordinate (i1) at (0,-.6);
\coordinate (i2) at (.3,-.3);
\draw (Aa) to[in=-90,out=90] (.1,.4);
\draw (Ab) to[in=-90,out=90] (.3,.4);
\RedRectangleMorphism{(Aa)}{\kColor}
\RectangleMorphism{(Ab)}{\hColor}
\RectangleMorphism{(i1)}{white}
\RectangleMorphism{(i2)}{white}
}
&
\tikzmath{
\coordinate (Aa) at (0,0);
\coordinate (Ab) at (.3,-.9);
\coordinate (i1) at (0,-.6);
\coordinate (i2) at (.3,-.3);
\draw (Aa) to[in=-90,out=90] (0,.4);
\draw (Ab) to[in=-90,out=90] (.3,.4);
\RedRectangleMorphism{(Aa)}{\kColor}
\RectangleMorphism{(Ab)}{\hColor}
\RectangleMorphism{(i1)}{white}
\RectangleMorphism{(i2)}{white}
}
\arrow[l, Rightarrow, "\mu^A"]
\\
\tikzmath{
\coordinate (Aa) at (0,0);
\coordinate (Ab) at (.4,-.5);
\draw (Aa) to[in=-90,out=90] (0,.4);
\draw (Ab) to[in=-90,out=90] (.4,.4);
\draw[red] (-.2,.5) -- (-.2,-.1) arc (-180:0:.2cm) -- (.2,.5);
\RectangleMorphism{(Aa)}{\kColor}
\RectangleMorphism{(Ab)}{\hColor}
}
\arrow[dr, Rightarrow, "A^1_g\xz A^1_e"]
\arrow[d, Rightarrow, "A^1_g"]
&
\tikzmath{
\coordinate (Aa) at (0,-.3);
\coordinate (Ab) at (.3,0);
\coordinate (i1) at (0,.3);
\coordinate (i2) at (.3,-.6);
\draw (Aa) to[in=-90,out=90] (0,.7);
\draw (Ab) to[in=-90,out=90] (.3,.7);
\RectangleMorphism{(i1)}{white}
\RectangleMorphism{(i2)}{white}
\RectangleMorphism{(Aa)}{\hColor}
\RectangleMorphism{(Ab)}{\kColor}
}
\arrow[d, equals]
\arrow[r, Rightarrow, "\mu^A"]
&
\tikzmath{
\coordinate (Aa) at (.2,-.2);
\coordinate (Ab) at (0,-.6);
\coordinate (i1) at (0,0);
\coordinate (i2) at (.3,-.9);
\draw (Aa) to[in=-90,out=90] (.2,.4);
\draw (Ab) to[in=-90,out=90] (0,.4);
\RectangleMorphism{(Aa)}{\kColor}
\RectangleMorphism{(Ab)}{\hColor}
\RectangleMorphism{(i1)}{white}
\RectangleMorphism{(i2)}{white}
}
\arrow[d, equals]
\arrow[r, Rightarrow, "\mu^A"]
\arrow[uu, phantom, "\text{\scriptsize \ref{Functor:mu.monoidal}}"]
&
\tikzmath{
\coordinate (Aa) at (.2,-.2);
\coordinate (Ab) at (0,-.6);
\coordinate (i1) at (0,0);
\coordinate (i2) at (.2,-.8);
\draw (Aa) to[in=-90,out=90] (.2,.4);
\draw (Ab) to[in=-90,out=90] (0,.4);
\RectangleMorphism{(Aa)}{\kColor}
\RectangleMorphism{(Ab)}{\hColor}
\RectangleMorphism{(i1)}{white}
\RectangleMorphism{(i2)}{white}
}
=
\tikzmath{
\coordinate (Aa) at (0,0);
\coordinate (Ab) at (.2,-.6);
\coordinate (i2) at (.2,-.2);
\draw (Aa) to[in=-90,out=90] (0,.4);
\draw (Ab) to[in=-90,out=90] (.2,.4);
\RedRectangleMorphism{(Aa)}{\kColor}
\RectangleMorphism{(Ab)}{\hColor}
\RectangleMorphism{(i2)}{white}
}
\arrow[ur, Rightarrow, "\text{\scriptsize \ref{Functor:lr}}", "A^1_e"']
\arrow[uul, Rightarrow, "A^2"]
&
\tikzmath{
\coordinate (Aa) at (0,0);
\coordinate (Ab) at (.3,-.6);
\coordinate (i2) at (.3,-.3);
\draw (Aa) to[in=-90,out=90] (0,.5);
\draw (Ab) to[in=-90,out=90] (.3,.5);
\RedRectangleMorphism{(Aa)}{\kColor}
\RectangleMorphism{(Ab)}{\hColor}
\RectangleMorphism{(i2)}{white}
}
\arrow[ur, Rightarrow, "A^1_e","\text{\scriptsize \ref{Interchanger:Natural}}"']
\arrow[r, Rightarrow, "A^1_e"]
\arrow[l, Rightarrow, "\mu^A"]
&
\tikzmath{
\coordinate (Aa) at (0,0);
\coordinate (Ab) at (.3,-.9);
\coordinate (i1) at (0,-.3);
\coordinate (i2) at (.3,-.6);
\draw (Aa) to[in=-90,out=90] (0,.4);
\draw (Ab) to[in=-90,out=90] (.3,.4);
\RedRectangleMorphism{(Aa)}{\kColor}
\RectangleMorphism{(Ab)}{\hColor}
\RectangleMorphism{(i1)}{white}
\RectangleMorphism{(i2)}{white}
}
\arrow[ddr, Rightarrow, near start, "A^2\xz A^2"]
\arrow[u, Rightarrow, "\phi"]
\arrow[dl, phantom, near start, "\text{\scriptsize \ref{Functor:A.unital}}"]
\\
\tikzmath{
\coordinate (Aa) at (0,0);
\coordinate (Ab) at (.4,-.5);
\coordinate (i1) at (-.2,.3);
\draw (Aa) to[in=-90,out=90] (0,.6);
\draw (Ab) to[in=-90,out=90] (.4,.6);
\draw[red] (-.2,.6) -- (-.2,-.1) arc (-180:0:.2cm) -- (.2,.6);
\RectangleMorphism{(i1)}{white}
\RectangleMorphism{(Aa)}{\kColor}
\RectangleMorphism{(Ab)}{\hColor}
}
\arrow[r, Rightarrow, "A^1_e"]
\arrow[drr, Rightarrow, "\mu^A"]
&
\tikzmath{
\coordinate (Aa) at (0,0);
\coordinate (Ab) at (.4,-.5);
\coordinate (i1) at (-.2,.3);
\coordinate (i2) at (.6,-.8);
\draw (Aa) to[in=-90,out=90] (0,.7);
\draw (Ab) to[in=-90,out=90] (.4,.7);
\draw[red] (-.2,.6) -- (-.2,-.1) arc (-180:0:.2cm) -- (.2,.7);
\RectangleMorphism{(i1)}{white}
\RectangleMorphism{(i2)}{white}
\RectangleMorphism{(Aa)}{\kColor}
\RectangleMorphism{(Ab)}{\hColor}
}
\arrow[r, Rightarrow, "\mu^A"]
&
\tikzmath{
\coordinate (Aa) at (0,0);
\coordinate (Ab) at (.4,-.5);
\coordinate (i1) at (-.2,.2);
\coordinate (i2) at (.6,-.8);
\draw (0,0) to[in=-90,out=90] (0,.6);
\draw (Ab) to[in=-90,out=90] (.4,.6);
\draw[red] (-.2,.6) -- (-.2,-.1) arc (-180:0:.2cm) -- (.2,.6);
\RectangleMorphism{(i1)}{white}
\RectangleMorphism{(i2)}{white}
\RectangleMorphism{(Aa)}{\kColor}
\RectangleMorphism{(Ab)}{\hColor}
}
&
\tikzmath{
\coordinate (Aa) at (0,0);
\coordinate (Ab) at (.5,-.8);
\coordinate (i1) at (-.2,.2);
\coordinate (i2) at (.4,-.3);
\draw (0,0) to[in=-90,out=90] (0,.6);
\draw (Ab) to[in=-90,out=90] (.5,.6);
\draw[red] (-.2,.6) -- (-.2,-.4) arc (-180:0:.25cm) -- (.3,.6);
\filldraw[thick, fill=white] ($ (i2) + (-.2,-.1) $) rectangle ($ (i2) + (.2,.1) $);
\RectangleMorphism{(i1)}{white}
\RectangleMorphism{(Aa)}{\kColor}
\RectangleMorphism{(Ab)}{\hColor}
}
\arrow[u, Rightarrow, "\mu^A"]
&
\tikzmath{
\coordinate (Aa) at (0,0);
\coordinate (Ab) at (.4,-1);
\coordinate (i1) at (-.2,.2);
\coordinate (i2) at (.2,-.3);
\coordinate (i3) at (.4,-.7);
\draw (0,0) to[in=-90,out=90] (0,.6);
\draw (Ab) to[in=-90,out=90] (.4,.6);
\draw[red] (-.2,.6) -- (-.2,-.4) arc (-180:0:.2cm) -- (.2,.6);
\RectangleMorphism{(i1)}{white}
\RectangleMorphism{(i2)}{white}
\RectangleMorphism{(i3)}{white}
\RectangleMorphism{(Aa)}{\kColor}
\RectangleMorphism{(Ab)}{\hColor}
}
\arrow[u, Rightarrow, "\text{\scriptsize \ref{Functor:omega}}\hspace{.5cm}\mu^A"]
\arrow[l, Rightarrow, "\text{\scriptsize \ref{Functor:mu.unital}}", "\mu^A"']
\\
&&
\tikzmath{
\coordinate (Aa) at (0,0);
\coordinate (Ab) at (.4,-.5);
\coordinate (i1) at (-.2,.2);
\draw (Aa) to[in=-90,out=90] (0,.6);
\draw (Ab) to[in=-90,out=90] (.4,.6);
\draw[red] (-.2,.6) -- (-.2,-.1) arc (-180:0:.2cm) -- (.2,.6);
\RectangleMorphism{(i1)}{white}
\RectangleMorphism{(Aa)}{\kColor}
\RectangleMorphism{(Ab)}{\hColor}
}
\arrow[ur, Rightarrow, "A^1_{e}"]
\arrow[urr, Rightarrow, swap, "A^1_{g^{-1}_\cC} \xz A^1_{g}"]
\arrow[uur, equals, "\text{\scriptsize \ref{Functor:lr}}","\text{\scriptsize \ref{Functor:lr}}"']
\arrow[u, Rightarrow, "A^1_e"]
\arrow[rr, Rightarrow, "A^1_{g^{-1}_\cC}"]
&&
\tikzmath{
\coordinate (Aa) at (0,0);
\coordinate (Ab) at (.4,-.7);
\coordinate (i1) at (-.2,.2);
\coordinate (i2) at (.2,-.3);
\draw (Aa) to[in=-90,out=90] (0,.6);
\draw (Ab) to[in=-90,out=90] (.4,.6);
\draw[red] (-.2,.6) -- (-.2,-.1) arc (-180:0:.2cm) -- (.2,.6);
\RectangleMorphism{(i1)}{white}
\RectangleMorphism{(i2)}{white}
\RectangleMorphism{(Aa)}{\kColor}
\RectangleMorphism{(Ab)}{\hColor}
}
\arrow[u, Rightarrow, "A^1_{g}"]
\arrow[rr, Rightarrow, "\mu^A"]
&&
\tikzmath{
\coordinate (Aa) at (0,0);
\coordinate (Ab) at (.3,-.3);
\draw (Aa) to[in=-90,out=90] (0,.4);
\draw (Ab) to[in=-90,out=90] (.3,.4);
\RedRectangleMorphism{(Aa)}{\kColor}
\RectangleMorphism{(Ab)}{\hColor}
}
\arrow[uull, Rightarrow, "A^1_{g}"]
\arrow[uuuu, Rightarrow, "\mu^A"]
\end{tikzcd}};\end{tikzpicture}
$$
Again, each square above is labelled by the property for $A$ which makes it commute.
\end{proof}

\begin{rem}
By an argument similar to the right half of the commutative diagram in the proof of \eqref{eq:GCrossedFunctor-Gamma2} above, for a functor $A \in \TriCat_G^{\st}$,
$x\in \cC(1_\cC \to g_\cC)$, and $y\in \cC(h)\cC \to k_\cC)$,
 the following square commutes:
\begin{equation}
\label{eq:Functor:mu.monoidal.fewer}
\begin{tikzcd}
\tikzmath{
\node (g) at (0,3.2) {$\scriptstyle g_\cD$};
\node (h) at (1,0) {$\scriptstyle h_\cD$};
\node (k) at (1,3.2) {$\scriptstyle k_\cD$};
\node[draw,rectangle, thick, rounded corners=5pt] (Ax) at (0,2.4) {$\scriptstyle A(x)$};
\node[draw,rectangle, thick, rounded corners=5pt] (Ak) at (1,1.6) {$\scriptstyle A(\id_{k_\cC})$};
\node[draw,rectangle, thick, rounded corners=5pt] (Ay) at (1,.8) {$\scriptstyle A(y)$};
\draw (Ax) to[in=-90,out=90] (g);
\draw (h) to[in=-90,out=90] (Ay);
\draw (Ay) to[in=-90,out=90] (Ak);
\draw (Ak) to[in=-90,out=90] (k);
}
\arrow[r, Rightarrow, "\mu^A"]
\arrow[d, Rightarrow, "A^2_{\id_{k_\cC}, y}"]
&
\tikzmath{
\node (h) at (0,0) {$\scriptstyle h_\cD$};
\node (gk) at (0,2.4) {$\scriptstyle gk_\cD$};
\node[draw,rectangle, thick, rounded corners=5pt] (Ax) at (0,1.6) {$\scriptstyle A(x \xz \id_{k_\cC})$};
\node[draw,rectangle, thick, rounded corners=5pt] (Ay) at (0,.8) {$\scriptstyle A(y)$};
\draw (h) to[in=-90,out=90] (Ay);
\draw (Ay) to[in=-90,out=90] (Ax);
\draw (Ax) to[in=-90,out=90] (gk);
}
\arrow[d, Rightarrow, "A^2_{x\xz \id_{k_\cC}, y}"]
\\
\tikzmath{
\node (g) at (0,2.4) {$\scriptstyle g_\cD$};
\node (h) at (1,0) {$\scriptstyle h_\cD$};
\node (k) at (1,2.4) {$\scriptstyle k_\cD$};
\node[draw,rectangle, thick, rounded corners=5pt] (Ax) at (0,1.6) {$\scriptstyle A(x)$};
\node[draw,rectangle, thick, rounded corners=5pt] (Ay) at (1,.8) {$\scriptstyle A(y)$};
\draw (Ax) to[in=-90,out=90] (g);
\draw (h) to[in=-90,out=90] (Ay);
\draw (Ay) to[in=-90,out=90] (k);
}
\arrow[r, Rightarrow, "\mu^A_{x,y}"]
&
\tikzmath{
\node (gk) at (0,1.6) {$\scriptstyle gk_\cD$};
\node (h) at (0,0) {$\scriptstyle h_\cD$};
\node[draw,rectangle, thick, rounded corners=5pt] (Axy) at (0,.8) {$\scriptstyle A(x\xz y)$};
\draw (h) to[in=-90,out=90] (Axy);
\draw (Axy) to[in=-90,out=90] (gk);
}
\end{tikzcd}
\end{equation}
\end{rem}

\begin{proof}[Proof of Prop.~\ref{prop:2FunctorStrictlyUnitalAndAssociative}: $(A,\mu^A, \iota^A)\mapsto (\bfA, \bfa)$ is strict.]
\mbox{}

It is straightforward to see that if $(A, A^1, A^2, \mu^A, \iota^A)\in \TriCat_G^{\st}(\cC \to \cC)$ is the identity 3-functor, then so is $(\bfA, \bfA^1, \bfA^2, \bfa)\in G\CrsBrd^{\st}(\fC \to \fC)$.
Suppose now we have two composable 1-morphisms
$(A, A^1, A^2, \mu^A, \iota^A) \in \TriCat_G^{\st}(\cD \to \cE)$
and
$(B, B^1, B^2, \mu^B, \iota^B) \in \TriCat_G^{\st}(\cC \to \cD)$.
We now calculate the composition formulas for the composite $G$-crossed braided monoidal functor
$(\bfA \circ \bfB, (\bfA\circ \bfB)^1, (\bfA\circ \bfB)^2, \bfa\circ \bfb)$
associated to
$(A\circ B, (A\circ B)^2, (A\circ B)^2, \mu^{A\circ B}, \iota^{A\circ B})$.
The unitor $(\bfA \circ \bfB)^1$ and tensorator $(\bfA \circ \bfB)^2$ are straightforward:
\begin{align*}
(\bfA \circ \bfB)^1
&= 
(A \circ B)^1_e 
= 
A(B^1_e) \xt A^1_{e}
=
\bfA(\bfB^1) \xt \bfB^1_e
\\
(\bfA \circ \bfB)^2_{x,y}
&=
\mu_{x,y}^{A\circ B}
=
A(\mu^B_{x,y}) \xt \mu^A_{B(x),B(y)}  
=
\bfA(\bfB^2_{x,y}) \xt \bfA^2_{\bfB(x), \bfB(y)}. 
\end{align*}

To compute $(\bfa\circ \bfb)_x$ for $x\in \cC(1_\cC \to g_\cC)$, we use the following shorthand as in Notation \ref{nota:ShadedBoxes}, where black rectangles and strings corresponds to 1-cells in $\cE$ after applying $A$, and blue rectangles and strands corresponds to 1-cells in $\cD$ after applying $B$.
We also draw red strands to denote $\id_{g}$ in both $\cD$ and $\cE$.
For example:
$$
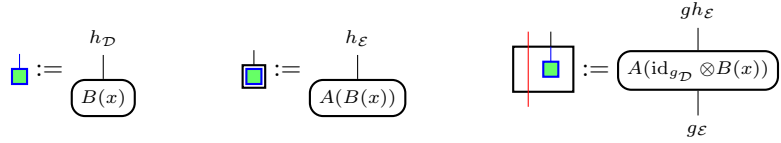
\tikzmath{
\draw[blue] (0,0) -- (0,.3);
\filldraw[fill=\hColor, draw=blue, thick] (-.1,-.1) rectangle (.1,.1);
}
:=
\tikzmath{
\node (h) at (0,.8) {$\scriptstyle h_\cD$};
\node[draw,rectangle, thick, rounded corners=5pt] (Ax) at (0,0) {$\scriptstyle B(x)$};
\draw (Ax) to[in=-90,out=90] (h);
}
\qquad\qquad
\tikzmath{
\draw (0,0) -- (0,.35);
\filldraw[fill=white, thick] (-.15,-.15) rectangle (.15,.15);
\filldraw[fill=\hColor, draw=blue, thick] (-.1,-.1) rectangle (.1,.1);
}
:=
\tikzmath{
\node (h) at (0,.8) {$\scriptstyle h_\cE$};
\node[draw,rectangle, thick, rounded corners=5pt] (ABx) at (0,0) {$\scriptstyle A(B(x))$};
\draw (ABx) to[in=-90,out=90] (h);
}
\qquad\qquad
\tikzmath{
\draw (0,.3) -- (0,.5);
\draw[thick] (-.5,-.3) rectangle (.3,.3);
\filldraw[fill=\hColor, draw=blue, thick] (-.1,-.1) rectangle (.1,.1);
\draw[blue] (0,.1) -- (0,.3);
\draw[red] (-.3,-.5) -- (-.3,.5);
}
:=
\tikzmath{
\node (h) at (0,.8) {$\scriptstyle gh_\cE$};
\node (g) at (0,-.8) {$\scriptstyle g_\cE$};
\node[draw,rectangle, thick, rounded corners=5pt] (Bx) at (0,0) {$\scriptstyle A(\id_{g_\cD} \xz B(x))$};
\draw (g) to[in=-90,out=90] (Bx);
\draw (Bx) to[in=-90,out=90] (h);
}
$$
We draw unshaded boxes on red strands to denote $B(\id_{g_\cC}), B(\id_{g_\cD^{-1}}$, $A(\id_{g_\cD}),A(\id_{g_\cD^{-1}}$.
$$
\tikzmath{
\draw[blue, thick] (-.1,-.1) rectangle (.1,.1);
\draw[red] (0,-.3) -- (0,.3);
}
:=
\tikzmath{
\node (g1) at (0,-.8) {$\scriptstyle g_\cD$};
\node (g2) at (0,.8) {$\scriptstyle g_\cD$};
\node[draw,rectangle, thick, rounded corners=5pt] (Ai) at (0,0) {$\scriptstyle B(\id_{g_\cC})$};
\draw (g1) to[in=-90,out=90] (Ai);
\draw (Ai) to[in=-90,out=90] (g2);
}
\qquad\qquad
\tikzmath{
\draw[thick] (-.15,-.15) rectangle (.15,.15);
\draw[blue, thick] (-.1,-.1) rectangle (.1,.1);
\draw[red] (0,-.35) -- (0,.35);
}
:=
\tikzmath{
\node (g1) at (0,-.8) {$\scriptstyle g_\cE$};
\node (g2) at (0,.8) {$\scriptstyle g_\cE$};
\node[draw,rectangle, thick, rounded corners=5pt] (Ai) at (0,0) {$\scriptstyle A(B(\id_{g_\cC}))$};
\draw (g1) to[in=-90,out=90] (Ai);
\draw (Ai) to[in=-90,out=90] (g2);
}
\qquad\qquad
\tikzmath{
\draw[thick] (-.1,-.1) rectangle (.1,.1);
\draw[red] (0,-.3) -- (0,.3);
}
:=
\tikzmath{
\node (g1) at (0,-.8) {$\scriptstyle g_\cE$};
\node (g2) at (0,.8) {$\scriptstyle g_\cE$};
\node[draw,rectangle, thick, rounded corners=5pt] (Ai) at (0,0) {$\scriptstyle A(\id_{g_\cD})$};
\draw (g1) to[in=-90,out=90] (Ai);
\draw (Ai) to[in=-90,out=90] (g2);
}
$$
The composite along the diagonal in the commuting diagram below is the definition of $(\bfa \circ \bfb)_x$.
Each face without a label above commutes by functoriality of 1-cell composition $\xo$.
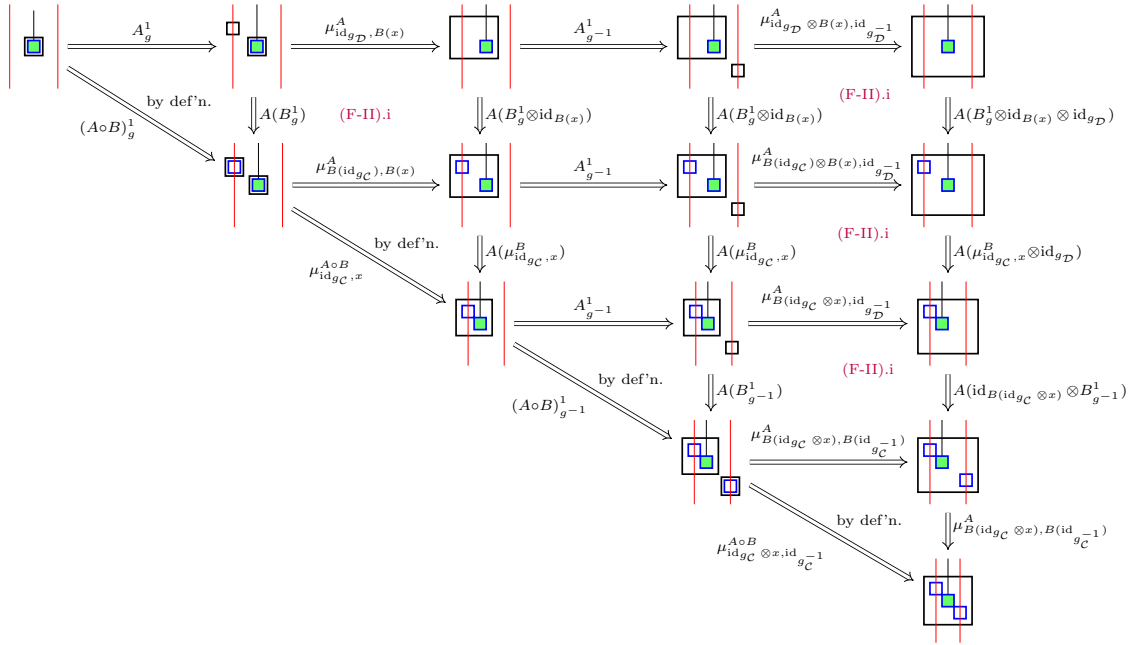
$$
\begin{tikzpicture}[baseline= (a).base]\node[scale=.8] (a) at (0,0){\begin{tikzcd}[column sep=7em]
\tikzmath{
\coordinate (x) at (0,0);
\draw (x) -- ($ (x) + (0,.6) $);
\BlueRectangleMorphism{(x)}{\hColor}
\draw[thick] ($ (x) + (-.15, -.15) $) rectangle ($ (x) + (.15, .15) $);
\draw[red] (-.4,-.7) -- (-.4,.7);
\draw[red] (.4,-.7) -- (.4,.7);
}
\arrow[r, Rightarrow, "A^1_g"]
\arrow[dr, Rightarrow, "\text{\scriptsize by def'n.}", "(A\circ B)^1_g"']
&
\tikzmath{
\coordinate (ig) at (-.4,.3);
\coordinate (x) at (0,0);
\draw (x) -- ($ (x) + (0,.7) $);
\BlueRectangleMorphism{(x)}{\hColor}
\draw[thick] ($ (x) + (-.15, -.15) $) rectangle ($ (x) + (.15, .15) $);
\draw[thick] ($ (ig) + (-.1, -.1) $) rectangle ($ (ig) + (.1, .1) $);
\draw[red] (-.4,-.7) -- (-.4,.7);
\draw[red] (.4,-.7) -- (.4,.7);
}
\arrow[r, Rightarrow, "\mu^A_{\id_{g_\cD}, B(x)}"]
\arrow[d, Rightarrow, "A(B^1_g)"]
\arrow[dr, phantom, "\text{\scriptsize \ref{Functor:mu.natural}}"]
&
\tikzmath{
\coordinate (x) at (0,0);
\draw (x) -- ($ (x) + (0,.7) $);
\BlueRectangleMorphism{(x)}{\hColor}
\draw[thick] ($ (x) + (-.6, -.2) $) rectangle ($ (x) + (.2, .5) $);
\draw[red] (-.4,-.7) -- (-.4,.7);
\draw[red] (.4,-.7) -- (.4,.7);
}
\arrow[r, Rightarrow, "A^1_{g^{-1}}"]
\arrow[d, Rightarrow, "A(B^1_g\xz \id_{B(x)})"]
&
\tikzmath{
\coordinate (ig2) at (.4,-.4);
\coordinate (x) at (0,0);
\draw (x) -- ($ (x) + (0,.7) $);
\BlueRectangleMorphism{(x)}{\hColor}
\draw[thick] ($ (x) + (-.6, -.2) $) rectangle ($ (x) + (.2, .5) $);
\draw[thick] ($ (ig2) + (-.1, -.1) $) rectangle ($ (ig2) + (.1, .1) $);
\draw[red] (-.4,-.7) -- (-.4,.7);
\draw[red] (.4,-.7) -- (.4,.7);
}
\arrow[r, Rightarrow, "\mu^A_{\id_{g_\cD}\xz B(x), \id_{g^{-1}_\cD}}"]
\arrow[d, Rightarrow, "A(B^1_g\xz \id_{B(x)})"]
&
\tikzmath{
\coordinate (x) at (0,0);
\draw (x) -- ($ (x) + (0,.7) $);
\BlueRectangleMorphism{(x)}{\hColor}
\draw[thick] ($ (x) + (-.6, -.5) $) rectangle ($ (x) + (.6, .5) $);
\draw[red] (-.4,-.7) -- (-.4,.7);
\draw[red] (.4,-.7) -- (.4,.7);
}
\arrow[d, Rightarrow, "A(B^1_g\xz \id_{B(x)}\xz \id_{g_\cD})"]
\arrow[dl, phantom, near start, "\text{\scriptsize \ref{Functor:mu.natural}}"]
\\
&
\tikzmath{
\coordinate (ig) at (-.4,.3);
\coordinate (x) at (0,0);
\draw (x) -- ($ (x) + (0,.7) $);
\BlueRectangleMorphism{(x)}{\hColor}
\draw[thick] ($ (x) + (-.15, -.15) $) rectangle ($ (x) + (.15, .15) $);
\draw[thick, blue] ($ (ig) + (-.1, -.1) $) rectangle ($ (ig) + (.1, .1) $);
\draw[thick] ($ (ig) + (-.15, -.15) $) rectangle ($ (ig) + (.15, .15) $);
\draw[red] (-.4,-.7) -- (-.4,.7);
\draw[red] (.4,-.7) -- (.4,.7);
}
\arrow[r, Rightarrow, "\mu^A_{B(\id_{g_\cC}), B(x)}"]
\arrow[dr, Rightarrow, "\text{\scriptsize by def'n.}", "\mu^{A\circ B}_{\id_{g_\cC}, x}"']
&
\tikzmath{
\coordinate (ig) at (-.4,.3);
\coordinate (x) at (0,0);
\draw (x) -- ($ (x) + (0,.7) $);
\BlueRectangleMorphism{(x)}{\hColor}
\draw[thick, blue] ($ (ig) + (-.1, -.1) $) rectangle ($ (ig) + (.1, .1) $);
\draw[thick] ($ (x) + (-.6, -.2) $) rectangle ($ (x) + (.2, .5) $);
\draw[red] (-.4,-.7) -- (-.4,.7);
\draw[red] (.4,-.7) -- (.4,.7);
}
\arrow[r, Rightarrow, "A^1_{g^{-1}}"]
\arrow[d, Rightarrow, "A(\mu^B_{\id_{g_\cC}, x})"]
&
\tikzmath{
\coordinate (ig) at (-.4,.3);
\coordinate (ig2) at (.4,-.4);
\coordinate (x) at (0,0);
\draw (x) -- ($ (x) + (0,.7) $);
\BlueRectangleMorphism{(x)}{\hColor}
\draw[thick, blue] ($ (ig) + (-.1, -.1) $) rectangle ($ (ig) + (.1, .1) $);
\draw[thick] ($ (x) + (-.6, -.2) $) rectangle ($ (x) + (.2, .5) $);
\draw[thick] ($ (ig2) + (-.1, -.1) $) rectangle ($ (ig2) + (.1, .1) $);
\draw[red] (-.4,-.7) -- (-.4,.7);
\draw[red] (.4,-.7) -- (.4,.7);
}
\arrow[r, Rightarrow, "\mu^A_{B(\id_{g_\cC})\xz B(x), \id_{g^{-1}_\cD}}"]
\arrow[d, Rightarrow, "A(\mu^B_{\id_{g_\cC}, x})"]
&
\tikzmath{
\coordinate (ig) at (-.4,.3);
\coordinate (x) at (0,0);
\draw (x) -- ($ (x) + (0,.7) $);
\BlueRectangleMorphism{(x)}{\hColor}
\draw[thick, blue] ($ (ig) + (-.1, -.1) $) rectangle ($ (ig) + (.1, .1) $);
\draw[thick] ($ (x) + (-.6, -.5) $) rectangle ($ (x) + (.6, .5) $);
\draw[red] (-.4,-.7) -- (-.4,.7);
\draw[red] (.4,-.7) -- (.4,.7);
}
\arrow[d, Rightarrow, "A(\mu^B_{\id_{g_\cC}, x} \xz \id_{g_\cD})"]
\arrow[dl, phantom, near start, "\text{\scriptsize \ref{Functor:mu.natural}}"]
\\
&&
\tikzmath{
\coordinate (ig) at (-.2,.2);
\coordinate (x) at (0,0);
\draw (x) -- ($ (x) + (0,.7) $);
\BlueRectangleMorphism{(x)}{\hColor}
\draw[thick, blue] ($ (ig) + (-.1, -.1) $) rectangle ($ (ig) + (.1, .1) $);
\draw[thick] ($ (x) + (-.4, -.2) $) rectangle ($ (x) + (.2, .4) $);
\draw[red] (-.2,-.7) -- (-.2,.7);
\draw[red] (.4,-.7) -- (.4,.7);
}
\arrow[r, Rightarrow, "A^1_{g^{-1}}"]
\arrow[dr, Rightarrow, "\text{\scriptsize by def'n.}", "(A\circ B)^1_{g^{-1}}"']
&
\tikzmath{
\coordinate (ig) at (-.2,.2);
\coordinate (ig2) at (.4,-.4);
\coordinate (x) at (0,0);
\draw (x) -- ($ (x) + (0,.7) $);
\BlueRectangleMorphism{(x)}{\hColor}
\draw[thick, blue] ($ (ig) + (-.1, -.1) $) rectangle ($ (ig) + (.1, .1) $);
\draw[thick] ($ (x) + (-.4, -.2) $) rectangle ($ (x) + (.2, .4) $);
\draw[thick] ($ (ig2) + (-.1, -.1) $) rectangle ($ (ig2) + (.1, .1) $);
\draw[red] (-.2,-.7) -- (-.2,.7);
\draw[red] (.4,-.7) -- (.4,.7);
}
\arrow[r, Rightarrow, "\mu^A_{B(\id_{g_\cC}\xz x), \id_{g^{-1}_\cD}}"]
\arrow[d, Rightarrow, "A(B^1_{g^{-1}})"]
&
\tikzmath{
\coordinate (ig) at (-.2,.2);
\coordinate (x) at (0,0);
\draw (x) -- ($ (x) + (0,.7) $);
\BlueRectangleMorphism{(x)}{\hColor}
\draw[thick, blue] ($ (ig) + (-.1, -.1) $) rectangle ($ (ig) + (.1, .1) $);
\draw[thick] ($ (x) + (-.4, -.5) $) rectangle ($ (x) + (.6, .4) $);
\draw[red] (-.2,-.7) -- (-.2,.7);
\draw[red] (.4,-.7) -- (.4,.7);
}
\arrow[d, Rightarrow, "A(\id_{B(\id_{g_\cC}\xz x)}\xz B^1_{g^{-1}})"]
\arrow[dl, phantom, near start, "\text{\scriptsize \ref{Functor:mu.natural}}"]
\\
&&&
\tikzmath{
\coordinate (ig) at (-.2,.2);
\coordinate (ig2) at (.4,-.4);
\coordinate (x) at (0,0);
\draw (x) -- ($ (x) + (0,.7) $);
\BlueRectangleMorphism{(x)}{\hColor}
\draw[thick, blue] ($ (ig) + (-.1, -.1) $) rectangle ($ (ig) + (.1, .1) $);
\draw[thick] ($ (x) + (-.4, -.2) $) rectangle ($ (x) + (.2, .4) $);
\draw[thick, blue] ($ (ig2) + (-.1, -.1) $) rectangle ($ (ig2) + (.1, .1) $);
\draw[thick] ($ (ig2) + (-.15, -.15) $) rectangle ($ (ig2) + (.15, .15) $);
\draw[red] (-.2,-.7) -- (-.2,.7);
\draw[red] (.4,-.7) -- (.4,.7);
}
\arrow[r, Rightarrow, "\mu^A_{B(\id_{g_\cC}\xz x), B(\id_{g^{-1}_\cC})}"]
\arrow[dr, Rightarrow, "\text{\scriptsize by def'n.}", "\mu^{A\circ B}_{\id_{g_\cC}\xz x, \id_{g^{-1}_\cC}}"']
&
\tikzmath{
\coordinate (ig) at (-.2,.2);
\coordinate (ig2) at (.4,-.3);
\coordinate (x) at (0,0);
\draw (x) -- ($ (x) + (0,.7) $);
\BlueRectangleMorphism{(x)}{\hColor}
\draw[thick, blue] ($ (ig) + (-.1, -.1) $) rectangle ($ (ig) + (.1, .1) $);
\draw[thick] ($ (x) + (-.4, -.5) $) rectangle ($ (x) + (.6, .4) $);
\draw[thick, blue] ($ (ig2) + (-.1, -.1) $) rectangle ($ (ig2) + (.1, .1) $);
\draw[red] (-.2,-.7) -- (-.2,.7);
\draw[red] (.4,-.7) -- (.4,.7);
}
\arrow[d, Rightarrow, "\mu^A_{B(\id_{g_\cC}\xz x), B(\id_{g_\cC^{-1}})}"]
\\
&&&&
\tikzmath{
\coordinate (ig) at (-.2,.2);
\coordinate (ig2) at (.2,-.2);
\coordinate (x) at (0,0);
\draw (x) -- ($ (x) + (0,.7) $);
\BlueRectangleMorphism{(x)}{\hColor}
\draw[thick, blue] ($ (ig) + (-.1, -.1) $) rectangle ($ (ig) + (.1, .1) $);
\draw[thick] ($ (x) + (-.4, -.4) $) rectangle ($ (x) + (.4, .4) $);
\draw[thick, blue] ($ (ig2) + (-.1, -.1) $) rectangle ($ (ig2) + (.1, .1) $);
\draw[red] (-.2,-.7) -- (-.2,.7);
\draw[red] (.2,-.7) -- (.2,.7);
}
\end{tikzcd}};\end{tikzpicture}
$$
As the above diagram commutes, $(\bfa \circ \bfb)_x = \bfA(\bfa_{F^\cD_{g}(B(x)}) \xt \bfa_{\bfB(x)}$.

Finally, we observe this data agrees with the composite data for the data for the composite of the  $G$-crossed braided monoidal functors $(\bfA,\bfA^1,\bfA^2,\bfa)$ and $(\bfB, \bfB^1, \bfB^2, \bfb)$ in $G\CrsBrd$.
\end{proof}

For $\cC, \cD \in \TriCat_G^{\st}$, $A, B\in \TriCat_G^{\st}(\cC \to \cD)$, and $\eta\in \TriCat_G^{\st}(A \Rightarrow B)$,
let $\fC, \fD$ be the $G$-crossed braided categories obtained from $\cC,\cD$ respectively
from Theorem \ref{thm:ConstructionOfGCrossedBraidedCategory}, and let $(\bfA,\bfa),(\bfB,\bfb):\fC \to \fD$ be the $G$-crossed braided functors obtained from $A,B$ respectively from Theorem \ref{thm:From3FunctorToGCrossedBraidedFunctor}.
In Construction \ref{const:GMonoidalTransformationFrom3Transformation}
we defined $h:(\bfA,\bfa) \Rightarrow (\bfB,\bfb)$ by $h_a := \eta_a \in \cD(\bfA(a) \Rightarrow \bfB(a))$ for $a\in \fC_g = \cC(1_\cC \to g_\cC)$.

\begin{proof}[Proof of Thm.~\ref{thm:GMonoidalTransformation}: $h:(\bfA,\bfa) \Rightarrow (\bfB,\bfb)$ is a $G$-crossed braided monoidal transformation]
\item[\underline{Naturality:}]
This is immediate by the definition $h_x := \eta_x$ for $x\in \fC_g = \cC(1_\cC \to g_\cC)$ and  \ref{Transformation:eta_c.natural}.

\item[\underline{Unitality:}]
By \ref{Transformation:eta_c.unital}.
$h_{1_\fC} = \eta_{\id_{e_\cC}} = B^1_e \xt (A^1_e)^{-1} = \bfB^1 \xt (\bfA^1)^{-1}$.

\item[\underline{Monoidality:}]
That $\bfB^2_{x,y} \xt (h_x \xz h_y) = h_{x\xz y} \xt \bfA^2_{x,y}$ follows immediately by \ref{Transformation:eta^2}.

\item[\eqref{eq:ActionCohrenceForTransformation}]
For $x\in \fC_h = \cC(1_\cC \to h_\cC)$, we use the following shorthand as in Notation \ref{nota:ShadedBoxes}.
We also draw red strands to denote $\id_{g}$ in $\cD$, and we draw unshaded boxes on red strands to denote each of $A(\id_{g_\cC}),A(\id_{g_\cC^{-1}}$ and $B(\id_{g_\cC}), B(\id_{g_\cD^{-1}}$, 
For example:
$$
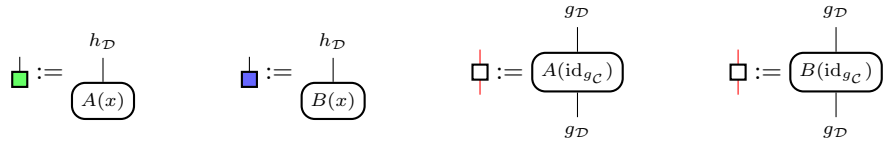
\tikzmath{
\draw (0,0) -- (0,.3);
\filldraw[fill=\hColor, thick] (-.1,-.1) rectangle (.1,.1);
}
:=
\tikzmath{
\node (h) at (0,.8) {$\scriptstyle h_\cD$};
\node[draw,rectangle, thick, rounded corners=5pt] (Ax) at (0,0) {$\scriptstyle A(x)$};
\draw (Ax) to[in=-90,out=90] (h);
}
\qquad\qquad
\tikzmath{
\draw (0,0) -- (0,.3);
\filldraw[fill=\kColor, thick] (-.1,-.1) rectangle (.1,.1);
}
:=
\tikzmath{
\node (h) at (0,.8) {$\scriptstyle h_\cD$};
\node[draw,rectangle, thick, rounded corners=5pt] (Ax) at (0,0) {$\scriptstyle B(x)$};
\draw (Ax) to[in=-90,out=90] (h);
}
\qquad\qquad
\tikzmath{
\draw[red] (0,-.3) -- (0,.3);
\filldraw[thick, fill=white] (-.1,-.1) rectangle (.1,.1);
}
:=
\tikzmath{
\node (g1) at (0,-.8) {$\scriptstyle g_\cD$};
\node (g2) at (0,.8) {$\scriptstyle g_\cD$};
\node[draw,rectangle, thick, rounded corners=5pt] (Ai) at (0,0) {$\scriptstyle A(\id_{g_\cC})$};
\draw (g1) to[in=-90,out=90] (Ai);
\draw (Ai) to[in=-90,out=90] (g2);
}
\qquad\qquad
\tikzmath{
\draw[red] (0,-.3) -- (0,.3);
\filldraw[thick, fill=white] (-.1,-.1) rectangle (.1,.1);
}
:=
\tikzmath{
\node (g1) at (0,-.8) {$\scriptstyle g_\cD$};
\node (g2) at (0,.8) {$\scriptstyle g_\cD$};
\node[draw,rectangle, thick, rounded corners=5pt] (Ai) at (0,0) {$\scriptstyle B(\id_{g_\cC})$};
\draw (g1) to[in=-90,out=90] (Ai);
\draw (Ai) to[in=-90,out=90] (g2);
}
$$
The outside of the commuting diagram below corresponds to \eqref{eq:ActionCohrenceForTransformation}.
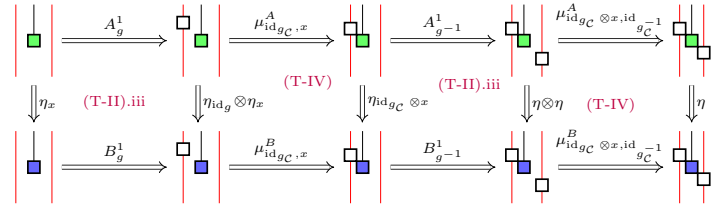
$$
\begin{tikzpicture}[baseline= (a).base]\node[scale=.8] (a) at (0,0){\begin{tikzcd}[column sep=5em]
\tikzmath{
\coordinate (x) at (0,0);
\draw (x) -- ($ (x) + (0,.6) $);
\draw[red] (-.3,-.6) -- (-.3,.6);
\draw[red] (.3,-.6) -- (.3,.6);
\RectangleMorphism{(x)}{\hColor}
}
\arrow[r, Rightarrow, "A^1_g"]
\arrow[d, Rightarrow, "\eta_x"]
\arrow[dr, phantom, "\text{\scriptsize \ref{Transformation:eta_c.unital}}"]
&
\tikzmath{
\coordinate (x) at (0,0);
\coordinate (ig) at (-.3,.3);
\draw (x) -- ($ (x) + (0,.6) $);
\draw[red] (-.3,-.6) -- (-.3,.6);
\draw[red] (.3,-.6) -- (.3,.6);
\RectangleMorphism{(x)}{\hColor}
\RectangleMorphism{(ig)}{white}
}
\arrow[r, Rightarrow, "\mu^A_{\id_{g_\cC}, x}"]
\arrow[d, Rightarrow, "\eta_{\id_g}\xz \eta_x"]
&
\tikzmath{
\coordinate (x) at (0,0);
\coordinate (ig) at (-.2,.2);
\draw (x) -- ($ (x) + (0,.6) $);
\draw[red] (-.2,-.6) -- (-.2,.6);
\draw[red] (.3,-.6) -- (.3,.6);
\RectangleMorphism{(x)}{\hColor}
\RectangleMorphism{(ig)}{white}
}
\arrow[dl, phantom, near start, "\text{\scriptsize \ref{Transformation:eta^2}}"]
\arrow[r, Rightarrow, "A^1_{g^{-1}}"]
\arrow[d, Rightarrow, "\eta_{\id_{g_\cC}\xz x}"]
&
\tikzmath{
\coordinate (x) at (0,0);
\coordinate (ig) at (-.2,.2);
\coordinate (ig2) at (.3,-.3);
\draw (x) -- ($ (x) + (0,.6) $);
\draw[red] (-.2,-.6) -- (-.2,.6);
\draw[red] (.3,-.6) -- (.3,.6);
\RectangleMorphism{(x)}{\hColor}
\RectangleMorphism{(ig)}{white}
\RectangleMorphism{(ig2)}{white}
}
\arrow[dl, phantom, near start, "\text{\scriptsize \ref{Transformation:eta_c.unital}}"]
\arrow[dr, phantom, "\text{\scriptsize \ref{Transformation:eta^2}}"]
\arrow[r, Rightarrow, "\mu^A_{\id_{g_\cC}\xz x, \id_{g_\cC^{-1}}}"]
\arrow[d, Rightarrow, "\eta\xz \eta"]%_{\id_{g_\cC}\xz x} \xz \eta_{\id_{g_\cC^{-1}}}"]
&
\tikzmath{
\coordinate (x) at (0,0);
\coordinate (ig) at (-.2,.2);
\coordinate (ig2) at (.2,-.2);
\draw (x) -- ($ (x) + (0,.6) $);
\draw[red] (-.2,-.6) -- (-.2,.6);
\draw[red] (.2,-.6) -- (.2,.6);
\RectangleMorphism{(x)}{\hColor}
\RectangleMorphism{(ig)}{white}
\RectangleMorphism{(ig2)}{white}
}
\arrow[d, Rightarrow, "\eta"]%_{\id_{g_\cC}\xz x\xz \id_{g_\cC^{-1}}}"]
\\
\tikzmath{
\coordinate (x) at (0,0);
\draw (x) -- ($ (x) + (0,.6) $);
\draw[red] (-.3,-.6) -- (-.3,.6);
\draw[red] (.3,-.6) -- (.3,.6);
\RectangleMorphism{(x)}{\kColor}
}
\arrow[r, Rightarrow, "B^1_g"]
&
\tikzmath{
\coordinate (x) at (0,0);
\coordinate (ig) at (-.3,.3);
\draw (x) -- ($ (x) + (0,.6) $);
\draw[red] (-.3,-.6) -- (-.3,.6);
\draw[red] (.3,-.6) -- (.3,.6);
\RectangleMorphism{(x)}{\kColor}
\RectangleMorphism{(ig)}{white}
}
\arrow[r, Rightarrow, "\mu^B_{\id_{g_\cC}, x}"]
&
\tikzmath{
\coordinate (x) at (0,0);
\coordinate (ig) at (-.2,.2);
\draw (x) -- ($ (x) + (0,.6) $);
\draw[red] (-.2,-.6) -- (-.2,.6);
\draw[red] (.3,-.6) -- (.3,.6);
\RectangleMorphism{(x)}{\kColor}
\RectangleMorphism{(ig)}{white}
}
\arrow[r, Rightarrow, "B^1_{g^{-1}}"]
&
\tikzmath{
\coordinate (x) at (0,0);
\coordinate (ig) at (-.2,.2);
\coordinate (ig2) at (.3,-.3);
\draw (x) -- ($ (x) + (0,.6) $);
\draw[red] (-.2,-.6) -- (-.2,.6);
\draw[red] (.3,-.6) -- (.3,.6);
\RectangleMorphism{(x)}{\kColor}
\RectangleMorphism{(ig)}{white}
\RectangleMorphism{(ig2)}{white}
}
\arrow[r, Rightarrow, "\mu^B_{\id_{g_\cC}\xz x, \id_{g_\cC^{-1}}}"]
&
\tikzmath{
\coordinate (x) at (0,0);
\coordinate (ig) at (-.2,.2);
\coordinate (ig2) at (.2,-.2);
\draw (x) -- ($ (x) + (0,.6) $);
\draw[red] (-.2,-.6) -- (-.2,.6);
\draw[red] (.2,-.6) -- (.2,.6);
\RectangleMorphism{(x)}{\kColor}
\RectangleMorphism{(ig)}{white}
\RectangleMorphism{(ig2)}{white}
}
\end{tikzcd}};\end{tikzpicture}
$$
This completes the proof.
\end{proof}

%%%%%%%%%%%%%%%%%%%%%%%%%%%%%%%%%%%%%%%%%%%%%%%%%%
\subsection{Coherence proofs for the equivalence \S\ref{sec:2FunctorEquivalence}}
\label{sec:CoherenceProofsForEquivalence}

In this section, we supply the proofs from \S\ref{sec:2FunctorEquivalence} which prove that 
the strict 2-functor $\TriCat_G^{\st} \to G\CrsBrd^{\st}$ from Theorem \ref{thm:2Functor} is an equivalence.
We begin by expanding on Notation \ref{nota:ShadedBoxes}.

\begin{nota}
\label{nota:ShadedBoxesGCrossed}
In this section, we use an expanded shorthand notation for 1-cells in $\cD$ and $\fD$ for proofs using commutative diagrams.
For $x_1\in \cC(g_\cC \to h_\cC)$, $x_2 \in \cC(h_\cC \to k_\cC)$,
$y_1 \in \cC(p_\cC \to q_\cC)$, and $y_2 \in \cC(q_\cC \to r_\cC)$, we will denote the image under $\bfA$ after tensoring with the identity of the source object using small shaded squares with one strand coming out of the top, e.g.,
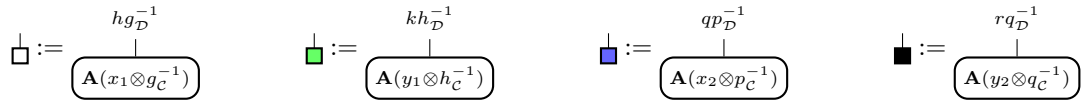
$$
\tikzmath{
\draw (0,0) -- (0,.3);
\filldraw[fill=\gColor, thick] (-.1,-.1) rectangle (.1,.1);
}
:=
\tikzmath{
\node (Ab) at (0,.8) {$\scriptstyle hg^{-1}_\cD$};
\node[draw,rectangle, thick, rounded corners=5pt] (Ax) at (0,0) {$\scriptstyle \bfA(x_1 \xz g_\cC^{-1})$};
\draw (Ax) to[in=-90,out=90] (Ab);
}
\qquad\qquad
\tikzmath{
\draw (0,0) -- (0,.3);
\filldraw[fill=\hColor, thick] (-.1,-.1) rectangle (.1,.1);
}
:=
\tikzmath{
\node (Ab) at (0,.8) {$\scriptstyle kh^{-1}_\cD$};
\node[draw,rectangle, thick, rounded corners=5pt] (Ax) at (0,0) {$\scriptstyle \bfA(y_1 \xz h_\cC^{-1})$};
\draw (Ax) to[in=-90,out=90] (Ab);
}
\qquad\qquad
\tikzmath{
\draw (0,0) -- (0,.3);
\filldraw[fill=\kColor, thick] (-.1,-.1) rectangle (.1,.1);
}
:=
\tikzmath{
\node (Ab) at (0,.8) {$\scriptstyle qp^{-1}_\cD$};
\node[draw,rectangle, thick, rounded corners=5pt] (Ax) at (0,0) {$\scriptstyle \bfA(x_2 \xz p_\cC^{-1})$};
\draw (Ax) to[in=-90,out=90] (Ab);
}
\qquad\qquad
\tikzmath{
\draw (0,0) -- (0,.3);
\RectangleMorphism{(0,0)}{black}
}
:=
\tikzmath{
\node (Ab) at (0,.8) {$\scriptstyle rq^{-1}_\cD$};
\node[draw,rectangle, thick, rounded corners=5pt] (Ax) at (0,0) {$\scriptstyle \bfA(y_2 \xz q^{-1}_\cC)$};
\draw (Ax) to[in=-90,out=90] (Ab);
}
$$
We denote the $G$-actions $F^\fD_g$ and $F^\fD_h$ as in Construction \ref{const:GCrossedBraiding} by a red strand underneath the 1-morphism in $\cC$, where red corresponds to $g$ and green corresponds to $h$.
We denote $\bfA$ applied to the $G$-actions $F^\fC_g$ and $F^\fC$ by outlining the shaded square with red or green respectively, e.g.,
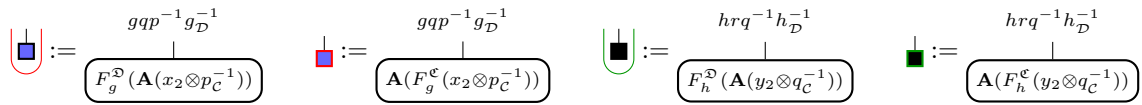
$$
\tikzmath{
\draw (0,0) -- (0,.3);
\draw[red] (-.2,.3) -- (-.2,-.1) arc (-180:0:.2cm) -- (.2,.3);
\RectangleMorphism{(0,0)}{\kColor}
}
:=
\tikzmath{
\node (Ab) at (0,.8) {$\scriptstyle gqp^{-1}g^{-1}_\cD$};
\node[draw,rectangle, thick, rounded corners=5pt] (Ax) at (0,0) {$\scriptstyle F^\fD_g(\bfA(x_2 \xz p_\cC^{-1}))$};
\draw (Ax) to[in=-90,out=90] (Ab);
}
\qquad
\tikzmath{
\draw (0,0) -- (0,.3);
\RedRectangleMorphism{(0,0)}{\kColor}
}
:=
\tikzmath{
\node (Ab) at (0,.8) {$\scriptstyle gqp^{-1}g^{-1}_\cD$};
\node[draw,rectangle, thick, rounded corners=5pt] (Ax) at (0,0) {$\scriptstyle \bfA(F^\fC_g(x_2 \xz p_\cC^{-1}))$};
\draw (Ax) to[in=-90,out=90] (Ab);
}
\qquad
\tikzmath{
\draw (0,0) -- (0,.3);
\draw[DarkGreen] (-.2,.3) -- (-.2,-.1) arc (-180:0:.2cm) -- (.2,.3);
\RectangleMorphism{(0,0)}{black}
}
:=
\tikzmath{
\node (Ab) at (0,.8) {$\scriptstyle hrq^{-1}h^{-1}_\cD$};
\node[draw,rectangle, thick, rounded corners=5pt] (Ax) at (0,0) {$\scriptstyle F^\fD_h(\bfA(y_2 \xz q_\cC^{-1}))$};
\draw (Ax) to[in=-90,out=90] (Ab);
}
\qquad
\tikzmath{
\draw (0,0) -- (0,.3);
\GreenRectangleMorphism{(0,0)}{black}
}
:=
\tikzmath{
\node (Ab) at (0,.8) {$\scriptstyle hrq^{-1}h^{-1}_\cD$};
\node[draw,rectangle, thick, rounded corners=5pt] (Ax) at (0,0) {$\scriptstyle \bfA(F^\fC_h(y_2 \xz q_\cC^{-1}))$};
\draw (Ax) to[in=-90,out=90] (Ab);
}
$$
We use a similar convention for the $\xz$ composite of 1-cells as in Notation \ref{nota:ShadedBoxes}.
For example, if $x\in \fC_g=\cC(1_\cC \to g_\cC)$ and $y\in \fC_h = \cC(1_\cC \to h_\cC)$, we write
$$
\tikzmath{
\draw (0,0) -- (0,.3);
\filldraw[fill=\gColor, thick] (-.1,-.1) rectangle (.1,.1);
}
:=
\tikzmath{
\node (Ab) at (0,.8) {$\scriptstyle hg^{-1}_\cD$};
\node[draw,rectangle, thick, rounded corners=5pt] (Ax) at (0,0) {$\scriptstyle \bfA(x)$};
\draw (Ax) to[in=-90,out=90] (Ab);
}
\qquad\qquad
\tikzmath{
\draw (0,0) -- (0,.3);
\filldraw[fill=\kColor, thick] (-.1,-.1) rectangle (.1,.1);
}
:=
\tikzmath{
\node (Ab) at (0,.8) {$\scriptstyle qp^{-1}_\cD$};
\node[draw,rectangle, thick, rounded corners=5pt] (Ax) at (0,0) {$\scriptstyle \bfA(y)$};
\draw (Ax) to[in=-90,out=90] (Ab);
}
\qquad\qquad
\tikzmath{
\draw (-.2,.2) -- (-.2,.5);
\draw (.2,-.2) -- (.2,.5);
\filldraw[fill=\gColor, thick] (-.3,.1) rectangle (-.1,.3);
\filldraw[fill=\kColor, thick] (.1,-.3) rectangle (.3,-.1);
}
:=
\tikzmath{
\node (Ag) at (-.4,.8) {$\scriptstyle g_\cC$};
\node (Ah) at (.4,.8) {$\scriptstyle h_\cC$};
\node[draw,rectangle, thick, rounded corners=5pt] (Ax) at (0,0) {$\scriptstyle A(x)\xz A(y)$};
\draw (-.4,.23) to[in=-90,out=90] (Ag);
\draw (.4,.23) to[in=-90,out=90] (Ah);
}
\qquad\qquad
\tikzmath{
\draw (.1,-.1) -- (.1,.4);
\draw (-.1,.1) -- (-.1,.4);
\filldraw[fill=\gColor, thick] (-.2,0) rectangle (0,.2);
\filldraw[fill=\kColor, thick] (0,-.2) rectangle (.2,0);
}
:=
\tikzmath{
\node (Ab) at (0,.8) {$\scriptstyle g_\cC\xz h_\cC$};
\node[draw,rectangle, thick, rounded corners=5pt] (Ax) at (0,0) {$\scriptstyle A(x\xz y)$};
\draw[double] (Ax) to[in=-90,out=90] (Ab);
}\,.
$$
This means that by Construction \ref{const:GCrossedBraiding} of the $G$-crossed braiding in $\fC$, we have that
$$
\tikzmath{
\coordinate (Ab) at (.1,.1);
\coordinate (i1) at (-.1,.3);
\draw (Ab) to[in=-90,out=90] (.1,.6);
\draw (i1) to[in=-90,out=90] (-.1,.6);
\RectangleMorphism{(i1)}{white}
\RedRectangleMorphism{(Ab)}{black}
}
\xrightarrow{\bfA\left(\beta^{hg^{-1}, grq^{-1}g^{-1}}_{x_1\xz g_\cC^{-1}, F_g^\fC(y_2\xz q_\cC^{-1})}\right)}
\tikzmath{
\coordinate (Ab) at (-.1,.3);
\coordinate (i1) at (.1,.1);
\draw (Ab) to[in=-90,out=90] (-.1,.6);
\draw (i1) to[in=-90,out=90] (.1,.6);
\RectangleMorphism{(i1)}{white}
\GreenRectangleMorphism{(Ab)}{black}
}
\,.
$$
\end{nota}

\begin{proof}[Proof of Lem.~\ref{lem:PreimageUnderlying2Functor}: $(A,A^1,A^2): \cC \to \cD$ is a 2-functor]
\item[\ref{Functor:A.associative}]
For $x\in \cC(g_\cC \to h_\cC)$, $y\in \cC(h_\cC \to k_\cC)$, and $z\in \cC(k_\cC \to \ell_\cC)$, using the nudging convention \eqref{eq:Nudging}, the square for $A^2$ is exactly
$$
\begin{tikzcd}[column sep=7em]
\underbrace{\bfA(z \xz k^{-1}_\cC) \xz \bfA(y \xz h^{-1}_\cC) \xz \bfA(x \xz g^{-1}_\cC) \xz g_\cD}_{A(z)\xo A(y) \xo A(x) }
\arrow[r, "\id_{\bfA(z\xz k^{-1}_\cC)}\xz \bfA^2 \xz \id_{g_\cD}"]
\arrow[d, "\bfA^2 \xz \id_{\bfA(x \xz g^{-1}_\cC)} \xz \id_{g_\cD}"]
&
\underbrace{\bfA(z \xz k^{-1}_\cC) \xz \bfA(y \xz h^{-1}_\cC \xz x \xz g^{-1}_\cC) \xz g_\cD}_{A(z)\xo A(y \xo x) }
\arrow[d, "\bfA^2 \xz \id_{g_\cD}"]
\\
\underbrace{\bfA(z \xz k^{-1}_\cC \xz y \xz h^{-1}_\cC) \xz \bfA(x \xz g^{-1}_\cC) \xz g_\cD}_{A(z\xo y) \xo A(x) }
\arrow[r, "\bfA^2\xz \id_{g_\cD}"]
&
\underbrace{\bfA((z \xz k^{-1}_\cC) \xz (y \xz h^{-1}_\cC) \xz (x \xz g^{-1}_\cC)) \xz g_\cD}_{A(z\xo y \xo x) }
\end{tikzcd}
$$
which commutes by strictness of $\fC, \fD$ and associativity of $\bfA^2$.

\item[\ref{Functor:A.unital}]
For $x\in \cC(g_\cC \to h_\cC)$, using the nudging convention \eqref{eq:Nudging}, the lower triangle for $A^1$ and $A^2$ is exactly
$$
\begin{tikzcd}[column sep=6em]
\underbrace{e_\cD \xz \bfA(x \xz g^{-1}_\cC) \xz g_\cD}_{A(x)}
\arrow[d, swap, "\bfA^1_e \xz \id_{\bfA(x \xz g^{-1}_\cC)} \xz \id_{g_\cD}"]
\arrow[dr, "\id_x"]
\\
\underbrace{\bfA(e_\cC) \xz \bfA(x \xz g^{-1}_\cC) \xz g_\cD}_{A(\id_h) \xo A(x)}
\arrow[r, "\bfA^2\xz \id_{g_\cD}"]
&
\underbrace{\bfA(e_\cC \xz  x \xz g^{-1}_\cC) \xz g_\cD}_{A(x) }
\end{tikzcd}
$$
which commutes by unitality of $\bfA^1, \bfA^2$.
The other triangle is similar.
\end{proof}

\begin{proof}[Proof of Thm.~\ref{thm:Preimage3Functor}: $(A, \mu^A, \iota^A)\in \TriCat_G^{\st}(\cC \to \cD)$]
\item[\ref{Functor:mu.natural}]
Each component in the definition of $\mu^A_{y,x}$ is natural in $x$ and $y$.

\item[\ref{Functor:mu.monoidal}] 
For $g,h\in G$, $x_1\in \cC(g_\cC \to h_\cC)$, $x_2 \in \cC(h_\cC \to k_\cC)$,
$y_1 \in \cC(p_\cC \to q_\cC)$, and $y_2 \in \cC(q_\cC \to r_\cC)$, we use the following shorthand as in Notation \ref{nota:ShadedBoxes}:
$$
\tikzmath{
\draw (0,0) -- (0,.3);
\filldraw[fill=\gColor, thick] (-.1,-.1) rectangle (.1,.1);
}
:=
\tikzmath{
\node (Ab) at (0,.8) {$\scriptstyle hg^{-1}_\cD$};
\node[draw,rectangle, thick, rounded corners=5pt] (Ax) at (0,0) {$\scriptstyle \bfA(x_1 \xz g_\cC^{-1})$};
\draw (Ax) to[in=-90,out=90] (Ab);
}
\qquad\qquad
\tikzmath{
\draw (0,0) -- (0,.3);
\filldraw[fill=\hColor, thick] (-.1,-.1) rectangle (.1,.1);
}
:=
\tikzmath{
\node (Ab) at (0,.8) {$\scriptstyle kh^{-1}_\cD$};
\node[draw,rectangle, thick, rounded corners=5pt] (Ax) at (0,0) {$\scriptstyle \bfA(y_1 \xz h_\cC^{-1})$};
\draw (Ax) to[in=-90,out=90] (Ab);
}
\qquad\qquad
\tikzmath{
\draw (0,0) -- (0,.3);
\filldraw[fill=\kColor, thick] (-.1,-.1) rectangle (.1,.1);
}
:=
\tikzmath{
\node (Ab) at (0,.8) {$\scriptstyle qp^{-1}_\cD$};
\node[draw,rectangle, thick, rounded corners=5pt] (Ax) at (0,0) {$\scriptstyle \bfA(x_2 \xz p_\cC^{-1})$};
\draw (Ax) to[in=-90,out=90] (Ab);
}
\qquad\qquad
\tikzmath{
\draw (0,0) -- (0,.3);
\RectangleMorphism{(0,0)}{black}
}
:=
\tikzmath{
\node (Ab) at (0,.8) {$\scriptstyle rq^{-1}_\cD$};
\node[draw,rectangle, thick, rounded corners=5pt] (Ax) at (0,0) {$\scriptstyle \bfA(y_2 \xz q^{-1}_\cC)$};
\draw (Ax) to[in=-90,out=90] (Ab);
}
$$
Observe that by the definition of $A$ from $\bfA$ and the nudging convention \eqref{eq:Nudging}, we have
\begin{align*}
\tikzmath{
\node[draw,rectangle, thick, rounded corners=5pt] (x1) at (-.5,-.3) {$\scriptstyle A(x_1)$};
\node[draw,rectangle, thick, rounded corners=5pt] (x2) at (.5,-.9) {$\scriptstyle A(x_2)$};
\node[draw,rectangle, thick, rounded corners=5pt] (y1) at (-.5,.9) {$\scriptstyle A(y_1)$};
\node[draw,rectangle, thick, rounded corners=5pt] (y2) at (.5,.3) {$\scriptstyle A(y_2)$};
\draw (-.5,-1.5) to[in=-90,out=90] (x1);
\draw (x1) to[in=-90,out=90] (y1);
\draw (y1) to[in=-90,out=90] (-.5,1.5);
\draw (.5,-1.5) to[in=-90,out=90] (x2);
\draw (x2) to[in=-90,out=90] (y2);
\draw (y2) to[in=-90,out=90] (.5,1.5);
}
=
\scriptstyle
\bfA(y_1\xz h_\cC^{-1}) 
\xz 
F^\fD_h(\bfA(y_2 \xz q_\cC^{-1}))
\xz 
\bfA(x_1 \xz g_\cC^{-1})
\xz
F^\fD_g(\bfA(x_2 \xz p_\cC^{-1}))
\xz gp_\cD
&=
\tikzmath{
\coordinate (y1) at (0,0);
\coordinate (y2) at (.5,-.3);
\coordinate (x1) at (1,-.6);
\coordinate (x2) at (1.5,-.9);
\draw (y1) -- (0,.3);
\draw (y2) -- (.5,.3);
\draw (x1) -- (1,.3);
\draw (x2) -- (1.5,.3);
\draw[DarkGreen] (.3,.3) -- (.3,-.4) node[left] {$\scriptstyle h$} arc (-180:0:.2cm) -- (.7,.3);
\draw[red] (1.3,.3) -- (1.3,-1) node[left] {$\scriptstyle g$} arc (-180:0:.2cm) -- (1.7,.3);
\draw (2,-1.2) node [below] {$\scriptstyle gp_\cD$} -- (2,.3);
\RectangleMorphism{(y1)}{\hColor}
\RectangleMorphism{(y2)}{black}
\RectangleMorphism{(x1)}{white}
\RectangleMorphism{(x2)}{\kColor}
}
\\
\tikzmath{
\node[draw,rectangle, thick, rounded corners=5pt] (A) at (0,0) {$\scriptstyle A((y_1 \xz y_2)\xo (x_1\xz x_2))$};
\draw (0,-1) to[in=-90,out=90] (A);
\draw (A) to[in=-90,out=90] (0,1);
}
=
\scriptstyle
\bfA(
(y_1\xz h_\cC^{-1})
\xz 
F^\fC_h(y_2 \xz q_\cC^{-1})
\xz 
(x_1 \xz g_\cC^{-1})
\xz
F^\fC_g(x_2 \xz p_\cC^{-1}))
\xz gp_\cD
&=
\tikzmath{
\coordinate (y1) at (0,0);
\coordinate (y2) at (.2,-.2);
\coordinate (x1) at (.4,-.4);
\coordinate (x2) at (.6,-.6);
\draw (y1) -- (0,.3);
\draw (y2) -- (.2,.3);
\draw (x1) -- (.4,.3);
\draw (x2) -- (.6,.3);
\draw (.9,-.9) node [below] {$\scriptstyle gp_\cD$} -- (.9,.3);
\RectangleMorphism{(y1)}{\hColor}
\GreenRectangleMorphism{(y2)}{black}
\RectangleMorphism{(x1)}{white}
\RedRectangleMorphism{(x2)}{\kColor}
}
\end{align*}

Going around the outside of the diagram below corresponds to \ref{Functor:mu.monoidal}, except we leave off the extra $gp_\cD$ strand on the right hand side of each string diagram.
$$
\begin{tikzpicture}[baseline= (a).base]\node[scale=.8] (a) at (0,0){\begin{tikzcd}
\tikzmath{
\coordinate (y1) at (0,0);
\coordinate (y2) at (.5,-.3);
\coordinate (x1) at (1,-.6);
\coordinate (x2) at (1.5,-.9);
\draw (y1) -- (0,.3);
\draw (y2) -- (.5,.3);
\draw (x1) -- (1,.3);
\draw (x2) -- (1.5,.3);
\draw[DarkGreen] (.3,.3) -- (.3,-.4) node[left] {$\scriptstyle h$} arc (-180:0:.2cm) -- (.7,.3);
\draw[red] (1.3,.3) -- (1.3,-1) node[left] {$\scriptstyle g$} arc (-180:0:.2cm) -- (1.7,.3);
\RectangleMorphism{(y1)}{\hColor}
\RectangleMorphism{(y2)}{black}
\RectangleMorphism{(x1)}{white}
\RectangleMorphism{(x2)}{\kColor}
}
\arrow[dr, equals]
\arrow[d, Rightarrow, "\phi^{-1}=\beta^{-1}"]
\arrow[rrr, Rightarrow, "\bfa_h"]
&\mbox{}
\arrow[d, phantom, "\text{\scriptsize \eqref{eq:GCrossedFunctor-Gamma1}}"]
&&
\tikzmath{
\coordinate (y1) at (0,0);
\coordinate (y2) at (.5,-.3);
\coordinate (x1) at (1,-.6);
\coordinate (x2) at (1.5,-.9);
\draw (y1) -- (0,.3);
\draw (y2) -- (.5,.3);
\draw (x1) -- (1,.3);
\draw (x2) -- (1.5,.3);
\draw[red] (1.3,.3) -- (1.3,-1) node[left] {$\scriptstyle g$} arc (-180:0:.2cm) -- (1.7,.3);
\RectangleMorphism{(y1)}{\hColor}
\GreenRectangleMorphism{(y2)}{black}
\RectangleMorphism{(x1)}{white}
\RectangleMorphism{(x2)}{\kColor}
}
\arrow[dd, Rightarrow, "\bfA^2"]
\arrow[r, Rightarrow, "\bfA^2"]
&
\tikzmath{
\coordinate (y1) at (0,0);
\coordinate (y2) at (.2,-.2);
\coordinate (x1) at (.7,-.5);
\coordinate (x2) at (1.2,-.8);
\draw (y1) -- (0,.3);
\draw (y2) -- (.2,.3);
\draw (x1) -- (.7,.3);
\draw (x2) -- (1.2,.3);
\draw[red] (1,.3) -- (1,-.9) node[left] {$\scriptstyle g$} arc (-180:0:.2cm) -- (1.4,.3);
\RectangleMorphism{(y1)}{\hColor}
\GreenRectangleMorphism{(y2)}{black}
\RectangleMorphism{(x1)}{white}
\RectangleMorphism{(x2)}{\kColor}
}
\arrow[ddr, Rightarrow, "\bfa_g"]
\arrow[dd, Rightarrow, "\bfA^2"]
\\
\tikzmath{
\coordinate (y1) at (0,0);
\coordinate (y2) at (1,-.6);
\coordinate (x1) at (.5,-.3);
\coordinate (x2) at (1.5,-.9);
\draw (y1) -- (0,.3);
\draw (y2) -- (1,.3);
\draw (x1) -- (.5,.3);
\draw (x2) -- (1.5,.3);
\draw[red] (.8,.3) -- (.8,-.7) node[left] {$\scriptstyle g$} arc (-180:0:.2cm) -- (1.2,.3);
\draw[red] (1.3,.3) -- (1.3,-1) node[left] {$\scriptstyle g$} arc (-180:0:.2cm) -- (1.7,.3);
\RectangleMorphism{(y1)}{\hColor}
\RectangleMorphism{(y2)}{black}
\RectangleMorphism{(x1)}{white}
\RectangleMorphism{(x2)}{\kColor}
}
\arrow[dr,Rightarrow, "\bfa_g"]
\arrow[d, equals]
&
\tikzmath{
\coordinate (y1) at (0,0);
\coordinate (y2) at (.5,-.3);
\coordinate (x1) at (1,-.6);
\coordinate (x2) at (1.5,-.9);
\draw (y1) -- (0,.3);
\draw (y2) -- (.5,.3);
\draw (x1) -- (1,.3);
\draw (x2) -- (1.5,.3);
\draw[DarkGreen] (.2,.3) -- (.2,-.4) node[left] {$\scriptstyle h$, \textcolor{red}{$\scriptstyle g^{-1}, g$}} arc (-180:0:.3cm) -- (.8,.3);
\draw[red] (.25,.3) -- (.25,-.4) arc (-180:0:.25cm) -- (.75,.3);
\draw[red] (.3,.3) -- (.3,-.4) arc (-180:0:.2cm) -- (.7,.3);
\draw[red] (1.3,.3) -- (1.3,-1) node[left] {$\scriptstyle g$} arc (-180:0:.2cm) -- (1.7,.3);
\RectangleMorphism{(y1)}{\hColor}
\RectangleMorphism{(y2)}{black}
\RectangleMorphism{(x1)}{white}
\RectangleMorphism{(x2)}{\kColor}
}
\arrow[r, Rightarrow, "\bfa_g"]
&
\tikzmath{
\coordinate (y1) at (0,0);
\coordinate (y2) at (.5,-.3);
\coordinate (x1) at (1,-.6);
\coordinate (x2) at (1.5,-.9);
\draw (y1) -- (0,.3);
\draw (y2) -- (.5,.3);
\draw (x1) -- (1,.3);
\draw (x2) -- (1.5,.3);
\draw[DarkGreen] (.25,.3) -- (.25,-.4) node[left] {$\scriptstyle h$, \textcolor{red}{$\scriptstyle g^{-1}$}} arc (-180:0:.25cm) -- (.75,.3);
\draw[red] (.3,.3) -- (.3,-.4) arc (-180:0:.2cm) -- (.7,.3);
\draw[red] (1.3,.3) -- (1.3,-1) node[left] {$\scriptstyle g$} arc (-180:0:.2cm) -- (1.7,.3);
\RectangleMorphism{(y1)}{\hColor}
\RedRectangleMorphism{(y2)}{black}
\RectangleMorphism{(x1)}{white}
\RectangleMorphism{(x2)}{\kColor}
}
\arrow[ur, Rightarrow, "\bfa_{hg^{-1}}"]
\arrow[dr, phantom, "\text{\scriptsize \eqref{eq:GCrossedFunctor-Gamma2}}"]
\\
\tikzmath{
\coordinate (y1) at (0,0);
\coordinate (y2) at (1,-.6);
\coordinate (x1) at (.5,-.3);
\coordinate (x2) at (1.3,-.9);
\draw (y1) -- (0,.3);
\draw (y2) -- (1,.3);
\draw (x1) -- (.5,.3);
\draw (x2) -- (1.3,.3);
\draw[red] (.8,.3) -- (.8,-.9) node[left] {$\scriptstyle g$} arc (-180:0:.4cm) -- (1.6,.3);
\RectangleMorphism{(y1)}{\hColor}
\RectangleMorphism{(y2)}{black}
\RectangleMorphism{(x1)}{white}
\RectangleMorphism{(x2)}{\kColor}
}
\arrow[d, Rightarrow, "\bfA^2"]
\arrow[dr, phantom, "\text{\scriptsize $\bfa_g$ monoidal}"]
&
\tikzmath{
\coordinate (y1) at (0,0);
\coordinate (y2) at (1,-.6);
\coordinate (x1) at (.5,-.3);
\coordinate (x2) at (1.5,-.9);
\draw (y1) -- (0,.3);
\draw (y2) -- (1,.3);
\draw (x1) -- (.5,.3);
\draw (x2) -- (1.5,.3);
\draw[red] (1.3,.3) -- (1.3,-1) node[left] {$\scriptstyle g$} arc (-180:0:.2cm) -- (1.7,.3);
\RectangleMorphism{(y1)}{\hColor}
\RedRectangleMorphism{(y2)}{black}
\RectangleMorphism{(x1)}{white}
\RectangleMorphism{(x2)}{\kColor}
}
\arrow[ur, Rightarrow, "\beta"]
\arrow[r, Rightarrow, "\bfA^2"]
\arrow[d, Rightarrow, "\bfa_g"]
&
\tikzmath{
\coordinate (y1) at (0,0);
\coordinate (y2) at (.7,-.5);
\coordinate (x1) at (.5,-.3);
\coordinate (x2) at (1.2,-.8);
\draw (y1) -- (0,.3);
\draw (y2) -- (.7,.3);
\draw (x1) -- (.5,.3);
\draw (x2) -- (1.2,.3);
\draw[red] (1,.3) -- (1,-.9) node[left] {$\scriptstyle g$} arc (-180:0:.2cm) -- (1.4,.3);
\RectangleMorphism{(y1)}{\hColor}
\RedRectangleMorphism{(y2)}{black}
\RectangleMorphism{(x1)}{white}
\RectangleMorphism{(x2)}{\kColor}
}
\arrow[r, Rightarrow, "\bfA(\beta)"]
\arrow[d, Rightarrow, "\bfa_g"]
&
\tikzmath{
\coordinate (y1) at (0,0);
\coordinate (y2) at (.5,-.3);
\coordinate (x1) at (.7,-.5);
\coordinate (x2) at (1.2,-.8);
\draw (y1) -- (0,.3);
\draw (y2) -- (.5,.3);
\draw (x1) -- (.7,.3);
\draw (x2) -- (1.2,.3);
\draw[red] (1,.3) -- (1,-.9) node[left] {$\scriptstyle g$} arc (-180:0:.2cm) -- (1.4,.3);
\RectangleMorphism{(y1)}{\hColor}
\GreenRectangleMorphism{(y2)}{black}
\RectangleMorphism{(x1)}{white}
\RectangleMorphism{(x2)}{\kColor}
}
\arrow[r, Rightarrow, "\bfA^2"]
\arrow[d, Rightarrow, "\bfa_g"]
&
\tikzmath{
\coordinate (y1) at (0,0);
\coordinate (y2) at (.2,-.2);
\coordinate (x1) at (.4,-.4);
\coordinate (x2) at (.9,-.7);
\draw (y1) -- (0,.3);
\draw (y2) -- (.2,.3);
\draw (x1) -- (.4,.3);
\draw (x2) -- (.9,.3);
\draw[red] (.7,.3) -- (.7,-.8) node[left] {$\scriptstyle g$} arc (-180:0:.2cm) -- (1.1,.3);
\RectangleMorphism{(y1)}{\hColor}
\GreenRectangleMorphism{(y2)}{black}
\RectangleMorphism{(x1)}{white}
\RectangleMorphism{(x2)}{\kColor}
}
\arrow[dd, Rightarrow, "\bfa_g"]
&
\tikzmath{
\coordinate (y1) at (0,0);
\coordinate (y2) at (.2,-.2);
\coordinate (x1) at (.7,-.5);
\coordinate (x2) at (1.2,-.8);
\draw (y1) -- (0,.3);
\draw (y2) -- (.2,.3);
\draw (x1) -- (.7,.3);
\draw (x2) -- (1.2,.3);
\RectangleMorphism{(y1)}{\hColor}
\GreenRectangleMorphism{(y2)}{black}
\RectangleMorphism{(x1)}{white}
\RedRectangleMorphism{(x2)}{\kColor}
}
\arrow[ddl, Rightarrow, "\bfA^2"]
\arrow[dd, Rightarrow, "\bfA^2"]
\\
\tikzmath{
\coordinate (y1) at (0,0);
\coordinate (y2) at (.8,-.6);
\coordinate (x1) at (.2,-.2);
\coordinate (x2) at (1,-.8);
\draw (y1) -- (0,.3);
\draw (y2) -- (.8,.3);
\draw (x1) -- (.2,.3);
\draw (x2) -- (1,.3);
\draw[red] (.6,.3) -- (.6,-.9) node[left] {$\scriptstyle g$} arc (-180:0:.3cm) -- (1.2,.3);
\RectangleMorphism{(y1)}{\hColor}
\RectangleMorphism{(y2)}{black}
\RectangleMorphism{(x1)}{white}
\RectangleMorphism{(x2)}{\kColor}
}
\arrow[d, Rightarrow, "\bfa_g"]
&
\tikzmath{
\coordinate (y1) at (0,0);
\coordinate (y2) at (1,-.6);
\coordinate (x1) at (.5,-.3);
\coordinate (x2) at (1.5,-.9);
\draw (y1) -- (0,.3);
\draw (y2) -- (1,.3);
\draw (x1) -- (.5,.3);
\draw (x2) -- (1.5,.3);
\RectangleMorphism{(y1)}{\hColor}
\RedRectangleMorphism{(y2)}{black}
\RectangleMorphism{(x1)}{white}
\RedRectangleMorphism{(x2)}{\kColor}
}
\arrow[d, Rightarrow, "\bfA^2"]
\arrow[dl, Rightarrow, swap, near start, "\bfA^2\xz \bfA^2"]
\arrow[r, Rightarrow, "\bfA^2"]
&
\tikzmath{
\coordinate (y1) at (0,0);
\coordinate (y2) at (.7,-.5);
\coordinate (x1) at (.5,-.3);
\coordinate (x2) at (1.2,-.8);
\draw (y1) -- (0,.3);
\draw (y2) -- (.7,.3);
\draw (x1) -- (.5,.3);
\draw (x2) -- (1.2,.3);
\RectangleMorphism{(y1)}{\hColor}
\RedRectangleMorphism{(y2)}{black}
\RectangleMorphism{(x1)}{white}
\RedRectangleMorphism{(x2)}{\kColor}
}
\arrow[d, Rightarrow, "\bfA^2"]
%\arrow[dl, Rightarrow, "\bfA^2"]
\arrow[r, Rightarrow, "\bfA(\beta)"]
&
\tikzmath{
\coordinate (y1) at (0,0);
\coordinate (y2) at (.5,-.3);
\coordinate (x1) at (.7,-.5);
\coordinate (x2) at (1.2,-.8);
\draw (y1) -- (0,.3);
\draw (y2) -- (.5,.3);
\draw (x1) -- (.7,.3);
\draw (x2) -- (1.2,.3);
\RectangleMorphism{(y1)}{\hColor}
\GreenRectangleMorphism{(y2)}{black}
\RectangleMorphism{(x1)}{white}
\RedRectangleMorphism{(x2)}{\kColor}
}
\arrow[dr, Rightarrow, "\bfA^2"]
\\
\tikzmath{
\coordinate (y1) at (0,0);
\coordinate (y2) at (.8,-.6);
\coordinate (x1) at (.2,-.2);
\coordinate (x2) at (1,-.8);
\draw (y1) -- (0,.3);
\draw (y2) -- (.8,.3);
\draw (x1) -- (.2,.3);
\draw (x2) -- (1,.3);
\RectangleMorphism{(y1)}{\hColor}
\RedRectangleMorphism{(y2)}{black}
\RectangleMorphism{(x1)}{white}
\RedRectangleMorphism{(x2)}{\kColor}
}
\arrow[d, Rightarrow, "\bfA^2"]
&
\tikzmath{
\coordinate (y1) at (0,0);
\coordinate (y2) at (.7,-.5);
\coordinate (x1) at (.2,-.2);
\coordinate (x2) at (1.2,-.8);
\draw (y1) -- (0,.3);
\draw (y2) -- (.7,.3);
\draw (x1) -- (.2,.3);
\draw (x2) -- (1.2,.3);
\RectangleMorphism{(y1)}{\hColor}
\RedRectangleMorphism{(y2)}{black}
\RectangleMorphism{(x1)}{white}
\RedRectangleMorphism{(x2)}{\kColor}
}
\arrow[l, Rightarrow, "\bfA^2"]
\arrow[r, Rightarrow, "\bfA^2"]
&
\tikzmath{
\coordinate (y1) at (0,0);
\coordinate (y2) at (.4,-.4);
\coordinate (x1) at (.2,-.2);
\coordinate (x2) at (.9,-.7);
\draw (y1) -- (0,.3);
\draw (y2) -- (.4,.3);
\draw (x1) -- (.2,.3);
\draw (x2) -- (.9,.3);
\RectangleMorphism{(y1)}{\hColor}
\RedRectangleMorphism{(y2)}{black}
\RectangleMorphism{(x1)}{white}
\RedRectangleMorphism{(x2)}{\kColor}
}
\arrow[dll, Rightarrow, "\bfA^2"]
\arrow[rr, Rightarrow, "\bfA(\id\xz \beta)"]
&
&
\tikzmath{
\coordinate (y1) at (0,0);
\coordinate (y2) at (.2,-.2);
\coordinate (x1) at (.4,-.4);
\coordinate (x2) at (.9,-.7);
\draw (y1) -- (0,.3);
\draw (y2) -- (.2,.3);
\draw (x1) -- (.4,.3);
\draw (x2) -- (.9,.3);
\RectangleMorphism{(y1)}{\hColor}
\GreenRectangleMorphism{(y2)}{black}
\RectangleMorphism{(x1)}{white}
\RedRectangleMorphism{(x2)}{\kColor}
}
\arrow[dr, Rightarrow, "\bfA^2"]
&
\tikzmath{
\coordinate (y1) at (0,0);
\coordinate (y2) at (.2,-.2);
\coordinate (x1) at (.7,-.5);
\coordinate (x2) at (.9,-.7);
\draw (y1) -- (0,.3);
\draw (y2) -- (.2,.3);
\draw (x1) -- (.7,.3);
\draw (x2) -- (.9,.3);
\RectangleMorphism{(y1)}{\hColor}
\GreenRectangleMorphism{(y2)}{black}
\RectangleMorphism{(x1)}{white}
\RedRectangleMorphism{(x2)}{\kColor}
}
\arrow[d, Rightarrow, "\bfA^2"]
\\
\tikzmath{
\coordinate (y1) at (0,0);
\coordinate (y2) at (.4,-.4);
\coordinate (x1) at (.2,-.2);
\coordinate (x2) at (.6,-.6);
\draw (y1) -- (0,.3);
\draw (y2) -- (.4,.3);
\draw (x1) -- (.2,.3);
\draw (x2) -- (.6,.3);
\RectangleMorphism{(y1)}{\hColor}
\RedRectangleMorphism{(y2)}{black}
\RectangleMorphism{(x1)}{white}
\RedRectangleMorphism{(x2)}{\kColor}
}
\arrow[rrrrr, Rightarrow, "\bfA(\phi)=\bfA(\beta)"]
&&&&&
\tikzmath{
\coordinate (y1) at (0,0);
\coordinate (y2) at (.2,-.2);
\coordinate (x1) at (.4,-.4);
\coordinate (x2) at (.6,-.6);
\draw (y1) -- (0,.3);
\draw (y2) -- (.2,.3);
\draw (x1) -- (.4,.3);
\draw (x2) -- (.6,.3);
\RectangleMorphism{(y1)}{\hColor}
\GreenRectangleMorphism{(y2)}{black}
\RectangleMorphism{(x1)}{white}
\RedRectangleMorphism{(x2)}{\kColor}
}
\end{tikzcd}};\end{tikzpicture}
$$
The faces without labels above commute either by naturality or by associativity of $\bfA^2$.

\item[\ref{Functor:mu.unital}]
This follows since each $F_g^\fD$ is strictly unital, and thus for all $g\in G$,
$$
\bfA^1_e = \bfA^2_{\id_e, \id_e} \xt (A^1_e \xz A^1_e)=\bfA^2_{\id_e, F^\fD_g(\id_e)}\xt (A^1_e, F^\fD_g(A^1_e).
$$

\item[\ref{Functor:iota}]
This part is automatic as $\iota^A_1 = A^1_e$.

\item[\ref{Functor:omega}] 
This follows by monoidality of $\bfa_g$ and associativity of $\bfA^2$.
We omit the full proof as it is much easier than \ref{Functor:mu.monoidal} above.

\item[\ref{Functor:lr}]
This reduces to unitality of $\bfA^1$ and $\bfA^2$, i.e., for all $x\in \cC(g_\cC \to h_\cC)$,
$$
\bfA^2_{\id_e, x\xz g_\cC^{-1}} \xt (\bfA^1 \xz \id_{\bfA(x\xz g_\cC^{-1})}) = \id_{\bfA(x\xz g_\cC^{-1})}.
$$
The other unitality axiom is similar.

\item[\ref{Functor:PentagonCoherence}]
Every map is the identity map.
\item[\ref{Functor:TriangleCoherence}]
Every map is the identity map.
\end{proof}

\begin{proof}[Proof of Thm.~\ref{thm:FullyFaithfulOn2Morphisms}: the map $\TriCat_G^{\st}(A\Rightarrow B) \to G\CrsBrd(\bfA \Rightarrow \bfB)$ is bijective]
\mbox{}

Suppose that $\eta , \zeta \in \TriCat_G^{\st}(A \Rightarrow B)$ satisfy $\eta_x = \zeta_x$ for every $x\in \cC(1_\cC \to g_\cC)$ for all $g\in G$.
Since $\eta, \zeta$ are 2-morphisms in $\TriCat_G^{\st}$, we have $\eta_* = e_\cD = \zeta_*$, $\eta_g = \id_{g_\cD} = \zeta_g$ for all $g\in G$.
For an arbitrary $y\in \cC(g_\cC \to h_\cC)$, we have $\eta_{y\xz g^{-1}_\cC} = \zeta_{y\xz g^{-1}_\cC}$. 
By \ref{Transformation:eta^2} for $\eta:A\Rightarrow B$, the following diagram commutes:
\begin{equation}
\label{eq:EtaDeterminedByMapsFrom1}
\begin{tikzpicture}[baseline= (a).base]\node[scale=1] (a) at (0,0){\begin{tikzcd}
\tikzmath{
\node (hg) at (0,1.5) {$\scriptstyle hg^{-1}_\cD$};
\node (g1) at (1.5,-1) {$\scriptstyle g_\cD$};
\node (g2) at (1.5,1.5) {$\scriptstyle g_\cD$};
\node[draw,rectangle, thick, rounded corners=5pt] (Axg) at (0,.5) {$\scriptstyle A(y\xz g^{-1}_\cC)$};
\node[draw,rectangle, thick, rounded corners=5pt] (Ag) at (1.5,0) {$\scriptstyle A(g_\cC)$};
\draw (Axg) to[in=-90,out=90] (hg);
\draw (g1) to[in=-90,out=90] (Ag);
\draw (Ag) to[in=-90,out=90] (g2);
}
\arrow[r, Rightarrow, "\mu^A"]
\arrow[d, Rightarrow, "\eta_{y\xz g_\cC^{-1}}\xz \eta_{\id_{g_\cC}}"]
&
\tikzmath{
\node (g1) at (0,-1) {$\scriptstyle g_\cD$};
\node (g2) at (0,1) {$\scriptstyle g_\cD$};
\node[draw,rectangle, thick, rounded corners=5pt] (Ax) at (0,0) {$\scriptstyle A(y)$};
\draw (g1) to[in=-90,out=90] (Ax);
\draw (Ax) to[in=-90,out=90] (g2);
}
\arrow[d, Rightarrow, "\eta_y"]
\\
\tikzmath{
\node (hg) at (0,1.5) {$\scriptstyle hg^{-1}_\cD$};
\node (g1) at (1.5,-1) {$\scriptstyle g_\cD$};
\node (g2) at (1.5,1.5) {$\scriptstyle g_\cD$};
\node[draw,rectangle, thick, rounded corners=5pt] (Axg) at (0,.5) {$\scriptstyle B(y\xz g^{-1}_\cC)$};
\node[draw,rectangle, thick, rounded corners=5pt] (Ag) at (1.5,0) {$\scriptstyle B(g_\cC)$};
\draw (Axg) to[in=-90,out=90] (hg);
\draw (g1) to[in=-90,out=90] (Ag);
\draw (Ag) to[in=-90,out=90] (g2);
}
\arrow[r, Rightarrow, "\mu^B"]
&
\tikzmath{
\node (g1) at (0,-1) {$\scriptstyle g_\cD$};
\node (g2) at (0,1) {$\scriptstyle g_\cD$};
\node[draw,rectangle, thick, rounded corners=5pt] (Ax) at (0,0) {$\scriptstyle B(y)$};
\draw (g1) to[in=-90,out=90] (Ax);
\draw (Ax) to[in=-90,out=90] (g2);
}
\end{tikzcd}};\end{tikzpicture}
\end{equation}
as does a similar diagram for $\zeta$ replacing $\eta$.
Since $\eta_{y \xz g_\cC^{-1}} = \zeta_{y \xz g_\cC^{-1}}$ by assumption,
$\eta_{\id_g} = B_g^1 \xt (A^1_g)^{-1} = \zeta_{\id_g}$ by \ref{Transformation:eta_c.unital}, 
and $\mu^A, \mu^B$ are invertible 2-cells, we conclude that $\eta_x = \zeta_x$.

Now suppose $h: \bfA \Rightarrow \bfB$ is a $G$-monoidal natural transformation.
We define $\eta: A \Rightarrow B$ by $\eta_* = e_\cD$, $\eta_g = \id_{g_\cD}$ for all $g\in G$, and for $y\in \cC(g_\cC \to h_\cC)$, we use \eqref{eq:EtaDeterminedByMapsFrom1} above to define
$$
\eta_y :=
\mu^B_{y \xz g_\cC^{-1}, g_\cC}
\xt
(h_{y\xz g^{-1}_\cC} \xz (B^1_g \xt (A^1_g)^{-1}))
\xt
(\mu^A_{y \xz g_\cC^{-1}, g_\cC})^{-1}.
$$
By construction, provided $\eta$ is a transformation $\eta \mapsto h$.
It remains to verify that $\eta: A \Rightarrow B$ is indeed a transformation.
We prove one of the coherences below, and we give a hint as how to proceed for the other coherences.

\item[\ref{Transformation:eta_c.natural}]
Every composite step in the definition of $\eta$ is natural.

\item[\ref{Transformation:eta_c.monoidal}]
This follows from functoriality of 1-cell composition $\xo$ together with
the fact that $h$ is monoidal, and two instances each (one for each of $A$ and $B$) of \ref{Functor:A.unital}, \eqref{eq:Functor:mu.monoidal.fewer}, and \ref{Functor:omega}.

\item[\ref{Transformation:eta_c.unital}]
This follows using 2 instances of \ref{Functor:mu.unital} (one for each of $A$ and $B$) together with the fact that 
$$
h_{\id_e} = \bfB^1 \xz (\bfA^1)^{-1}=B^1_e \xt (A^1_e)^{-1}
$$
which is unitality for a monoidal natural transformation.

\item[\ref{Transformation:eta^1}] 
This condition is automatically satisfied.

\item[\ref{Transformation:eta^2}] 
For $x\in \cC(g_\cC \to h_\cC)$ and $y\in \cC(k_\cC \to \ell_\cC)$, we use the following shorthand as in Notation \ref{nota:ShadedBoxes} and Notation \ref{nota:ShadedBoxesGCrossed}:
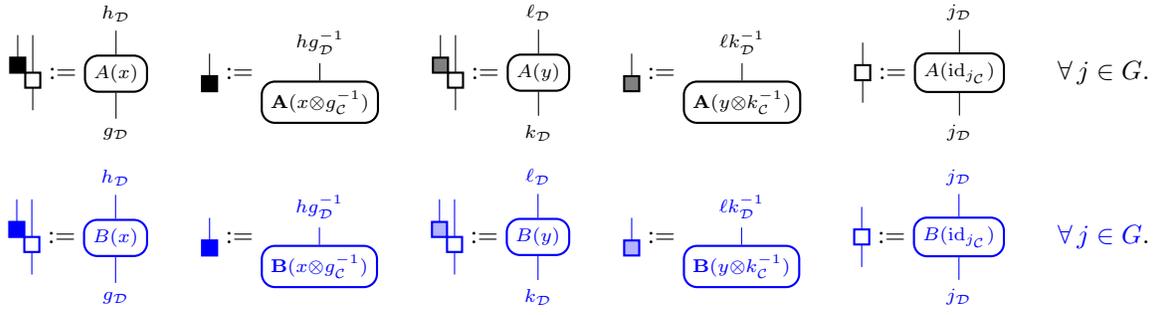
\begin{align*}
\tikzmath{
\draw (-.1,.1) -- (-.1,.5);
\draw (.1,-.5) -- (.1,.5);
\filldraw[fill=black, thick] (-.2,0) rectangle (0,.2);
\filldraw[fill=white, thick] (0,-.2) rectangle (.2,0);
}
:=
\tikzmath{
\node (g) at (0,-.8) {$\scriptstyle g_\cD$};
\node (h) at (0,.8) {$\scriptstyle h_\cD$};
\node[draw,rectangle, thick, rounded corners=5pt] (Ax) at (0,0) {$\scriptstyle A(x)$};
\draw (g) to[in=-90,out=90] (Ax);
\draw (Ax) to[in=-90,out=90] (h);
}
\qquad
\tikzmath{
\draw (0,0) -- (0,.4);
\filldraw[fill=black, thick] (-.1,-.1) rectangle (.1,.1);
}
:=
\tikzmath{
\node (Ab) at (0,.8) {$\scriptstyle hg^{-1}_\cD$};
\node[draw,rectangle, thick, rounded corners=5pt] (Ax) at (0,0) {$\scriptstyle \bfA(x \xz g_\cC^{-1})$};
\draw (Ax) to[in=-90,out=90] (Ab);
}
\qquad
\tikzmath{
\draw (-.1,.1) -- (-.1,.5);
\draw (.1,-.5) -- (.1,.5);
\filldraw[fill=gray, thick] (-.2,0) rectangle (0,.2);
\filldraw[fill=white, thick] (0,-.2) rectangle (.2,0);
}
:=
\tikzmath{
\node (g) at (0,-.8) {$\scriptstyle k_\cD$};
\node (h) at (0,.8) {$\scriptstyle \ell_\cD$};
\node[draw,rectangle, thick, rounded corners=5pt] (Ax) at (0,0) {$\scriptstyle A(y)$};
\draw (g) to[in=-90,out=90] (Ax);
\draw (Ax) to[in=-90,out=90] (h);
}
\qquad
\tikzmath{
\draw (0,0) -- (0,.4);
\filldraw[fill=gray, thick] (-.1,-.1) rectangle (.1,.1);
}
:=
\tikzmath{
\node (Ab) at (0,.8) {$\scriptstyle \ell k^{-1}_\cD$};
\node[draw,rectangle, thick, rounded corners=5pt] (Ax) at (0,0) {$\scriptstyle \bfA(y \xz k_\cC^{-1})$};
\draw (Ax) to[in=-90,out=90] (Ab);
}
\qquad
\tikzmath{
\draw (0,-.4) -- (0,.4);
\filldraw[fill=white, thick] (-.1,-.1) rectangle (.1,.1);
}
:=
\tikzmath{
\node (g) at (0,-.8) {$\scriptstyle j_\cD$};
\node (h) at (0,.8) {$\scriptstyle j_\cD$};
\node[draw,rectangle, thick, rounded corners=5pt] (Ax) at (0,0) {$\scriptstyle A(\id_{j_\cC})$};
\draw (g) to[in=-90,out=90] (Ax);
\draw (Ax) to[in=-90,out=90] (h);
}
\qquad
\forall\, j \in G.
\\
\tikzmath{
\draw[blue] (-.1,.1) -- (-.1,.5);
\draw[blue] (.1,-.5) -- (.1,.5);
\filldraw[fill=blue, draw=blue, thick] (-.2,0) rectangle (0,.2);
\filldraw[fill=white, draw=blue, thick] (0,-.2) rectangle (.2,0);
}
:=
\tikzmath{
\node[blue] (g) at (0,-.8) {$\scriptstyle g_\cD$};
\node[blue] (h) at (0,.8) {$\scriptstyle h_\cD$};
\node[draw, blue, rectangle, thick, rounded corners=5pt] (Ax) at (0,0) {$\scriptstyle B(x)$};
\draw[blue] (g) to[in=-90,out=90] (Ax);
\draw[blue] (Ax) to[in=-90,out=90] (h);
}
\qquad
\tikzmath{
\draw[blue] (0,0) -- (0,.4);
\filldraw[fill=blue, draw=blue, thick] (-.1,-.1) rectangle (.1,.1);
}
:=
\tikzmath{
\node[blue] (Ab) at (0,.8) {$\scriptstyle hg^{-1}_\cD$};
\node[draw, blue, rectangle, thick, rounded corners=5pt] (Ax) at (0,0) {$\scriptstyle \bfB(x \xz g_\cC^{-1})$};
\draw[blue] (Ax) to[in=-90,out=90] (Ab);
}
\qquad
\tikzmath{
\draw[blue] (-.1,.1) -- (-.1,.5);
\draw[blue] (.1,-.5) -- (.1,.5);
\filldraw[fill=\BlueGray, draw=blue, thick] (-.2,0) rectangle (0,.2);
\filldraw[fill=white, draw=blue, thick] (0,-.2) rectangle (.2,0);
}
:=
\tikzmath{
\node[blue] (g) at (0,-.8) {$\scriptstyle k_\cD$};
\node[blue] (h) at (0,.8) {$\scriptstyle \ell_\cD$};
\node[draw, blue, rectangle, thick, rounded corners=5pt] (Ax) at (0,0) {$\scriptstyle B(y)$};
\draw[blue] (g) to[in=-90,out=90] (Ax);
\draw[blue] (Ax) to[in=-90,out=90] (h);
}
\qquad
\tikzmath{
\draw[blue] (0,0) -- (0,.4);
\filldraw[fill=\BlueGray, draw=blue, thick] (-.1,-.1) rectangle (.1,.1);
}
:=
\tikzmath{
\node[blue] (Ab) at (0,.8) {$\scriptstyle \ell k^{-1}_\cD$};
\node[draw, blue, rectangle, thick, rounded corners=5pt] (Ax) at (0,0) {$\scriptstyle \bfB(y \xz k_\cC^{-1})$};
\draw[blue] (Ax) to[in=-90,out=90] (Ab);
}
\qquad
\tikzmath{
\draw[blue] (0,-.4) -- (0,.4);
\filldraw[fill=white, draw=blue, thick] (-.1,-.1) rectangle (.1,.1);
}
:=
\tikzmath{
\node[blue] (g) at (0,-.8) {$\scriptstyle j_\cD$};
\node[blue] (h) at (0,.8) {$\scriptstyle j_\cD$};
\node[draw, blue, rectangle, thick, rounded corners=5pt] (Ax) at (0,0) {$\scriptstyle B(\id_{j_\cC})$};
\draw[blue] (g) to[in=-90,out=90] (Ax);
\draw[blue] (Ax) to[in=-90,out=90] (h);
}
\qquad
\textcolor{blue}{\forall\, j \in G}.
\end{align*}
Suppose $x \in \cC(g_\cC \to h_\cC)$ and $y\in \cC(k_\cC \to \ell_\cC)$.
We begin with the following observation that the following diagram commutes:
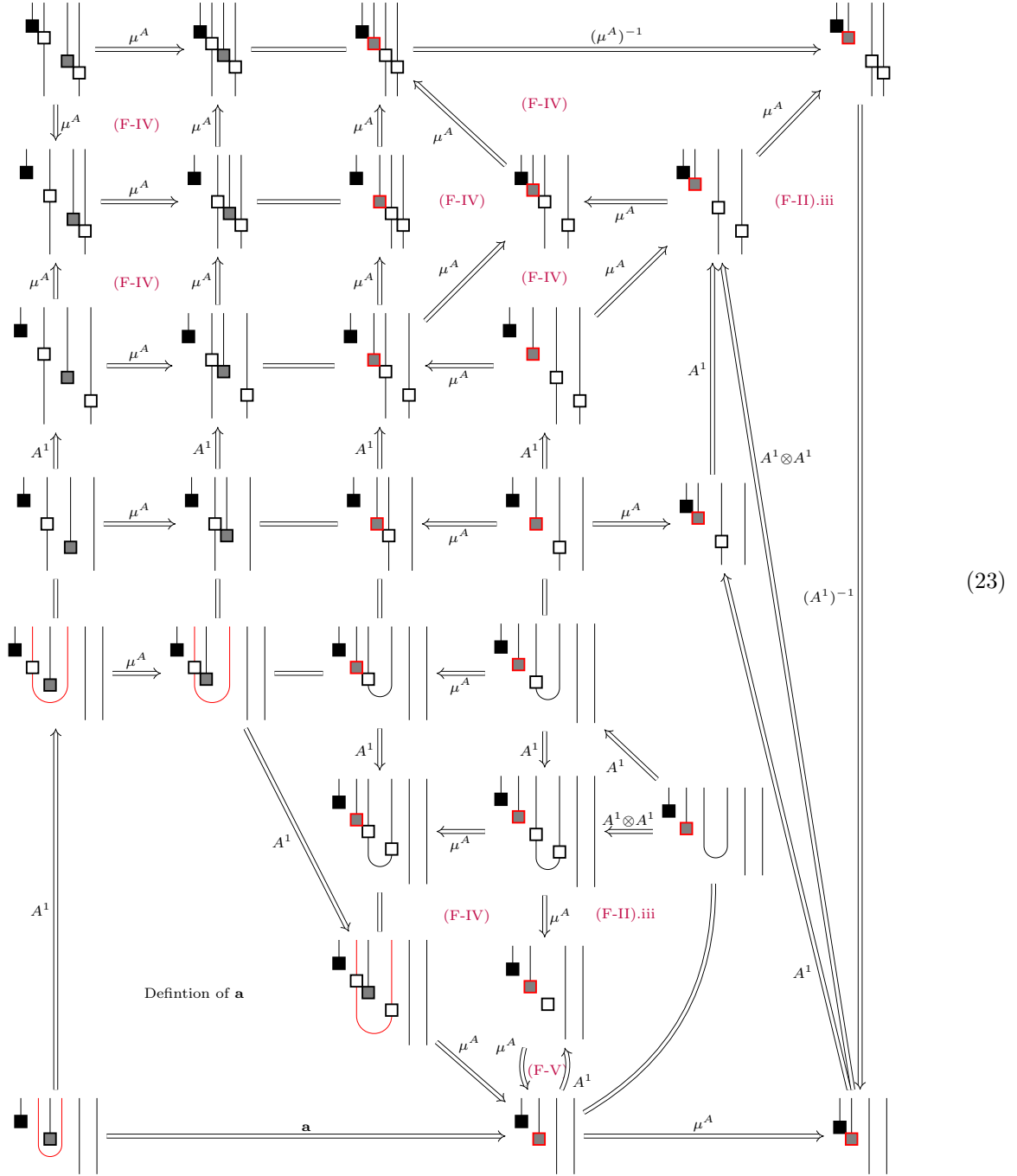
\begin{equation}
\label{eq:HelperDiagram}
\begin{tikzpicture}[baseline= (a).base]\node[scale=.9] (a) at (0,0){\begin{tikzcd}
\tikzmath{
\coordinate (x) at (-.4,.4);
\coordinate (ig) at (-.2,.2);
\coordinate (y) at (.2,-.2);
\coordinate (ik) at (.4,-.4);
\draw (-.4,.4) -- (-.4,.8);
\draw (-.2,-.8) -- (-.2,.8);
\draw (.4,-.8) -- (.4,.8);
\draw (.2,-.2) -- (.2,.8);
\RectangleMorphism{(x)}{black}
\RectangleMorphism{(ig)}{white}
\RectangleMorphism{(y)}{gray}
\RectangleMorphism{(ik)}{white}
}
\arrow[dr, phantom, "\text{\scriptsize \ref{Functor:omega}}"]
\arrow[r, Rightarrow, "\mu^A"]
\arrow[d, Rightarrow, "\mu^A"]
&
\tikzmath{
\coordinate (x) at (-.3,.3);
\coordinate (ig) at (-.1,.1);
\coordinate (y) at (.1,-.1);
\coordinate (ik) at (.3,-.3);
\draw (-.3,.3) -- (-.3,.8);
\draw (-.1,-.8) -- (-.1,.8);
\draw (.3,-.8) -- (.3,.8);
\draw (.1,-.2) -- (.1,.8);
\RectangleMorphism{(x)}{black}
\RectangleMorphism{(ig)}{white}
\RectangleMorphism{(y)}{gray}
\RectangleMorphism{(ik)}{white}
}
\arrow[r, equals]
&
\tikzmath{
\coordinate (x) at (-.3,.3);
\coordinate (y) at (-.1,.1);
\coordinate (ig) at (.1,-.1);
\coordinate (ik) at (.3,-.3);
\draw (-.3,.3) -- (-.3,.8);
\draw (-.1,.1) -- (-.1,.8);
\draw (.3,-.8) -- (.3,.8);
\draw (.1,-.8) -- (.1,.8);
\RectangleMorphism{(x)}{black}
\RectangleMorphism{(ig)}{white}
\RedRectangleMorphism{(y)}{gray}
\RectangleMorphism{(ik)}{white}
}
\arrow[rrr, Rightarrow, "(\mu^A)^{-1}"]
&
\mbox{}
\arrow[d, phantom, "\text{\scriptsize \ref{Functor:omega}}"]
&&
\tikzmath{
\coordinate (x) at (-.4,.4);
\coordinate (y) at (-.2,.2);
\coordinate (ig) at (.2,-.2);
\coordinate (ik) at (.4,-.4);
\draw (-.4,.4) -- (-.4,.8);
\draw (-.2,.2) -- (-.2,.8);
\draw (.4,-.8) -- (.4,.8);
\draw (.2,-.8) -- (.2,.8);
\RectangleMorphism{(x)}{black}
\RectangleMorphism{(ig)}{white}
\RedRectangleMorphism{(y)}{gray}
\RectangleMorphism{(ik)}{white}
}
\arrow[ddddddd, Rightarrow, swap, "(A^1)^{-1}"]
\\
\tikzmath{
\coordinate (x) at (-.6,.6);
\coordinate (ig) at (-.2,.2);
\coordinate (y) at (.2,-.2);
\coordinate (ik) at (.4,-.4);
\draw (-.6,.6) -- (-.6,1);
\draw (-.2,-.8) -- (-.2,1);
\draw (.4,-.8) -- (.4,1);
\draw (.2,-.2) -- (.2,1);
\RectangleMorphism{(x)}{black}
\RectangleMorphism{(ig)}{white}
\RectangleMorphism{(y)}{gray}
\RectangleMorphism{(ik)}{white}
}
\arrow[dr, phantom, "\text{\scriptsize \ref{Functor:omega}}"]
\arrow[r, Rightarrow, "\mu^A"]
&
\tikzmath{
\coordinate (x) at (-.4,.4);
\coordinate (ig) at (0,0);
\coordinate (y) at (.2,-.2);
\coordinate (ik) at (.4,-.4);
\draw (-.4,.4) -- (-.4,.8);
\draw (0,-.8) -- (0,.8);
\draw (.4,-.8) -- (.4,.8);
\draw (.2,-.2) -- (.2,.8);
\RectangleMorphism{(x)}{black}
\RectangleMorphism{(ig)}{white}
\RectangleMorphism{(y)}{gray}
\RectangleMorphism{(ik)}{white}
}
\arrow[u, Rightarrow, "\mu^A"]
\arrow[r, equals]
&
\tikzmath{
\coordinate (x) at (-.4,.4);
\coordinate (y) at (0,0);
\coordinate (ig) at (.2,-.2);
\coordinate (ik) at (.4,-.4);
\draw (-.4,.4) -- (-.4,.8);
\draw (0,0) -- (0,.8);
\draw (.4,-.8) -- (.4,.8);
\draw (.2,-.8) -- (.2,.8);
\RectangleMorphism{(x)}{black}
\RectangleMorphism{(ig)}{white}
\RedRectangleMorphism{(y)}{gray}
\RectangleMorphism{(ik)}{white}
}
\arrow[u, Rightarrow, "\mu^A"]
\arrow[r, phantom, "\text{\scriptsize \ref{Functor:omega}}"]
&
\tikzmath{
\coordinate (x) at (-.4,.4);
\coordinate (y) at (-.2,.2);
\coordinate (ig) at (0,0);
\coordinate (ik) at (.4,-.4);
\draw (-.4,.4) -- (-.4,.8);
\draw (-.2,.2) -- (-.2,.8);
\draw (.4,-.8) -- (.4,.8);
\draw (0,-.8) -- (0,.8);
\RectangleMorphism{(x)}{black}
\RectangleMorphism{(ig)}{white}
\RedRectangleMorphism{(y)}{gray}
\RectangleMorphism{(ik)}{white}
}
\arrow[ul, Rightarrow, "\mu^A"]
\arrow[d, phantom, "\text{\scriptsize \ref{Functor:omega}}"]
&
\tikzmath{
\coordinate (x) at (-.4,.4);
\coordinate (y) at (-.2,.2);
\coordinate (ig) at (.2,-.2);
\coordinate (ik) at (.6,-.6);
\draw (-.4,.4) -- (-.4,.8);
\draw (-.2,.2) -- (-.2,.8);
\draw (.6,-1) -- (.6,.8);
\draw (.2,-1) -- (.2,.8);
\RectangleMorphism{(x)}{black}
\RectangleMorphism{(ig)}{white}
\RedRectangleMorphism{(y)}{gray}
\RectangleMorphism{(ik)}{white}
}
\arrow[l, Rightarrow, "\mu^A"]
\arrow[ur, Rightarrow, "\mu^A"]
\arrow[r, phantom, "\text{\scriptsize \ref{Functor:mu.unital}}"]
&
\mbox{}
\\
\tikzmath{
\coordinate (x) at (-.6,.6);
\coordinate (ig) at (-.2,.2);
\coordinate (y) at (.2,-.2);
\coordinate (ik) at (.6,-.6);
\draw (-.6,.6) -- (-.6,1);
\draw (-.2,-1) -- (-.2,1);
\draw (.6,-1) -- (.6,1);
\draw (.2,-.2) -- (.2,1);
\RectangleMorphism{(x)}{black}
\RectangleMorphism{(ig)}{white}
\RectangleMorphism{(y)}{gray}
\RectangleMorphism{(ik)}{white}
}
\arrow[r, Rightarrow, "\mu^A"]
\arrow[u, Rightarrow, "\mu^A"]
&
\tikzmath{
\coordinate (x) at (-.6,.6);
\coordinate (ig) at (-.2,.2);
\coordinate (y) at (0,0);
\coordinate (ik) at (.4,-.4);
\draw (-.6,.6) -- (-.6,1);
\draw (-.2,-.8) -- (-.2,1);
\draw (.4,-.8) -- (.4,1);
\draw (0,0) -- (0,1);
\RectangleMorphism{(x)}{black}
\RectangleMorphism{(ig)}{white}
\RectangleMorphism{(y)}{gray}
\RectangleMorphism{(ik)}{white}
}
\arrow[r, equals]
\arrow[u, Rightarrow, "\mu^A"]
&
\tikzmath{
\coordinate (x) at (-.6,.6);
\coordinate (y) at (-.2,.2);
\coordinate (ig) at (0,0);
\coordinate (ik) at (.4,-.4);
\draw (-.6,.6) -- (-.6,1);
\draw (-.2,.2) -- (-.2,1);
\draw (.4,-.8) -- (.4,1);
\draw (0,-.8) -- (0,1);
\RectangleMorphism{(x)}{black}
\RectangleMorphism{(ig)}{white}
\RedRectangleMorphism{(y)}{gray}
\RectangleMorphism{(ik)}{white}
}
\arrow[u, Rightarrow, "\mu^A"]
\arrow[ur, Rightarrow, "\mu^A"]
&
\tikzmath{
\coordinate (x) at (-.6,.6);
\coordinate (y) at (-.2,.2);
\coordinate (ig) at (.2,-.2);
\coordinate (ik) at (.6,-.6);
\draw (-.6,.6) -- (-.6,1);
\draw (-.2,.2) -- (-.2,1);
\draw (.6,-1) -- (.6,1);
\draw (.2,-1) -- (.2,1);
\RectangleMorphism{(x)}{black}
\RectangleMorphism{(ig)}{white}
\RedRectangleMorphism{(y)}{gray}
\RectangleMorphism{(ik)}{white}
}
\arrow[l, Rightarrow, "\mu^A"]
\arrow[ur, Rightarrow, "\mu^A"]
\\
\tikzmath{
\coordinate (x) at (-.6,.6);
\coordinate (ig) at (-.2,.2);
\coordinate (y) at (.2,-.2);
\draw (-.6,.6) -- (-.6,1);
\draw (-.2,-.6) -- (-.2,1);
\draw (.6,-.6) -- (.6,1);
\draw (.2,-.2) -- (.2,1);
\RectangleMorphism{(x)}{black}
\RectangleMorphism{(ig)}{white}
\RectangleMorphism{(y)}{gray}
}
\arrow[r, Rightarrow, "\mu^A"]
\arrow[u, Rightarrow, "A^1"]
&
\tikzmath{
\coordinate (x) at (-.6,.6);
\coordinate (ig) at (-.2,.2);
\coordinate (y) at (0,0);
\draw (-.6,.6) -- (-.6,1);
\draw (-.2,-.6) -- (-.2,1);
\draw (.4,-.6) -- (.4,1);
\draw (0,0) -- (0,1);
\RectangleMorphism{(x)}{black}
\RectangleMorphism{(ig)}{white}
\RectangleMorphism{(y)}{gray}
}
\arrow[r, equals]
\arrow[u, Rightarrow, "A^1"]
&
\tikzmath{
\coordinate (x) at (-.6,.6);
\coordinate (y) at (-.2,.2);
\coordinate (ig) at (0,0);
\draw (-.6,.6) -- (-.6,1);
\draw (-.2,.2) -- (-.2,1);
\draw (.4,-.6) -- (.4,1);
\draw (0,-.6) -- (0,1);
\RectangleMorphism{(x)}{black}
\RectangleMorphism{(ig)}{white}
\RedRectangleMorphism{(y)}{gray}
}
\arrow[u, Rightarrow, "A^1"]
&
\tikzmath{
\coordinate (x) at (-.6,.6);
\coordinate (y) at (-.2,.2);
\coordinate (ig) at (.2,-.2);
\draw (-.6,.6) -- (-.6,1);
\draw (-.2,.2) -- (-.2,1);
\draw (.6,-.6) -- (.6,1);
\draw (.2,-.6) -- (.2,1);
\RectangleMorphism{(x)}{black}
\RectangleMorphism{(ig)}{white}
\RedRectangleMorphism{(y)}{gray}
}
\arrow[r, Rightarrow, "\mu^A"]
\arrow[l, Rightarrow, "\mu^A"]
\arrow[u, Rightarrow, "A^1"]
&
\tikzmath{
\coordinate (x) at (-.4,.4);
\coordinate (y) at (-.2,.2);
\coordinate (ig) at (.2,-.2);
\draw (-.4,.4) -- (-.4,.8);
\draw (-.2,.2) -- (-.2,.8);
\draw (.6,-.6) -- (.6,.8);
\draw (.2,-.6) -- (.2,.8);
\RectangleMorphism{(x)}{black}
\RectangleMorphism{(ig)}{white}
\RedRectangleMorphism{(y)}{gray}
}
\arrow[uu, Rightarrow, "A^1"]
\\
\tikzmath{
\coordinate (x) at (-.6,.6);
\coordinate (ig) at (-.3,.3);
\coordinate (y) at (0,0);
\draw (-.6,.6) -- (-.6,1);
\draw[red] (-.3,1) -- (-.3,0) arc (-180:0:.3cm) -- (.3,1);
\draw (0,0) -- (0,1);
\draw (.6,-.6) -- (.6,1);
\draw (.9,-.6) -- (.9,1);
\RectangleMorphism{(x)}{black}
\RectangleMorphism{(ig)}{white}
\RectangleMorphism{(y)}{gray}
}
\arrow[r, Rightarrow, "\mu^A"]
\arrow[u, equals]
&
\tikzmath{
\coordinate (x) at (-.6,.6);
\coordinate (ig) at (-.3,.3);
\coordinate (y) at (-.1,.1);
\draw (-.6,.6) -- (-.6,1);
\draw[red] (-.3,1) -- (-.3,0) arc (-180:0:.3cm) -- (.3,1);
\draw (-.1,.1) -- (-.1,1);
\draw (.6,-.6) -- (.6,1);
\draw (.9,-.6) -- (.9,1);
\RectangleMorphism{(x)}{black}
\RectangleMorphism{(ig)}{white}
\RectangleMorphism{(y)}{gray}
}
\arrow[r, equals]
\arrow[u, equals]
\arrow[ddr, Rightarrow, swap, "A^1"]
&
\tikzmath{
\coordinate (x) at (-.6,.6);
\coordinate (y) at (-.3,.3);
\coordinate (ig) at (-.1,.1);
\draw (-.6,.6) -- (-.6,1);
\draw (-.1,1) -- (-.1,0) arc (-180:0:.2cm) -- (.3,1);
\draw (-.3,.3) -- (-.3,1);
\draw (.6,-.6) -- (.6,1);
\draw (.9,-.6) -- (.9,1);
\RectangleMorphism{(x)}{black}
\RectangleMorphism{(ig)}{white}
\RedRectangleMorphism{(y)}{gray}
}
\arrow[u, equals]
\arrow[d, Rightarrow, swap, "A^1"]
&
\tikzmath{
\coordinate (x) at (-.7,.7);
\coordinate (y) at (-.4,.4);
\coordinate (ig) at (-.1,.1);
\draw (-.7,.7) -- (-.7,1.1);
\draw (-.1,1.1) -- (-.1,0) arc (-180:0:.2cm) -- (.3,1.1);
\draw (-.4,.4) -- (-.4,1.1);
\draw (.6,-.6) -- (.6,1.1);
\draw (.9,-.6) -- (.9,1.1);
\RectangleMorphism{(x)}{black}
\RectangleMorphism{(ig)}{white}
\RedRectangleMorphism{(y)}{gray}
}
\arrow[l, Rightarrow, "\mu^A"]
\arrow[u, equals]
\arrow[d, Rightarrow, swap, "A^1"]
\\
&
&
\tikzmath{
\coordinate (x) at (-.6,.6);
\coordinate (y) at (-.3,.3);
\coordinate (ig) at (-.1,.1);
\coordinate (ig2) at (.3,-.2);
\draw (-.6,.6) -- (-.6,1);
\draw (-.1,1) -- (-.1,-.3) arc (-180:0:.2cm) -- (.3,1);
\draw (-.3,.3) -- (-.3,1);
\draw (.6,-.8) -- (.6,1);
\draw (.9,-.8) -- (.9,1);
\RectangleMorphism{(x)}{black}
\RectangleMorphism{(ig)}{white}
\RectangleMorphism{(ig2)}{white}
\RedRectangleMorphism{(y)}{gray}
}
\arrow[dr, phantom, "\text{\scriptsize \ref{Functor:omega}}"]
&
\tikzmath{
\coordinate (x) at (-.7,.7);
\coordinate (y) at (-.4,.4);
\coordinate (ig) at (-.1,.1);
\coordinate (ig2) at (.3,-.2);
\draw (-.7,.7) -- (-.7,1.1);
\draw (-.1,1.1) -- (-.1,-.3) arc (-180:0:.2cm) -- (.3,1.1);
\draw (-.4,.4) -- (-.4,1.1);
\draw (.6,-.8) -- (.6,1.1);
\draw (.9,-.8) -- (.9,1.1);
\RectangleMorphism{(x)}{black}
\RectangleMorphism{(ig)}{white}
\RectangleMorphism{(ig2)}{white}
\RedRectangleMorphism{(y)}{gray}
}
\arrow[l, Rightarrow, "\mu^A"]
\arrow[d, Rightarrow, "\mu^A"]
&
\tikzmath{
\coordinate (x) at (-.7,.7);
\coordinate (y) at (-.4,.4);
\draw (-.7,.7) -- (-.7,1.1);
\draw (-.1,1.1) -- (-.1,.1) arc (-180:0:.2cm) -- (.3,1.1);
\draw (-.4,.4) -- (-.4,1.1);
\draw (.6,-.4) -- (.6,1.1);
\draw (.9,-.4) -- (.9,1.1);
\RectangleMorphism{(x)}{black}
\RedRectangleMorphism{(y)}{gray}
}
\arrow[ul, Rightarrow, "A^1"]
\arrow[l, Rightarrow, swap, "A^1\xz A^1"]
\\
\mbox{}
\arrow[rr, phantom, "\text{\scriptsize Defintion of $\bfa$}"]
&&
\tikzmath{
\coordinate (x) at (-.6,.6);
\coordinate (ig) at (-.3,.3);
\coordinate (ig2) at (.3,-.2);
\coordinate (y) at (-.1,.1);
\draw (-.6,.6) -- (-.6,1);
\draw[red] (-.3,1) -- (-.3,-.3) arc (-180:0:.3cm) -- (.3,1);
\draw (-.1,.1) -- (-.1,1);
\draw (.6,-.8) -- (.6,1);
\draw (.9,-.8) -- (.9,1);
\RectangleMorphism{(x)}{black}
\RectangleMorphism{(ig)}{white}
\RectangleMorphism{(ig2)}{white}
\RectangleMorphism{(y)}{gray}
}
\arrow[u, equals]
\arrow[dr, Rightarrow, near start, "\mu^A"]
&
\tikzmath{
\coordinate (x) at (-.6,.6);
\coordinate (y) at (-.3,.3);
\coordinate (i) at (0,0);
\draw (-.6,.6) -- (-.6,1);
\draw (-.3,.3) -- (-.3,1);
\draw (.6,-.6) -- (.6,1);
\draw (.3,-.6) -- (.3,1);
\RectangleMorphism{(x)}{black}
\RectangleMorphism{(i)}{white}
\RedRectangleMorphism{(y)}{gray}
}
\arrow[ur, phantom, "\text{\scriptsize \ref{Functor:mu.unital}}"]
\arrow[d, bend right=20, Rightarrow, near start, "\text{\scriptsize \ref{Functor:lr}}", "\mu^A"']
\\
\tikzmath{
\coordinate (x) at (-.5,.4);
\coordinate (y) at (0,.1);
\draw (-.5,.4) -- (-.5,.8);
\draw[red] (-.2,.8) -- (-.2,0) arc (-180:0:.2cm) -- (.2,.8);
\draw (0,0) -- (0,.8);
\draw (.5,-.5) -- (.5,.8);
\draw (.8,-.5) -- (.8,.8);
\RectangleMorphism{(x)}{black}
\RectangleMorphism{(y)}{gray}
}
\arrow[uuu, Rightarrow, "A^1"]
\arrow[rrr, Rightarrow, "\bfa"]
&&&
\tikzmath{
\coordinate (x) at (-.6,.6);
\coordinate (y) at (-.3,.3);
\draw (-.6,.6) -- (-.6,1);
\draw (-.3,.3) -- (-.3,1);
\draw (.3,-.3) -- (.3,1);
\draw (0,-.3) -- (0,1);
\RectangleMorphism{(x)}{black}
\RedRectangleMorphism{(y)}{gray}
}
\arrow[u, bend right=20, Rightarrow, swap, "A^1"]
\arrow[uur, equals, bend right=30]
\arrow[rr, Rightarrow, "\mu^A"]
&&
\tikzmath{
\coordinate (x) at (-.5,.5);
\coordinate (y) at (-.3,.3);
\draw (-.5,.5) -- (-.5,1);
\draw (-.3,.3) -- (-.3,1);
\draw (.3,-.3) -- (.3,1);
\draw (0,-.3) -- (0,1);
\RectangleMorphism{(x)}{black}
\RedRectangleMorphism{(y)}{gray}
}
\arrow[uuuuuul, Rightarrow, swap, near end, "A^1\xz A^1"]
\arrow[uuuul, near start, Rightarrow, "A^1"]
\end{tikzcd}};\end{tikzpicture}
\end{equation}
Observe that \eqref{eq:HelperDiagram} above also holds with $(A,A^1, \mu^A, \bfA, \bfa)$ replaced by $(B,B^1, \mu^B, \bfB, \bfb)$.

Going around the outside of the diagram below corresponds to \ref{Transformation:eta^2}, where we also use the abuse of notation of $h$ for $B^1 \xt (A^1)^{-1}$.
$$
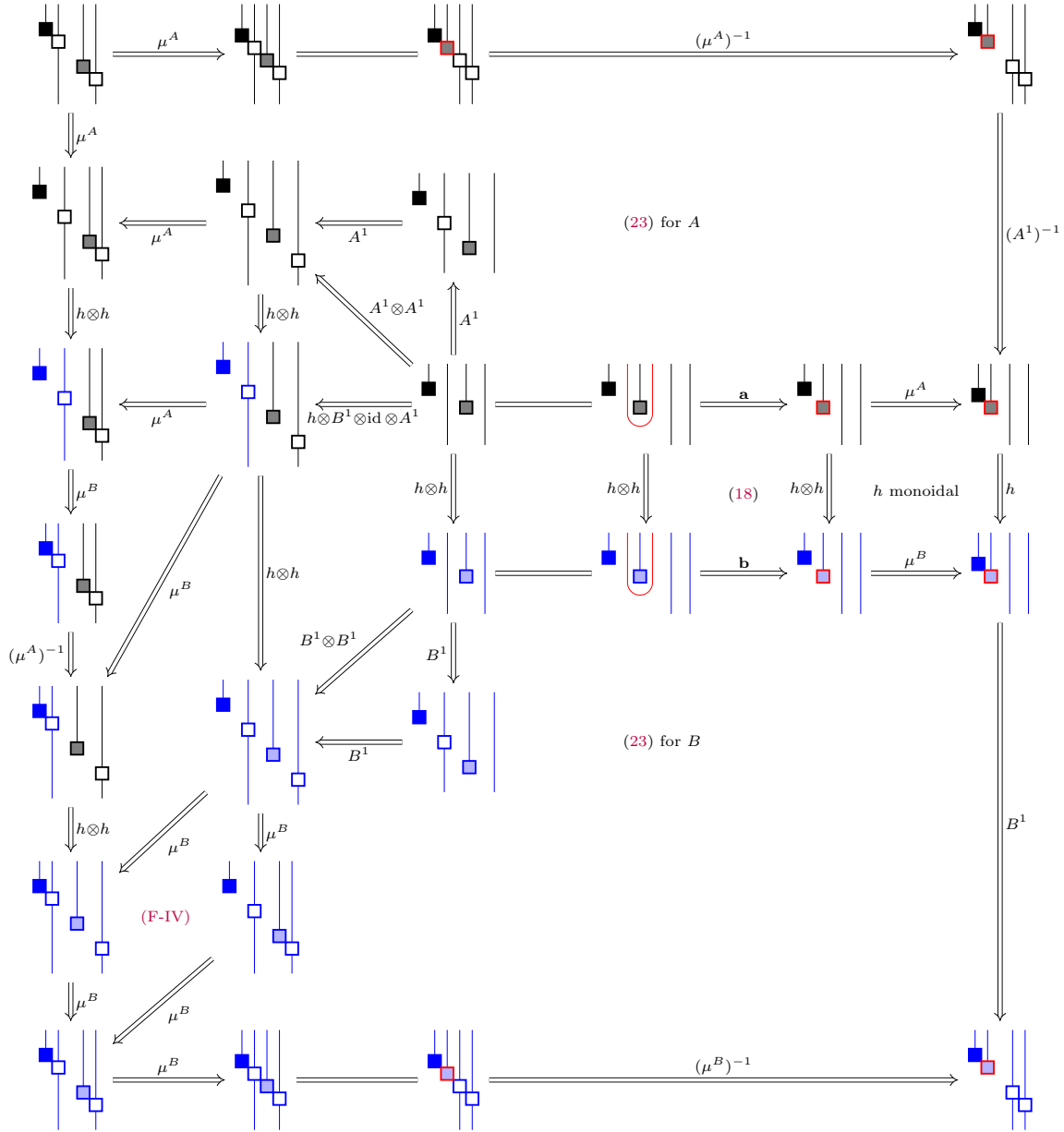
\begin{tikzpicture}[baseline= (a).base]\node[scale=.9] (a) at (0,0){\begin{tikzcd}[column sep=4em]
\tikzmath{
\coordinate (x) at (-.4,.4);
\coordinate (ig) at (-.2,.2);
\coordinate (y) at (.2,-.2);
\coordinate (ik) at (.4,-.4);
\draw (-.4,.4) -- (-.4,.8);
\draw (-.2,-.8) -- (-.2,.8);
\draw (.4,-.8) -- (.4,.8);
\draw (.2,-.2) -- (.2,.8);
\RectangleMorphism{(x)}{black}
\RectangleMorphism{(ig)}{white}
\RectangleMorphism{(y)}{gray}
\RectangleMorphism{(ik)}{white}
}
\arrow[r, Rightarrow, "\mu^A"]
\arrow[d, Rightarrow, "\mu^A"]
&
\tikzmath{
\coordinate (x) at (-.3,.3);
\coordinate (ig) at (-.1,.1);
\coordinate (y) at (.1,-.1);
\coordinate (ik) at (.3,-.3);
\draw (-.3,.3) -- (-.3,.8);
\draw (-.1,-.8) -- (-.1,.8);
\draw (.3,-.8) -- (.3,.8);
\draw (.1,-.2) -- (.1,.8);
\RectangleMorphism{(x)}{black}
\RectangleMorphism{(ig)}{white}
\RectangleMorphism{(y)}{gray}
\RectangleMorphism{(ik)}{white}
}
\arrow[r, equals]
&
\tikzmath{
\coordinate (x) at (-.3,.3);
\coordinate (y) at (-.1,.1);
\coordinate (ig) at (.1,-.1);
\coordinate (ik) at (.3,-.3);
\draw (-.3,.3) -- (-.3,.8);
\draw (-.1,.1) -- (-.1,.8);
\draw (.3,-.8) -- (.3,.8);
\draw (.1,-.8) -- (.1,.8);
\RectangleMorphism{(x)}{black}
\RectangleMorphism{(ig)}{white}
\RedRectangleMorphism{(y)}{gray}
\RectangleMorphism{(ik)}{white}
}
\arrow[rrr, Rightarrow, "(\mu^A)^{-1}"]
&&&
\tikzmath{
\coordinate (x) at (-.4,.4);
\coordinate (y) at (-.2,.2);
\coordinate (ig) at (.2,-.2);
\coordinate (ik) at (.4,-.4);
\draw (-.4,.4) -- (-.4,.8);
\draw (-.2,.2) -- (-.2,.8);
\draw (.4,-.8) -- (.4,.8);
\draw (.2,-.8) -- (.2,.8);
\RectangleMorphism{(x)}{black}
\RectangleMorphism{(ig)}{white}
\RedRectangleMorphism{(y)}{gray}
\RectangleMorphism{(ik)}{white}
}
\arrow[dd, Rightarrow, "(A^1)^{-1}"]
\\
\tikzmath{
\coordinate (x) at (-.6,.6);
\coordinate (ig) at (-.2,.2);
\coordinate (y) at (.2,-.2);
\coordinate (ik) at (.4,-.4);
\draw (-.6,.6) -- (-.6,1);
\draw (-.2,-.8) -- (-.2,1);
\draw (.4,-.8) -- (.4,1);
\draw (.2,-.2) -- (.2,1);
\RectangleMorphism{(x)}{black}
\RectangleMorphism{(ig)}{white}
\RectangleMorphism{(y)}{gray}
\RectangleMorphism{(ik)}{white}
}
\arrow[d, Rightarrow, "h\xz h"]
&
\tikzmath{
\coordinate (x) at (-.6,.6);
\coordinate (ig) at (-.2,.2);
\coordinate (y) at (.2,-.2);
\coordinate (ik) at (.6,-.6);
\draw (-.6,.6) -- (-.6,1);
\draw (-.2,-1) -- (-.2,1);
\draw (.6,-1) -- (.6,1);
\draw (.2,-.2) -- (.2,1);
\RectangleMorphism{(x)}{black}
\RectangleMorphism{(ig)}{white}
\RectangleMorphism{(y)}{gray}
\RectangleMorphism{(ik)}{white}
}
\arrow[l, Rightarrow, "\mu^A"]
\arrow[d, Rightarrow, "h\xz h"]
&
\tikzmath{
\coordinate (x) at (-.6,.6);
\coordinate (ig) at (-.2,.2);
\coordinate (y) at (.2,-.2);
\draw (-.6,.6) -- (-.6,1);
\draw (-.2,-.6) -- (-.2,1);
\draw (.6,-.6) -- (.6,1);
\draw (.2,-.2) -- (.2,1);
\RectangleMorphism{(x)}{black}
\RectangleMorphism{(ig)}{white}
\RectangleMorphism{(y)}{gray}
}
\arrow[rr, phantom, "\text{\scriptsize \eqref{eq:HelperDiagram} for $A$}"]
\arrow[l, Rightarrow, "A^1"]
&&
\mbox{}
\\
\tikzmath{
\coordinate (x) at (-.6,.6);
\coordinate (ig) at (-.2,.2);
\coordinate (y) at (.2,-.2);
\coordinate (ik) at (.4,-.4);
\draw[blue] (-.6,.6) -- (-.6,1);
\draw[blue] (-.2,-.8) -- (-.2,1);
\draw (.4,-.8) -- (.4,1);
\draw (.2,-.2) -- (.2,1);
\BlueRectangleMorphism{(x)}{blue}
\BlueRectangleMorphism{(ig)}{white}
\RectangleMorphism{(y)}{gray}
\RectangleMorphism{(ik)}{white}
}
\arrow[d, Rightarrow, "\mu^B"]
&
\tikzmath{
\coordinate (x) at (-.6,.6);
\coordinate (ig) at (-.2,.2);
\coordinate (y) at (.2,-.2);
\coordinate (ik) at (.6,-.6);
\draw[blue] (-.6,.6) -- (-.6,1);
\draw[blue] (-.2,-1) -- (-.2,1);
\draw (.6,-1) -- (.6,1);
\draw (.2,-.2) -- (.2,1);
\BlueRectangleMorphism{(x)}{blue}
\BlueRectangleMorphism{(ig)}{white}
\RectangleMorphism{(y)}{gray}
\RectangleMorphism{(ik)}{white}
}
\arrow[l, Rightarrow, "\mu^A"]
\arrow[ddl, Rightarrow, "\mu^B"]
\arrow[dd, Rightarrow, "h\xz h"]
&
\tikzmath{
\coordinate (x) at (-.6,.6);
\coordinate (y) at (0,.3);
\draw (-.6,.6) -- (-.6,1);
\draw (-.3,-.3) -- (-.3,1);
\draw (.3,-.3) -- (.3,1);
\draw (0,.3) -- (0,1);
\RectangleMorphism{(x)}{black}
\RectangleMorphism{(y)}{gray}
}
\arrow[l, Rightarrow, "h \xz B^1 \xz \id \xz A^1"]
\arrow[ul, Rightarrow, swap, "A^1\xz A^1"]
\arrow[u, Rightarrow, swap, "A^1"]
\arrow[r, equals]
\arrow[d, Rightarrow, swap, "h\xz h"]
&
\tikzmath{
\coordinate (x) at (-.5,.4);
\coordinate (y) at (0,.1);
\draw (-.5,.4) -- (-.5,.8);
\draw[red] (-.2,.8) -- (-.2,0) arc (-180:0:.2cm) -- (.2,.8);
\draw (0,0) -- (0,.8);
\draw (.5,-.5) -- (.5,.8);
\draw (.8,-.5) -- (.8,.8);
\RectangleMorphism{(x)}{black}
\RectangleMorphism{(y)}{gray}
}
\arrow[r, Rightarrow, "\bfa"]
\arrow[d, Rightarrow, swap, "h\xz h"]
\arrow[dr, phantom, "\text{\scriptsize \eqref{eq:ActionCohrenceForTransformation}}"]
&
\tikzmath{
\coordinate (x) at (-.6,.6);
\coordinate (y) at (-.3,.3);
\draw (-.6,.6) -- (-.6,1);
\draw (-.3,.3) -- (-.3,1);
\draw (.3,-.3) -- (.3,1);
\draw (0,-.3) -- (0,1);
\RectangleMorphism{(x)}{black}
\RedRectangleMorphism{(y)}{gray}
}
\arrow[r, Rightarrow, "\mu^A"]
\arrow[d, Rightarrow, swap, "h\xz h"]
\arrow[dr, phantom, "\text{\scriptsize $h$ monoidal}"]
&
\tikzmath{
\coordinate (x) at (-.5,.5);
\coordinate (y) at (-.3,.3);
\draw (-.5,.5) -- (-.5,1);
\draw (-.3,.3) -- (-.3,1);
\draw (.3,-.3) -- (.3,1);
\draw (0,-.3) -- (0,1);
\RectangleMorphism{(x)}{black}
\RedRectangleMorphism{(y)}{gray}
}
\arrow[d, Rightarrow, "h"]
\\
\tikzmath{
\coordinate (x) at (-.4,.4);
\coordinate (ig) at (-.2,.2);
\coordinate (y) at (.2,-.2);
\coordinate (ik) at (.4,-.4);
\draw[blue] (-.4,.4) -- (-.4,.8);
\draw[blue] (-.2,-.8) -- (-.2,.8);
\draw (.4,-.8) -- (.4,.8);
\draw (.2,-.2) -- (.2,.8);
\BlueRectangleMorphism{(x)}{blue}
\BlueRectangleMorphism{(ig)}{white}
\RectangleMorphism{(y)}{gray}
\RectangleMorphism{(ik)}{white}
}
\arrow[d, Rightarrow, swap, "(\mu^A)^{-1}"]
&&
\tikzmath{
\coordinate (x) at (-.6,.6);
\coordinate (y) at (0,.3);
\draw[blue] (-.6,.6) -- (-.6,1);
\draw[blue] (-.3,-.3) -- (-.3,1);
\draw[blue] (.3,-.3) -- (.3,1);
\draw[blue] (0,.3) -- (0,1);
\BlueRectangleMorphism{(x)}{blue}
\BlueRectangleMorphism{(y)}{\BlueGray}
}
\arrow[dl, Rightarrow, swap, "B^1\xz B^1"]
\arrow[d, Rightarrow, swap, "B^1"]
\arrow[r, equals]
&
\tikzmath{
\coordinate (x) at (-.5,.4);
\coordinate (y) at (0,.1);
\draw[blue] (-.5,.4) -- (-.5,.8);
\draw[red] (-.2,.8) -- (-.2,0) arc (-180:0:.2cm) -- (.2,.8);
\draw[blue] (0,0) -- (0,.8);
\draw[blue] (.5,-.5) -- (.5,.8);
\draw[blue] (.8,-.5) -- (.8,.8);
\BlueRectangleMorphism{(x)}{blue}
\BlueRectangleMorphism{(y)}{\BlueGray}
}
\arrow[r, Rightarrow, "\bfb"]
&
\tikzmath{
\coordinate (x) at (-.6,.6);
\coordinate (y) at (-.3,.3);
\draw[blue] (-.6,.6) -- (-.6,1);
\draw[blue] (-.3,.3) -- (-.3,1);
\draw[blue] (.3,-.3) -- (.3,1);
\draw[blue] (0,-.3) -- (0,1);
\BlueRectangleMorphism{(x)}{blue}
\RedRectangleMorphism{(y)}{\BlueGray}
}
\arrow[r, Rightarrow, "\mu^B"]
&
\tikzmath{
\coordinate (x) at (-.5,.5);
\coordinate (y) at (-.3,.3);
\draw[blue] (-.5,.5) -- (-.5,1);
\draw[blue] (-.3,.3) -- (-.3,1);
\draw[blue] (.3,-.3) -- (.3,1);
\draw[blue] (0,-.3) -- (0,1);
\BlueRectangleMorphism{(x)}{blue}
\RedRectangleMorphism{(y)}{\BlueGray}
}
\arrow[ddd, Rightarrow, "B^1"]
\\
\tikzmath{
\coordinate (x) at (-.4,.4);
\coordinate (ig) at (-.2,.2);
\coordinate (y) at (.2,-.2);
\coordinate (ik) at (.6,-.6);
\draw[blue] (-.4,.4) -- (-.4,.8);
\draw[blue] (-.2,-1) -- (-.2,.8);
\draw (.6,-1) -- (.6,.8);
\draw (.2,-.2) -- (.2,.8);
\BlueRectangleMorphism{(x)}{blue}
\BlueRectangleMorphism{(ig)}{white}
\RectangleMorphism{(y)}{gray}
\RectangleMorphism{(ik)}{white}
}
\arrow[d, Rightarrow, "h\xz h"]
&
\tikzmath{
\coordinate (x) at (-.6,.6);
\coordinate (ig) at (-.2,.2);
\coordinate (y) at (.2,-.2);
\coordinate (ik) at (.6,-.6);
\draw[blue] (-.6,.6) -- (-.6,1);
\draw[blue] (-.2,-1) -- (-.2,1);
\draw[blue] (.6,-1) -- (.6,1);
\draw[blue] (.2,-.2) -- (.2,1);
\BlueRectangleMorphism{(x)}{blue}
\BlueRectangleMorphism{(ig)}{white}
\BlueRectangleMorphism{(y)}{\BlueGray}
\BlueRectangleMorphism{(ik)}{white}
}
\arrow[dl, Rightarrow, "\mu^B"]
\arrow[d, Rightarrow, "\mu^B"]
&
\tikzmath{
\coordinate (x) at (-.6,.6);
\coordinate (ig) at (-.2,.2);
\coordinate (y) at (.2,-.2);
\draw[blue] (-.6,.6) -- (-.6,1);
\draw[blue] (-.2,-.6) -- (-.2,1);
\draw[blue] (.6,-.6) -- (.6,1);
\draw[blue] (.2,-.2) -- (.2,1);
\BlueRectangleMorphism{(x)}{blue}
\BlueRectangleMorphism{(ig)}{white}
\BlueRectangleMorphism{(y)}{\BlueGray}
}
\arrow[rr, phantom, "\text{\scriptsize \eqref{eq:HelperDiagram} for $B$}"]
\arrow[l, Rightarrow, "B^1"]
&&
\mbox{}
\\
\tikzmath{
\coordinate (x) at (-.4,.4);
\coordinate (ig) at (-.2,.2);
\coordinate (y) at (.2,-.2);
\coordinate (ik) at (.6,-.6);
\draw[blue] (-.4,.4) -- (-.4,.8);
\draw[blue] (-.2,-1) -- (-.2,.8);
\draw[blue] (.6,-1) -- (.6,.8);
\draw[blue] (.2,-.2) -- (.2,.8);
\BlueRectangleMorphism{(x)}{blue}
\BlueRectangleMorphism{(ig)}{white}
\BlueRectangleMorphism{(y)}{\BlueGray}
\BlueRectangleMorphism{(ik)}{white}
}
\arrow[d, Rightarrow, "\mu^B"]
\arrow[r, phantom, "\text{\scriptsize \ref{Functor:omega}}"]
&
\tikzmath{
\coordinate (x) at (-.6,.6);
\coordinate (ig) at (-.2,.2);
\coordinate (y) at (.2,-.2);
\coordinate (ik) at (.4,-.4);
\draw[blue] (-.6,.6) -- (-.6,1);
\draw[blue] (-.2,-.8) -- (-.2,1);
\draw[blue] (.4,-.8) -- (.4,1);
\draw[blue] (.2,-.2) -- (.2,1);
\BlueRectangleMorphism{(x)}{blue}
\BlueRectangleMorphism{(ig)}{white}
\BlueRectangleMorphism{(y)}{\BlueGray}
\BlueRectangleMorphism{(ik)}{white}
}
\arrow[dl, Rightarrow, "\mu^B"]
\\
\tikzmath{
\coordinate (x) at (-.4,.4);
\coordinate (ig) at (-.2,.2);
\coordinate (y) at (.2,-.2);
\coordinate (ik) at (.4,-.4);
\draw[blue] (-.4,.4) -- (-.4,.8);
\draw[blue] (-.2,-.8) -- (-.2,.8);
\draw[blue] (.4,-.8) -- (.4,.8);
\draw[blue] (.2,-.2) -- (.2,.8);
\BlueRectangleMorphism{(x)}{blue}
\BlueRectangleMorphism{(ig)}{white}
\BlueRectangleMorphism{(y)}{\BlueGray}
\BlueRectangleMorphism{(ik)}{white}
}
\arrow[r, Rightarrow, "\mu^B"]
&
\tikzmath{
\coordinate (x) at (-.3,.3);
\coordinate (ig) at (-.1,.1);
\coordinate (y) at (.1,-.1);
\coordinate (ik) at (.3,-.3);
\draw[blue] (-.3,.3) -- (-.3,.8);
\draw[blue] (-.1,-.8) -- (-.1,.8);
\draw[blue] (.3,-.8) -- (.3,.8);
\draw[blue] (.1,-.2) -- (.1,.8);
\BlueRectangleMorphism{(x)}{blue}
\BlueRectangleMorphism{(ig)}{white}
\BlueRectangleMorphism{(y)}{\BlueGray}
\BlueRectangleMorphism{(ik)}{white}
}
\arrow[r, equals]
&
\tikzmath{
\coordinate (x) at (-.3,.3);
\coordinate (y) at (-.1,.1);
\coordinate (ig) at (.1,-.1);
\coordinate (ik) at (.3,-.3);
\draw[blue] (-.3,.3) -- (-.3,.8);
\draw[blue] (-.1,.1) -- (-.1,.8);
\draw[blue] (.3,-.8) -- (.3,.8);
\draw[blue] (.1,-.8) -- (.1,.8);
\BlueRectangleMorphism{(x)}{blue}
\BlueRectangleMorphism{(ig)}{white}
\RedRectangleMorphism{(y)}{\BlueGray}
\BlueRectangleMorphism{(ik)}{white}
}
\arrow[rrr, Rightarrow, "(\mu^B)^{-1}"]
&&&
\tikzmath{
\coordinate (x) at (-.4,.4);
\coordinate (y) at (-.2,.2);
\coordinate (ig) at (.2,-.2);
\coordinate (ik) at (.4,-.4);
\draw[blue] (-.4,.4) -- (-.4,.8);
\draw[blue] (-.2,.2) -- (-.2,.8);
\draw[blue] (.4,-.8) -- (.4,.8);
\draw[blue] (.2,-.8) -- (.2,.8);
\BlueRectangleMorphism{(x)}{blue}
\BlueRectangleMorphism{(ig)}{white}
\RedRectangleMorphism{(y)}{\BlueGray}
\BlueRectangleMorphism{(ik)}{white}
}
\end{tikzcd}};\end{tikzpicture}
$$
The faces without labels above commute by functoriality of 1-cell composition $\xo$ or by the shorthand $h= B^1 \xt (A^1)^{-1}$.

\item[\ref{Transformation:AssociatorCoherence}]
Every map is the identity map.

\item[\ref{Transformation:UnitCoherence}]
Every map is the identity map.

\end{proof}

\bibliographystyle{amsalpha}
{\footnotesize{
\bibliography{bibliography}
}}
\end{document}